\documentclass[letterpaper]{amsart}
\pdfoutput=1

\usepackage{amssymb}
\usepackage{amsfonts}
\usepackage{amsmath}
\usepackage{amsthm}
\usepackage{mathtools}
\usepackage{graphicx}
\usepackage{mathrsfs}
\usepackage{dsfont}
\usepackage{amscd}
\usepackage{multirow}

\usepackage[all]{xy}
\normalfont
\usepackage[T1]{fontenc}
\usepackage{calligra}
\usepackage{verbatim}
\usepackage{fourier}
\usepackage[usenames,dvipsnames]{color}
\usepackage[colorlinks=true,linkcolor=Blue,citecolor=Violet]{hyperref}
\usepackage{enumerate}
\usepackage{slashed}
\usepackage{nicematrix}
\usepackage{microtype}
\usepackage{tikz}\usetikzlibrary{snakes,math}
\usepackage{caption}
\usepackage{subcaption}
\usepackage{faktor}

\newcommand{\N}{{\mathds{N}}}
\newcommand{\Z}{{\mathds{Z}}}

\newcommand{\R}{{\mathds{R}}}
\newcommand{\C}{{\mathds{C}}}
\newcommand{\T}{{\mathds{T}}}
\newcommand{\U}{{\mathcal{Z}}}
\newcommand{\D}{{\mathfrak{D}}}
\newcommand{\A}{{\mathfrak{A}}}
\newcommand{\B}{{\mathfrak{B}}}
\newcommand{\M}{{\mathfrak{M}}}

\newcommand{\bigslant}[2]{\faktor{#1}{#2}}

\newcommand{\Lip}[1][L]{{\mathsf{#1}}}

\newcommand{\GH}{{\mathsf{GH}}}

\newcommand{\tunnel}[6]{{ {#1}: {#2} \xleftarrow{#3} {#4} \xrightarrow{#5} {#6} }}

\newcommand{\gradiant}{\operatorname*{grad}}
\newcommand{\metametric}[1]{\operatorname*{\eth_{#1}}}
\newcommand{\dpropinquity}[1]{{\mathsf{\Lambda}^\ast_{#1}}}
\newcommand{\dppropinquity}[1]{{\mathsf{\Lambda}^{\bullet\ast}_{#1}}}

\newcommand{\Kantorovich}[1]{{\mathsf{mk}_{#1}}}

\newcommand{\boundedLipschitz}[1]{{\mathsf{bl}_{#1}}}
\newcommand{\Haus}[1]{{\mathsf{Haus}_{{#1}}\,}}

\newcommand{\StateSpace}{{\mathscr{S}}}
\newcommand{\unital}[1]{{\mathfrak{u}{#1}}}
\newcommand{\MongeKant}{{Mon\-ge-Kan\-to\-ro\-vich metric}}

\newcommand{\pqms}{proper quantum metric space}
\newcommand{\pqpms}{pointed proper quantum metric space}

\newcommand{\qcms}{quantum compact metric space}
\newcommand{\lcqms}{separable quantum locally compact metric space}
\newcommand{\plcqms}{pointed separable quantum locally compact metric space}

\newcommand{\unit}{1}

\newcommand{\sa}[1]{{\mathfrak{sa}\left({#1}\right)}}

\newcommand{\QuasiStateSpace}{{\mathscr{Q}}}

\newcommand{\dom}[1]{{\operatorname*{dom}\left({#1}\right)}}
\newcommand{\domsa}[1]{{\operatorname*{dom_{\mathsf{sa}}}\left({#1}\right)}}
\newcommand{\codom}[1]{{\operatorname*{codom}\left({#1}\right)}}
\newcommand{\diam}[2]{{\mathrm{diam}\left({#1},{#2}\right)}}
\newcommand{\qdiam}[1]{{\,\mathrm{qdiam}\left({#1}\right)}}

\newcommand{\norm}[2]{\left\|{#1}\right\|_{#2}}

\newcommand{\targetsettunnel}[3]{{\mathfrak{t}_{#1}\left({#2}\middle\vert{#3}\right)}}

\newcommand{\worknote}[1]{}
\newcommand{\opnorm}[3]{{\left|\mkern-1.5mu\left|\mkern-1.5mu\left| {#1} \right|\mkern-1.5mu\right|\mkern-1.5mu\right|_{#3}^{#2}}}

\newcommand{\tunnelextent}[1]{{\chi\left({#1}\right)}}

\newcommand{\alg}[1]{{\mathfrak{#1}}}

\newcommand{\spectrum}[1]{\mathrm{Sp}\left({#1}\right)}

\newcommand{\ModState}[1]{\widehat{\StateSpace}}
\newcommand{\lipunit}[2]{{$#1$}-Lipschitz {$#2$}-pinned exhaustive sequence}
\newcommand{\closure}[1]{\mathrm{cl}\left({#1}\right)}

\newcommand{\Eth}{\operatorname*{\DH}}

\renewcommand{\geq}{\geqslant}
\renewcommand{\leq}{\leqslant}

\newcommand{\Dirac}{{\slashed{D}}}

\theoremstyle{plain}
\newtheorem{theorem}{Theorem}[section]
\newtheorem*{theorem*}{Theorem}

\newtheorem{corollary}[theorem]{Corollary}

\newtheorem{lemma}[theorem]{Lemma}

\newtheorem{theorem-definition}[theorem]{Theorem-Definition}

\theoremstyle{definition}
\newtheorem{definition}[theorem]{Definition}
\newtheorem*{definition*}{Definition}

\newtheorem{notation}[theorem]{Notation}

\newtheorem{convention}[theorem]{Convention}

\theoremstyle{remark}

\newtheorem{example}[theorem]{Example}

\newtheorem{remark}[theorem]{Remark}

\newtheorem{claim}[theorem]{Claim}
\newtheorem{step}[theorem]{Step}

\numberwithin{equation}{section}

\let\oldtocsection=\tocsection
\let\oldtocsubsection=\tocsubsection
\let\oldtocsubsubsection=\tocsubsubsection

\renewcommand{\tocsection}[2]{\bfseries \hspace{0em}\oldtocsection{#1}{#2}}
\renewcommand{\tocsubsection}[2]{\itshape \hspace{1em}\oldtocsubsection{#1}{#2}}
\renewcommand{\tocsubsubsection}[2]{\hspace{2em}\oldtocsubsubsection{#1}{#2}}

\begin{document}

\title[]{The quantum Gromov-Hausdorff Hypertopology on the class of  pointed Proper Quantum Metric Spaces}

\author{Fr\'{e}d\'{e}ric Latr\'{e}moli\`{e}re}
\email{frederic@math.du.edu}
\urladdr{http://www.math.du.edu/\symbol{126}frederic}
\address{Department of Mathematics \\ University of Denver \\ Denver CO 80208}

\date{\today}
\subjclass[2000]{Primary:  46L89, 46L30, 58B34.}
\keywords{}

\begin{abstract}
	We introduce a hypertopology, induced by an inframetric up to full quantum isometry, on the class of {\pqpms s}, which are separable, possibly non-unital, C*-algebras endowed with an analogue of the Lipschitz seminorm, with a distinguished state, and with a particular type of approximate units. Our hypertopology provides an analogue of the Gromov-Hausdorff distance on proper metric spaces, and in fact, convergence in the latter implies convergence in the former. Moreover, when restricted to the class of {\qcms s}, our new topology is compatible with the topology of the Gromov-Hausdorff propinquity. We include new examples of noncompact, noncommutative {\pqpms s} which are limits, for our new topology, of finite dimensional {\qcms s}. This article thus provides a first answer to the challenging  question of how to extend noncommutative metric geometry to the locally compact quantum space realm.
	
	\vspace*{5mm}
	\emph{This paper in its first stage of editing, and subsequent improved versions will replace it shortly.}
\end{abstract}
\maketitle
\tableofcontents


The field of Noncommutative Metric Geometry provides a framework for generalizing the concept of a metric space to the noncommutative setting. This research is founded on the powerful allegory that certain noncommutative $C^\ast$-algebras serve as analogues of the algebra of Lipschitz functions over a classical metric space \cite{Connes89,Rieffel98a,Rieffel99}. The primary goal of noncommutative metric geometry has been to study the convergence of these algebras in the topology induced by a noncommutative analogue of the Gromov-Hausdorff distance. This approach has been highly effective for studying quantum compact metric spaces  \cite{Rieffel01, Rieffel02, Ozawa05, Kerr02, Rieffel09, Rieffel10, Latremoliere13,Latremoliere13b, Latremoliere13c, Latremoliere14,Latremoliere15,Latremoliere15b,Latremoliere15c,Rieffel15b,Latremoliere16,Latremoliere16b,Kaad18,Kaad21,Latremoliere20b,Latremoliere21c}   , and associated structures such as modules  \cite{Latremoliere16,Latremoliere17a,Latremoliere18a,Latremoliere18d}, dynamical systems \cite{Latremoliere18b,Latremoliere18c}, and spectral triples \cite{Latremoliere18g,Latremoliere20a,Latremoliere21a,Latremoliere22,Latremoliere23a,LAtremoliere23b,Latremoliere24a,Latremoliere24b}. However, a significant challenge remains: developing an adequate noncommutative geometric framework for the noncommutative locally compact metric spaces. This is a critical omission, as the original Gromov-Hausdorff distance was introduced in the locally compact setting \cite{Gromov81}.

In this article, we propose the first answer to this difficult open problem. We introduce a hypertopology—constructed from a continuous inframetric, called the \emph{metametric} —on a large class of possibly non-unital $C^*$-algebras. These algebras are endowed with an appropriate analogue of the Lipschitz seminorm on a proper metric space, which is the natural classical class for Gromov-Hausdorff convergence. Our new topology successfully generalizes the existing compact theory,    and critically, it is compatible with the classical theory, as Gromov-Hausdorff convergence of proper metric spaces implies convergence for our proposed topology.

We proved our new inframetric is zero exactly between fully isometric {\pqpms s}, satisfies a version of the triangle inequality, is continuous for its induced topology,  recovers the topology of the propinquity --- our analogue of the Gromov-Hausdorff distance for {\qcms s} --- and is weaker than Gromov-Hausdorff convergence in general. We also construct an example of finite dimensional approximations, in the sense of our new topology, of certain C*-crossed-products which provide examples of noncommutative proper noncompact quantum metric spaces, including the algebra of compact operators.

\medskip
Our study is motivated by the desire to extend the successful program of metric geometry --- specifically, the topological study of the space of metric spaces --- to the noncommutative geometric realm. In this framework, geometric spaces are replaced by $C^\ast$-algebras which, when commutative, are exactly the algebras of continuous functions on an underlying locally compact Hausdorff space, by Gelfan'd duality. The benefit of this extension lies in its ability to provide a flexible framework for studying approximations and convergences. The core of this geometric perspective is found in Noncommutative Differential Geometry, pioneered by Alain Connes \cite{Connes80, Connes89,Connes}. This field replaces the traditional Dirac operators with spectral triples over possibly noncommutative C*-algebras. The established success of convergence results in Riemannian geometry naturally raises the question: can these metric and convergence techniques be extended to the noncommutative setting? Moreover, Noncommutative Geometry is very helpful for describing geometric objects that are highly singular or those arising intrinsically from quantum physics --- such as from quantum physics and quantum field theory, where noncommutativity formalizes the uncertainty principle.  Introducing metric considerations into the study of these algebras is a natural path to a deeper understanding of the underlying physics and to addressing fundamental issues like quantum gravity. The notion of a distance between 'space-times' has been a recurrent, implicit, or explicit pattern in this area, dating back to the first introduction of a distance on compact metric spaces by Edwards \cite{Edwards75}, which is now a special case of the Gromov-Hausdorff distance.

The challenge and benefits of defining an appropriate framework for approximations of C*-algebras in noncommutative geometry, or even operator algebra theory, is not new. The field of $C^\ast$-algebras is already rich with approximation techniques, such as nuclearity, quasi-diagonality, and the study of inductive limits, including approximately finite dimensional algebras or ($A\mathbb{T}$) algebras. However,  these are primarily topological concepts (since based upon $C^\ast$-algebras alone), which well-tailored for classification theory. The introduction of geometric structures by Connes guides us toward seeking a geometric notion of approximation rooted in metric considerations.The theory of Quantum Compact Metric Spaces, which is the foundational success of this geometric approach, originated from Connes' observation \cite{Connes89} that a spectral triple naturally induces a metric akin to the Monge-Kantorovich distance on the state space of the underlying $C^\ast$-algebra. Rieffel \cite{Rieffel98a, Rieffel99} generalized this by defining a {\qcms} using an analogue of the Lipschitz seminorm and introduced a Gromov-Hausdorff-like pseudodistance between such spaces. This initiated a highly successful research program. Building upon this, our prior work has focused on developing various analogues of the Gromov-Hausdorff distance—notably the (dual) propinquity—on spaces of {\qcms s}, and such related structures as spectral triples \cite{Latremoliere21a}.

\medskip

Our exploration starts from the comfortable shores of the theory of metric space, originating from the work of Fr{\'e}chet. If $(X,d)$ is a locally compact Hausdorff space, then a natural encoding of the metric $d$ at the level of algebras of functions is given by the \emph{Lipschitz seminorm}, defined for any function $f : X \rightarrow \C$ by
\begin{equation*}
	\Lip_d(f) \coloneqq \sup\left\{ \frac{|f(x)-f(y)|}{d(x,y)} : x,y \in X, x\neq y \right\} \text,
\end{equation*}
allowing for the value $\infty$. In particular, the seminorm $\Lip_d$ is defined on a dense *-subalgebra of the C*-algebra $C_0(X)$ of $\C$-valued continuous functions from $X$ to $\C$. On this algebra, $\Lip_d$ satisfies the Leibniz inequality and is lower semicontinuous on $C_0(X)$. Most notably, it induces a metric, called the bounded Lipschitz distance or \emph{Fortet-Mourier distance} \cite{Fortet-Mourier-53}, on the state space $\StateSpace(C_0(X))$, which is the space of all Radon measures on $X$ by the Radon representation theorem. For any two $\varphi,\psi \in \StateSpace(C_0(X))$, we define
\begin{equation*}
	\boundedLipschitz{\Lip_d}(\varphi,\psi) \coloneqq \sup\left\{ |\varphi(f) - \psi(f)| : f \in C_0(X), \Lip_d(f) \leq 1, \norm{f}{C_0(X)} \leq 1 \right\} \text.
\end{equation*}
Remarkably, the Fortet-Mourier metric induces the weak* topology on $\StateSpace(C_0(X))$. This property will be central in identifying analogues of the Lipschitz seminorm in the noncommutative setting --- a perspective adopted by Rieffel   \cite{Rieffel98a} as a foundation for noncommutative metric geometry. 

If, moreover, $(X,d)$ is compact, then $\boundedLipschitz{\Lip_d}$ is equivalent to another metric on $\StateSpace(C(X))$, called the {\MongeKant}, and introduced by Kantorovich \cite{Kantorovich40, Kantorovich58} in his study of Monge's transportation problem. The {\MongeKant} between any two states $\varphi,\psi \in \StateSpace(C(X))$ is given by:
\begin{equation*}
	\Kantorovich{\Lip_d}(\varphi,\psi) \coloneqq \sup\left\{ |\varphi(f) - \psi(f)| : f \in C_0(X), \Lip_d(f) \leq 1 \right\} \text.
\end{equation*}
This metric, which by equivalence with the Fortet-Mourier distance, also induces the weak* topology on $\StateSpace(C(X))$, has the added advantage that for any two $x,y \in X$, if $\delta_x$ and $\delta_y$ are the characters of $C(X)$ given by evaluation at, respectively, $x$ and $y$, then
\begin{equation*}
	\Kantorovich{\Lip_d}(\delta_x,\delta_y) = d(x,y) \text.
\end{equation*}
In other words, $\Kantorovich{\Lip_d}$ induces, on the Gelfan'd spectrum of $C(X)$, which is homeomorphic to $X$, a distance which turns the Gelfan'd spectrum of $C(X)$ to an isometric space to $X$. Thus in this context, $\Lip_d$ recovers the entire metric information in $(X,d)$.

Taken together, these observations leads to the definition of a {\qcms} \cite{Latremoliere13}, as an ordered pair $(\A,\Lip)$ of a unital C*-algebra and a seminorm $\Lip$ defined on a dense *-subalgebra $\dom{\Lip}$ of $\A$ which satisfies the Leibniz inequality (or some generalized form \cite{Latremoliere15}), is lower semicontinuous on $\A$ (when it is assumed to be $\infty$ outside of $\dom{\Lip}$), is zero exactly on the multiple of the unit of $\A$, and such that
\begin{equation*}
	\Kantorovich{\Lip}: \varphi,\psi \in \StateSpace(\A) \mapsto \sup\left\{ |\varphi(a)-\psi(a)| : a\in\dom{\Lip}\cap\sa{\A},\Lip(a) \leq 1 \right\} 
\end{equation*}
metrizes the weak* topology of the state space $\StateSpace(\A)$ of $\A$. In fact, $\Lip$ only needs to be defined on some dense Jordan-Lie subalgebra of the self-adjoint part $\sa{\A}$ of $\A$, but we will adopt this equivalent description. This definition is in fact the result of some evolution from the initial proposal from Rieffel \cite{Rieffel98a} where $\A$ was simply an order unit space, and thus no Leibniz property was present. Other variants tend to replace the class to which $\A$ belongs, including operator systems \cite{Kerr02} or operator spaces. Our own interest is focused on the definition given here, using C*-algebras, and adding the Leibniz inequality, which leads to a rich theory and allows one to take advantage of all the C*-algebraic related structures.

\medskip

The situation described so far becomes more intricate when $(X,d)$ is only locally compact. The {\MongeKant} is not in general an actual metric, does not induces the weak* topology on the state space, and may not recover the original distance. Since the {\MongeKant} and the Fortet-Mourier distance are equivalent when $(X,d)$ is compact, and the Fortet-Mourier distance remains well-behaved when $(X,d)$ is locally compact, it seems natural to replace the {\MongeKant} by the Fortet-Mourier metric for the purpose of our noncommutative generalization. 

This is indeed the approach we took in \cite{Latremoliere05b}. Thus, if $\A$ is a C*-algebra, and $\Lip$ is defined on a dense *-subspace $\dom{\Lip}$ of $\A$, then the associated Forter-Mourier metric on $\StateSpace(\A)$ is defined for any two states $\varphi,\psi \in \StateSpace(\A)$ by:
\begin{equation*}
	\boundedLipschitz{\Lip}(\varphi,\psi) \coloneqq \sup\left\{ |\varphi(a)-\psi(a)| : a\in\dom{\Lip}\cap\sa{\A}, \Lip(a)\leq 1, \norm{a}{\A}\leq 1 \right\}\text. 
\end{equation*}
A natural necessary and sufficient condition on $\Lip$ for $\boundedLipschitz{\Lip}$ to induce the weak* topology on the state space of $\A$ is found in \cite[Theorem 4.1]{Latremoliere05b}.  This will serve as the starting point of our present work.

It would be natural to ask whether something can be still be said about the {\MongeKant} in the locally compact case, which still makes sense in the noncommutative setting. Indeed, Dobrushin proved in \cite{Dobrushin70} that the {\MongeKant} metrizes the weak* topology on certain subsets of Radon probability measures subject to a strong form of tightness. While nontrivial, it is possible to give an analogue of this result, which we did in \cite{Latremoliere12b}. For our purpose here, the {\MongeKant} will still play a role, though more limited. The main takeaway for us here from \cite{Latremoliere12b} may be that if $\A$ is no longer unital, then we need to pick a replacement for the subalgebra $\C\unit_\A$ in the form of a commutative C*-subalgebra of $\A$ containing an approximate unit for $\A$. We call such C*-subalgebras topographies, as we see them as providing a mean to filter the C*-algebra $\A$, or rather the noncommutative space it describes, in terms of ``level sets'' of elements in the topography. The commutativity ensures that this concept makes sense, and provides a tool to discuss local notions within our noncommutative setting. As shown for instance in \cite{Akemann89}, there is in fact no canonical way to discuss ``approaching infinity'' in a general C*-algebra, and the choice of a topography is our way to handle this situation.

\medskip

Our class of {\lcqms s}, introduced in the first part of this article, build on the observations above. Our focus will be on the subclass of {\lcqms s} which are analogues of proper metric spaces, i.e. metric spaces whose closed balls are all compact. This subclass is distinguished from the more general {\lcqms s} by the existence of certain special approximate units. Proper metric spaces have the property that, unlike the general locally compact situation, the {\MongeKant} does recover the original distance. Our interest, however, is that they form the natural class to define the Gromov-Hausdorff convergence.

\medskip

In \cite{Hausdorff}, Hausdorff introduced a distance between closed subsets of a metric space, which now bears his name. The Hausdorff distance is a very interesting construction which enables the study of continuity of set-valued functions, the construction of fractal sets and other attractors as limits of certain dynamics on the space of closed subsets of a metric space, and finds applications in pattern and picture recognition.  If $(X,d)$ is a metric space, if $\varepsilon > 0$, and if $A, B \subseteq X$ are not empty, then we denote:
\begin{equation*}
	A \subseteq_\varepsilon^d B
\end{equation*}
to mean that for all $x \in A$, the distance $d(x,B) \coloneqq \inf\{d(x,b) : b\in B\}$ from $x$ to $B$ is at most $\varepsilon$. The Hausdorff distance $\Haus{d}$ between $A$ and $B$ is then defined by
\begin{equation*}
	\Haus{d}(A,B) \coloneqq \inf\left\{ \varepsilon > 0 : A \subseteq_\varepsilon^d B \text{ and }B \subseteq_\varepsilon^d A \right\}\text.
\end{equation*}
Indeed, $\Haus{d}$ is a distance on the so-called hyperspace of $X$, i.e. the set of all closed subsets of $X$. This is the origin of our term hypertopology, meant as a topology on a collection of metric (sub)spaces.

In \cite{Gromov81}, Gromov, motivated in part by the asymptotic behavior of certain spaces associated with hyperbolic groups, introduced his far reaching generalization of the Hausdorff distance, now called the \emph{Gromov-Hausdorff distance}. Gromov' construction is, at once, an intrinsic distance between proper metric spaces, and capture local information allowing for spaces of infinite diameters to still be limits of compact spaces. If $(Z,d)$ is any metric space, then we write $Z[x,r]$ for the closed ball of center $x \in Z$ and radius $r$. Modeled after the definition of the Hausdorff distance, but with a modification to ``localize'' it, we define, for any $A,B \subseteq Z$ \emph{and} for any $a\in A$, $b\in B$:
\begin{equation*}
	\delta_Z((A,a),(B,b)) \coloneqq \inf\left\{\varepsilon > 0 : A\left[a,\frac{1}{\varepsilon}\right]\subseteq_\varepsilon^d B\text{ and }B\left[a,\frac{1}{\varepsilon}\right]\subseteq_\varepsilon^d A \right\}\text.
\end{equation*}	
Now, if we are given two pointed proper metric spaces (i.e. proper metric spaces together with a choice of a point) $(X,d_X,x)$ and $(Y,d_Y,y)$, the \emph{Gromov-Hausdorff distance} between $(X,d_X,x)$ and $(Y,d_Y,y)$ is defined by:
\begin{multline*}
	\GH((X,d_X,x),(Y,d_Y,y)) \coloneqq\inf\Bigg\{\delta_Z(j_X(X),j(x),j_Y(Y),j(y)) : j_X:X\rightarrow Z, j_Y:Y\rightarrow Z, \\  \text{ isometries from $X$, $Y$ into a proper metric space $(Z,d_X)$} \Bigg\} \text.
\end{multline*}
While this quantity may be defined between any locally compact metric space, it is indeed a (complete) metric over the class of pointed proper metric spaces. 

\medskip

The Gromov-Hausdorff distance has a long and rich history in metric geometry, and has found deep applications in Riemannian geometry, where for instance, form of continuities for this metric, or variant thereof, of spectra of Laplacians and Dirac operators, have been proven. This very elegant and geometric way of thinking about spaces, quite distinct from more algebraic methods, is bound to also prove very interesting in noncommutative geometry. As a historical note, a prior metric defined more specifically on the class of compact metric spaces was introduced earlier by Edwards \cite{Edwards75}, as a first step in placing a topology of a space of ``space-times'' in the spirit of Wheeler's work on geometricodynamics and on the still open search of a quantum gravity theory. The definition of Edwards is similar, and well-known, and is topologically equivalent to Gromov' distance when working with compact metric spaces.

\medskip

Rieffel was the first to construct a noncommutative version of the Gromov-Hausdorff distance over a class of \emph{compact} quantum metric spaces \cite{Rieffel00}. This pioneering work spawned an entire new research area in noncommutative geometry. As we wish to work with {\qcms s} as described above, built on top of C*-algebras, and desire a distance between them which is zero only when the underlying C*-algebras are *-isomorphic, among other properties, we introduced the \emph{(dual) Gromov-Hausdorff propinquity} \cite{Latremoliere13,Latremoliere13b,Latremoliere14,Latremoliere15}. Our metric indeed induces the same topology as the Gromov-Hausdorff distance when restricted to the classical case of actual compact metric spaces, but also allows for such results as approximations of quantum tori by fuzzy tori, of spheres by matrix algebras, of noncommutative solenoids by noncommutative tori, and more. A lot of these results were motivated by the mathematical physics literature, where such approximations of spaces, or noncommutative algebras, by matrix models, is a common tool.

Motivated by our work on the propinquity between {\qcms s}, we then extended its reach to various structures. In particular, we introduced a Gromov-Hausdorff like distance between metric spectral triples \cite{Latremoliere18g}, for which we proved the continuity of the spectrum of the underlying Dirac operators \cite{Latremoliere22} and various examples \cite{Latremoliere20a,Latremoliere21a,Latremoliere23a,Latremoliere24a}.
\medskip

This article now addresses the generalization of the more involved Gromov-Hausdorff distance between {\pqpms s}. Owing presumably to its difficulty, this matter has not yet been addressed. Notably, the theory of locally compact metric spaces and their noncommutative analogues adds a significant hurdle, as working with unitalization (the noncommutative equivalent of compactifications) is not a great path forward, since the distance function on a locally compact metric space does not extend to a distance on such compactifications in general (in other words, $\R$ sits as a nice dense open subset of its one point compactification $\T$, but of course, this embedding is not an isometry or even a Lipscthiz map).

\medskip

Among the difficulties we must address, maybe the most salient is that the Gromov-Hausdorff distance is fundamentally a local concept, involving an exhaustion of the underlying spaces by balls of ever growing radii. From Gelfan'd duality, a closed subspace of a locally compact Hausdorff space $X$, in its simplest form, correspond to a quotient of the C*-algebra $C_0(X)$. Of course, noncommutative C*-algebras can be simple (note that they are topologically simple if and only if they are algebraically simple, so this is not just a matter of accepting some topological inconvenience) --- in fact, many of the most studied examples like irrational rotation algebras, C*-crossed-products by minimal actions, or the C*-algebras of interest in Elliott's classification program, are simple. This means that we do not have an obvious notion of closed subspace within our category of C*-algebras. While, presumably, one approach is to work within a larger category with weaker morphisms in the hope to remedy this problem, we insist on working within the category of C*-algebras, as these are the objects which naturally occur in the study of quantum physics (as algebras of observables) and geometry of singular spaces (as algebras of foliations, crossed products, group and groupoid C*-algebras, among others). Dropping the multiplicative structure is a steep price, and we adopt the view that if we can not afford it, then we may need to work harder and try to find a way around this difficulty.

\medskip

Indeed, the idea we offer is that, like everything else in noncommutative geometry, it is helpful to start by looking at the problem at hand in a function-based way. A closed subset may naturally correspond to a quotient in the Gelfan'd picture, but it also corresponds to a function: in its simplest form, a $\{0,1\}$-value function called its characteristic function. Such functions are of course given by projections in general. Once again, C*-algebras can well be projectionless, so this may not seem like quite a solution yet --- though by embedding a C*-algebra in its universal Von Neumann algebra, this concept becomes very helpful, already --- we will see a brief application later in this paper, based upon our work in \cite{Latremoliere05b}. 

But since we live in the world of metrics, maybe we can work with Lipschitz functions instead, simply asking that their support be contained in a subset. In order to define the Gromov-Hausdorff distance, we must find an analogue of exhausting the space in larger and larger balls. But this is a bit misleading --- of course, in the classical pictures, these are obvious spaces to work with. But in our noncommutative context, we are led to realize that what matters is to work with ever larger subsets which cover more and more of the space, which in turn can be seen as the supports of Lipschitz functions which are forced to be (almost) $1$ at a base point, and decay very slowly, i.e. have small Lipschitz seminorms --- which is now a concept we do have an analogue for in noncommutative geometry.

\medskip

We explain in the pages below how this process is indeed able to define a convergence {\`a} la Gromov in the work of geometry {\`a} la Connes. We define a class function which is our analogue of the Gromov-Hausdorff distance using this idea. It will become clear that this requires a bit of technical gymnastic, but it makes the voyage more interesting. This nontrivial generalization is not exactly a metric, as it satisfies a relaxed version of the triangle inequality, but a very manageable one which creates no difficulty in inducing a topology for which it is continuous. Moreover, distance zero between {\pqpms} will imply that they are full isometric in the quantum sense, and in particular, the underlying C*-algebras are *-isomorphic. We call it the \emph{metametric}: the prefix \emph{meta} is used here in both its common meaning, of an object ``beyond metrics'' (a generalization of a metric) and as a form of self-referential concept, i.e. a metric about metrics.

Moreover, the classical Gromov-Hausdorff convergence will imply convergence for our new topology, and so will convergence of {\qcms s} in the sense of the propinquity. A converse to this last statement also holds. Last, we prove that we can now approximate certain noncommutative non-unital C*-crossed-products by matrix algebras with our new topology. Therefore, our proposal meets all the basic wishes one could have for such a Gromov-Hausdorff distance between {\pqpms s}.

\medskip

It is our hope that this construction will open the field of noncommutative metric geometry to the wealth of the locally compact realm.

\section{Pointed Proper Quantum Metric Spaces}

A locally compact quantum metric space is described by a C*-algebras endowed with a ``quantum metric structure'' given by an analogue of the Lipschitz seminorm, i.e. a structure analogous, though possibly noncommutative, to the algebra of Lipschitz functions over a locally compact metric space. For our purpose, it will suffice to focus on separable {\lcqms s}. We begin our discussion with isolating the key properties that a seminorm defined on subspace of a C*-algebra should possess to be considered an analogue of the Lipschitz seminorm, starting from our work in \cite{Latremoliere05b,Latremoliere12b}. We then progress toward our notion of a {\pqpms}, a type of {\lcqms} which generalizes the notion of a proper, or boundedly-compact, metric space. This notion is encoded in the existence of certain notable approximate units constructed from the quantum metric structure.

\subsection{Separable Locally Compact Quantum Metric Spaces}

Let $\A$ be a separable C*-algebra, and let $\Lip$ be a seminorm defined on a dense subspace $\dom{\Lip}$ of the self-adjoint part $\sa{\A}$ of $\A$. We begin our discussion with the observation that such a seminorm induces a metric on the quasi-state space $\QuasiStateSpace(\A)$ of $\A$, which, in the classical picture, induces the weak* topology on the state space $\StateSpace(\A)$ of $\A$, and we wish to keep this property in our more general context. We begin with a little bit of notation which we will use accross this article.

\begin{notation}
  If $E$ is a normed vector space, we denote the norm on $E$ by $\norm{\cdot}{E}$ by default. If $F\subseteq E$ is a subspace of $E$, we may continue to write $\norm{\cdot}{E}$ for the induced norm on $F$, if this does not create confusion. Moreover, the topological dual of $E$ is denoted by $E^\ast$.
\end{notation}

\begin{notation}
  If $\A$ is a C*-algebra, and if $a \in \A$, then we denote the real part $\frac{a+a^\ast}{2}$ of $a$ by $\Re(a)$, and the imaginary part $\frac{a-a^\ast}{2i}$ of $a$ by $\Im(a)$. The Jordan-Lie subalgebra $\left\{ \Re(a) : a \in \A \right\}$ consisting of the self-adjoint elements of $\A$ is denoted henceforth by $\sa{\A}$. The Jordan product, in particular, over $\sa{\A}$, is simply the real part of the product, and the Lie product is simply the imaginary part of the product.

  The state space of $\A$ is denoted by $\StateSpace(\A)$. The set of all positive linear functionals over $\A$ of norm at most $1$ is the quasi-state space of $\A$ \cite[3.1.2]{Pedersen79}, denoted by $\QuasiStateSpace(\A)$.
  
  If $\A$ has a unit,  we set $\unital{\A} = \A$. Otherwise, $\unital{\A}\coloneqq \A\oplus\C$ is the smallest unitalization of $\A$. We denote the unit of $\unital{\A}$ by $\unit_\A$
\end{notation}

As we will use this notion often, we recall from \cite[3.10.4]{Pedersen79} that a self-adjoint element $h$ in a C*-algebra $\A$ is \emph{strictly positive} when, for all $\varphi \in \StateSpace(\A)$, we have $\varphi(h) > 0$.

The quasi-state space of a C*-algebra $\A$ is the convex hull of $\StateSpace(\A)\cup\{0\}$ \cite[3.2.1]{Pedersen79}. If $\A$ in, in fact, a separable C*-algebra without a unit, then $\QuasiStateSpace(\A)$ in the weak* closure of $\StateSpace(\A)$, as we now check.

\begin{lemma}
	If $\A$ is a separable C*-algebra without a unit, then $\QuasiStateSpace(\A)$ is the weak* closure of $\StateSpace(\A)$.
\end{lemma}

\begin{proof}
	Since $\A$ is separable, it has a strictly positive element $h \in \sa{\A}$ \cite[3.10.6]{Pedersen79}. In particular, there exists an approximate unit $(e_n)_{n\in\N}$ of $\A$ in $C^\ast(h)$  \cite[Proof of 3.10.5]{Pedersen79}, and therefore, $h \A h$ is dense in $\A$.
	
	If $\varphi \in \StateSpace(\A)$, then the restriction of $\varphi$ to $C^\ast(h)$ is again a state, since it is obviously positive and $\lim_{n\rightarrow\infty}\varphi(e_n) = 1$, so by \cite[Propostion 2.1.5, v]{Dixmier}.

	Since $\A$ is not unital, $h$ is not invertible, so $0 \in \spectrum{h}$. If $0$ is isolated in $\spectrum{h}$, then the function 
	\begin{equation*}
		f : x \in \spectrum{h} \mapsto \begin{cases} 1 \text{ if $x\neq 0$,} \\ 0 \text{ otherwise,}
		\end{cases}
	\end{equation*}
	 is continuous over $\spectrum{h}$ and in fact, $f\in c_0(\spectrum{h}\setminus\{0\})$, so $g = f(h) \in C^\ast(h)$ is well-defined by the functional calculus. For all $a\in\A$, since $e_n g = e_n$ for all $n\in\N$, and since $(e_n)_{n\in\N}$ is an approximate unit for $\A$, we then have:
	 \begin{align*}
	 	0\leq \norm{a - g a}{\A}
	 	&\leq \norm{a - e_n g a}{\A} + \norm{e_n g a - g a}{\A} \\
	 	&\leq \norm{a - e_n a}{\A} + \norm{e_n g - g}{\A}\norm{a}{\A}\\
	 	&\xrightarrow{n\rightarrow\infty} 0 + 0 \text,
	 \end{align*}
	and similarly, $\norm{a - a g}{\A} = 0$. So $g$ is a unit for $\A$. Since $\A$ is not unital, we conclude by contraposition that $0\in\spectrum{h}$ is not isolated in $\spectrum{h}$.
	
	 Let $(x_n)_{n\in\N}$ be a sequence in $\spectrum{h}\setminus\{0\}$ converging to $0$. Let $\varphi_n$ be a state of $\A$ obtained by applying the Hahn-Banach theorem for positive linear functionals to $\delta_{x_n}$. By construction, $\varphi_n(h)=\delta_{x_n}(h) \xrightarrow{n\rightarrow\infty} 0$. 
	 
	 Let $a \in \A$, $\varepsilon > 0$. Since $h\A h$ is norm-dense in $\A$, there exists $b \in \A$ such that $\norm{a-hbh}{\A} < \frac{\varepsilon}{2}$.  Since $\varphi_n(h^2) = \delta_{x_n}(h)^2 \xrightarrow{n\rightarrow\infty} 0$, there exists $N\in\N$ such that, if $n\geq N$, then $\varphi_n(h^2) <\frac{\varepsilon}{2}$. 	 
	 For all $n\in\N$, the map $\varphi_n$ is positive, so we conclude by \cite[3.1.4]{Pedersen79},\cite[Proposition 2.1.5 ii]{Dixmier} that
	 $\norm{\varphi_n(h\cdot h)}{\A^\ast} = \lim_{k\rightarrow\infty}\varphi_n(h e_k h) = \varphi(h^2)$. Therefore, if $n\geq N$ then:
	 \begin{align*}
	 	|\varphi_n(a)| &\leq |\varphi_n(a - hbh)| + |\varphi_n(hbh)| \\
	 	&\leq \norm{a-hbh}{\A} + \varphi_n(h^2)\norm{b}{\B} \\
	 	&\leq \frac{\varepsilon}{2} + \frac{\varepsilon}{2} = \varepsilon \text. 
	 \end{align*}
	Therefore, for all $a\in\A$, we have shown that $\lim_{n\rightarrow\infty}\varphi_n(a) = 0$, i.e. $(\varphi_n)_{n\in\N}$ converges to $0$ for the weak* topology.
	
	Now, let $\psi \in \A^\ast$ be a positive linear functional such that $t \coloneqq \norm{\psi}{\A^\ast} \leq 1$. For each $n\in\N$, set $\psi_n \coloneqq \psi + (1-t) \varphi_n$. Of course, $\psi_n$ is a positive linear functional, and $\norm{\psi_n}{\A^\ast} = \lim_{k\rightarrow\infty}\psi_n(e_k) = t + (1-t) = 1$. So $(\psi_n)_{n\in\N}$ is a sequence in $\StateSpace(\A)$, converging to $\psi$ in the weak* topology. So $\QuasiStateSpace(\A)$ contains all positive linear functional over $\A$ of norm at most $1$.
	
	On the other hand, if $\psi \in \A^\ast$ lies in the weak* closure of $\StateSpace(\A)$, then  $\norm{\psi}{\A^\ast} \leq 1$, and $\psi$ is positive. This concludes the proof of our lemma.
\end{proof}

\medskip

Let us begin our discussion about noncommutative Lipschitz seminorms. At a minimum, an analogue of the Lipschitz seminorm ought to has the following properties.
\begin{definition}
  Let $\A$ be a C*-algebra. A seminorm $L$ defined on a subspace of $\A$ is called a \emph{norm modulo constants} when $L$ is a norm if $\A$ has no unit, or $\{ a \in \dom{L} : L(a) = 0 \}$ is $\C\unit_\A$ if $\A$ has unit $\unit_\A$.
\end{definition}

\begin{definition}
	Let $\A$ be a C*-algebra. A seminom $L$ defined on a subspace $\dom{L}$ of $\A$ is called \emph{hermitian} when for all $a\in \dom{L}$, we have $a^\ast \in \dom{\Lip}$, and $\Lip(a) = \Lip(a^\ast)$.
\end{definition}

\begin{notation}
	If $L$ is a hermitian seminorm on $\dom{L}$, then $\domsa{L} \coloneqq \dom{L}\cap\sa{\A}$. Note that if $\dom{L}$ is a dense subspace in $\A$, then $\domsa{L}$ is a dense $\R$-subspace of $\sa{\A}$.
\end{notation}

We defined in \cite{Latremoliere05b} a metric on the state space of a C*-algebra endowed with a norm modulo constants, which is analogous to the Fortet-Mourier metric \cite{Fortet-Mourier-53} induced by a Lipschitz seminorm. We extend our definition in \cite[Definition 2.3]{Latremoliere05b} to the quasi-state space.

\begin{definition}\label{Fortet-Mourier-def}
  Let $\A$ be a separable C*-algebra, and $\Lip$ be a hermitian norm modulo constants defined on a dense subspace $\dom{\Lip}$ of $\A$. The \emph{Fortet-Mourier distance} is defined between any two quasi-states $\varphi, \psi \in \QuasiStateSpace(\A)$ by
  \begin{equation*}
    \boundedLipschitz{\Lip}(\varphi,\psi) \coloneqq \sup\left\{ |\varphi(a) - \psi(a)| : a\in\domsa{\Lip}, \norm{a}{\Lip} \leq 1 \right\} \text,
    \end{equation*}
    where, for all $a\in\dom{\Lip}$,
    \begin{equation*}
      \norm{a}{\Lip} \coloneqq \max\left\{ \Lip(a), \norm{a}{\A} \right\} \text.
    \end{equation*}
\end{definition}

\begin{remark}
	Since for all $a\in\A$ and $\varphi\in\QuasiStateSpace(\A)$, we have $\varphi(a) = \varphi\circ\Re(a)$, we conclude that:
	\begin{equation*}
		\boundedLipschitz{\Lip}(\varphi,\psi) \coloneqq \sup\left\{ |\varphi(a) - \psi(a)| : a\in\dom{\Lip}, \norm{a}{\Lip} \leq 1 \right\} \text,
	\end{equation*}
	though we will prefer the formulation given in Definition (\ref{Fortet-Mourier-def}).
\end{remark}

A natural characterization of when the Fortet-Mourier metric induced induces the weak* topology on the \emph{state space} was discovered in \cite{Latremoliere05b}.
\begin{theorem}[{\cite[Theorem 4.1]{Latremoliere05b}}]\label{bl-thm}
  If $\A$ is a separable C*-algebra, and if $\Lip$ is a hermitian norm modulo constants defined on a dense subspace $\dom{\Lip}$ of ${\A}$, then $\boundedLipschitz{\Lip}$ metrizes the weak* topology of $\StateSpace(\A)$ if, and only if, there exists a strictly positive element $h \in \sa{\A}$ such that the set
  \begin{equation*}
    \left\{ h a h : a \in \domsa{\Lip}, \norm{a}{\Lip} \leq 1 \right\} 
  \end{equation*}
  is totally bounded in $\A$.
\end{theorem}

Definition (\ref{Fortet-Mourier-def}) involves a somewhat arbitrary cutoff value, and it will be helpful that we allow for this cutoff to change. To this end, we introduce the following notation.
\begin{notation}
	Let $\A$ be a separable C*-algebra, and let $\Lip$ be a hermitian norm modulo constants defined on a dense subspace $\dom{\Lip}$ of $\A$. For any $M \geq 1$, and for all $a\in\dom{\Lip}$, we define
\begin{equation*}
	\norm{a}{\Lip_\A,M} \coloneqq \max\left\{\frac{1}{M} \norm{a}{\A}, \Lip(a) \right\} \text,
\end{equation*}
	with the convention that $\norm{\cdot}{\Lip} = \norm{\cdot}{\Lip,1}$ as done in Definition (\ref{Fortet-Mourier-def}).
	
We then define, between any two quasi-states $\varphi,\psi \in \QuasiStateSpace(\A)$ of $\A$, the distance:
	\begin{equation*}
		\boundedLipschitz{\Lip,M} (\varphi,\psi) \coloneqq \sup\left\{ |\varphi(a) - \psi(a)| : a\in\domsa{\Lip}, \norm{a}{\Lip,M} \leq 1 \right\} \text,
	\end{equation*}
	again with the convention, adopted earlier, that $\boundedLipschitz{\Lip} = \boundedLipschitz{\Lip,1}$.
\end{notation}
We immediately remark that, for all $M > 0$,
\begin{equation*}
	\boundedLipschitz{\Lip} \leq \boundedLipschitz{\Lip,M} \leq  M \boundedLipschitz{\Lip} \text,
\end{equation*}
which follows from the obvious inclusion
\begin{equation*}
 	\left\{ a \in \domsa{\Lip} : \norm{a}{\Lip} \leq 1 \right\} \subseteq \left\{ a \in \domsa{\Lip} : \norm{a}{\Lip,M} \right\} \subseteq M \left\{ a \in \dom{\Lip} : \norm{a}{\Lip} \leq 1 \right\}\text.
\end{equation*}
Since all the versions of the Fortet-Mourier distances are equivalent, they all either induce the weak* topology on the state space, or none of them do, and no new characterization besides Theorem (\ref{bl-thm}) is needed.

The following lemma shows that we, in fact, obtain many compact sets in $\A$ from a hermitian norm modulo constants $\Lip$ when its associated Forter-Mourier distance metrize the weak* topology on $\StateSpace(\A)$. This is, arguably, the central feature of such seminorms.
 
\begin{lemma}\label{cad-totally-bounded-lemma}
  Let $\A$ be a separable C*-algebra and $\Lip$ a hermitian norm modulo constants defined on a dense subspace $\dom{\Lip}$ of $\A$. There exists $h \in \sa{\A}$ with $h > 0$ such that
  \begin{equation*}
    \left\{ h a h : a \in\domsa{\Lip}, \norm{a}{\Lip} \leq 1 \right\}
  \end{equation*}
  is totally bounded if, and only if, for all $c,d \in \A$, for all $M\geq 1$, and for all $K>0$, the set
	\begin{equation*}
		\left\{ c a d : a \in \domsa{\Lip}, \norm{a}{\Lip,M} \leq K \right\}
	\end{equation*} 
	is totally bounded.
\end{lemma}

\begin{proof}
  Assume $\{ h a h : a\in\domsa{\Lip}, \norm{a}{\Lip} \leq 1\}$ is totally bounded. Since multiplication by a nonzero positive scalar is uniformly continuous, it maps totally bounded sets to totally bounded sets, so $\{h a h : a\in\domsa{\Lip},\norm{a}{\Lip}\leq K \}$ is totally bounded for all $K > 0$.

  Fix $M\geq 1$ and $K>0$. We then observe that
  \begin{equation*}
  	\{ h a h : a \in \domsa{\Lip}, \norm{a}{\Lip,M} \leq K \} \subseteq \{ h a h : a\in\domsa{\Lip}, \norm{a}{\Lip} \leq MK \} 
  \end{equation*}
  and thus $\{ h a h : a \in \domsa{\Lip}, \norm{a}{\Lip,M} \leq K \}$ is totally bounded.

  It immediately follows that $\{ h^n a h^n : a\in\domsa{\Lip}, \norm{a}{\Lip,M} \leq K \}$ is totally bounded for all $n\in\N\setminus\{0\}$ as well, as the image of the previous set by the uniformly continuous map $a\in\A\mapsto h^{n-1} a h^{n-1}$. Thus, if $P \in \R[X]$ with $P(0) = 0$, then $\{ P(h) a P(h) : a\in\domsa{\Lip}, \norm{a}{\Lip,M} \leq K\}$ is totally bounded, and further, if $f \in C_0(\sigma(h))$, then $\{ f(h) a f(h) : a\in\domsa{\Lip}, \norm{a}{\Lip,M} \leq K \}$ is totally bounded.

  Since $h > 0$, there exists $(f_n)_{n\in\N} \in c_0(\sigma(h))$ such that $\norm{f_n}{c_0(\sigma)} = 1$, and $(f_n(h))_{n\in\N}$ is an approximate unit of $\A$. Set $e_n \coloneqq f_n(h)$ for all $n\in\N$; note that $\norm{e_n}{\A} = 1$ and $e_n \in \sa{\A}$ for all $n\in\N$. In particular, $\{e_n a e_n : a\in\domsa{\Lip},\norm{a}{\Lip,M}\leq K \}$ is totally bounded in $\A$.
	
	Let now $c,d \in \A$. Let $\varepsilon > 0$. There exists $N\in\N$ such that $\norm{c-c e_n}{\A} < \frac{\varepsilon}{4(\norm{d}{\A}KM + 1)}$ and $\norm{d - e_n d}{\A} < \frac{\varepsilon}{4(\norm{c}{\A}KM + 1)}$. If $a\in \domsa{\Lip}$ with $\norm{a}{\Lip,M}\leq K$, then 
	\begin{align*}
		\norm{c a d - c e_n a e_n d}{\A} 
		&\leq \norm{c-c e_n}{\A} \norm{a}{\A}\norm{d}{\A} + \norm{c e_n}{\A} \norm{a}{\A} \norm{d-e_n d}{\A} \\
		&\leq \norm{c-c e_n}{\A} KM \norm{d}{\A} + \norm{c}{\A} KM \norm{d-e_n d}{\A} < \frac{\varepsilon}{2} \text.
	\end{align*}
	Fix $n\geq N$. Now, $\{ c e_n a e_n d : a\in\domsa{\Lip},\norm{a}{\Lip,M}\leq K\}$ is totally bounded in $\A$, again as the image of the totally bounded set  $\{ e_n a e_n : a\in\domsa{\Lip},\norm{a}{\Lip_M}\leq K \}$ by the uniformly continuous map $a\in\A \mapsto c a d$. So, there exists a finite $\frac{\varepsilon}{2}$-dense subset $F$ of $\{ c e_n a e_n d : a\in\domsa{\Lip},\norm{a}{\Lip_M}\leq K \}$, from which it follows that for every $a\in\domsa{\Lip}$ with $\norm{a}{\Lip_M}\leq K$, there exists $b \in F$ such that $\norm{c a d - b}{\A} < \varepsilon$.
	
	As $\varepsilon>0$ is arbitrary, $\{ c a d : a\in\domsa{\Lip},\norm{a}{\Lip,M}\leq K \}$ is totally bounded, as claimed.

        \medskip

        The converse is trivial .
\end{proof}

\begin{example}
	If $(X,d)$ is a locally compact metric space, then $\boundedLipschitz{\Lip_d}$ metrizes the weak* topology on $\StateSpace(C_0(X))$, by choosing $h \coloneqq x\in X \mapsto \max\{1, \frac{1}{d(x,x_0) + 1}\}$ with $x_0 \in X$ arbitrary.
\end{example}

\begin{remark}\label{quasistate-top-bl-rmk}
	Let $\A$ be a separable C*-algebra, and $\Lip$ a hermitian norm modulo constants defined on a dense subspace $\dom{\Lip}$ of ${\A}$. The topology induced by $\boundedLipschitz{\Lip}$ on the cone of positive linear functionals is stronger than the weak* topology in general. For instance, the sequence $(\delta_n)_{n\in\N}$ converges to $0$ for the weak* topology on $C_0(\R)^\ast$, yet $\boundedLipschitz{\Lip}(\delta_n,\delta_{n+1}) = 1$ for all $n\in\N$. 
	 
	We observe that, if $(\varphi_n)_{n\in\N}$ is a sequence of positive functionals in $\A$ which converges to some $\varphi$ in the weak* topology, and if $\lim_{n\rightarrow\infty} \norm{\varphi_n}{\A^\ast} = \norm{\varphi}{\A^\ast}$, then $\lim_{n\rightarrow\infty} \boundedLipschitz{\Lip}(\varphi_n,\varphi) = 0$. The converse is not valid in general: when $\A = C_0((0,1))$, the sequence $(\delta_{\frac{1}{n+1}})_{n\in\N}$ converges to $0$ for the weak* topology, and $\boundedLipschitz{\Lip}(\delta_{\frac{1}{n+1}},0) = \frac{1}{n+1} \xrightarrow{n\rightarrow\infty} 0$, yet of course, $(1)_{n\in\N} = (\norm{\delta_{\frac{1}{n+1}}}{\A^\ast})_{n\in\N}$ does not converge to $0 = \norm{0}{\A^\ast}$. We will see later on that a form of completeness for {\lcqms s}  does guarantee that indeed, convergence for the Forter-Mourier between positive linear functionals is equivalent to weak* convergence and convergence of the norms.
\end{remark}

\medskip

There is another natural extended metric to consider in our context, which proved to be central in the theory of {\qcms s} but will play a lesser role in our present work, though it still must be considered.

\begin{definition}\label{Kantorovich-def}
	Let $\A$ be a C*-algebra, and $\Lip$ a hermitian norm modulo constants defined on a dense subspace $\dom{\Lip}$ of ${\A}$. The \emph{\MongeKant} $\Kantorovich{\Lip}$ is defined between any two states $\varphi,\psi \in \StateSpace(\A)$ of $\A$ by setting:
	\begin{equation*}
		\Kantorovich{\Lip}(\varphi,\psi) \coloneqq \sup\left\{ |\varphi(a) - \psi(a)| : a \in \domsa{\Lip}, \Lip(a) \leq 1 \right\} \text.
	\end{equation*}
\end{definition}

One purpose of introducing this metric is to keep track the diameter of the space we work with. To this end, we introduce the following notation:
\begin{notation}
	For any metric space $(E,d)$ and any nonempty $A\subseteq E$,
	\begin{equation*}
		\diam{A}{d} \coloneqq \sup \left\{ d(x,y) : x,y \in A \right\}\text,
	\end{equation*}
	allowing for the value $\infty$. If $E$ is a normed vector space, then we write $\diam{A}{E}$ for $\diam{A}{d}$ where $d$ is the distance induced by the norm $\norm{\cdot}{E}$ of $E$.
\end{notation}

\begin{notation}
	Let $\A$ be a C*-algebra, and $\Lip$ be a hermitian norm modulo constants defined on some dense subspace $\dom{\Lip}$ of $\A$. We denote the diameter of $(\StateSpace(\A),\Kantorovich{\Lip})$ by
	\begin{equation*}
		\qdiam{\A,\Lip} \coloneqq \diam{\StateSpace(\A)}{\Kantorovich{\Lip}} \text,
	\end{equation*}
	allowing for the value $\infty$.
\end{notation}

While the Fortet-Mourier metric is well-behaved, the {\MongeKant} is more complicated in general. Indeed, if $(X,d)$ is a locally compact metric space, in general, $\Kantorovich{\Lip_d}$ does not metrize the weak* topology and, when $\diam{X}{d} = \infty$, may even take the value $\infty$. Instead, $\Kantorovich{\Lip_d}$ induces the weak* topology on certain subsets of states, as described by Dobrushin \cite{Dobrushin70}, subjected to some strong form of tightness. We generalizes Dobrushin's criterion to noncommutative C*-algebras in \cite{Latremoliere12b}. We proved that if $\A$ is a separable C*-algebra, and if $\Lip$ is a hermtian norm modulo constant, defined on a dense subspace of $\A$, and if there exists a strictly positive element $h \in \sa{\A}$ and a state $\mu \in \StateSpace(\A)$ such that
\begin{equation}\label{weird-set-eq}
	\{ h a h \in \unital{\A} : a = b + t\unit_\A , b \in\domsa{\Lip}, \mu(b) = -t, \Lip(b) \leq 1 \}
\end{equation}
is totally bounded, then $\Kantorovich{\Lip}$ does induce the weak* topology on certain sets of states called \emph{tame}, which includes such sets as  
\begin{equation*}
	\{ \varphi \in \StateSpace(\A) : \varphi_{|C^\ast(h)} \text{ is supported on $K$ } \}
\end{equation*}
where $K\subseteq\spectrum{h}\setminus\{0\}$ is compact (and where $\Lip(\unit_\A) = 0$). For our purpose, however, it will suffice to require that the set in Expression (\ref{weird-set-eq}) is merely bounded, which will help us control the {\MongeKant} $\Kantorovich{\Lip}$ enough to meet our needs --- for instance ensuring that our definition of {\lcqms s} is equivalent with the notion of {\qcms} when the base C*-algebra is unital, and that there is a weak* dense set of states within finite {\MongeKant} from a fixed state, in Lemma (\ref{standard-lemma}).

\begin{remark}
	If $\A$ is a unital C*-algebra, then we note that $\boundedLipschitz{\Lip}$ metrizes the weak* topology if, and only if, $\{ a \in \domsa{\Lip} : \norm{a}{\Lip}\leq 1\}$ is totally bounded. Thus, as seen for instance in \cite{Rieffel99}, $\Kantorovich{\Lip}$ metrizes the weak* topology exactly if $\qdiam{\A}{\Lip} < \infty$, in which case we also note that $\boundedLipschitz{\Lip,M} = \Kantorovich{\Lip}$ whenever $M\geq\qdiam{\A,\Lip}$. This condition is in turn given by the boundedness of the set $\{ a \in \domsa{\Lip} : \Lip(a)\leq 1, \mu(a) = 0\}$ for any state $\mu \in \StateSpace(\A)$. 
\end{remark}

\medskip

The Lipschitz seminorm induced by a metric, in addition to inducing nicely behaved metrics on the space of Radon probability measures, has a few other properties which have proven helpful to keep in the noncommutative setting. It has a lower semicontinuity property, specifically, the closed unit ball for the Lipschitz seminorm is closed in norm. Moreover, the Lipschitz seminorm satisfies the Leibniz inequality, which connects it to the multiplicative structure of the underlying C*-algebra. This property became central in our work on the propinquity between {\qcms s} \cite{Latremoliere13}, and will again be important to us here.

\begin{definition}\label{Leibniz-def}
	A seminorm $\Lip$ on a subalgebra $\dom{\Lip}$ of $\A$ is \emph{Leibniz} when, for all $a,b \in \dom{\Lip}$,
	\begin{equation*}
		\Lip(ab)  \leq \norm{a}{\A} \Lip(b) + \Lip(a) \norm{b}{\A} \text.
	\end{equation*}
\end{definition}

\begin{remark}
	If $\Lip$ is a Leibniz hermitian seminorm on a subalgebra $\dom{\Lip}$ of $\A$, then $\domsa{\Lip}$ is a Jordan-Lie subalgebra of $\sa{\A}$, and the restriction of $\Lip$ to $\domsa{\Lip}$ satisfies:
	\begin{equation*}
		\max\{\Lip(\Re(ab)), \Lip(\Im(ab))\} \leq \norm{a}{\A} \Lip(b) + \Lip(a) \norm{b}{\A} \text.
	\end{equation*}
\end{remark}

For this article, a particular consequence of the Lwibniz property will be used repeatedly, and we establish it now.
\begin{lemma}\label{Leibniz-lemma}
	If $\A$ is a C*-algebra, $\Lip$ is a seminorm defined on a subalgebra $\dom{\Lip}$ of $\sa{\A}$, and if $a,b \in \dom{\Lip}$, then
	\begin{equation*}
		\Lip(a b a) \leq \norm{a}{\A}\left(2 \norm{b}{\A}\Lip(b) + \norm{a}{\A}\Lip(d)\right) \text.
	\end{equation*}
\end{lemma}

\begin{proof}
	Let $a,b \in \dom{\Lip}$. By the Leibniz property:
	\begin{align*}
		\Lip(a b a) 
		&\leq \norm{ab}{\A}\Lip(a) + \Lip(ab)\norm{a}{\A} \\
		&\leq \norm{a}{\A}\norm{b}{\B}\Lip(a) + \norm{a}{\A}^2\Lip(b) + \norm{a}{\A}\norm{b}{\B}\Lip(a) \\
		&\leq \norm{a}{\A}\left(2 \norm{b}{\B}\Lip(a) + \norm{a}{\A}\Lip(b)\right) \text,
	\end{align*}
	as claimed.
\end{proof}

\begin{remark}
	If $L$ is a seminorm defined on a Jordan-Lie subalgebra $\domsa{\Lip}$ of $\sa{\A}$, for some C*-algebra $\A$, such that 
	\begin{equation}\label{weak-Leibniz-eq}
		\max\{\Lip(\Re(ab)),\Lip(\Im(ab))\} \leq\norm{a}{\A}\Lip(b) + \Lip(a)\norm{b}{\A}
	\end{equation}
	 for all $a,b\in\domsa{\Lip}$, then since 
	\begin{equation*}
		a b a = \Re(\Re(ab)a) + \Im(\Im(ab)a) \text,
	\end{equation*}
	we would obtain from a similar computation as above that
	\begin{equation*}
		\Lip(a b a) \leq 2\norm{a}{\A}\left(2\norm{b}{\A}\Lip(b) + \norm{a}{\A}\Lip(d)\right)
	\end{equation*}
for all $a,b \in \domsa{\Lip}$. The factor 2 in front would cause us difficulties. Therefore, while the work on {\qcms} uses seminorms defined only on Jordan-Lie subalgebras of $\sa{\A}$ satisfying the weaker Leibniz relation in Expression \eqref{weak-Leibniz-eq}, we actually will require our Lipschitz seminorms analogues to satisfy the usual Leibniz inequality of Definition (\ref{Leibniz-def}) here.	
\end{remark}

\medskip

We are now ready to define our notion of {\lcqms}.

\begin{definition}\label{lcqms-def}
	A \emph{\lcqms} $(\A,\Lip)$ is an ordered pair of a C*-algebra $\A$ and a Leibniz, hermitian, norm modulo constants $\Lip$ defined on a dense *-subalgebra $\dom{\Lip}$ of $\A$ such that:
	\begin{enumerate}
		\item the \emph{Fortet-Mourier distance} $\boundedLipschitz{\Lip}$ defined on the state space $\StateSpace(\A)$ of $\A$ by:
	\begin{equation*}
		\boundedLipschitz{\Lip} \coloneqq \varphi,\psi\in\StateSpace(\A) \mapsto \sup\{ |\varphi(a)-\psi(a)| : a\in\domsa{\Lip}, \norm{a}{\A} \leq 1, \Lip_\A(a) \leq 1 \}
	\end{equation*}
	induces the weak* topology on $\StateSpace(\A)$,
		\item $\{a\in\dom{\Lip} : \Lip(a) \leq 1\}$ is closed in $\sa{\A}$,
	\item there exists $h \in \sa{\A}$ strictly positive, and $\mu\in\StateSpace(\A)$ such that $\mu(h) = 1$, and
	\begin{equation*}
		\left\{ h a h \in \unital{\A} : a=b+t\unit_\A, t\in\R, b\in \domsa{\Lip}, \Lip(b) \leq 1, \mu(b) = -t  \right\} 
	\end{equation*}
	is bounded.
	\end{enumerate}
\end{definition}

We reconcile our new definition with the concept of {\qcms} as defined in \cite{Latremoliere13,Latremoliere15}.
\begin{lemma}
	If $(\A,\Lip)$ is a {\lcqms} and $\A$ is unital, then $(\A,\Lip_{|\domsa{\Lip}})$ is a {\qcms}.
\end{lemma} 

\begin{proof}
	Assume that $(\A,\Lip)$ is a {\lcqms} while $\A$ is unital. Let $h \in\sa{\A}$ be a strictly positive element in $\A$ and $\mu \in \StateSpace(\A)$ such that
	\begin{equation*}
		\left\{ h a h : a\in\dom{\Lip}, \Lip(a) \leq 1, \mu(a) = 0 \right\}
	\end{equation*}
	is bounded (note that here $\unital{\A} = \A$ and $\unit_\A\in\dom{\Lip}$ since $\Lip$ is a norm modulo constants), while
	\begin{equation*}
		\left\{ h a h : a\in\dom{\Lip}, \norm{a}{\Lip}\leq 1 \right\}
	\end{equation*}
	is compact.
	
	 As a strictly positive element in a unital C*-algebra $\A$, the element $h$ is invertible (since $\{\varphi(h):\varphi\in\StateSpace(\A)\}$, as $\A$ is unital, is the closed convex hull of the spectrum of $h$, and it does not contain $0$ by definition). The map $a \in \A \mapsto h^{-1} a h^{-1}$ is of course Lipscthiz, hence the set:
	\begin{equation*}
		\left\{ a \in \dom{\Lip} : \Lip(a) \leq 1, \mu(a) = 0 \right\}
	\end{equation*}
	is also bounded, which implies that $\diam{\StateSpace(\A)}{\Lip} < \infty$. Similarly,
	\begin{equation*}
		\left\{ a \in \dom{\Lip} : \Lip(a) \leq 1, \norm{a}{\A}\leq 1 \right\}
	\end{equation*}
	is compact. Hence, $(\A,\Lip)$ is a {\qcms}.
	
	Of course, if $(\A,\Lip)$ is a {\qcms}, then $(\A,\Lip)$ is a {\lcqms}.
\end{proof}

\medskip

We also note that the topological assumption in \cite{Latremoliere12b} is indeed stronger than what we require in this paper.

\begin{lemma}
Let $\A$ be a separable C*-algebra and $\Lip$ be a hermitian norm modulo constant defined on dense subspace $\dom{\Lip}$ of $\A$. If 
	\begin{equation}\label{weird-set-again-eq}
		\left\{ h a h \in \unital{\A} : a=b+t\unit_\A, t\in\R, b\in \dom{\Lip}, \Lip(b) \leq 1, \mu(b) = -t  \right\} 
	\end{equation}
is totally bounded for some $h \in \domsa{\Lip}$ strictly positive, and some $\mu \in \StateSpace(\A)$ with $\mu(h) = 1$, then $\{ h a h : a\in\domsa{\Lip}, \norm{a}{\Lip} \leq 1 \}$ is totally bounded as well.
\end{lemma}

\begin{proof}
Assume that the set given in Expression (\ref{weird-set-again-eq}) is totally bounded. 
 Let $h \in \domsa{\Lip}$, and $\mu(h)=1$, for some strictly positive element $h \in \sa{\A}$. Let $(a_n)_{n\in\N}$ in $\domsa{\Lip}$ such that $\norm{a_n}{\Lip} \leq 1$ for all $n\in\N$. Then $(a_n-\mu(a_n) h) \in \domsa{\Lip}$, with $\Lip(a_n-\mu(a_n)h) \leq 1 + |\mu(a_n)|\Lip(h) \leq 1+\Lip(h)$ and $\mu(a_n-\mu(a_n)h) = 0$. So there exists a Cauchy subsequence $(h(a_{f(n)}-\mu(a_{f(n)})h)h)_{n\in\N}$ of $(h(a_n-\mu(a_n)h)h)_{n\in\N}$.  Now, $(\mu(a_{f(n)}))_{n\in\N}$ is bounded (with values in $[-1,1]$) and thus has a convergent subsequence $(\mu(a_{f(g(n))}))_{n\in\N}$ as well, and therefore, $(\mu(a_{f(g(n))}h^3)_{n\in\N}$ converges in $\A$. Altogether, the subsequence $(h a_{f(g(n))} h)_{n\in\N}$ of $(h a_n h)_{n\in\N}$ is therefore Cauchy. So $\{ h a h : a\in\domsa{\Lip}, \norm{a}{\Lip}\leq 1\}$ is totally bounded, as claimed.
\end{proof}

If we assume that for a {\lcqms} $(\A,\Lip)$, there exists $h \in \sa{\A}$ with $h > 0$ and $\mu \in \StateSpace(\A)$ such that the set in Expression \eqref{weird-set-again-eq} is bounded for the norm $\norm{\cdot}{\Lip}$, rather than $\norm{\cdot}{\A}$, and if $\{ f a f : a\in\domsa{\Lip}, \norm{a}{\Lip}\leq 1\}$ is totally bounded, then in fact, $(\A,\Lip)$ is a {\lcqms} in the sense of \cite{Latremoliere12b}. We will however not use this observation further.

\medskip

The reason we still need to work with the {\MongeKant} is the very useful observation that, while the state space is not metrized by it, some weak* dense subset is indeed a metric space (again, with a stronger topology than the weak* topology in general).

\begin{lemma}\label{standard-lemma}
	Let $(\A,\Lip)$ be a {\lcqms}. If $h\in\sa{\A}$ is a strictly positive element and $\mu \in \StateSpace(\A)$ with $\mu(h)=1$  and such that
	\begin{equation}\label{weird-set-yet-again-eq}
		\left\{ h a h\in\unital{\A}: a\in\unital{\A}, a=b+t, b\in\dom{\Lip},t\in\R,\Lip(a)\leq 1, \mu(a) = -t \right\}
	\end{equation}
	is bounded, then $\{\varphi\in\StateSpace(\A):\Kantorovich{\Lip}(\varphi,\mu) < \infty \}$ is weak* dense in $\StateSpace(\A)$.
\end{lemma}	

\begin{proof}
Without changing notation, we extend $\Lip$ to $\unital{\A}$ by setting $\Lip(a+t\unit_\A) = \Lip(a)$ for all $a\in\dom{\Lip}$, $t\in\C$, noting that since $\Lip$ is a norm modulo constant, this definition is consistent if $\A$ is unital.

	Let $h\in\sa{\A}$ strictly positive, and $\mu\in\StateSpace(\A)$, such that $\mu(h)=1$ and the set in Expression \eqref{weird-set-yet-again-eq} is bounded. Denote the spectrum of $h$ by $\spectrum{h}$, and let $C_0(\spectrum{h})$ be the C*-algebra of $\C$-valued continuous functions over $\spectrum{h}$ vanishing at $0$.
	
	Let $(f_n)_{n\in\N}$ be a an approximate unit in $C_0(\spectrum{h})$, with $f_n$ supported on a compact in $\spectrum{h}\setminus\{0\}$ for each $n\in\N$. For each $n\in\N$, there exists $g_n \in C_0(\spectrum{h})$ such that $f_n(x) = x g_n(x)$ for all $x\in\spectrum{h}$, since $f_n$ is compactly supported. By construction, if $h_n \coloneqq f_n(h)$ and $s_n \coloneqq g_n(h)$ for all $n\in\N$, we conclude that $(h_n)_{n\in\N}$ is an approximate unit of $C^\ast(h)$. Now, $C^\ast(h)$ contains an approximate unit of $\A$ since $h$ is strictly positive \cite[3.10.5]{Pedersen79}, so $(h_n)_{n\in\N}$ is an approximate unit for $\A$. Moreover, $h_n = s_n h$ for all $n\in\N$.

	Now, for each $n\in\N$, the set
	\begin{equation*}
		\left\{ h_n a h_n \in \unital{\A} : a\in\domsa{\Lip}+\R\unit_\A, \Lip(a) \leq 1, \mu(a) = 0 \right\} = \left\{ s_n (h a h) s_n : a\in\domsa{\Lip}+\R\unit_\A, \Lip(a) \leq 1, \mu(a) = 0 \right\}
	\end{equation*}
	is bounded by assumption and since $a\in\A\mapsto s_n a s_n$ is Lipschitz. Let
	\begin{equation*}
	 M_n \coloneqq \sup\{ \norm{h_n a h_n}{\unital{\A}} : a\in\domsa{\Lip}+\R\unit_\A, \Lip(a) \leq 1, \mu(a) = 0 \}\text.
	 \end{equation*}
	
	 Since $(h_n)_{n\in\N}$ is an approximate unit, $\lim_{n\rightarrow\infty}\mu(h_n) = 1$. Up to dropping finitely many entries in the sequence $(h_n)_{n\in\N}$, and then replacing our sequence with $\left(\frac{1}{\mu(h_n)}h_n\right)_{n\in\N}$, we will henceforth assume that $\mu(h_n)=1$ without loss of generality (in particular $(\frac{1}{\mu(h_n)}h_n)_{n\in\N}$ is still an approximate unit of $\A$).

	Let now $\varphi \in \StateSpace(\A)$. For each $n\in\N$, let $\varphi_n \coloneqq a\in\A\mapsto\frac{1}{\varphi(h_n^2)}\varphi(h_n a h_n)$. By construction, $\varphi_n \in \StateSpace(\A)$ and $\lim_{n\rightarrow\infty}\varphi(h_n^2)=1$ \cite[Proposition 2.1.5]{Dixmier}. 
	
	Fix $a\in\A$. Since $a = \lim_{n\rightarrow\infty} h_n a h_n$, we conclude $\lim_{n\rightarrow\infty} \varphi_n(a) = \varphi(a)$. Hence, $(\varphi_n)_{n\in\N}$ converges to $\varphi$ in the weak* topology.
	
	Therefore 
	\begin{equation*}
		S\coloneqq \left\{ \frac{1}{\varphi(h_n^2)}\varphi(h_n\cdot h_n) : \varphi\in\StateSpace(\A), n \in \N, \varphi(h_n^2) > 0 \right\}
	\end{equation*} 
	is weak* dense in $\StateSpace(\A)$.
	
	On the other hand, let $\varphi \in \StateSpace(\A)$ and $n\in\N$ such that $\varphi(h_n^2)>0$ --- which exists since $\lim_{n\rightarrow\infty}\varphi(h_n^2)=1$. For all $a \in \domsa{\Lip}$, with $\Lip(a) \leq 1$,
	\begin{align*}
		|\varphi_n(a) - \mu(a)|
		&=\left|\frac{1}{\varphi(h_n^2)} \varphi(h_n a h_n)  - \frac{1}{\varphi(h_n^2)}\mu(\varphi(h_n^2) a)\right|\\
		&=\frac{1}{\varphi(h_n^2)} | \varphi(h_n a h_n)  - \mu(a)\varphi(h_n^2) |\\
		&=\frac{1}{\varphi(h_n^2)} |\varphi(h_n (a-\mu(a)\unit_\A) h_n )| \leq \frac{M_n}{\varphi(h_n^2)}\text.
	\end{align*}

	Therefore, 
	\begin{align*}
		\Kantorovich{\Lip}(\varphi_n,\mu)
		&=	\sup\left\{|\varphi_n(a) - \mu(a)| : a \in \domsa{\Lip}, \Lip(a) \leq 1 \right\} \\
		&\leq \frac{M}{\varphi(h_n^2)} < \infty \text.
	\end{align*}
	Therefore, $S \subseteq \{ \varphi \in \StateSpace(\A) : \Kantorovich{\Lip}(\varphi,\mu) < \infty \}$. Since $S$ is weak* dense in $\StateSpace(\A)$, our proof is complete.
\end{proof}

\medskip

The Gromov-Hausdorff distance is defined between \emph{pointed} proper metric spaces. The matter of finding an adequate notion of a {\pqpms} is delicate, because the role of a base point is more than just to pin the space in place --- it provides a center for an exhaustion of the space using balls of increasing radius, and working with subsets such as balls in noncommutative geometry is basically not available, by  the very nature of noncommutativity. Nonetheless, we introduce the simplest notion of a quantum space with a ``base state'', which borrows from the conclusion of Lemma (\ref{standard-lemma}).

\begin{definition}\label{plcqms-def}
	An \emph{pinned {\lcqms}} $(\A,\Lip,\mu)$ is a {\lcqms} $(\A,\Lip)$ together with a state $\mu \in \StateSpace(\A)$ such that
	\begin{equation*}
		\left\{ \varphi \in \StateSpace(\A) : \Kantorovich{\Lip}(\varphi,\mu) < \infty \right\}
	\end{equation*}
	is weak* dense in $\StateSpace(\A)$. The state $\mu$ is then call the \emph{pin} of $(\A,\Lip,\mu)$.
\end{definition}

\medskip

We conclude this section with a sufficient condition for completeness of the Fortet-Mourier distance, as a preliminary to our next section. It will also refine our description of the topology induced by the Fortet-Mourier distance on the quasi-state space, under this natural assumption, which will be met by the type of {\lcqms s} we will work with in this paper.

\begin{lemma}\label{completeness-lemma}
	If $(\A,\Lip)$ is a {\lcqms}, and if there exists a bounded approximate unit $(e_n)_{n\in\N}$ for $\A$ in $\domsa{\Lip}$ such that $(\Lip(e_n))_{n\in\N}$ is bounded as well, then $(\StateSpace(\A),\boundedLipschitz{\Lip})$ is complete. Moreover, a sequence $(\varphi_n)_{n\in\N}$ of positive linear functionals over $\A$ converge to $\psi \in \A^\ast$ for $\boundedLipschitz{\Lip}$ if, and only if, $(\varphi_n)_{n\in\N}$ converges to $\psi$ in the weak* topology and $\lim_{n\rightarrow\infty} \norm{\varphi_n}{\A^\ast} = \norm{\psi}{\A^\ast}$.
\end{lemma}

\begin{proof}
	Let $(\varphi_n)_{n\in\N}$ be a Cauchy sequence in $(\StateSpace(\A),\boundedLipschitz{\Lip})$. It is a standard argument that $(\varphi_n)_{n\in\N}$ converges, in the weak* topology, to a continuous positive linear functional $\psi$ over $\A$. Now, by assumption, there exists $M > 0$ such that $\{ a \in \domsa{\Lip} : \norm{a}{\Lip} \leq M \}$ contains an approximate unit for $\A$. By uniform convergence: 
	\begin{equation*}
		\norm{\psi}{\A^\ast} = \lim_{n\rightarrow\infty}\psi(h_n) = \lim_{n\rightarrow\infty} \lim_{m\rightarrow\infty} \varphi_m(h_n) = \lim_{m\rightarrow\infty} \lim_{n\rightarrow\infty} \varphi_m(h_n) = \lim_{m\rightarrow\infty} \norm{\varphi_m}{\A^\ast} \text.
	\end{equation*}
	In particular, if $(\varphi_n)_{n\in\N}$ is a sequence of states of $\A$, then $\psi$ is a state of $\A$ as well.
	
	We already established the converse in Remark (\ref{quasistate-top-bl-rmk}).
\end{proof}

\begin{definition}\label{strongly-complete-def}
	A {\lcqms} $(\A,\Lip)$ is \emph{strongly complete} when there exists an approximate unit of $\A$ in $\{ a \in \domsa{\A} : \norm{a}{\Lip} \leq 1\}$.
\end{definition}
%

Our next step is the introduction of the class of {\pqms s}, which is also based upon the existence of special approximate units.

\subsection{Proper Quantum Metric Spaces}

We now introduce a subclass of {\lcqms s} which will be the natural domain for our Gromov-Hausdorff metametric, as they have an appropriately strong notion of completeness. To this end, we begin with the following source of approximate units for certain {\lcqms s}, which, by Lemma (\ref{completeness-lemma}), immediately implies completeness. This result is actually a fundamental tool in the later study of the convergence of {\pqpms s} as well.

\begin{definition}\label{lipunit-def}
If $(\A,\Lip,\mu)$ is a pinned {\lcqms}, a \emph{\lipunit{\Lip}{\mu}} is a sequence $(h_n)_{n\in\N} \in \domsa{\Lip}$  such that 
	\begin{enumerate}
	\item $\lim_{m\rightarrow\infty} \Lip(h_n) = 0$,
	\item $\lim_{n\rightarrow\infty} \mu(h_n) = 1$,
	\item $\lim_{n\rightarrow\infty} \norm{h_n}{\A}=1$.
	\end{enumerate}
\end{definition}

\begin{theorem}\label{approx-unit-thm}
	If $(\A,\Lip,\mu)$ is a {\plcqms}, and if $(h_n)_{n\in\N} \in \domsa{\Lip}$ is a \lipunit{\Lip}{\mu}, then $(h_n)_{n\in\N}$ is an approximate unit for $\A$.
\end{theorem}

\begin{proof}
 Since $(\norm{h_n}{\A})_{n\in\N}$ and $(\Lip(h_n))_{n\in\N}$ are convergent, $(\norm{h_n}{\Lip})_{n\in\N}$ is bounded. Let $M>0$ such that, for all $n\geq M$, we have $\norm{h_n}{\Lip} \leq M$.

	\begin{claim}
	The sequences $(h_n)_{n\in\N}$ is weakly convergent to the unit in $\unital{\A}$, i.e.
	\begin{equation*}	
	\forall \varphi\in\StateSpace(\A) \quad \lim_{n\rightarrow\infty} \varphi(h_n) = 1\text.	
\end{equation*}		
	\end{claim}	

	Let $\varphi\in\StateSpace(\A)$ with $\Kantorovich{\Lip}(\varphi,\mu) < \infty$. For all $n\in\N$,
	\begin{equation*}
		|\mu(h_n)-\varphi(h_n)| \leq \Lip(h_n)\Kantorovich{\Lip}(\mu,\varphi) \xrightarrow{n\rightarrow\infty} 0\text.
	\end{equation*}
	Since $\lim_{n\rightarrow\infty} \mu(h_n) = 1$, we conclude that $\lim_{n\rightarrow\infty}\varphi(h_n) = 1$.

	By Definition (\ref{lcqms-def}), the metric $\boundedLipschitz{\Lip}$ metrizes the weak* topology on $\StateSpace(\A)$, and by assumption on $\mu$, the set $\{ \varphi\in\StateSpace(\A) : \Kantorovich{\Lip}(\varphi,\mu) < \infty\}$ is weak* dense in $\StateSpace(\A)$. 
	
	Let $\psi \in \StateSpace(\A)$ and $\varepsilon > 0$. There exists $\varphi \in \StateSpace(\A)$ such that $\Kantorovich{\Lip}(\varphi,\mu)<\infty$, and $\boundedLipschitz{\Lip}(\varphi,\psi) < \frac{\varepsilon}{2 M}$. 
we conclude that
\begin{align*}
	\sup_{n \in \N} |\varphi(h_n)-\psi(h_n)| \leq M \boundedLipschitz{\Lip}(\varphi,\psi) < \frac{\varepsilon}{2} \text.
\end{align*}	

	Since $\lim_{n\rightarrow\infty}\varphi(h_n) = 1$, there exists $N\in\N$ such that, if $n\geq N$, then $|1-\varphi(h_n)| < \frac{\varepsilon}{2}$. Hence, for all $n\geq N$:
\begin{equation*}
	|1-\psi(h_n)| \leq |1-\varphi(h_n)| + |\varphi(h_n) - \psi(h_n)| \leq \frac{\varepsilon}{2} + \frac{\varepsilon}{2} = \varepsilon \text.
\end{equation*}	
Hence, $\lim_{n\rightarrow\infty} \psi(h_n) = 1$, and our claim is proven.

	\begin{claim}
	For all $\varphi \in \StateSpace(\A)$, it holds that $\lim_{n\rightarrow\infty} \varphi(h_n^2) = 1 = \lim_{n\rightarrow\infty}\norm{h_n^2}{\A} \text.$ 
	\end{claim}	

For all $n\in\N$, by the Cauchy-Schwarz inequality, and since $h\in\sa{\A}$, we observe that:
	\begin{equation*}
		\mu(h_n)^2 \leq \mu(h_n^2) \leq \norm{h_n^2}{\A} \leq \norm{h_n}{\A}^2\text,
	\end{equation*}
	and since $\lim_{n\rightarrow\infty} \mu(h_n) = \lim_{n\rightarrow\infty}\norm{h_n}{\A} = 1$, by the squeeze theorem, $\lim_{n\rightarrow\infty} \mu(h_n^2) = 1$.
	
	Moreover, for all $n\in\N$, by the Leibniz property of $\Lip$ (since $\Re(h^2) = h^2$),
	\begin{equation*}
		\Lip(h_n^2) \leq 2 \norm{h_n}{\A} \Lip(h_n) \xrightarrow{n\rightarrow\infty} 0 \text.
	\end{equation*}
	
	Our first claim, applied to $(h_n^2)_{n\in\N}$, finishes the proof of our second claim.

\begin{claim}
	Every state of $\A$ extends uniquely to $\unital{\A}$ by fixing the value at $\unit_\A$ to be $1$. If $K\subseteq\StateSpace(\A)$ is a nonempty weak* compact subset of $\StateSpace(\A)$, then
	\begin{equation*}
		\lim_{n\rightarrow\infty} \sup\{ |\varphi((\unit_\A-h_n)^2)| : \varphi \in K \} = 0 \text.
	\end{equation*}
\end{claim}

	Let $K\subseteq\StateSpace(\A)$ be nonempty and weak* compact. By assumption, there exists $N\in\N$ such that, if $n\geq N$, then,
\begin{equation*}
	\Lip(h_n^2)\leq 2\Lip(h_n)\norm{h_n}{\A} \leq 1 \text.
\end{equation*}

	Hence, for all $\varphi,\psi \in K$:
	\begin{align}\label{approx-unit-equi-eq}
		|\varphi((1-h_n^2)) - \psi((1-h_n^2))| 
		&\leq \boundedLipschitz{\Lip,1}(\varphi,\psi) \max\{\norm{(1-h_n)^2}{\A},\Lip((1-h_n)^2)\} \nonumber \\
		&\leq (1+2M+M^2) \boundedLipschitz{\Lip}(\varphi,\psi) \text.
	\end{align}
	
	For any $a\in\sa{\A}$, we define $\widehat{a} : \varphi \in \StateSpace(\A) \mapsto \varphi(a)$.
	
	Expression \eqref{approx-unit-equi-eq} thus shows that $\left\{ \widehat{(1-h_n)^2} : \varphi \in K \mapsto \varphi((1-h_n)^2) : n \in \N \right\}$ is equicontinuous over the metric space $(K,\boundedLipschitz{\Lip})$, whose topology is the weak* topology --- hence, it is a compact metric space. On the other hand, by our first claim, if $\varphi\in K$, we have $\lim_{n\rightarrow\infty}\varphi(h_n)=\lim_{n\rightarrow\infty} \varphi(h_n^2) = 1$, and therefore:
	\begin{align*}
		\varphi((1-h_n)^2)
		&= 1-2\varphi(h_n) + \varphi(h_n^2) \\
		&\xrightarrow{n\rightarrow\infty} 1-2+1 = 0 \text. 
	\end{align*} 
	So $\left(\widehat{(1-h_n^2)}\right)_{n\in\N}$ converges pointwise to $0$ on $K$.

	Last,  $\widehat{(1-h_n^2)} \in [-(1+2M+M^2), 1+2M+M^2]$ for all $n\in\N$.

	Hence, by Arz{\'e}la-Ascoli, $\lim_{n\rightarrow\infty}\sup\{\varphi((1-h_n)^2) : \varphi \in K\} = 0$, as desired.
	
	\begin{claim}
		The sequence $(h_n)_{n\in\N}$ is an approximate unit in $\A$.
	\end{claim}

	It is sufficient to prove that $(h_n)_{n\in\N}$ converges to $1$ in the strict topology on $\unital{\A}$. By \cite{Latremoliere05b}, this follows from our previous claim directly. We offer a quick proof here, focused on our problem at hand, making it a little bit easier.

	Without loss of generality, by truncating the sequence $(h_n)_{n\in\N}$ is needed, we assume henceforth that $\Lip(h_n)\leq \frac{1}{2 M}$ for all $n\in\N$.

	Fix $h$ strictly positive in $\A$, with $\norm{h}{\A} = 1$. It is sufficient to prove that $\lim_{n\rightarrow\infty} \norm{h - h h_n}{\A} = 0$ by \cite[Lemma 2.3.6]{Wegge-Olsen92}. For each $n\in\N$, let $K_n \coloneqq \spectrum{h}\cap [\frac{1}{n+1},1]$; note that $K_n$ is compact in $\C$.  We now embed $\A$ in the enveloping Von Neumann algebra $\A''$ \cite{Pedersen79}, and set $p_n \coloneqq \chi_{K_n}(h)$ where 
	\begin{equation*}
		\chi_{K_n} : x \in \spectrum{h} \mapsto \begin{cases} 1 \text{ if $x \in K_n$, }\\ 0 \text{ otherwise.} \end{cases}
	\end{equation*}
	Thus $p_n$ is a projection in $\A''$ for all $n\in\N$; moreover, 
	\begin{equation*}
		\lim_{n\rightarrow\infty} \underbracket[1pt]{\norm{h - p_n h}{\A''}}_{\leq\frac{1}{n+1}} = 0
	\end{equation*}
	 by construction. We let $P_n \coloneqq \{ \varphi \in \StateSpace(\A) : \varphi(p_n) = 1 \}$, where we denote by the same symbol a state over $\A$ and its unique extension as a state to $\A''$.

	 Moreover, there exists $f_n \in C_0(\spectrum{h})$ with range in $[0,1]$, and such that $f_n(x) = 1$ for all $x \in K_n$, so that $f_n(h) p_n = p_n$. Let $g_n \coloneqq f_n(h)$. Let
	 \begin{equation*}
	 	Q_n \coloneqq \{\varphi \in \StateSpace(\A) : \varphi(g_n) = 1 \} \text.
	 \end{equation*}
	 The set $Q_n$ is closed in the weak* topology (as the pull back of the closet set $\{1\}$ by the continuous function $\widehat{g_n}$), and since $Q_n$ is a subset of the closed unit ball in $\A^\ast$, it is therefore compact in the weak* topology by Banach-Alaoglu Theorem.

	For all $a\in\sa{\A}$, we then record that:
	\begin{equation*}
		\norm{p_n a p_n}{\A''} \leq \sup\{|\varphi(a)| : \varphi \in P_n \} \leq \sup\{ |\varphi(a)| : \varphi \in Q_n \} \text.
	\end{equation*}
	Indeed, by Kadison functional calculus \cite[Theorem 3.10.3]{Pedersen79}, 
	\begin{equation*}
		\norm{p_n a p_n}{\A''} = \sup\{ |\theta(p_n a p_n)| : \theta\in\QuasiStateSpace(\A) \}\text.
	\end{equation*}
 If $\theta(p_n) = 0$, then $\theta(p_n a p_n) = 0$, so 
	\begin{equation*}
	\norm{p_n a p_n}{\A''} = \sup\{ |\theta(p_n a p_n)| : \theta\in\QuasiStateSpace(\A), \theta(p_n) > 0 \}\text.
	\end{equation*}
	
	Now, let $\theta\in\QuasiStateSpace(\A)$ with $\theta(p_n) > 0$. Let $\varphi \coloneqq \frac{1}{\theta(p_n)}\theta(p_n \cdot p_n)$. Since $\theta$ is positive on $\A''$, so is $\varphi$. Therefore, $\varphi$ restricts to a positive linear functional on $\A$, and moreover, $\varphi(g_n) = \frac{1}{\theta(p_n)}\theta(p_n g_n p_n) = \frac{\theta(p_n)}{\theta(p_n)} = 1$. So $\varphi \in \StateSpace(\A)$. Moreover, $\varphi(p_n) = 1$ as well, so, $\varphi \in P_n$. Yet,
	\begin{equation*}
		|\theta(p_n a p_n)| = \theta(p_n)|\varphi(a)| \leq |\varphi(a)| \text.
	\end{equation*}
	This proves our inequality.

	Let $\varepsilon > 0$. There exists $N_1$ such that
	\begin{equation*}
		\norm{h - p_n h}{\A} < \frac{\varepsilon}{3M} \text.
	\end{equation*}
	Therefore:
	\begin{align}\label{h-hhn-eq}
		\norm{h - h h_n}{\A}
		&\leq \norm{h - p_n h}{\A} + \norm{p_n h  - p_n h h_n}{\A} + \norm{p_n h h_n - h h_n}{\A} \nonumber\\
		&\leq \frac{\varepsilon}{3M} + \sqrt{\norm{p_n(h- h h_n)(h - h_n h) p_n}{\A''}} + \frac{\varepsilon}{3 M}\norm{h_n}{\A} \\
		&\leq \frac{2\varepsilon}{3} + \sup\{|\varphi((h- h h_n)(h - h_n h))|:\varphi \in P_n \}^{\frac{1}{2}} \nonumber \\
		&= \frac{2\varepsilon}{3} +  \sup\{|\varphi(h(\unit_\A- h_n)^2 h)|:\varphi \in P_n \}^{\frac{1}{2}} \nonumber \text.
	\end{align}
	
	If $\varphi\in\StateSpace(\A)$ and $h \in \sa{\A}$, then $\varphi(h\cdot h)$ is also a positive linear functional, and its norm is therefore $\varphi(h^2)$. Now, if $\varphi \in \StateSpace(\A)$, and $\varphi(p_n) = 1$, we note by the Cauchy-Schwarz inequality that, for all $a\in \A$:
	\begin{align*}
		|\varphi(p_n a) - \varphi(a)| 
		&=|\varphi((p_n-\unit_\A)a)| \\
		&\leq \norm{a}{\A}\sqrt{\varphi( (p_n-\unit_\A)^2 )} \\
		&\leq \norm{a}{\A}\sqrt{\varphi(p_n^2) - 2\varphi(p_n) + 1}= 0 \text{ since $p_n^2 = p_n$ and $\varphi(p_n)=1$.} 
	\end{align*} 
	Similarly, $\varphi(a p_n) = \varphi(a)$ and thus $\varphi(p_n a p_n) = \varphi(a)$.
	
	Moreover, since $[p_n,h] = 0$, for all $c \in \A$,
	\begin{align*}
		|\varphi(h c h)| 
		&= |\varphi(p_n h c h p_n)| = |\varphi(h p_n c p_n h)| \\ 
		&\leq \varphi(h^2) \norm{p_n c p_n}{\A''} \leq \sup\{|\varphi(c)| : \varphi\in Q_n \} \text.
	\end{align*}

By the previous claim,
\begin{equation}
	\lim_{n\rightarrow\infty} \sup\{|\varphi((1-h_n)^2)|:\varphi \in  Q_n\} = 0 \text,
\end{equation}
so there exists $N_2 \in \N$ such that if $n\geq N_2$, then 
\begin{align*}
	 \sup\{|\varphi(h(1- h_n)^2 h)|:\varphi \in Q_n \} 
	 &\leq \norm{p_n(1-h_n)^2 p_n}{\A''} \\
	 &\leq \sup\{|\varphi((1-h_n)^2)|:\varphi \in  Q_n\} < \frac{\varepsilon^2}{9} \text.
\end{align*}
Therefore, by Expression \eqref{h-hhn-eq}, if $n\geq \max\{N_1,N_2\}$,
\begin{equation*}
	\norm{h - h h_n}{\A} < \varepsilon \text.
\end{equation*}

Therefore, $(h_n)_{n\in\N}$ is an approximate unit for $\A$, as claimed.
\end{proof}

\begin{remark}
	Central to the proof of Theorem (\ref{approx-unit-thm}) is the fact that the Fortet-Mourier distance induces the weak* topology on the state space.
\end{remark}

\begin{remark}
	We could simplify somewhat the proof of Theorem (\ref{approx-unit-thm}) if we were to assume that the sequence $(h_n)_{n\in\N}$ is contained in some \emph{topography}, as defined later on and in \cite{Latremoliere12b}, but we see here that this is not necessary, though it will be the case when working on the notion of Gromov-Hausdorff convergence.
\end{remark}

Our work will focus on {\lcqms s} with an approximate identify obtained from Theorem (\ref{approx-unit-thm}).

\begin{definition}\label{pqms-def}
	An \emph{pinned \pqms} $(\A,\Lip,\mu)$ is an pinned {\lcqms} which contains a \lipunit{\Lip}{\mu}.

	An ordered pair $(\A,\Lip)$ is called a \emph{\pqms} when there exists $\mu \in \StateSpace(\A)$ such that $(\A,\Lip,\mu)$ is an pinned {\pqms}. 
\end{definition}

We check that indeed, our definition restricts to the usual notion of a proper metric space in the classical picture.

\begin{theorem}\label{classical-proper-thm}
Let  $(X,d)$ be a locally compact metric space. The following assertions are equivalent:
\begin{enumerate}
	 \item $(X,d)$ is proper 
	 \item for all $x\in X$, the triple $(C_0(X),\Lip_d,\delta_x)$ is a pinned {\pqms},
	 \item there exists $x\in X$ such that $(C_0(X),\Lip_d,\delta_x)$ is a pinned {\pqms}.
	 \end{enumerate}
\end{theorem}

\begin{proof}
	Assume first that $(X,d)$ is proper. If $(X,d)$ is compact, set $h_n = 1$ for all $n\in\N$: then $(h_n)_{n\in\N}$ meets all our requirements for a \lipunit{\Lip_d}{\delta_x}, for any $x\in X$. 
	
	Assume now that $(X,d)$ is proper but not compact. Fix $x\in X$. For each $n\in\N$, set 
	\begin{equation*}
		h_n : t \in X \mapsto \max\left\{ 1 - \frac{1}{n+1} d(x,t), 0 \right\}\text.
	\end{equation*}
	 The function $h_n$ is continuous over $X$ and supported on the closed ball $X[x,n+1]$, which is compact. So $h_n \in C_0(X)$. Moreover, $\Lip_d(h_n) \leq \frac{1}{n+1}$ and $h_n(x) = 1 = \norm{h_n}{\A}$. Thus, $(C(X),\Lip_d)$ is a {\pqms} .

	Therefore, (1) implies (2) which trivially implies (3).

	Let us now assume (3). Thus $(C_0(X),\Lip_d)$ contains a \lipunit{\Lip_d}{\delta_x} $(h_n)_{n\in\N}$ for some $x\in \R$.  Let $R > 0$. Since $\lim_{n\rightarrow\infty} \Lip_d(h_n) = 0$, there exists $N\in\N$ such that, for all $n \geq N$, we have $\Lip_d(h_n) \leq \frac{1}{2R}$. Thus, for all $t \in X[x,R]$, we gave $|1-h_n(t)| = |h_n(x) - h_n(t)| \leq \frac{1}{2R} d(x,t) \leq \frac{R}{2R} = \frac{1}{2}$. So $X[x,R] \subseteq h_n^{-1}([0,\frac{1}{2}]$. Since $h_n \in C_0(X)$, the function $h_n$ is proper; since $[0,\frac{1}{2}]$ is compact, we conclude that $X[x,R]$ is a closed subset of a compact subset of $(X,d)$, and thus it is compact as well. This is sufficient to conclude that $(X,d)$ is proper, since any other closed ball in $(X,d)$, which is closed, is contained in a ball of center $x$, which is compact. So $(X,d)$ is proper, as claimed.
\end{proof}

We notice, following a similar argument as the proof of Theorem (\ref{approx-unit-thm}), that just as in the classical picture, a {\pqms} is either compact, or has infinite diameter.
\begin{corollary}
	If $(\A,\Lip)$ is a {\pqms}, and if $\qdiam{\A,\Lip}<\infty$, then $(\A,\Lip)$ is a {\qcms}. 
\end{corollary}

\begin{proof}
Assume $\qdiam{\A,\Lip}<\infty$. Let $(h_n)_{n\in\N}$ be some \lipunit{\Lip}{\mu} for some $\mu\in\StateSpace(\A)$. Let $\varepsilon > 0$. There exists $N\in\N$ such that, if $n\geq N$, then $\Lip(h_n) < \frac{1}{2(\qdiam{\A,\Lip}+1)}\varepsilon$. Since $\lim_{n\rightarrow\infty} \mu(h_n)=1$, there exists $N'\in\N$ such that, if $n\geq N'$, then $|1-\mu(h_n)| < \frac{\varepsilon}{2}$. Thus, by definition of the {\MongeKant} $\Kantorovich{\Lip}$, for all $n\geq \max\{N,N'\}$, and for all $\varphi\in\StateSpace(\A)$,
	\begin{align*}
		|\varphi(h_n) - 1|
		&\leq |\varphi(h_n)-\mu(h_n)| + |1-\mu(h_n)| \\
		&\leq \qdiam{\A,\Lip} \Lip(h_n) +  |1-\mu(h_n)| < \varepsilon \text.
	\end{align*}
	Hence, for all $m,n \geq \max\{N,N'\}$,
	\begin{align*}
		\norm{h_n - h_m}{\A} 
		&=\sup\{|\varphi(h_n) - \varphi(h_m)| : \varphi\in\StateSpace(\A) \} \\
		&\leq 2\varepsilon \text.
	\end{align*}
	Thus $(h_n)_{n\in\N}$ is Cauchy, so converges to some $g \in \sa{\A}$ since $\A$ is complete. We also note that $\norm{g}{\A} = 1$, so $1 = \varphi(g)^2 \leq \varphi(g^2) \leq \norm{g^2}{\A} = \norm{g}{\A}^2 = 1$ since $g \in\sa{\A}$. So $\varphi(g^2) = 1$ for all $\varphi \in \StateSpace(\A)$.
	
	Let now $a\in C^\ast(g)$. for all $\varphi \in \StateSpace(\A)$, we have
	\begin{align*}
		|\varphi(a-g a)| \leq \sqrt{\varphi((1-g)^2)} \norm{a}{\A} \leq (1-2\varphi(g) +\varphi(g^2)) \norm{a}{\A}= 0
	\end{align*}
	so $a=ga$ for all $a\in C^\ast(g)$. Of course, $C^\ast(g)$ is Abelian since $g\in\sa{\A}$, so $g$ is the unit of $\C^\ast(g)$. Since $g$ is strictly positive in $\A$, this implies in turn that $g$ is the unit of $\A$.
\end{proof}

\medskip

The Gromov-Hausdorff distance is defined between \emph{pointed} proper metric spaces. Our concept of an pinned {\pqms} is formally very similar, but for the theory we shall develop in this pages, we will actually add a few requirements to an pinned {\pqms}. The reason stems from the very reason why a base point is chosen when working with proper metric spaces and the Gromov-Hausdorff distance: the point is used as the center of balls, whose Hausdorff distance to some other set is the basis for the Gromov-Hausdorff distance. It is of course an very serious concern when moving from the classical to the quantum picture, where such a local notion is no longer available --- a key feature of the quantum world. Informally, our approach is to replace balls, or more generally, subsets by sets of quasi-states of the form $\{ \varphi(e\cdot e) : \varphi \in \StateSpace(\A)\}$, where $e$ is a reasonable, self-adjoint element in the domain of the Lipschitz seminorm. In particular, it is natural to ask that if $\mu$ is a chosen pin, then $\mu(e)$ is almost $1$. This works fine when dealing with one such set, but we will in fact need to work with entire collections of such sets which ``filter'' the entire space, again heuristically. This can be done, thanks to Theorem (\ref{approx-unit-thm}), by choosing elements $e$ as above with ever smaller Lipschitz seminorm. However, if we pick two such elements which do not commute, then in some sense, while each can be used to define a sort of local concept, their product can not, or said differently, these localizations are not compatible. This turns out to be an issue in the proof of the triangle inequality for our upcoming analogue of the Gromov-Hausdorff distance.

This issue is actually related to similar ones previously encountered. In our discussion of the properties of the {\MongeKant} in \cite{Latremoliere12b}, we already introduced the following notion, which will serve us now to avoid the issue discussed above, and it is generally known that such notions as ``going to infinity'' within a non-Abelian C*-algebra do not appear to be intrinsic \cite{Akemann89}. The idea is to single out a particular Abelian C*-subalgebra of a {\pqms}. 

With this in mind, we put forth the following definition.

\begin{definition}
	A \emph{topography} $\M$ of an {\lcqms} $(\A,\Lip)$ is an Abelian C*-subalgebra of $\A$ containing a strictly positive $h  \in \dom{\Lip}$ such that $\{ h a h \in \unital{\A} : a=b+t\unit_\A, b\in\domsa{\Lip}, \Lip(b) \leq 1, \mu(a) =-t\}$ is bounded for some character $\mu$ of $\M$.
	
	A \emph{topographic {\lcqms}} $(\A,\Lip,\M)$ is a {\lcqms} $(\A,\Lip)$ and a topography $\M$ of $\A$.
\end{definition}

\begin{remark}
	It is immediate, from the definition of a topographic {\lcqms} $(\A,\Lip,\M)$, that $(\M,\Lip)$ is a {\lcqms} as well. Moreover, note that $\M\cap\domsa{\Lip}$ contains an approximate unit for $\A$, as a consequence of the Leibniz property.
\end{remark}

With our proposed notion of a topography, we can ask that the chosen pin of a {\pqpms} endowed with a topography will indeed be a point, at least, from the perspective of this additional localization structure. We conclude this section with the class of quantum metric spaces of interest to us in this project.

\begin{definition}\label{pqpms-def}
	A \emph{\pqpms} $(\A,\Lip,\M,\mu)$ is an pinned {\lcqms} $(\A,\Lip,\mu)$ and a topography $\M$ of $(\A,\Lip)$ such that:
	\begin{enumerate}
		\item $\mu$ restricts to a character of $\M$,
		\item there exists a \lipunit{\Lip}{\mu} in $\M$.
	\end{enumerate}
\end{definition}

Of course, in particular, if $(\A,\Lip,\M,\mu)$ is a {\pqpms}, then $(\A,\Lip,\mu)$ is an pinned {\pqms}, and $(\M,\Lip_{|\M},\M,\mu)$ is an Abelian {\pqpms}.

\begin{remark}
	If $(\A,\Lip)$ is a {\qcms}, then $(\A,\Lip,\C\unit_\A,\mu)$ is a {\pqpms} for any $\mu \in \StateSpace(\A)$. That said, when approximating non-compact {\pqpms s} in the sense of our new hypertopology, we will want to allow for more general choices of topographies, even for {\qcms s}.
\end{remark}

\section{The Local Quantum Metametrics}

We now develop the principal new idea of the present work: a form of convergence, inspired and analogous to Gromov-Hausdorff convergence, for {\pqpms s}. More specifically, we shall introduce a topology on the class of all {\pqpms s}. To this end, we introduce a class function we refer to as a metametric, which gives rise to inframetrics, which in turn induce our sought-after topology. There are two very important features, the first of which is explored in this section, which we require of our class function: first, that it satisfies a (relaxed) form of the triangle inequality, to ensure that indeed, it gives rise to a topology in a reasonable and natural manner, and second, that it is zero exactly between fully quantum isometric {\pqpms s}. We will address this second property in the next section. Both of these requirements are non-trivial, each demanding a bit of effort to be established. The definition of this class function, and the proof of these two fundamental properties, is the topic of this and the next section. We note that, while our work on the propinquity over the class of {\qcms s} obviously informed us on some aspects of the following construction and work, very significant differences in the methods and tools employed to arrive to the desired result are necessary --- such as replacing compactness with completeness in some arguments.

\medskip

We start by defining our notion of isometric embedding for two {\pqpms s}, which we call tunnels, and associate to each such tunnel a non-negative number, called the extent of the tunnel, which quantifies how far the two {\pqpms s} are from the perspective of said tunnel.

\subsection{Tunnels}

We now need a quantum analogue of an isometry. What we desire here is a *-morphism whose dual map induces an isometry between state spaces endowed with the Forter-Mourier distances. Recall that a *-morphism is proper when it maps an approximate unit to an approximate unit.

\begin{definition}\label{M-isometry-def}
Let $(\A,\Lip_\A)$ and $(\D,\Lip_\D)$ be {\lcqms s}. A quantum $M$-isometry $\pi:(\D,\Lip_\D)\twoheadrightarrow(\A,\Lip_\A)$, for some $M\geq 1$, is a proper *-epimorphism such that $\dom{\Lip_\A} = \pi(\dom{\Lip_\D})$, while:
\begin{equation*}
	\forall a \in \domsa{\Lip_\A} \quad \norm{a}{\Lip_\A,M} \coloneqq \inf\left\{  \norm{d}{\Lip_\D,M} :d\in\domsa{\Lip_\D}, \pi(d) = a \right\}\text,
\end{equation*}
and 
\begin{equation*}
	\forall d \in \domsa{\Lip_\D} \quad \Lip_\A\circ\pi(d) \leq \Lip_\D(d) \text.
\end{equation*}
\end{definition}

If $\pi : (\D,\Lip_\D) \rightarrow (\A,\Lip_\A)$, then it is easy to check that the map $\pi^\ast:\varphi\in\QuasiStateSpace(\A) \mapsto \varphi\circ\pi\in\QuasiStateSpace(\D)$ is an isometry from $(\QuasiStateSpace(\A),\boundedLipschitz{\Lip_\A,M})$ into $(\QuasiStateSpace(\D),\boundedLipschitz{\Lip_\D,M})$.

\begin{remark}\label{quantum-M-iso-composition-rmk}
	Quantum $M$-isometries can be used as morphisms for a category over the class of {\lcqms s}. Indeed, the identity is trivially a $M$-isometry. Let $(\A,\Lip_\A)$, $(\B,\Lip_\B)$ and $(\D,\Lip_\D)$ be three {\lcqms s}. If $\pi : (\A,\Lip_\A) \rightarrow (\B,\Lip_\B)$ and $\rho: (\B,\Lip_\B) \rightarrow (\D,\Lip_\D)$, then $\rho\circ\pi : \A \rightarrow \D$ is, of course, a *-epimorphism, and by construction, $\dom{\Lip_\D} = \rho\circ\pi(\dom{\Lip_\A})$. Moreover, let $d \in \domsa{\Lip_\D}$ and $\varepsilon > 0$. Then, since $\rho$ is a quantum $M$-isometry, there exists $b \in \domsa{\Lip}$ such that $\rho(b) = d$ and $\norm{d}{\Lip_\B,M} \leq \norm{b}{\Lip_\B,M} + \frac{\varepsilon}{2}$. Similarly, there exists $a\in \domsa{\Lip_\A}$ such that $\pi(a) = b$ and $\norm{a}{\Lip_\A,M} \leq \norm{b}{\Lip_\B,M} + \frac{\varepsilon}{2} \leq \norm{d}{\Lip_\D} + \varepsilon$. On the other hand, for all $a \in \domsa{\Lip}$, we note that $\Lip_\D(\rho\circ\pi(a)) \leq \Lip_\B(\pi(a)) \leq \Lip_\A(a)$ and of course, $\norm{\rho\circ\pi(a)}{\D} \leq \norm{a}{\D}$, so that $\norm{\rho\circ\pi(a)}{\Lip_\D,M} \leq \norm{a}{\Lip_\D,M}$. Therefore, as needed, we have shown that $\norm{d}{\Lip_\D,M} = \inf\{ \norm{a}{\Lip_\A,M} : a\in\domsa{\Lip}, \rho\circ\pi(a) = d \}$, as required.
\end{remark}

When a {\lcqms} is endowed with a chosen topography, it then becomes reasonable to work with isometries which map topographies to topographies. We actually require a stronger property, needed in both our proofs to establish the relaxed triangle inequality and the coincidence property. Essentially, while in general, one may not lift an approximate unit to an approximate unit via a *-epimorphism, we ask that we can at least lift Lipschitz pinned exhaustive sequnences from the topography. This property is invisible in the compact case (as it is trivially met), and equally trivially met in the classical case, yet it becomes a desirable requirement in general.

\begin{definition}\label{topographic-isometry-def}
	Let $(\A,\Lip_\A,\M_\A)$ and $(\D,\Lip_\D,\M_\D)$ be two topographic {\lcqms s}. A \emph{topographic quantum $M$-isometry} $\pi:(\D,\Lip_\D,\M_\D)\rightarrow(\A,\Lip_\A,\M_\A)$ is a quantum $M$-isometry such that:
	\begin{enumerate}
		\item $\pi(\M_\D) \subseteq \M_\A$, 
		\item For all $a\in\domsa{\Lip_\A}\cap\M_\A$ and for all $\varepsilon > 0$, there exists $d\in\M_\D\cap\domsa{\Lip_\A}$ such that
		\begin{equation*}
			\norm{d}{\D} \leq \exp(\varepsilon) \norm{a}{\A} \text{ and }\Lip_\D(d)\leq \exp(\varepsilon) \Lip_\A(a) \text.
		\end{equation*} 
	\end{enumerate}
\end{definition}

When $(\A,\Lip_\A)$ and $(\D,\Lip_\D)$ are {\qcms s}, then any quantum isometry $\pi : (\D,\Lip_\D,\C\unit_\D) \rightarrow(\A,\Lip_\A,\C\unit_\A)$ is automatically a topographic $(\qdiam{\A,\Lip} + 1)$-isometry (note in particular that as a proper *-morphism, it is unital in this case); thus what appears as additional constraints here compared to \cite{Latremoliere13}, isn't. This observation is one justification for our introduction of the real number $M$ in Definition (\ref{M-isometry-def}).

\begin{remark}
	If $\pi:(\D,\Lip_\D,\M_\D)\rightarrow(\A,\Lip_\A,\M_\A)$ is a topographic quantum isometry, then in particular, if $\pi_\M$ is the restriction of $\pi$ to $\M$, then $\pi^\ast$ is an isometry from $(\StateSpace(\A),\Kantorovich{\Lip_\A})$ to $(\StateSpace(\D),\Kantorovich{\Lip_\D})$.
\end{remark}

Topographic quantum $M$-isometries can also serve as morphisms for a category over the class of {\lcqms s} endowed with topographies. Notably, we observe the following simple, yet helpful, property.
 
\begin{lemma}
	If $\pi : (\A,\Lip_\A,\M_\A) \rightarrow (\B,\Lip_\B,\M_\B)$ and $\rho:(\B,\Lip_\B,\M_\B)\rightarrow(\D,\Lip_\D,\M_\D)$ are two topographic quantum $M$ isometries, then so is $\rho\circ\pi$.
\end{lemma}

\begin{proof}
	By Remark \ref{quantum-M-iso-composition-rmk}, the map $\rho\circ\pi$ is a quantum $M$-isometry. Moreover, $\rho\circ\pi(\M_\A) \subseteq \M_\D$ as required.

	Let $d\in\domsa{\Lip_\D}\cap\M_\D$. Let $\varepsilon > 0$. By definition, there exists $b \in \M_\B\cap\domsa{\Lip_\B}$ such that $\rho(b) = d$, $\norm{b}{\B}\leq \exp\left(\frac{\varepsilon}{2}\right) \norm{d}{\D}$ and $\Lip_\B(b)\leq \exp\left(\frac{\varepsilon}{2}\right) \Lip_\D(d)$. Again by definition, since $b\in \M_\B\cap\domsa{\Lip_\B}$, there exists $a\in\M_\A\cap\domsa{\Lip_\A}$ such that $\pi(a) = b$, while
	\begin{equation*}
		\norm{a}{\A}\leq \exp\left(\frac{\varepsilon}{2}\right)\norm{b}{\B}\leq \exp\left(\frac{\varepsilon}{2}\right)\exp\left(\frac{\varepsilon}{2}\right)\norm{d}{\D} \leq \exp(\varepsilon)\norm{d}{\D} \text,
	\end{equation*}
	and 
	\begin{equation*}
		\Lip_\A(a) \leq \exp\left(\frac{\varepsilon}{2}\right) \Lip_\B(b) \leq \exp\left(\frac{\varepsilon}{2}\right)\exp\left(\frac{\varepsilon}{2}\right)\Lip_\D(d)  = \exp(\varepsilon)\Lip_\D(d) \text.
	\end{equation*}
As $\varepsilon > 0$, our proof is concluded.
\end{proof}
	
\medskip

We are now ready to introduce the appropriate notion of isometric embedding we shall use to define our new topology.
	
\begin{definition}\label{tunnel-def}
	Let $\mathds{A} \coloneqq (\A,\Lip_\A,\M_\A,\mu_\A)$ and $\mathds{B} \coloneqq  (\B,\Lip_\B,\M_\B,\mu_\B)$ be two {\pqpms s}, and $M\geq 1$. 	An \emph{$M$-tunnel} $(\D,\Lip_\D,\M_\D,\pi_\A,\pi_\B,e)$ from $\mathds{A}$ to $\mathds{B}$ is given by:
	\begin{enumerate}
		\item a pinned, topographic {\lcqms} $(\D,\Lip_\D,\M_\D,\mu_\A\circ\pi_\A)$,
		\item an element $e \in\domsa{\Lip_\D}\cap\M_\D$,
		\item two topographic quantum $M$-isometries $\pi_\A:  (\D,\Lip_\D,\M_\D)\twoheadrightarrow(\A,\Lip_\A,\M_\A)$ and $\pi_\B:(\D,\Lip_\D,\M_\D)\twoheadrightarrow(\B,\Lip_\B,\M_\B)$.
	\end{enumerate}
\end{definition}

\begin{notation}
	If $\tau\coloneqq(\D,\Lip_\D,\M_\D,\pi_\A,\pi_\B,e)$ is a tunnel from $\mathds{A}$ to $\mathds{B}$, we may also write this tunnel as:
	\begin{equation*}
		\tunnel{\tau}{\mathds{A}}{\pi_\A}{(\D,\Lip_\D,\M_\D,e)}{\pi_\B}{\mathds{B}} \text.
	\end{equation*}
	Moreover, the domain $\dom{\tau}$ is $\mathds{A}$, while the codomain $\codom{\tau}$ is $\mathds{B}$.
\end{notation}

Tunnels have an obvious symmetry; we will use the following notation.
\begin{notation}
If
\begin{equation*}
	\tunnel{\tau}{\mathds{A}}{\pi}{(\D,\Lip_\D,\M_\D,e)}{\rho}{\mathds{B}} \text,
\end{equation*}
then the \emph{reversed tunnel} $\tau^{-1}$ is defined by
\begin{equation*}
	\tunnel{\tau^{-1}}{\mathds{B}}{\rho}{(\D,\Lip_\D,\M_\D,e)}{\pi}{\mathds{A}} \text.
\end{equation*}
\end{notation}

We did not explicitly assume that a tunnel is constructed using a {\pqpms} --- only a pinned {\lcqms}. Instead, we  prove that it must be. Arguably, this construction, and a similar argument later on, are both needed specifically to establish a form of the triangle inequality, and are the only places where we use the special lifting requirement between topographies in Definition (\ref{topographic-isometry-def}).

\begin{lemma}\label{lifting-lipunit-lemma}
	Let $(\A,\Lip_\A,\M_\A,\mu_\A)$ be a {\pqpms} and $(\D,\Lip_\D,\M_\D)$ be a topographic {\lcqms} for which $\mu_\A\circ\pi$ is a pin. If $\pi : (\D,\Lip_\D,\M_\D)\rightarrow(\A,\Lip_\A,\M_\A)$ is a topographic quantum isometry, and if $(e_n)_{n\in\N}$ is a \lipunit{\Lip_\A}{\mu_\A} contained in $\M_\A$, then there exists a \lipunit{\Lip_\D}{\mu_\A\circ\pi} $(d_n)_{n\in\N}$ contained in $\M_\D$, such that $\pi(d_n) = e_n$ for all $n\in\N$. In particular, $(\D,\Lip_\D,\M_\D,\mu_\A\circ\pi)$ is a {\pqpms}.
\end{lemma}

\begin{proof}
	Let $(e_n)_{n\in\N} \in \M_\A$ be a \lipunit{\Lip_\A}{\mu_\A}.
	
	Since $\pi_\A$ is a topographic quantum $M$-isometry, for each $n\in\N$, there exists $d_n \in \M_\D\cap\domsa{\Lip_\D}$ such that $\pi(d_n) = e_n$, while $\norm{e_n}{\D} \leq \norm{d_n}{\D} \leq \exp(\frac{1}{n+1}) \norm{e_n}{\D}$ and $\Lip_\D(d_n) \leq \exp(\frac{1}{n+1}) \Lip_\A(e_n)$. 
	
	By construction, $\lim_{n\rightarrow\infty} \norm{d_n}{\D} = \lim_{n\rightarrow\infty} \norm{e_n}{\A} = 1$, and $\lim_{n\rightarrow\infty} \Lip(d_n) = 0$. Now, $\mu_\A\circ\pi_\A(d_n) = \mu_\A(e_n)$ for all $n\in\N$, so $\lim_{n\rightarrow\infty} \mu_\A\circ\pi_\A(d_n) = 1$.

	Since $\mu_\A\circ\pi$ is assumed to be a pin for $(\D,\Lip_\D)$, our lemma is now proven.
\end{proof}

\begin{corollary}\label{pqpms-tunnel-cor}
	If $(\D,\Lip,\M,\pi,\rho,e)$ is a tunnel with domain $(\A,\Lip_\A,\M_\A,\mu_\A)$, then $(\D,\Lip,\M,\mu_\A\circ\pi)$ is a {\pqpms}.
\end{corollary}

\begin{proof}
	First, note that $\mu_\A\circ\pi$ is a character on $\M_\D$ since $\mu_\A$ is a character of $\M_\A$, and $\pi$ is a *-morphism such that $\pi(\M_\D)\subseteq\M_\A$. Our corollary then follows from Lemma (\ref{lifting-lipunit-lemma}) and Definition (\ref{pqpms-def}), since $\M_\A$ contains a \lipunit{\Lip_\A}{\mu_\A}.
\end{proof}

\begin{remark}
	We note in particular, that topographic quantum isometries between {\pqms s} are always proper morphisms.
\end{remark}	

\medskip

The purpose of introducing tunnels is that they place two {\pqpms s} within a common {\lcqms} (dually) and allows us to define a number which quantifies how far apart they are from this particular context. It involves not only some metric-related quantify involving state spaces and quasi-state spaces as expected, but also some control on the various ingredients making a tunnel.

\begin{notation}
	If $K \subseteq \A^\ast$ and $a, b \in \A$, then $a K b \coloneqq =\{ \varphi (a \cdot b) : \varphi \in K \}$. If, in particular, $a \in \sa{\A}$ with $\norm{a}{\A}\leq 1$, and $K\subseteq\QuasiStateSpace(\A)$, then $a K a \subseteq \QuasiStateSpace(\A)$.
\end{notation}

\begin{notation}
	Let $(X,d)$ be a metric space. If $A,B \subseteq E$ are two nonempty subsets of $E$, we write $d(x,A) \coloneqq \inf\{ d(x,a) : a\in A \}$, and if $\sup_{x\in B} d(x,A) \leq \varepsilon$ for some $\varepsilon \geq 0$, we write
	\begin{equation*}
		B \subseteq_\varepsilon^d A\text.
	\end{equation*}
	
	By definition, the \emph{Hausdorff distance} between $A$ and $B$ is defined by: 
	\begin{equation*}
		\Haus{d}(A,B) = \inf\{\varepsilon > 0 : A\subseteq_\varepsilon^d B \text{ and } B\subseteq_\varepsilon^d A \} \text.
	\end{equation*}
	
	$\Haus{d}$ thus defined is a distance on the set of closed subsets of $(X,d)$. More generally, $\Haus{d}(A,B) = 0$ if, and only if, the closure of $A$ equals to the closure of $B$ --- that is, $\Haus{d}$ is a metric up to closure equality. We will allow taking the Hausdorff distance between non closed subsets in this paper.
	
	 $\Haus{d}$ is complete whenever $d$ is complete, and induces a compact topology whenever $(X,d)$ is compact. 
	
	When $(X,d)$ is a normed space $E$, then we will write $\Haus{E}$ for the Hausdorff distance induced by the norm of $E$.
\end{notation}

\begin{notation}
	If $\pi : \A\rightarrow\B$ is a *-morphism, then $\pi^\ast : \B^\ast \mapsto \A^\ast$ is defined by $\varphi \in \B^\ast \mapsto \pi^\ast(\varphi)\coloneqq\varphi\circ\pi$.
\end{notation}

\begin{definition}\label{extent-def}
	Let $(\A,\Lip_\A,\M_A,\mu_\A)$ and $(\B,\Lip_\B,\M_\B,\mu_\B)$ be two {\pqpms s}. Let $M\geq 1$. The extent of an $M$-tunnel $\tau\coloneqq(\D,\Lip,\M,\pi_\A,\pi_\B,e)$ is the maximum of:
	\begin{enumerate}
		\item 
	\begin{equation*}
		\inf\Bigg\{ \varepsilon > 0 \; :  \; e \QuasiStateSpace(\D) e \subseteq_\varepsilon^\boundedLipschitz{\Lip,M} \pi_\A^\ast(\QuasiStateSpace(\A))\text{, and } e \QuasiStateSpace(\D) e\subseteq_\varepsilon^\boundedLipschitz{\Lip,M}\pi_\B^\ast(\QuasiStateSpace(\B)) \Bigg\} \text,
	\end{equation*}
		\item $\Kantorovich{\Lip}(\mu_\A\circ\pi_\A,\mu_\B\circ\pi_\B)$,
		\item $\max\{ 2 M \Lip(e), |1-\mu_\A\circ\pi(e)|, |1-\mu_\B\circ\rho(e)| \}$.
\end{enumerate}
\end{definition}

\medskip

\medskip

The reader may have observed that, while the extent of a tunnel controls the Lipschitz seminorm of the chosen distinguished element, the definition of the extent seems to ignore its norm. The following lemma explains this observation, and is extremely important, both for the triangle inequality and for the coincidence property.  

\begin{lemma}\label{key-lemma}
If $\tau \coloneqq(\D,\Lip,\M,\pi,\rho,e)$ is an $M$-tunnel with finite extent, for some $M\geq 1$, then 
	\begin{equation*}
		\norm{e}{\D} \leq \sqrt{1 + \frac{\tunnelextent{\tau}}{M}} \leq \sqrt{1+\tunnelextent{\tau}} \text.
	\end{equation*}
\end{lemma}

\begin{proof}
	By Corollary (\ref{pqpms-tunnel-cor}), there exists an approximate unit $(h_n)_{n\in\N}$ in $\domsa{\Lip_\D}$ such that $\lim_{n\rightarrow\infty}\norm{e_n}{\D} = 1$ and $\lim_{n\rightarrow\infty} \Lip(e_n)=0$, so that $\norm{e_n}{\Lip,M} \xrightarrow{n\rightarrow\infty} \frac{1}{M}$.
	
	Let $t > \tunnelextent{\tau}$. Let $\varphi \in\StateSpace(\D)$. By definition of the extent, there exists $\psi \in \QuasiStateSpace(\A)$ such that $\boundedLipschitz{\Lip,M}(\varphi(e\cdot e),\psi) < t \text.$
	\begin{align*}
		|\varphi(e h_n e)| 
		&\leq |\varphi(e h_n e) - \psi(h_n)| + |\psi(h_n)| \\
		&\leq \boundedLipschitz{\Lip,M}(\varphi(e\cdot e),\psi)\norm{h_n}{\Lip} + \norm{h_n}{\D} \\
		&\leq \norm{h_n}{\Lip,M} t + \norm{h_n}{\D} \xrightarrow{n\rightarrow\infty} \frac{t}{M} + 1 \text.  
	\end{align*}
	
	Therefore 
	\begin{align*}
		\norm{e^2}{\D} 
		&\leq \limsup_{n\rightarrow\infty}\left( \norm{e^2 - (e h_n)e}{\D} + \norm{e h_n e}{\D} \right) \\
		&= 0 + \frac{t}{M} + 1 \text,
	\end{align*}
	since $(h_n)_{n\in\N}$ is an approximate unit, and since $e h_n e$ is self-adjoint for all $n\in\N$.

	Since $t > \tunnelextent{\tau}$ was arbitrary, we thus get:
	\begin{equation*}
		\norm{e}{\D}^2 = \norm{e^2}{\D}  \leq 1 + \frac{\tunnelextent{\tau}}{M} \text,
	\end{equation*}
	as claimed.
\end{proof}
\medskip

A central feature of tunnels is that they may be ``approximately'' composed to give rise to new tunnels with some control over their extent. This result plays the key role in establishing a form of triangle inequality for our upcoming metametrics.

\begin{lemma}\label{triangle-lemma}
	Let $M\geq 1$. For all $\varepsilon > 0$, and for all {\pqpms s} $\mathds{A}$, $\mathds{B}$ and $\mathds{E}$,  if $\tau_1$ is an $M$-tunnel from $\mathds{A}$ to $\mathds{B}$, and if $\tau_2$ is an $M$-tunnel from $\mathds{B}$ to $\mathds{E}$, and if $\max\{\tunnelextent{\tau_1},\tunnelextent{\tau_2}\} < \infty$, then there exists an $M$-tunnel $\tau$ from $\mathds{A}$ to $\mathds{E}$ such that
	\begin{equation}\label{triangle-lemma-main-eq}
		\tunnelextent{\tau} 
		\leq  \exp(\varepsilon) \left[ \tunnelextent{\tau_1} (1+\tunnelextent{\tau_2})^2 +  \tunnelextent{\tau_2}(1+\tunnelextent{\tau_1})^2 \right]  + \varepsilon 		\text.	\end{equation}
\end{lemma}

\begin{proof}
	We need to introduce some notation.  We write 
	\begin{equation*}
		\mathds{A} \coloneqq (\A,\Lip_\A,\M_\A,\mu_\A)\text{, }\mathds{B} \coloneqq (\B,\Lip_\B,\M_\B,\mu_\B)\text{ and }\mathds{E} \coloneqq (\alg{E},\Lip_{\alg{E}},\M_{\alg{E}}, \mu_{\alg{E}})\text.
	\end{equation*}
	
	We also write:
	\begin{equation*}
		\tunnel{\tau_1}{\mathds{A}}{\pi_1}{(\D_1,\Lip[T]_1,\M_1,e_1)}{\rho_1}{\mathds{B}}
	\end{equation*}
	and
	\begin{equation*}
		\tunnel{\tau_2}{\mathds{B}}{\pi_2}{(\D_2,\Lip[T]_2,\M_2,e_2)}{\rho_2}{\mathds{E}}\text.	
	\end{equation*}
	
	Let $\varepsilon > 0$. For convenience, set $\delta = \exp\left(\frac{\varepsilon}{2}\right)$.

	\begin{step}We first construct our candidate for a new tunnel.\end{step}
	
	We define $\D\coloneqq \D_1\oplus\D_2$, and we set $\pi : (d_1,d_2) \in \D \mapsto \pi_\A(d_1) \in \A$ and $\rho : (d_1,d_2)\in\D \mapsto \rho_2(d_2) \in \alg{E}$. It will also be helpful to define $\eta_1:(d_1,d_2)\in\D \mapsto d_1\in\D_2$ and $\eta_2: (d_1,d_2) \in \D\mapsto d_2 \in D_2$.
	
	Let $\varepsilon_0 \coloneqq \frac{\varepsilon}{\max\{1,\norm{e_1}{\D_1}, \norm{e_2}{\D_2}\}}$. We also define, for all $(d_1,d_2) \in \dom{\Lip_1}\oplus\dom{\Lip_2}$:
	\begin{equation*}
		\Lip[T](d_1,d_2) \coloneqq \max\left\{\Lip_1(d_1),\Lip_2(d_2),\frac{1 }{\varepsilon_0}\norm{\rho_1(d_1)-\pi_2(d_2)}{\B} \right\} \text.
	\end{equation*}
	Set $K\coloneqq \max\{1,\norm{e_1}{\D_2},\norm{e_2}{\D_2}\}$ so that $K\varepsilon_0 = \varepsilon$ and note $\varepsilon_0 \leq \varepsilon$.

	Now, since $\rho_1$ is a topographic quantum $M$-isometry, by Definition (\ref{topographic-isometry-def}), there exists $e_2' \in \M_1\cap\domsa{\Lip_1}$ such that $\rho_1(e_2') = \pi_2(e_2)$, while 
	\begin{equation}\label{triangle-lemma-eq0}
		\Lip_1(e_2') \leq \delta \Lip_\B(\pi_2(e_2)) \leq \delta \Lip_2(e_2) \leq \delta \frac{\tunnelextent{\tau_2}}{2M} \leq \delta\tunnelextent{\tau_2} \text,
	\end{equation}
	 and 
	\begin{align}\label{triangle-lemma-eq1}
		\norm{e_2'}{\D_1} 
		&\leq \delta \norm{\pi_2(e_2)}{\B} \nonumber\\
		&\leq \delta \norm{e_2}{\D_2}  \\
		&\leq \delta \underbracket[1pt]{\sqrt{1+\tunnelextent{\tau_2} }}_{\text{by Lemma (\ref{key-lemma})} }  \text.\nonumber
	\end{align}

	We henceforth set
	\begin{equation}\label{triangle-lemma-eq5}
		f_1 \coloneqq  e_1 e_2' \in \D_1 \text.
	\end{equation}
	
	We define similarly $e_1'$ and $f_2 \coloneqq  e_2 e_1' \in \M_2$ simply by switching the roles of the indices $1$ and $2$ in the above construction. Lastly, we set:
	\begin{equation*}
		f \coloneqq (f_1, f_2) \in \D \text.
	\end{equation*} 
	
	\begin{step}We check that $(\D,\Lip[T],\M_\D)$ is a topographic {\lcqms} and $\mu_\A\circ\pi$ is a pin for $(\D,\Lip[T])$.
	 \end{step}\label{triangle-lemma-lcqms-step}

	It is a standard argument to check that $\dom{\Lip[T]} \coloneqq \dom{\Lip_1}\oplus\dom{\Lip_2}$ is indeed a dense Jordan-Lie subalgebra of $\sa{\D}$ and $\Lip[T]$ is indeed a hermitian seminorm on $\dom{\Lip[T]}$. Moreover, it is routine to check that $\Lip[T]$ is lower semi-continuous and Leibniz.
	
	By Definition (\ref{tunnel-def}), the state $\mu_\A\circ\pi_\A$ is a pin for $(\D_1,\Lip_1,\M_1)$. 
	
	Now, it is immediate that
	\begin{align}\label{triangle-lemma-kant-ineq}
		\Kantorovich{\Lip[T]}(\mu_\A\circ\pi, \mu_{\alg{E}}\circ\rho) 
		&\leq \Kantorovich{\Lip[T]}(\mu_\A\circ\pi, \mu_\B\circ\rho_1\circ\eta_1) \nonumber \\
		&\quad+ \Kantorovich{\Lip[T]}(\mu_\B\circ\rho_1\circ\eta_1, \mu_{\B}\circ\pi_2\circ\eta_2)) \nonumber \\
		&\quad+ \Kantorovich{\Lip[T]}(\mu_\B\circ\pi_2\circ\eta_2, \mu_{\alg{E}}\circ\rho)) \\
		&=\Kantorovich{\Lip_1}(\mu_\A\circ\pi_1, \mu_\B\circ\rho_1) + \Kantorovich{\Lip[T]}(\mu_\B\circ\rho_1\circ\eta_1, \mu_{\B}\circ\pi_2\circ\eta_2)) \nonumber \\
		&\quad + \Kantorovich{\Lip_2}(\mu_\B\circ\pi_2, \mu_{\alg{E}}\circ\rho_2)) \nonumber \\
		&\leq \tunnelextent{\tau_1} + \varepsilon_0 + \tunnelextent{\tau_2} \nonumber \\
		&\leq \tunnelextent{\tau_1} + \varepsilon + \tunnelextent{\tau_2}  \nonumber \text.
	\end{align}
	Note that $ \mu_{\alg{E}}\circ\rho$ is also a pin, for $(\D_2,\Lip_2,\M_2)$. Let $\varphi \in \StateSpace(\D)$. There exists $t\in[0,1]$, $\varphi_1\in\StateSpace(\D_1)$ and $\varphi_2\in\StateSpace(\D_2)$ such that $\varphi = t\varphi_1 + (1-t)\varphi_2$. Thus, $\varphi_1$ is the limit of a sequence $(\varphi_{1,n})_{n\in\N}$ in the weak* topology in $\StateSpace(\D_1)$, with $\Kantorovich{\Lip_1}(\varphi_{1,n},\mu_\A\circ\pi_1) < \infty$. Similarly, there exists  a sequence $(\varphi_{2,n})_{n\in\N}$ converging for the weak* topology in $\StateSpace(\D_2)$ to $\varphi_2$, with $\Kantorovich{\Lip_2}(\varphi_{2,n},\mu_\alg{E}\circ\rho_2) < \infty$.
	
	Thus, $\varphi$ is the weak* limit, in $\StateSpace(\D)$, of $(t \varphi_{1,n} +(1-t)\varphi_{2,n})_{n\in\N}$. On the other hand, for each $n\in\N$,
	\begin{align*}
		\Kantorovich{\Lip[T]}(t \varphi_{1,n} +(1-t)\varphi_{2,n}, \mu_\A\circ\pi_1)
		&\leq t\Kantorovich{\Lip_1}(\varphi_{1,n},\mu_\A\circ\pi_1) \\
		&\quad +  (1-t)\left[\Kantorovich{\Lip[T]}(\mu_\A\circ\pi_1,\mu_{\alg{E}}\circ\rho_2) + \Kantorovich{\Lip_2}(\mu_{\alg{E}}\circ\rho_2, \varphi_{2,n})\right] \\
		&<\infty \text.
	\end{align*}	
	So $\mu_\A\circ\pi$ is a pin for $(\D,\Lip[T])$.

	\begin{step}We check $\rho$ and $\pi$ are topographic quantum $M$-isometries.\end{step}\label{triangle-lemma-isometry-step}

	We first record that, by construction, if $d \coloneqq (d_1,d_2) \in \domsa{\Lip[T]}$, then
	\begin{equation*}
		\Lip_{\alg{E}}(\rho(d_1,d_2)) = \Lip_{\alg{E}}(\rho_2(d_2)) \leq \Lip_2(d_2) \leq \Lip[T](d_1,d_2) \text.
	\end{equation*}

	Let $x \in \domsa{\Lip_{\alg{E}}}$ such that $\norm{x}{\Lip_{\alg{E}},M} = 1$. Let $\alpha > 0$. Since $\rho_2$ is a quantum $M$-isometry, there exists $d_2\in\domsa{\Lip_2}$ such that $\norm{x}{\Lip_{\alg{E}},M}\leq \norm{d_2}{\Lip_{2},M} \leq \norm{x}{\Lip_{\alg{E}},M} + \frac{\alpha}{2}$. Now, set $b\coloneqq \pi_2(d_2) \in \dom{\Lip_\B}$. By Definition (\ref{M-isometry-def}) of a quantum $M$-isometry, $\norm{b}{\Lip_\B,M}\leq \norm{d_2}{\Lip_2,M} \leq \norm{x}{\Lip_{\alg{E}},M} + \frac{\alpha}{2}$. Therefore, since $\rho_1$ is a quantum $M$-isometry, there exists $d_1 \in \domsa{\Lip_1}$ such that $\norm{b}{\Lip_\B,M} \leq \norm{d_1}{\Lip_1,M} \leq \norm{b}{\Lip_\B,M} + \frac{\alpha}{2}$. Hence, we conclude
	\begin{align*}
		\norm{(d_1,d_2)}{\Lip[T],M}
		&=\max\left\{ \frac{1}{M}\norm{(d_1,d_2)}{\D}, \Lip[T](d_1,d_2) \right\}\\
		&=\max\left\{ \frac{1}{M}\norm{d_1}{\D_1},\frac{1}{M}\norm{d_2}{\D_1}, \Lip_1(d_1),\Lip_2(d_2), \frac{K}{\varepsilon}\norm{b-b}{\B} \right\}\\
		&\leq \max\{ \norm{d_1}{\Lip_1,M}, \norm{d_2}{\Lip_2,M} \} \leq \norm{x}{\Lip_{\alg{E}},M} + \alpha \text.
	\end{align*}
	
	 Hence, as required,
	\begin{equation*}
		\norm{x}{\Lip_\alg{E},M} = \inf\{ \norm{(d_1,d_2)}{\Lip[T],M} : \rho(d_1,d_2) = x \} \text.
	\end{equation*}
	
	If, moreover, $x\in \M_{\alg{E}}$, then we can in fact find, for any $\alpha>0$, some $d_2 \in \M_2$ such that, at once:
	\begin{equation*}
		\rho_2(d_2)=x\text{, } \norm{d_2}\leq\exp(\frac{\alpha}{2})\norm{x}{\D_2} \text{ and }\Lip_2(x)\leq\exp(\frac{\alpha}{2})\Lip_{\alg{E}}(x) \text.
	\end{equation*}
	As above, there exists $d_1\in\M_1$ such that $\rho_1(d_1) = \pi_2(d_2)$, while $\Lip_1(d_1) \leq \exp(\frac{\alpha}{2})\Lip_{\B}(\pi_2(d_2)) \leq \exp(\alpha)\Lip_{\alg{E}}(x)$ and similarly, $\norm{d_1}{\D_1} \leq\exp(\alpha)\norm{x}{\alg{E}}$. Altogether,
	\begin{equation*}
		\rho(d_1,d_2) = x\text{, }\norm{d_1,d_2}{\D}\leq \exp(\alpha)\norm{x}{\alg{E}} \text{ and }\Lip[T](d_1,d_2)\leq \exp(\alpha)\Lip_{\alg{E}}(x) \text,
	\end{equation*}
	as desired.

	Hence, $\rho$ is a topographic quantum $M$-isometry. A similar computation shows that $\pi$ is also a topographic quantum $M$-isometry.

	\begin{step}We conclude that $(\D,\Lip[T],\pi,\rho,f)$ is an $M$-tunnel from $\mathds{A}$ to $\mathds{E}$. \end{step}
	
	By construction, $f\in \M_\D$. Since $\M_\D$ is Abelian, $f$, as a product of self-adjoint elements, is self-adjoint. Together with our two previous steps, we conclude that $(\D,\Lip[T],\pi,\rho,f)$ is an $M$-tunnel by Definition (\ref{tunnel-def}).

	\begin{step}We now estimate the extent of our tunnel.\end{step}

	We first compute estimates for $\Lip[T](f)$ for $f = (f_1,f_2)$. By the Leibniz property of $\Lip_1$,	
	\begin{align*}
		\Lip_1(f_1) 
		&\leq \norm{e_1}{\D_1}\Lip_1(e'_2) + \Lip_1(e_1)\norm{e'_2}{\D_1} \\
		&\leq \underbracket[1pt]{\sqrt{1+\tunnelextent{\tau_1}}}_{\geq\norm{e_1}{\D_1}\text{ by Lemma (\ref{key-lemma})}} \delta \Lip_2(e_2) +  \underbracket[1pt]{\delta\sqrt{1+\tunnelextent{\tau_2}}}_{\geq\norm{e_2}{\D_2}\text{ by Exp. \eqref{triangle-lemma-eq1}}} \Lip_1(e_1) \\
		&\leq\delta \left( \sqrt{1+\tunnelextent{\tau_1}} \frac{\tunnelextent{\tau_2}}{2M} + \sqrt{1+\tunnelextent{\tau_2}} \frac{\tunnelextent{\tau_1}}{2M} \right) \text.
	\end{align*}
	A similar computation holds for $f_2$, and since
	\begin{equation*}
		\pi_2(f_1) = \rho_1(e_1)\pi_2(e_2) = \rho_1(f_2)
	\end{equation*}
	 by construction, we conclude:
	\begin{equation*}
		\Lip[T](f) \leq \frac{\delta}{2M}\left( \sqrt{1+\tunnelextent{\tau_1}} \tunnelextent{\tau_2} + \sqrt{1+\tunnelextent{\tau_2}} \tunnelextent{\tau_1} \right) \text.
	\end{equation*}
So
\begin{equation}\label{triangle-lemma-first-ineq}
	2 M \Lip[T](f) \leq \delta\left( \sqrt{1+\tunnelextent{\tau_1}} \tunnelextent{\tau_2} + \sqrt{1+\tunnelextent{\tau_2}} \tunnelextent{\tau_1} \right) \text.
\end{equation}
	
We turn to the quantities in the extent involving states. First, since 
	\begin{equation}\label{triangle-lemma-eq-6}
		\Kantorovich{\Lip_1}(\mu_\A\circ\pi_1,\mu_\B\circ\rho_1)\leq\tunnelextent{\tau_1}\text,
	\end{equation}
 we observe that
	\begin{align}\label{triangle-lemma-second-ineq}
		|\mu_\A\circ\pi(f_1,f_2) - 1|
		&=|\mu_\A\circ\pi_1(f_1) - 1| \nonumber \\
		&=|\mu_\A\circ\pi_1(e_1 e_2') - 1| \nonumber \\ 
		&\leq |\underbracket[1pt]{\mu_\A\circ\pi_1(e_1)\mu_\A\circ\pi_1(e'_2)}_{\text{$\mu$ is multiplicative on $\M_\A$}} - 1 | \nonumber \\
		&\leq |\mu_\A\circ\pi_1(e_1)-1| |\mu_\A\circ\pi_1(e'_2)| + |\mu_\A\circ\pi_1(e'_2) - 1| \nonumber \\
		&\leq \tunnelextent{\tau_1} \cdot \underbracket[1pt]{\delta \sqrt{1+\tunnelextent{\tau_2}}}_{\geq\norm{e'_2}{\D_1}\text{ by Eq. \eqref{triangle-lemma-eq1}}} + |\mu_\A\circ\pi_1(e'_2)-\mu_\B\circ\rho_1(e'_2)| + |\mu_\B\circ\rho_1(e'_2) - 1| \nonumber\\
		&\leq \tunnelextent{\tau_1} \cdot \delta \sqrt{1+\tunnelextent{\tau_2}}  + \underbracket[1pt]{\tunnelextent{\tau_1}}_{\text{by Eq. \eqref{triangle-lemma-eq-6}}} \cdot \underbracket[1pt]{\delta \tunnelextent{\tau_2}}_{\substack{\geq \Lip_1(e'_2)\\ \text{ by Eq. \eqref{triangle-lemma-eq0}}}} + \tunnelextent{\tau_2} \nonumber \\
		&\leq \tunnelextent{\tau_1} \delta \left[\sqrt{1+\tunnelextent{\tau_2}} + \tunnelextent{\tau}_2\right] + \tunnelextent{\tau_2} \\
		&\leq \tunnelextent{\tau_1} \delta \left(\underbracket[1pt]{1+2\tunnelextent{\tau_2}}_{\sqrt{1+\tunnelextent{\tau_2}} \leq 1 + \tunnelextent{\tau_2}} \right) + \tunnelextent{\tau_2}  \nonumber \\ 
		&\leq \underbracket[1pt]{\delta}_{\geq 1} \left[\tunnelextent{\tau_1}(1+2\tunnelextent{\tau_2}) + \tunnelextent{\tau_2}\left( \underbracket[1pt]{1+2\tunnelextent{\tau_1}}_{\geq 1} \right) \right] \text. \nonumber
	\end{align}
	
The same reasoning holds by symmetry to obtain
\begin{equation}\label{triangle-lemma-third-ineq}
	|1-\mu_\alg{E}\circ\rho(f_1,f_2)| \leq \delta \left[\tunnelextent{\tau_1}(1+2\tunnelextent{\tau_2}) + \tunnelextent{\tau_2}( 1+2\tunnelextent{\tau_1}) \right]  \text.
\end{equation}

	We now turn to the quasi state space component of the extent. Let $\varphi \in \QuasiStateSpace(\D_1\oplus\D_2)$. There exists $t \in [0,1]$, $\varphi_1 \in \QuasiStateSpace(\D_1)$ and $\varphi_2\in\QuasiStateSpace(\D_2)$ such that 
	\begin{equation*}
		\varphi = t \varphi_1\circ\eta_1 + (1-t)\varphi_2\circ\eta_2\text.
	\end{equation*}
	
	Let $q_1 > \tunnelextent{\tau_1}$ and $q_2 > \tunnelextent{\tau_2}$. 
	
	By definition of the extent $\tunnelextent{\tau_1}$, there exists $\psi_1 \in \QuasiStateSpace(\B)$ such that 
	\begin{equation*}
		\boundedLipschitz{\Lip_1}(\varphi_1(e_1  \cdot  e_1), \psi_1) \leq q_1  \text.
	\end{equation*}
	
	Now, for all $d\in\domsa{\Lip[T]}$ with $\Lip[T](d) \leq 1$ and $\norm{d}{\D} \leq M$, we have, by Lemma (\ref{Leibniz-lemma}):
	\begin{align*}
		\Lip(e_2' d e_2') 
		&\leq \norm{e'_2}{\D_1} \left( 2 \Lip(e_2')\norm{d}{\D} + \norm{e_2'}{\D}\Lip(d) \right) \\
		&\leq \delta\sqrt{1+\tunnelextent{\tau_2}} \left(\underbracket[1pt]{2 \delta \Lip_2(e_2) M}_{\leq \tunnelextent{\tau_2}} +  \delta \sqrt{1 + \tunnelextent{\tau_2}} \right) \\ 
		&\leq \delta^2 \tunnelextent{\tau_2}\sqrt{1+\tunnelextent{\tau_2}} +  \delta^2 (1 + \tunnelextent{\tau_2}) \\
		&\leq \delta^2 \tunnelextent{\tau_2} \underbracket[1pt]{(1+\tunnelextent{\tau_2})}_{\geq \sqrt{1+\tunnelextent{\tau_2}}} +  \delta^2 (1 + \tunnelextent{\tau_2}) \\
		&\leq \delta^2( 1 + \tunnelextent{\tau_2})^2  \text.
	\end{align*}
	Moreover
	\begin{equation*}
		\norm{e_2' d e_2'}{\D_1} \leq M \norm{e_2'}{\D_1}^2 \leq M \delta^2 (1 + \tunnelextent{\tau_2}) \text. 
	\end{equation*}
	Therefore,
	\begin{equation*}
		\norm{e_2' d e_2'}{\Lip_1,M} \leq \max\{ \delta^2( 1 + \tunnelextent{\tau_2})^2, \delta^2 (1 + \tunnelextent{\tau_2}) \} = \delta^2( 1 + \tunnelextent{\tau_2})^2 \text,
	\end{equation*}
	from which it follows:
	\begin{equation}\label{triangle-lemma-eq-bl}
		\boundedLipschitz{\Lip_1}(\varphi(e_1 e_2' \cdot e_2' e_1) , \psi_1(e_2' \cdot e_2')) \leq \delta^2 (1+ \tunnelextent{\tau_2})^2  q_1 \text.
	\end{equation}

	Then, by Definition (\ref{extent-def}) of the extent of $\tau_2$ and since $q_2 > \tunnelextent{\tau_2}$, there exists $\theta_1 \in \StateSpace(\alg{E})$ such that 
	\begin{equation*}
		\boundedLipschitz{\Lip_2}(\psi(e_2 \cdot e_2), \theta_1\circ\rho_2) < q_2  \text.
	\end{equation*}
	
	Let now $d\coloneqq (d_1,d_2) \in \domsa{\Lip}$ with $\Lip[T](d)\leq 1$ and $\norm{d}{\D_1} \leq M$. We compute:
	\begin{align*}
		|\varphi_1(e_1 e_2' d_1 e_2' e_1) - \theta_1\circ\rho_2(d_2)|
		&\leq |\varphi_1(e_1 e_2' d_1 e_2' e_1) - \psi_1\circ\rho_1(e_2' d_1 e_2')| \\
		&\quad + |\psi_1\circ\rho_1(e_2' d_1 e_2') - \psi_1\circ\pi_2 (e_2 d_2 e_2)| \\
		&\quad + |\psi_1\circ\pi_2 (e_2 d_2 e_2) - \theta_1\circ\rho_2(d_2)| \\
		&\leq \underbracket[1pt]{\delta_1^2 (1+\tunnelextent{\tau_2})^2 q_1}_{\text{ by \eqref{triangle-lemma-eq-bl}}}  + \norm{\rho_1(d_1) - \pi_2(d_2)}{\B}\norm{e_2}{\D_2}^2 + q_2 \\
		&\leq  \delta^2 (1+\tunnelextent{\tau_2})^2 q_1 + \underbracket[1pt]{ \varepsilon_0}_{\text{since }\Lip[T](d)\leq 1} + q_2\\
		&\leq q_1 \delta^2  (1+\tunnelextent{\tau_2}) +  q_2 \underbracket[1pt]{\delta^2 (1+\tunnelextent{\tau_1})^2}_{\geq 1} + \underbracket[1pt]{\varepsilon}_{\geq\varepsilon_0} \\
		&\leq \delta^2 \left[ q_1 (1+\tunnelextent{\tau_2})^2 +  q_2 (1+\tunnelextent{\tau_1})^2 \right] + \varepsilon\text.
	\end{align*}
	Therefore, for all $q_1 > \tunnelextent{\tau_1}$ and $q_2 > \tunnelextent{\tau_2}$, we have shown that:
	\begin{equation*}
		\boundedLipschitz{\Lip}(\varphi_1\circ\eta_1(f\cdot f),\theta_1\circ\eta_2\circ\rho_2) \leq \delta^2 \left[ q_1 (1+\tunnelextent{\tau_2})^2 +  q_2 (1+\tunnelextent{\tau_1})^2 \right] + \varepsilon\text.
	\end{equation*}
	Therefore, since $q_1>\tunnelextent{\tau_1}$ and $q_2 > \tunnelextent{\tau_2}$ were arbitrary, we conclude:
	\begin{equation*}
		\boundedLipschitz{\Lip}(\varphi_1\circ\eta_1(f\cdot f),\theta_1\circ\eta_2\circ\rho_2) \leq \delta^2 \left[ \tunnelextent{\tau_1} (1+ \tunnelextent{\tau_2})^2 +  \tunnelextent{\tau_2} (1+ \tunnelextent{\tau_1})^2 \right] + \varepsilon\text.
	\end{equation*}

	 On the other hand, since $\varphi_2\in\QuasiStateSpace(\D_2)$, there existst $\theta_2 \in \QuasiStateSpace(\alg{E})$ such that $\boundedLipschitz{\Lip,l}(\varphi_2(e_2\cdot e_2), \theta_2) \leq q_2$. Just as above, we conclude that
	\begin{equation*}
		\boundedLipschitz{\Lip}(\varphi_2(e_2 e_1'\cdot e_1' e_2),\theta_2(e_1'\cdot e_1'))  \leq \underbracket[1pt]{\delta^2 (1+ \tunnelextent{\tau_1})^2}_{\geq \norm{e_1' d e'_1}{\Lip_2} \text{ when $\norm{d}{\Lip[T]}\leq 1$ }} q_2 \text.
	\end{equation*}
	Again, since $q_2 > \tunnelextent{\tau_2}$ is arbitrary, we conclude:
	\begin{equation*}
		\boundedLipschitz{\Lip}(\varphi_2(e_2 e_1'\cdot e_1' e_2),\theta_2(e_1'\cdot e_1'))  \leq \delta^2 (1+\tunnelextent{\tau_1})^2 \tunnelextent{\tau_2} \text.
	\end{equation*}
	
	Set $\theta\coloneqq t\theta_1\circ\eta_1 + (1-t)\theta_2\circ\eta_2$. By construction, $\theta\in\QuasiStateSpace(\alg{E})$.
	
	Let $d\coloneqq(d_1,d_2)\in\domsa{\Lip[T]}$ with $\Lip[T](d)\leq 1$ and $\norm{d}{\D} \leq M$.  Then:
	
	\begin{align*}
		|\varphi(f d f) - \theta(d)|
		&\leq t|\varphi_1(e_1 e_2' d_1 e_2' e_1) - \theta_1(e_2' d e_2')| \\
		&\quad + (1-t)|\varphi_2(e_2 e_1' d e_1' e_2) - \theta_2(e_1' d e_1')| \\
 		&\leq  \delta^2 \left[ \tunnelextent{\tau_1} (1+\tunnelextent{\tau_2})^2 +  \tunnelextent{\tau_2} (1+\tunnelextent{\tau_1})^2 \right] + \varepsilon_0\text.
	\end{align*}
	
	Of course, it follows that:
	\begin{equation}\label{triangle-lemma-qs1-ineq}
		f \QuasiStateSpace(\D) f \subseteq_{r}^{\boundedLipschitz{\Lip[T],M}} \rho^\ast(\QuasiStateSpace(\alg{E})) \text,
	\end{equation}
	with $r \coloneqq  \delta^2 \left[ \tunnelextent{\tau_1} (1+\tunnelextent{\tau_2})^2 +  \tunnelextent{\tau_2} (1+\tunnelextent{\tau_1})^2 \right] + \varepsilon\text.$
	
	Again, noting our chosen $r$ is symmetric in indices $1$ and $2$ by design, the same reasoning applies to prove that
	\begin{equation}\label{triangle-lemma-qs2-ineq}
		f \QuasiStateSpace(\D) f \subseteq_{r}^{\boundedLipschitz{\Lip[T],M}} \pi^\ast(\QuasiStateSpace(\A)) \text.
	\end{equation}
	
	\medskip

	In summary, we have shown that
	\begin{align*}
		\tunnelextent{\tau} \leq
		\max\Big\{ &\sqrt{1+\tunnelextent{\tau_1}} \tunnelextent{\tau_2} + \sqrt{1+\tunnelextent{\tau_2}} \tunnelextent{\tau_1}, \text{ by Eq. \eqref{triangle-lemma-first-ineq}, } \\
		 &\delta^2  [\tunnelextent{\tau_1} (1+\tunnelextent{\tau_2})^2 +  \tunnelextent{\tau_2} (1+\tunnelextent{\tau_1})^2 ] + \varepsilon, \text{from Eqs.  \eqref{triangle-lemma-qs1-ineq} and \eqref{triangle-lemma-qs2-ineq}}, \\
		&\tunnelextent{\tau_1}(1+2\tunnelextent{\tau_2}) + \tunnelextent{\tau_2}( 1+2\tunnelextent{\tau_1}), \text{by Eqs. \eqref{triangle-lemma-second-ineq} and \eqref{triangle-lemma-third-ineq}, }\\
		&\tunnelextent{\tau_1} + \varepsilon + \tunnelextent{\tau_2} \text{by Eq. \eqref{triangle-lemma-kant-ineq}} \Big\}\\
		&= \delta^2 \left[ \tunnelextent{\tau_1} (1+\tunnelextent{\tau_2})^2 +  \tunnelextent{\tau_2} (1+\tunnelextent{\tau_1})^2 \right] +  \varepsilon \text.
	\end{align*}
	Since $\delta^2 = \exp(\varepsilon)$, we conclude:
	\begin{equation*}
		\tunnelextent{\tau} \leq \exp(\varepsilon) \left( (1+\tunnelextent{\tau_2})^2\tau_1 + (1+\tunnelextent{\tau_1})^2\tau_2 \right) +  \varepsilon \text,
	\end{equation*} 
	as desired. 
\end{proof}

\subsection{The Local Quantum Metametrics}

We now introduce the main tool to construct our hypertopology on the class of {\pqpms s}.

\begin{definition}\label{prop-def}
	The \emph{$M$-local quantum metametric}, for $M\geq 1$, between two {\plcqms s} $\mathds{A}$ and $\mathds{B}$ is the non-negative real number:
	\begin{equation*}
		\metametric{M}(\mathds{A},\mathds{B}) \coloneqq \inf\left\{ \tunnelextent{\tau} : \tau \text{ is a tunnel from $\mathds{A}$ to $\mathds{B}$} \right\}\text.
	\end{equation*}
\end{definition}

\begin{theorem}\label{triangle-thm}
The following assertions hold, for any $M \geq 1$.
	\begin{enumerate}
		\item For any two {\pqpms s} $\mathds{A}$ and $\mathds{B}$, if there exists a full topographic quantum $M$-isometry $\pi : \mathds{A} \rightarrow\mathds{B}$ such that $\mu_\B\circ\pi = \mu_\A$, then $\metametric{M}(\mathds{A},\mathds{B}) = 0$.
		\item For any two {\pqpms s} $\mathds{A}$ and $\mathds{B}$, we have $\metametric{M}(\mathds{B},\mathds{A}) = \metametric{M}(\mathds{A},\mathds{B})$. 
		\item For any two {\pqpms s} $\mathds{A}$, $\mathds{B}$, and $\mathds{D}$,
		\begin{equation*}
			\metametric{M}(\mathds{A},\mathds{D})\leq 
			 \left( \metametric{}(\mathds{A},\mathds{B})\left(1 + \metametric{M}(\mathds{B},\mathds{D})\right)^2 + \metametric{M}(\mathds{B},\mathds{D})\left(1 + \metametric{M}(\mathds{A},\mathds{B}) \right)^2 \right) \text.
		\end{equation*}
	\end{enumerate}
\end{theorem}

\begin{proof}
	Let $\varepsilon > 0$. By Definition (\ref{prop-def}), there exists tunnels $\tau_1$ from $\mathds{A}$ to $\mathds{B}$ and $\tau_2$ from $\mathds{B}$ to $\mathds{D}$ such that
	\begin{equation*}
		\tunnelextent{\tau_1} \leq \metametric{M}(\mathds{A},\mathds{B})  + \varepsilon \text{ and } \tunnelextent{\tau_2} \leq \metametric{M}(\mathds{B},\mathds{D}) + \varepsilon \text.
	\end{equation*}
	
	By Expression \eqref{triangle-lemma-main-eq} of Lemma (\ref{triangle-lemma}), there exists a tunnel $\tau$ from $\mathds{A}$ to $\mathds{D}$ with 
	\begin{equation*}
		\tunnelextent{\tau} \leq \exp(\varepsilon)(\tunnelextent{\tau_1}(1+\tunnelextent{\tau_2})^2 + \tunnelextent{\tau_2}(1+\tunnelextent{\tau_1})^2)+ \varepsilon \text.
	\end{equation*}
	
	By Definition (\ref{prop-def}) again,
	\begin{align*}
		\metametric{M}(\mathds{A},\mathds{D})
		&\leq \tunnelextent{\tau} \\
		&\leq  \exp(\varepsilon)(\tunnelextent{\tau_1}(1+\tunnelextent{\tau_2})^2 + \tunnelextent{\tau_2}(1+\tunnelextent{\tau_1})^2) + \varepsilon  \\
		&\leq \exp(\varepsilon) \big((\metametric{M}(\mathds{A},\mathds{B})+\varepsilon)(1+\metametric{M}(\mathds{B},\mathds{D}) + 2\varepsilon)^2 \\
		&\quad\quad\quad + (\metametric{M}(\mathds{B},\mathds{D})+\varepsilon)(1+\metametric{M}(\mathds{A},\mathds{B}) + 2\varepsilon )^2  \big) + \varepsilon  \text.
	\end{align*}
	Therefore, as claimed, since $\varepsilon > 0$ is arbitrary, taking the limit as $\varepsilon\rightarrow 0$, we get:
	\begin{equation*}
		\metametric{M}(\mathds{A},\mathds{D}) \leq  \left( \metametric{M}(\mathds{A},\mathds{B})\left(1 + \metametric{M}(\mathds{B},\mathds{D})\right)^2 + \metametric{M}(\mathds{B},\mathds{D})\left(1 + \metametric{M}(\mathds{A},\mathds{B})\right)^2 \right) \text.
	\end{equation*}
	This concludes the proof that $\metametric{M}$ satisfies a local generalized  triangle inequality in (3).
	
	Symmetry follows from the observation that if $\tau$ is a tunnel from $\mathds{A}$ to $\mathds{B}$, then $\tau^{-1}$ is a tunnel from $\mathds{B}$ to $\mathds{A}$ with the same extent. So (2) holds.
	
	Last, assume that there exists a full topographic quantum $M$-isometry $\pi : \mathds{A} \rightarrow\mathds{B}$ such that $\mu_\B\circ\pi = \mu_\A$. Let $\varepsilon > 0$. Since $\mathds{A}$ is a {\pqms}, there exists $h \in \M_\A\cap\dom{\Lip_\A}$ with $\Lip_\A(h) < \frac{\varepsilon}{2 M}$ and $\norm{h}{\A} = \mu_\A(h) = 1$. We now build a tunnel of extent at most $\varepsilon$, as follows:
	\begin{equation*}
		\tunnel{\tau}{\mathds{A}}{\mathrm{id}}{(\A,\Lip_\A,\M_\A,h)}{\pi}{\mathds{B}} \text.
	\end{equation*}
	Since $\mu_\B\circ\pi = \mu_\A$ by definition, we conclude $\Kantorovich{\Lip_\A}(\mu_\B\circ\pi,\mu_\A) = 0$. By construction,
	\begin{equation*}
		\max\{ 2 M \Lip_\A(h), |1-\mu_\A(h)|, |1-\mu_\B\circ\pi(h)| \} \leq \varepsilon \text.
	\end{equation*}
	
	Last, let $\varphi\in\StateSpace(\A)$. Then $\opnorm{\varphi(h\cdot h)}{\A}{\C}=\varphi(h^2) \leq 1$ so $\varphi(h\cdot h) \in \QuasiStateSpace(\A)$. Therefore
	\begin{equation*}
		h \QuasiStateSpace(\A) h \subseteq \QuasiStateSpace(\A) \text.
	\end{equation*} 
	Last, $\pi^\ast(\QuasiStateSpace(\B)) = \QuasiStateSpace(\A)$, so
	\begin{equation*}
		h \QuasiStateSpace(\A) h \subseteq \pi^\ast(\QuasiStateSpace(\B)) \text.
	\end{equation*} 
	Altogether, the extent of $\tau$ is $\Lip_\A(h)$ and thus no more than $\varepsilon > 0$. So $\metametric{M}(\mathds{A},\mathds{B}) \leq \varepsilon$. As $\varepsilon > 0$ was arbitrary, $\metametric{M}(\mathds{A},\mathds{B}) = 0$.
\end{proof}

Our focus now is to prove that the equivalence relation defined by the metametric being null between two {\pqpms s} is indeed given by full quantum isometric classes, i.e. to study the separation property of this newly introduced topology.


\section{Coincidence property of the local metametics}

The main property of the local metametrics we wish to establish in this paper is that distance zero for the metametric is equivalent to the existence of a full quantum isometry between the {\pqpms s}. The construction of this isometry begins with the construction of two set-valued maps defined by any tunnel between {\pqpms s}, which, as we shall see, share some set-theoretic analogue of the properties of *-morphisms. In essence, the sought-after quantum isometry is a limit of these set-valued maps.

\subsection{Target Sets}

\begin{definition}\label{targetset-def}
	Let $\mathds{A} \coloneqq(\A,\Lip_\A,\M_\A,\mu_\A)$ and $\mathds{B}\coloneqq(\B,\Lip_\B,\M_\B,\mu_\B)$ be two {\pqpms s}. Let $M\geq 1$. Let 
	\begin{equation*}
		\tunnel{\tau}{\mathds{A}}{\pi_\A}{(\D,\Lip_\D,\M_\D,e)}{\pi_\B}{\mathds{B}}
	\end{equation*}
	be an $M$-tunnel from $\mathds{A}$ to $\mathds{B}$. For all $a\in\domsa{\Lip_\A}$ and for all $l > \norm{a}{\Lip_\A,M}$, the \emph{target set} $\targetsettunnel{\tau}{a}{l}$ is given by:
	\begin{equation*}
		\targetsettunnel{\tau}{a}{l} \coloneqq \left\{ \pi_\B(d) : d\in\domsa{\Lip_\D}, \pi_\A(d) = a, \norm{d}{\Lip_\D,M} \leq l \right\}\text.
	\end{equation*}
\end{definition}

Our first step is to check that the target set is never empty, and exhibits a form of almost continuity, which is central to our construction.

\begin{lemma}\label{main-lemma}
Let $\mathds{A} \coloneqq(\A,\Lip_\A,\M_\A,\mu_\A)$ and $\mathds{B}\coloneqq(\B,\Lip_\B,\M_\B,\mu_\B)$ be two {\pqpms s}, and let $M\geq 1$. Let $\tau\coloneqq (\D,\Lip_\D,\M_\D,\pi_\A,\pi_\B, e^\D)$ be an $M$-tunnel from $\mathds{A}$ to $\mathds{B}$.

	Write $e^\B \coloneqq \pi_\B(e^\D)$.
	
	For all $a\in\domsa{\Lip_\A}$, if $l > \norm{a}{\Lip_\A,M}$, then
	\begin{equation*}
		\targetsettunnel{\tau}{a}{l} \neq \emptyset\text{;}
	\end{equation*}
	moreover if $d\in\domsa{\Lip_\D}$ with $\norm{d}{\Lip_\D,M} \leq l$  and $\pi_\A(d) = a$, we conclude that:
	\begin{equation*}
		\norm{e^\D d e^\D}{\B} \leq \norm{a}{\A} + l\tunnelextent{\tau} \text,
	\end{equation*}
	and therefore, for all $b \in \targetsettunnel{\tau}{a}{l}$, we assert:
	\begin{equation*}
		\norm{e^\B b e^\B}{\B} \leq \norm{a}{\A} + l \tunnelextent{\tau} \text.
	\end{equation*}
\end{lemma}

\begin{proof}
	By Definition (\ref{M-isometry-def}) of a quantum $M$-isometry, for all $a\in \domsa{\Lip_\A}$ and $l > \norm{a}{\Lip_\A,M}$, there exists $d \in \domsa{\Lip_\D}$ with $\pi_\A(d) = a$ and $\norm{d}{\Lip_\D,M} \leq l$, so $\targetsettunnel{\tau}{a}{l} \neq \emptyset$.

	Let now $d \in \domsa{\Lip_\D}$ with $\pi_\A(d) = a$ and $\norm{d}{\Lip_\D,M} \leq l$. 	Let $\varphi \in \StateSpace(\D)$ and let $\varepsilon > 0$. By Definition (\ref{extent-def}), there exists $\psi \in \QuasiStateSpace(\A)$ such that 
	\begin{equation*}
		\boundedLipschitz{\Lip_\D,M}(\varphi(e^\D\cdot e^\D),\psi\circ\pi_\A) \leq \tunnelextent{\tau} + \frac{\varepsilon}{l}\text;
	\end{equation*}
	therefore, in particular,
	\begin{align*}
			|\varphi(e^\D d e^\D)|
			&\leq |\varphi(e^\D d e^\D) - \psi\circ\pi_\A(d)| + |\psi\circ\pi_\A(d)| \\
			&\leq l \boundedLipschitz{\Lip_\D,M}(\varphi(e^\D\cdot e^\D,\psi\circ\pi_\A) + |\psi(a)| \\
			&\leq l \tunnelextent{\tau} + \norm{a}{\A} + \varepsilon \text.
	\end{align*}
	Therefore, for all $\varepsilon > 0$, since $e^\B d e^\B$ is self-adjoint,
	\begin{equation*}
		\norm{e^\D d e^\D}{\D} = \sup\{|\varphi(e^\B d e^\B)| : \varphi \in \StateSpace(\D)\} \leq l \tunnelextent{\tau} + \norm{a}{\A} + \varepsilon  \text,
	\end{equation*}
	hence
	\begin{equation*}
		\norm{e^\D d e^\D}{\D}  \leq l \tunnelextent{\tau} + \norm{a}{\A}  \text,
	\end{equation*}
	as claimed.
	
	Since $\pi_\B$ is a *-morphism, it has norm $1$, and therefore the proof of our lemma is concluded.
\end{proof}

Target sets exhibit morphism-like properties, though as set-valued functions. We present these properties now.

\begin{lemma}\label{linearity-lemma}
Let $\mathds{A} \coloneqq(\A,\Lip_\A,\M_\A,\mu_\A)$ and $\mathds{B}\coloneqq(\B,,\Lip_\B,\M_\B,\mu_\B)$ be two {\pqpms s}. Let $\tau\coloneqq (\D,\Lip_\D,\M_\D,\pi_\A,\pi_\B,e^\D)$ be an $M$-tunnel from $\mathds{A}$ to $\mathds{B}$ for $M\geq 1$. 

For all $a,a' \in \domsa{\Lip_\A}$, $t\in \R$, and $l>\max\left\{\norm{a}{\Lip_\A,M},\norm{a'}{\Lip_\A,M}\right\}$, we have
\begin{equation*}
	t \cdot \targetsettunnel{\tau}{a}{l} + \targetsettunnel{\tau}{a}{l} \subseteq \targetsettunnel{\tau}{t a + a'}{(1+|t|)l}  \text.
\end{equation*}	
\end{lemma}

\begin{proof}
  Let $b\in\targetsettunnel{\tau}{a}{l}$ and $b' \in \targetsettunnel{\tau}{a'}{l}$. By Definition (\ref{targetset-def}), there exists $d,d'\in\domsa{\Lip_\D}$ with $\norm{d}{\Lip_\D,M}\leq l$, $\norm{d'}{\Lip_\D,M}\leq l$, $\pi_\A(d) = a$ and $\pi_\A(d')=a'$. 
	
	Then $\pi_\A( t d + d' ) = t a + a'$ and $\norm{t d+ d'}{\Lip_\D,M}\leq (|t|+1)l$. Thus 
	\begin{equation*}
		tb +b' \in \targetsettunnel{\tau}{ta + a'}{(1+|t|)l}\text,
	\end{equation*}
 as claimed. 
\end{proof}

\begin{corollary}\label{linearity-cor}
Let $\mathds{A} \coloneqq(\A,\Lip_\A,\M_\A,\mu_\A)$ and $\mathds{B}\coloneqq(\B,,\Lip_\B,\M_\B,\mu_\B)$ be two {\pqpms s}. Let $\tau\coloneqq (\D,\Lip_\D,\M_\D,\pi_\A,\pi_\B,e^\D)$ be an $M$-tunnel from $\mathds{A}$ to $\mathds{B}$ for some $M\geq 1$.  Let $e^\B\coloneqq \pi_\B(e^\D)$.

	If $a,a'\in\domsa{\Lip_\A}$, $t \in \R$, and $l\geq\max\left\{\norm{a}{\Lip_\A,M},\norm{a'}{\Lip_\A,M}\right\}$, then for all $b\in \targetsettunnel{\tau}{a}{l}$ and $b'\in\targetsettunnel{\tau}{a'}{l}$:
	\begin{equation}\label{linearity-cor-eq1}
		\norm{e^\B (t b + b') e^\B}{\B} \leq \norm{t a + a'}{\A} + (1+|t|)l\tunnelextent{\tau} \text. 
	\end{equation}
	In particular,
	\begin{equation}
		\diam{e^\B \targetsettunnel{\tau}{a}{l} e^\B}{\B} \leq 2 l \tunnelextent{\tau} \text.
	\end{equation}
\end{corollary}

\begin{proof}
	Let $b\in\targetsettunnel{\tau}{a}{l}$ and $b'\in\targetsettunnel{\tau}{a'}{l}$. By Lemma (\ref{linearity-lemma}), for any $t\in \R$, the linear combination $t b + b'$ is in $\targetsettunnel{\tau}{t a + a'}{(1+|t|)l}$. Therefore, by Lemma (\ref{main-lemma}), we conclude that
	\begin{equation*}
		\norm{e^\B (t b + b') e^\B}{\B} \leq (1+|t|) l \tunnelextent{\tau} + \norm{ta + a'}{\A} \text.
	\end{equation*}
	In turn, applying Expression \eqref{linearity-cor-eq1} with $a'=a$ and $t=-1$, we get that for all $b,b\in \targetsettunnel{\tau}{a}{l}$, 
	\begin{equation*}
		\norm{e^\B(b-b')e^\B}{\B} \leq \norm{a-a}{\A} + 2 l \tunnelextent{\tau} = 2 l \tunnelextent{\tau}\text,
	\end{equation*}
	which concludes our proof.
\end{proof}

\begin{lemma}\label{product-lemma}
Let $\mathds{A} \coloneqq(\A,\Lip_\A,\M_\A,\mu_\A)$ and $\mathds{B}\coloneqq(\B,,\Lip_\B,\M_\B,\mu_\B)$ be two {\pqpms s}. Let $\tau\coloneqq (\D,\Lip_\D,\M_\D,\pi_\A,\pi_\B,e^\D)$ be an $M$-tunnel from $\mathds{A}$ to $\mathds{B}$, for some $M\geq 1$.

	If $a,a'\in\domsa{\Lip_\A}$, $t \in \R$, and $l\geq\max\left\{\norm{a}{\Lip_\A,M}, \norm{a'}{\Lip_\A,M} \right\}$, then for all $b\in \targetsettunnel{\tau}{a}{l}$ and $b'\in\targetsettunnel{\tau}{a'}{l}$:
	\begin{equation*}
		\Re(bb') \in \targetsettunnel{\tau}{\Re(aa')}{2 M l^2} \text{ and }
		\Im(bb') \in \targetsettunnel{\tau}{\Im(aa')}{2 M l^2}\text.
	\end{equation*}
\end{lemma}

\begin{proof}
	Let $d,d'\in\domsa{\Lip_\D}$ such that $\pi_\A(d)=a$, $\pi_\A(d')=a'$, $\pi_\B(d)=b$, $\pi_\B(d')=b'$ and $\max\left\{\norm{d}{\Lip_\D,M},\norm{d'}{\Lip_\D,M} \right\}\leq l$. Then by the Leibniz property, 
	\begin{align*}
		\Lip_\D(\Re dd') 
		&\leq \Lip_\D(d)\norm{d'}{\D} + \norm{d}{\D}\Lip_\D(d') \leq 2 M l^2 \text.
	\end{align*}
	Since $\pi_\A(\Re (dd')) = \Re (aa')$, we conclude that $\Re(dd') \in \targetsettunnel{\tau}{\Re(aa')}{2M l^2}$. A similar result applies to for the imaginary part of the product (i.e. the Lie product).
\end{proof}

\subsection{The Coincidence property}

We now prove that target sets do converge to a full quantum $M$-isometry when they are defined using tunnels whose extent converges to $0$.

\begin{theorem}\label{main-thm}
	Let $M\geq 1$. Let $\mathds{A}\coloneqq(\A,\Lip_\A,\M_\A,\mu_\A)$ and $\mathds{B}\coloneqq(\B,\Lip_\B,\M_\B,\mu_\B)$ be two {\pqpms s}. There exists a full topographic quantum $M$-isometry $\pi$ from $\mathds{A}$ to $\mathds{B}$ with $\mu_\B\circ\pi = \mu_\A$ if, and only if:
	\begin{equation*}
		\metametric{M}(\mathds{A},\mathds{B}) = 0
	\end{equation*}
\end{theorem}

\begin{proof}
By Theorem (\ref{triangle-thm}), we already know that if there exists  a full topographic quantum $M$-isometry from $\mathds{A}$ to $\mathds{B}$ such that $\mu_\B\circ\pi=\mu_\A$, then $\metametric{M}(\mathds{A},\mathds{B}) = 0$. We thus focus our attention on the converse. Let us assume henceforth that $\metametric{M}(\mathds{A},\mathds{B}) = 0$.

We begin by fixing some notation. By Definition (\ref{prop-def}) of the metametric, for all $n \in \N$, there exists a tunnel $\tau_n$, given by:
	\begin{equation*}
		\tunnel{\tau_n}{(\A,\Lip_\A,\M_\A,\mu_\A)}{\pi_n}{(\D_n,\Lip_n,\M_n,e_n)}{\rho_n}{(\B,\Lip_\B,\M_\B,\mu_\B)}\text, 
	\end{equation*}
	such that
	\begin{equation*}
		\tunnelextent{\tau_n} <  \frac{1}{n+2} \text.
	\end{equation*}

	For each $n\in\N$, we define 
	\begin{equation*}
		e_n^\A \coloneqq \pi_n(e_n) \in \domsa{\Lip_\A}\cap\M_\A \text{ and }e_n^\B \coloneqq \rho_n(e_n) \in \domsa{\Lip_\B}\cap\M_\B \text.
	\end{equation*}

\begin{lemma}\label{e-approx-unit-lemma}
	The sequences $(e_m^\A)_{n\in\N}$ and $(e_m^\B)_{m\in\N}$ are  approximate units of, respectively, $\A$ and $\B$.
\end{lemma}
\begin{proof}
	For all $n\in\N$, by Lemma (\ref{key-lemma}), 
	\begin{equation*}
		\norm{e_n}{\D_n} \leq \sqrt{ 1 + \tunnelextent{\tau_n}} \leq \sqrt{ 1 + \frac{1}{n+2}} \text,
	\end{equation*}
	while by Definition (\ref{extent-def}),
	\begin{equation*}
		\mu_\A\circ\pi_n (e_n) \geq 1 - \tunnelextent{\tau_n} \geq 1 - \frac{1}{n+2} \text{ so }\norm{e_n}{\D_n} \geq 1 - \frac{1}{n+2}  \text.
	\end{equation*}
	Thus  $\lim_{n\rightarrow\infty} \norm{e_n}{\D_n} = 1$.
	
	Moreover, again by Definition (\ref{extent-def}), 
	\begin{equation*}
		|1-\mu_\A(e_n^\A)| = |1-\mu_\A\circ\pi_n(e_n)| \leq \frac{1}{n+2} \text{ and } |1-\mu_\B(e_n^\B)| = |1-\mu_\B\circ\rho_n(e_n)| \leq \frac{1}{n+2}  \text.
	\end{equation*}
	Therefore $\lim_{n\rightarrow\infty} \mu_\A(e_n^\A) = \lim_{n\rightarrow\infty} \mu_\B(e_n^\B) = 1$. 
	
	Since, for all $n\in\N$,
	\begin{equation*}
		|\mu_\A(e_n^\A)| \leq \norm{e_n^\A}{\A} \leq \norm{e_n}{\D_n}\text,
	\end{equation*}
	we conclude by the squeeze theorem that $\lim_{n\rightarrow\infty} \norm{e_n^\A}{\A} = 1$. Similarly, $\lim_{n\rightarrow\infty}\norm{e_n^\B}{\B} = 1$.

	Last, $\Lip_n(e_n) \leq \frac{1}{2M(n+2)} \xrightarrow{n\rightarrow\infty} 0 \text.$ Therefore, $0\leq\max\left\{ \Lip_\A(e_n^\A) , \Lip_\B(e_n^\B) \right\} \leq \Lip_n(e_n)$, and thus $\lim_{n\rightarrow\infty} \Lip_\A(e_n^\A) = \lim_{n\rightarrow\infty} \Lip_\B(e_n^\B) = 0$.
	Therefore, $(e_n^\A)_{n\in\N}$ and $(e_n^\A)_{n\in\N}$ are Lipschitz pinned exhaustive sequences, and therefore, approximate units for $\A$ and $\B$, respectively, by Theorem (\ref{approx-unit-thm}).
\end{proof}

	\medskip
	
	We introduce a few additional notations. We denote $\targetsettunnel{\tau_n}{\cdot}{\cdot}$ simply as $\targetsettunnel{n}{\cdot}{\cdot}$, for all $n\in\N$. We also will write $\targetsettunnel{-m}{\cdot}{\cdot}$ for the target function $\targetsettunnel{\tau_m^{-1}}{\cdot}{\cdot}$ defined using the reverse tunnel $\tau_m^{-1}$, for all $m\in\N$.

We now observe with the following lemma that the convergence of various sets will not depend on the assumption that the elements $e_m^\B$ commute. While this is a bit more work than needed since, for our proof of the relaxed triangle inequality, we needed these elements commute, we still include this slightly more general proof here. We begin with a simple lemma with an easy inequality which we will use regularly.

\begin{lemma}\label{easy-lemma}
Let $\A$ be a C*-algebra. If $a,e,f \in \A$, then:
	\begin{equation*}
	\norm{f a f}{\A} \leq \left(\norm{e}{\A} + \norm{f}{\A}\right) \norm{e-f}{\A} \norm{a}{\A} + \norm{e a e}{\A} \text.
	\end{equation*} 
\end{lemma}
  
 \begin{proof}
 A simple computation shows:
 	\begin{align*}
 		\norm{f a f}{\A}
 		&\leq \norm{(f-e)a f}{\A} + \norm{e a f}{\A}  \\
 		&\leq \norm{f-e}{\A} \norm{a}{\A} \norm{f}{\A} + \norm{e a (f-e)}{\A} + \norm{e a e}{\A} \\
 		&\leq \left(\norm{e}{\A} + \norm{f}{\A}\right) \norm{f-e}{\A}\norm{a}{\A} + \norm{e a e}{\A} \text.
 	\end{align*}
 	This completes our proof.
 \end{proof} 

\begin{lemma}\label{no-em-lemma}
	For all $p \in \N$, for all $a\in\domsa{\Lip_\A}$ and for all $l > \norm{a}{\Lip_\A.M}$, we have
	\begin{equation*}
		\lim_{m\rightarrow\infty} \Haus{\B}(e^\B_p \targetsettunnel{m}{a}{l} e^\B_p, e^\B_p e_m^\B \targetsettunnel{m}{a}{l} e_m^\B e^\B_p)= 0 \text.
	\end{equation*}
\end{lemma}

\begin{proof}
	Let $\varepsilon > 0$. Fix $p$. Since $(e^\B_m)_{m\in\N}$ is an approximate unit for $\B$, there exists $N \in \N$ such that for all $m\geq N$, we have $\norm{e^\B_p - e_p^\B e_m^\B}{\B} < \frac{\varepsilon}{4 l M}$. By construction and Lemma (\ref{key-lemma}), 
	\begin{equation*}
		\norm{e^\B_p}{\B} \leq \sqrt{1+\frac{1}{p+2}} \leq 2 \text{ for all $p\in\N$.}
	\end{equation*}
	
	Let now $b \in \targetsettunnel{m}{a}{l}$. By Definition (\ref{targetset-def}), we note that $\norm{b}{\B} \leq l M$. Therefore, by Lemma (\ref{easy-lemma}),
	\begin{align*}
		\left|\norm{ e_p b e_p - e_p e_m b e_m e_p }{\B} - \norm{e_p b e_p}{\B}\right|
		&\leq 4\norm{e_p e_m - e_p}{\B}\norm{b}{\B}  \\
		&\leq 4\frac{\varepsilon}{l M} l M = \varepsilon \text.  
	\end{align*}
	This concludes our proof.

	Note that, since $e^\B_p$ and $e^\B_m$ are self-adjoint, we also have that $\norm{e^\B_p - e^\B_m e^\B_p}{\B} < \frac{\varepsilon}{2l M}$ for all $m\geq N$, and thus
	\begin{equation*}
	\left|\norm{ e_p b e_p - e_m e_p b e_p e_m }{\B} - \norm{e_p b e_p}{\B}\right| < \varepsilon\text,
	\end{equation*}
	so it is also true that
\begin{equation*}
		\lim_{m\rightarrow\infty} \Haus{\B}(e^\B_p \targetsettunnel{m}{a}{l} e^\B_p, e^\B_m e_p^\B \targetsettunnel{m}{a}{l} e_p^\B e^\B_m)= 0 \text.
	\end{equation*}
\end{proof}

We now prove our core lemma, which starts the journey toward obtaining a *-morphism from our set-valued target set maps.

	\begin{lemma}\label{pre-core-lemma}
		Let $a\in\domsa{\Lip_\A}$ and $p\in\N$. If there exists a strictly increasing function $f : \N\rightarrow\N$ and an element $\beta_p(a) \in \domsa{\Lip_\B}$ such that, for all $l > \norm{a}{\Lip_\A,M}$,
		\begin{equation*}
			\lim_{m\rightarrow\infty} \Haus{\B}\left(e^\B_p \targetsettunnel{f(m)}{a}{l} e^\B_p, \{\beta_p(a)\} \right) = 0\text,
		\end{equation*}
		then $\norm{\beta_p(a)}{\A} \leq \norm{a}{\A}$ while 
		\begin{equation*}
			\Lip_\B(\beta_p(a))\leq \Lip_\A(a)\text{ and }\norm{\beta_p(a)}{\Lip_\A,\M} \leq \norm{a}{\Lip_\A,M}(1+\frac{4}{p+2})^2\text.
		\end{equation*}
		 If, moreover, $a\in \M_\A$, then $\beta_p(a) \in \M_\B$, with $\Lip(\beta_p(a)) \leq \Lip_\A(a)(1+\frac{4}{p+2})^2$.
	\end{lemma}

\begin{proof}	
	Let $c_m \in e^\B_p \targetsettunnel{f(m)}{a}{l} e^\B_p$ for all $m\in\N$, so that $(c_m)_{m\in\N}$ converges in $\B$ to $\beta_p(a)$. Furthermore, since $(e^\B_m)_{m\in\N}$ is an approximate unit for $\B$, 
	\begin{multline*}
		\norm{e^\B_{f(m)} c_m e^\B_{f(m)} - \beta_p(a)}{\B}
		\leq \underbracket[1pt]{\norm{e^\B_{f(m)}}{\B}^2}_{\leq 2} \norm{c_m - \beta_p(a)}{\B} \\ + \norm{e^\B_{f(m)} \beta_p(a) e^\B_{f(m)} - \beta^p(a)}{\B} \xrightarrow{m\rightarrow\infty} 0 \text, 
	\end{multline*}
	namely, $(e^\B_{f(m)} c_m e^\B_{f(m)})_{m\in\N}$ converges to $\beta_p(a)$ as well.
	
	Now, by Lemma (\ref{main-lemma}), we have
	\begin{equation*}
		\norm{e^\B_m c_m e^\B_m}{\B} \leq \norm{a}{\A} + l \tunnelextent{\tau_{f(m)}} \xrightarrow{n\rightarrow\infty} \norm{a}{\A} \text.
	\end{equation*}
	Therefore, $\norm{\beta_p(a)}{\B} \leq \norm{a}{\A}$, as claimed. Moreover, by Lemma (\ref{Leibniz-lemma}), 
	\begin{align*}
		\Lip_\B(c_m) 
		&\leq \norm{e^\B_p}{\B}\left( 2 \Lip(e^\B_p) l +\norm{e^\B_p}{\B} l \right) \\
		&\leq \sqrt{1+\tunnelextent{\tau_p}}\left(2\tunnelextent{\tau_p} + \sqrt{1+\tunnelextent{\tau_p}} \right)l \\
		&\leq \left(1+\frac{4}{p+2}\right)^2 l\text.
	\end{align*}
	Since $l > \norm{a}{\Lip_\A,M}$ is arbitrary, we conclude:
	\begin{equation*}
		\Lip_\B(c_m) \leq \left(1+\frac{4}{p+2}\right)^2 \norm{a}{\Lip_\A,M} \text,
	\end{equation*}
	as claimed.

	Now, assume in addition that $a \in \M_\A$. Then, by Definition (\ref{topographic-isometry-def}), there exists $b_m \in \targetsettunnel{m}{a}{l}\cap\M_\B $ for all $p,m\in\N$, with $\Lip_n(b_m)\leq \Lip_\A(a)+\frac{1}{m+2}$. 
	
	On the one hand, $e_p^\B(\targetsettunnel{f\circ g(m)}{a}{l}\cap\M_\B)e_p^\B \subseteq \M_\B$ since $e_p^\B \in \M_\B$, for all $m\in\N$. On the other hand, since $e_p^\B(\targetsettunnel{f\circ g(m)}{a}{l}\cap\M_\B)e_p^\B \subseteq e_p^\B(\targetsettunnel{f\circ g(m)}{a}{l})e_p^\B$ for all $m\in\N$, we conclude that $(e_p^\B(\targetsettunnel{f\circ g(m)}{a}{l}\cap\M_\B)e_p^\B)_{m\in\N}$ converges to $\{\beta_p(a)\}$ for $\Haus{\B}$. So in particular, $(e^\B_p b_m e^\B_p)_{m\in\N}$ converges to $\beta_p(a)$. Therefore, $\Lip_\B(\beta_p(a))\leq (1+\frac{4}{p+2})^2 \Lip_\A(a)$ by lower semi-continuity of $\Lip_\B$ and the Leibniz property, once more. Since $\M_\B$ is closed in $\A$, we conclude that $\beta_p(a) \in \M_\B$.
	
	This concludes the proof of our lemma.
\end{proof}

\begin{lemma}\label{core-lemma}
	Fix an strictly increasing function $f: \N\rightarrow\N$. For all $a\in\domsa{\Lip_\A}$ and for all $p\in\N$, there exists a strictly increasing function $g : \N \rightarrow\N$ and $\beta_p(a) \in \domsa{\Lip_\B}$, such that for all $l > \norm{a}{\Lip_\A,M}$,
	\begin{equation*}
		\lim_{n\rightarrow\infty} \Haus{\B}\left(e^\B_p \targetsettunnel{m}{a}{l} e^\B_p , \{ \beta_p(a) \} \right) = 0 \text{,}
	\end{equation*}
	and $\norm{\beta_p(a)}{\A} \leq \norm{a}{\A}$ while $\Lip_\B(\beta_p(a)) \leq \Lip_\A(a)$ and $\norm{\beta_p(a)}{\Lip_\A,\M} \leq l(1+\frac{4}{p+2})^2$. If, moreover, $a\in \M_\A$, then $\beta_p(a) \in \M_\B$, with $\Lip(\beta_p(a)) \leq \Lip_\A(a)(1+\frac{4}{p+2})^2$.
\end{lemma}

\begin{proof}
	Fix $p\in\N$, $a\in\domsa{\Lip_\A}$ and $l > \norm{a}{\Lip_\A,M}$. Since $\Lip_2(e^\B_m) \leq \tunnelextent{\tau_m} \leq \frac{1}{m+2}$ and $\norm{e^\B_m}{\B}\leq \sqrt{1+\tunnelextent{\tau_m}}\leq 2$, by Definition (\ref{targetset-def}), and Lemma (\ref{Leibniz-lemma}), we have the inclusion:
	\begin{align*}
		e^\B_p e^\B_m \targetsettunnel{m}{a}{l} e^\B_m e^\B_p 
		&\subseteq \left\{ e^\B_p b e^\B_p : \norm{b}{\Lip_\B} \leq l(4l+2) \right\} \text,
 	\end{align*}
	with the right hand side being a compact subset of $\sa{\B}$, since it is totally bounded by  Lemma (\ref{cad-totally-bounded-lemma}), and it is closed in $\B$,  which is complete. Therefore, the sequence
	\begin{equation*}
		\left( e^\B_p e^\B_{f(m)} \targetsettunnel{f(m)}{a}{l} e^\B_{f(m)} e^\B_p \right)_{m\in\N}
	\end{equation*}
	admits a convergent subsequence
	\begin{equation*}
		\left( e^\B_p e^\B_{f(g(m))} \targetsettunnel{f(m)}{a}{l} e^\B_{f(g(m))} e^\B_p \right)_{m\in\N}
	\end{equation*}
	for the Hausdorff distance $\Haus{\B}$, since the Hausdorff distance on a compact metric space induces a compact topology. Let $S_p \subseteq\B$ be its limit, and note it is a nonempty compact subset of $\sa{\B}$.
	
	By Lemma (\ref{main-lemma}),
	\begin{align*}
		0&\leq \diam{e_p^\B e_m^\B \targetsettunnel{m}{a}{l} e_m^\B e_p^\B}{\B} \\
		&\leq \norm{e^\B_p}{\B}^2 \diam{e_m^\B \targetsettunnel{m}{a}{l} e_m^\B}{\B} \\
		&\leq \underbracket[1pt]{\left(1+\tunnelextent{\tau_p}\right)}_{\geq\norm{e_p^\B}{\B}^2 \text{ by Lemma (\ref{key-lemma})}} 2 \tunnelextent{\tau_m} l \xrightarrow{m\rightarrow\infty} 0 \text.
	\end{align*}
	Therefore, by the continuity of the diameter for the Hausdorff distance, $\diam{S_p}{\B} = 0$. Since $S_p$ is not empty, it is a singleton, which we henceforth denote by $S_p = \{ \beta_p(a) \}$.
	
	By Lemma (\ref{no-em-lemma}), we then have by the triangle inequality for the Hausdorff distance $\Haus{\B}$:
	\begin{align*}
		0 &\leq
	 \Haus{\B}\left(e^\B_p \targetsettunnel{f\circ g(m)}{a}{l} e^\B_p, \{ \beta_p(a) \}\right) \\ 
		&\leq \Haus{\B}\left(e^\B_p \targetsettunnel{f\circ g(m)}{a}{l} e^\B_p, e^\B_p  e^\B_{f\circ g(m)} \targetsettunnel{f\circ g(m)}{a}{l} e^\B_{f\circ g(m)} e^\B_p \right) \\ 
		&\quad+ \Haus{\B}\left(e^\B_p e^\B_{f\circ g(m)} \targetsettunnel{f\circ g(m)}{a}{l} e^\B_{f\circ g(m)} e^\B_p, \{ \beta_p(a) \}\right) \\
		&\xrightarrow{m\rightarrow\infty} 0 + 0 = 0 \text. 
	\end{align*}
	
	Now, let $l' > \norm{a}{\Lip_\A,M}$. By Definition (\ref{targetset-def}), the following inclusion holds for all $m\in\N$:
	\begin{equation*}
		\targetsettunnel{m}{a}{\min\{l,'l\} } \subseteq \targetsettunnel{m}{a}{\max\{l',l\}} \text.
	\end{equation*}
	Therefore, by definition of the Hausdorff distance:
	\begin{align*}
		\, & \Haus{\B}\Big(e^\B_p e^\B_{f\circ g(m)} \targetsettunnel{f\circ g(m)}{a}{\min\{l,'l\} } e^\B_{f\circ g(m)} e^\B_p, \\
		&\quad\quad e^\B_p e^\B_{f\circ g(m)} \targetsettunnel{f\circ g(m)}{a}{\max\{l',l\}}e^\B_{f\circ g(m)} e^\B_p \Big)  \\
		&\leq \diam{ e^\B_{f\circ g(m)} \targetsettunnel{f\circ g(m)}{a}{\max\{l',l\}} e^\B_{f\circ g(m)} }{ \B} \norm{e_p^\B}{\B}^2 \\
		&\leq  2 \max\{l,l'\} \tunnelextent{\tau_{f\circ g(m)}} \cdot \left(1+\tunnelextent{\tau_{p}}\right)\\
		&\xrightarrow{m\rightarrow\infty} 0  \text.
	\end{align*}
	Therefore, by the triangle inequality,
	\begin{equation*}
		 \lim_{m\rightarrow\infty} \Haus{\B}\left(e^\B_p \targetsettunnel{f\circ g(m)}{a}{l'} e^\B_p, \{ \beta_p(a) \}\right) = 0 \text.
	\end{equation*}
	
	By Lemma (\ref{pre-core-lemma}), we conclude that $\beta_p(a)$ has the listed properties in its conclusion.
\end{proof}

\begin{corollary}\label{core-cor}
	There exists a strictly increasing function $f : \N\rightarrow\N$ and, for each $p\in\N$, a function $\beta_p : \domsa{\Lip_\A} \rightarrow\domsa{\Lip_\B}$ such that, for all $p\in\N$, for all $a\in \domsa{\Lip_\A}$, and for all $l > \norm{a}{\Lip_\A,M}$, 
	\begin{equation*}
		\lim_{m\rightarrow\infty} \Haus{\B}\left( e^\B_p \targetsettunnel{f(m)}{a}{l} e^\B_p , \{ \beta_p(a) \} \right) = 0 \text,
	\end{equation*}
	while $\norm{\beta_p(a)}{\B} \leq \norm{a}{\A}$ and $\norm{\beta_p(a)}{\Lip_\B,M} \leq \left(1+\frac{4}{p+1}\right)^2\norm{a}{\Lip_\A,M}$. Moreover, the restriction of $\beta_p$ to $\M_\A\cap\domsa{\Lip_\A}$ is valued in $\M_\B\cap\domsa{\Lip_\B}$ and $\Lip_\B(\beta_p(a))\leq (1+\frac{4}{p+2})^2\Lip_\A(a)$.
\end{corollary}

\begin{proof}
	For this proof, let $S\subseteq\domsa{\Lip_\A}$ be a countable dense subset of $\domsa{\Lip_\A}$ for $\norm{\cdot}{\A}$, which exists as $\A$ is separable. Write $S\coloneqq\{ a_j : j \in \N \}$. It will be helpful to choose any bijection $k : \N^2 \mapsto \N$.
	
	We first proceed by induction. Our induction hypothesis at $n\in\N$ is: if $k(J,P) = n$, then there exist strictly increasing functions $f_0, f_1, \ldots, f_n :\N\rightarrow\N$ such that for all $(p,j) \in \N^2$ such that $k(p,j) \leq n$, the sequence 
	\begin{equation*}
		\left(e^\B_p \targetsettunnel{f_0 \circ f_1 \circ \ldots \circ f_n(m)}{a_j}{\norm{a}{\Lip_\A,M}+ 1} e^\B_p \right)_{m\in\N}
	\end{equation*}
	converges to a singleton for $\Haus{\B}$, which we denote by $\{\beta_p(a_j)\}$. Moreover, $\norm{\beta_p(a_j)}{\B} \leq \norm{a}{\A}$, and $\norm{\beta_p(a)}{\Lip_\B,M} \leq \left(1+\frac{4}{p+1}\right)^2\norm{a}{\Lip_\A,M}$. If, moreover, $a_j \in \M_\A$, then $\beta_p(a)\in\M_\B$ and $\Lip_\B(\beta_p(a))\leq(1+\frac{4}{p+2})^2\Lip_\A(a)$.
	
	Let $(p,j) = k^{-1}(0)$. By Lemma (\ref{core-lemma}), there exists a strictly increasing function $f_0$ and $\beta_p(a_j) \in \domsa{\Lip_\B}$ with the desired properties, such that, for all $l > \norm{a_j}{\Lip_\A,M}$,
	\begin{equation*}
		\lim_{m\rightarrow\infty} \Haus{\B}\left( e^\B_p \targetsettunnel{f_0(m)}{a_j}{l} e^\B_p, \{\beta_p(a_j) \right) = 0 
	\end{equation*}
So our induction hypothesis holds for $n=0$.
	
	 Now, assume our hypothesis holds for some $n \in \N$. Let $(p,j) = k^{-1}(n+1)$. We again apply Lemma (\ref{core-lemma}) to obtain a strictly increasing function $f_{n+1} : \N\rightarrow\N$ such that, for all $l > \norm{a_j}{\Lip_\A,\M}$,
	\begin{equation*}
		\lim_{m\rightarrow\infty} \Haus{\B}\left( e^\B_p \targetsettunnel{(f_0\circ \ldots f_n) \circ f_{n+1}(m)}{a_j}{l} e^\B_p, \{ \beta_p(a_j) \} \right) = 0 \text,
	\end{equation*}
	for some $\beta_p(a_j) \in \domsa{\Lip_\B}$, again with all the desired properties. Since subsequences of convergent sequences converge to the same limit, our induction hypothesis now holds for $n+1$.

	We now simply set $f : n \in \N \mapsto f_0\circ f_1 \circ \ldots \circ f_n(n)$. It is immediate that for all $(p,j)\in\N^2$, for all $l > \norm{a_j}{\Lip_{\A,M}}$, we have
	\begin{equation*}
		\lim_{m\rightarrow\infty} \Haus{\B}\left( e^\B_p \targetsettunnel{f(m)}{a_j}{l} e^\B_p, \{ \beta_p(a_j) \right) = 0 \text,
	\end{equation*}
	again, with $\beta_p(a_j)$ having all the properties listed in our induction hypothesis (or as the conclusion of Lemma (\ref{core-lemma}) from which they arise).
	
	\medskip
	
	Let now $a \in \domsa{\Lip_\A}$, $p\in\N$. Let $\varepsilon > 0$. Since $S$ is dense in $\sa{\A}$, there exists $j \in \N$ such that $\norm{a-a_j}{\A} < \frac{\varepsilon}{24}$. Let $l > \max\left\{ \norm{a}{\Lip_\A,M}, \norm{a_j}{\Lip_\A,M} \right\}$. 
	
	Let $N_0 \in \N$ such that, if $m\geq N_0$, then $\tunnelextent{\tau_m} < \frac{\varepsilon}{24 l}$. 
	
	 By Lemma (\ref{main-lemma}), we then note that for all $m \geq N_0$, whenever $b \in \targetsettunnel{m}{a}{l}$ and $b'\in\targetsettunnel{m}{a_j}{l}$:
	\begin{equation*}
		\norm{e^\B_m ( b - b' ) e^\B_m}{\B} \leq \norm{a-a_j}{\A} + l \tunnelextent{\tau_m} \leq \frac{\varepsilon}{24} + l \frac{\varepsilon}{24 l} = \frac{\varepsilon}{12} \text.
	\end{equation*}
	So $\Haus{\B}(e^\B_m \targetsettunnel{m}{a}{l} e^\B_m, e^\B_m \targetsettunnel{m}{a_j}{l} e^\B_m) < \frac{\varepsilon}{12}$ for all $m\geq N_0$. Therefore,
	\begin{multline*}
		\Haus{\B}(e^\B_p e^\B_m \targetsettunnel{m}{a}{l} e^\B_m e^\B_p , e^\B_p e^\B_m \targetsettunnel{m}{a_j}{l} e^\B_m e^\B_p) \\ \leq \norm{e^\B_p}{\B}^2 \Haus{\B}(e^\B_m \targetsettunnel{m}{a}{l} e^\B_m, e^\B_m \targetsettunnel{m}{a_j}{l} e^\B_m) < \frac{\varepsilon}{6} 
	\end{multline*}
	since $\norm{e^\B_p}{\B}^2 \leq 1 + \frac{1}{p+2} \leq 2$.

	Since, by Lemma (\ref{no-em-lemma}),
	\begin{equation*}
		\lim_{m\rightarrow\infty} \Haus{\B}(e^\B_p \targetsettunnel{m}{a}{l} e^\B_p, e^\B_p e^\B_m \targetsettunnel{m}{a}{l} e^\B_m e^\B_p) = 0\text,
	\end{equation*}
	there exists $N_1 \in \N$ such that, for all $m\geq N_1$:
	\begin{equation*}
		\Haus{\B}(e^\B_p \targetsettunnel{m}{a}{l} e^\B_p, e^\B_p e^\B_m \targetsettunnel{m}{a}{l} e^\B_m e^\B_p) < \frac{\varepsilon}{6} \text.
	\end{equation*}
	
	Therefore, if $m\geq\max\{N_0,N_1\}$,
	\begin{align*}
		&\Haus{\B}\left(e^\B_p \targetsettunnel{m}{a}{l} e^\B_p, e^\B_p e^\B_{m} \targetsettunnel{m}{a_j}{l} e^\B_{m} e^\B_p\right)\\
		&\quad\leq \Haus{\B}\left(e^\B_p \targetsettunnel{m}{a}{l} e^\B_p, e^\B_p e^\B_{m} \targetsettunnel{m}{a}{l} e^\B_{m} e^\B_p\right) \\
		&\quad + \Haus{\B}\left(e^\B_p e^\B_{m} \targetsettunnel{m}{a}{l} e^\B_{m} e^\B_p, e^\B_p e^\B_{m} \targetsettunnel{m}{a_j}{l} e^\B_{m} e^\B_p\right) \\
	&\quad <\frac{\varepsilon}{6} + \frac{\varepsilon}{6} = \frac{\varepsilon}{3} \text.
	\end{align*}
	
	We have just established that the sequence	
	\begin{equation*}
		( e^\B_p \targetsettunnel{f(m)}{a_j}{l} e^\B_p)_{m\in\N} 
	\end{equation*}
	converges for $\Haus{\B}$; therefore by Lemma (\ref{no-em-lemma}), so does the sequence 
	\begin{equation*}
		( e^\B_p e^\B_{f(m)} \targetsettunnel{f(m)}{a_j}{l} e^\B_{f(m)} e^\B_p)_{m\in\N}\text. 
	\end{equation*}
	Therefore, it is Cauchy for $\Haus{\B}$, so there exists $N_2 \in \N$ such that, for all $n,m \geq N_2$,
	\begin{equation*}
		\Haus{\B}( e^\B_p e^\B_{f(m)} \targetsettunnel{f(m)}{a_j}{l}  e^\B_{f(m)} e^\B_p, e^\B_p e^\B_{f(n)} \targetsettunnel{f(n)}{a_j}{l}e^\B_{f(n)} e^\B_p \} ) < \frac{\varepsilon}{3} \text.
	\end{equation*}
	
	Let $K \coloneqq \max\{ N_0, N_1, N_2 \}$.
	
	Therefore, for all $m,n \in \N$ with $n\geq K$ and $m\geq K$:
	\begin{align*}
		&\Haus{\B}(e^\B_p \targetsettunnel{f(m)}{a}{l} e^\B_p , e^\B_p \targetsettunnel{f(n)}{a}{l} e^\B_p)\\
		&\quad\leq \Haus{\B}(e^\B_p \targetsettunnel{f(m)}{a}{l} e^\B_p , e^\B_p e^\B_{f(m)} \targetsettunnel{f(m)}{a_j}{l} e^\B_{f(m)} e^\B_p)\\
		&\quad\quad + \Haus{\B}(e^\B_p e^\B_{f(m)} \targetsettunnel{f(m)}{a_j}{l} e^\B_{f(m)} e^\B_p , e^\B_p e^\B_{f(n)} \targetsettunnel{f(n)}{a_j}{l} e^\B_{f(n)} e^\B_p) \\
		&\quad\quad + \Haus{\B}(e^\B_p e^\B_{f(n)}\targetsettunnel{f(n)}{a_j}{l} e^\B_{f(n)} e^\B_p , e^\B_p \targetsettunnel{f(n)}{a}{l} e^\B_p) \\
		&\quad< \frac{\varepsilon}{3} + \frac{\varepsilon}{3} + \frac{\varepsilon}{3} = \varepsilon \text.
	\end{align*}
	
	Therefore, $\left( e^\B_p \targetsettunnel{f(m)}{a}{l} e^\B_p \right)_{m\in\N}$ is a Cauchy sequence for $\Haus{\B}$. Since $\B$ is complete, so is $\Haus{\B}$. Therefore, $\left( e^\B_p \targetsettunnel{f(m)}{a}{l} e^\B_p \right)_{m\in\N}$ converges. As above, its limit is a singleton $\{\beta_p(a)\}$ by Lemma (\ref{main-lemma}). By Lemma (\ref{pre-core-lemma}), once more, $\beta_p(a)$ has the desired properties. This completes our proof.
\end{proof}

By symmetry, this results hold as well switching the roles of $\A$ and $\B$. In other words, let $f : \N\rightarrow\N$ be given by Corollary (\ref{core-cor}). Then, applying Corollary (\ref{core-cor}) once again, but to the sequence of tunnels $(\tau_{f(m)}^{-1})_{m\in\N}$, there exists a strictly increasing function $g : \N\rightarrow\N$, and a function $a_p :\domsa{\Lip_\B}\rightarrow\domsa{\Lip_\A}$ for each $p \in \N$, such that, for all $b \in \domsa{\Lip_\B}$, and for all $l > \norm{b}{\Lip_\B,M}$,
\begin{equation*}
	\lim_{m\rightarrow\infty} \Haus{\A}(e^\A_p \targetsettunnel{-f(g(m))}{b}{l} e^\A_p, \{ \alpha_p(b) \} ) = 0
\end{equation*}
with $\norm{\alpha_p(b)}{\A} \leq \norm{b}{\B}$ and $\norm{\alpha_p(a))}{\Lip_\A, M} \leq \left(1+\frac{4}{p+2}\right)^2 \norm{b}{\Lip_\B,M}$; moreover if $b\in\M_\B\cap\domsa{\Lip_\B}$, then $\alpha_p(b) \in \M_\A$ and $\Lip_\A(\alpha_p(a)) \leq  \left(1+\frac{4}{p+2}\right)^2\Lip_\B(b)$.

\medskip

\begin{remark}
\emph{To simplify our notation moving forward in this proof, we replace our original sequence of tunnels $(\tau_{n})_{n\in\N}$ with the sequence $(\tau_{f(g(n))})_{n\in\N}$ without further mention of the function $f\circ g$}.
\end{remark}

\medskip
	
	In general, however, there is no connection between the various elements $e_n^\B$ for varying $n \in \N$. Instead, for each $p\in\N$ and for each $m\in \N$, since $\pi_m$ is a \emph{topographic} quantum $M$-isometry, and since $e^\B_p \in \M_\B$, there exists $k_{p,m} \in \domsa{\Lip_n}\cap\M_m$ such that 
	\begin{equation*}
		\rho_m(k_{p,m}) = e^\B_p\text{, }\norm{k_{p,m}}{\D_n}\leq \norm{e^\B_m}{\A}+\frac{1}{m+2} \text{ and }\Lip_m(k_{p,m}) \leq \Lip_\A(e^\B_p) + \frac{1}{m+2} \text.
	\end{equation*}
	Note that $\norm{k_{p,m}}{\D_n} \leq 2$ by construction, for all $p,m \in \N$. We now denote:
	\begin{equation*}
		s_{p,m} \coloneqq \pi_m(k_{p,m}) \text{ and }h_{p,m} \coloneqq e^\A_p s_{m,p} e^\A_p \in e^\A_p \in  \targetsettunnel{-m}{e^\B_p}{2} e^\A_p \text.
	\end{equation*}
	Again, note that $\norm{s_{p,m}}{\A} \leq 2$ for all $p,m \in \N$.
	
	Note that by construction, since $\Lip_\A\circ\pi_m \leq \Lip_\D$, using Lemma (\ref{Leibniz-lemma}),
	\begin{align}
		\Lip_\A(h_{p,m}) 
		&\leq \norm{e^\A_p}{\A} \left( \underbracket[1pt]{2 \Lip(e^\A_p)}_{\leq\tunnelextent{\tau_p}}\norm{k_{p,m}}{\D_n} + \norm{e^\A_p}{\A} \Lip_\D(k_{p,m})\right) \nonumber \\
		&\leq \underbracket[1pt]{\sqrt{1+\frac{1}{p+2}}}_{\text{by Lemma (\ref{key-lemma})}}\left(\frac{1}{p+2}\left(\sqrt{1+\frac{1}{p+2}}+\frac{1}{m+2}\right) + \sqrt{1+\frac{1}{p+2}} \left(\frac{1}{p+2} + \frac{1}{m+2}\right)\right) \\
		&\leq \frac{6}{p+2} + \frac{4}{m+2} \leq 5\nonumber \text.
	\end{align}

In particular, from Corollary (\ref{core-cor}), we obtain:
\begin{corollary}
	For each $p\in\N$, the sequence $(e_p^\A  h_{p,m} e_p^\A)_{m\in\N}$  converges to $h_p \in \domsa{\Lip_\A}$, with $\Lip_\A(h_p) \leq \frac{6}{p+1}$, and such that $(h_p)_{p\in\N}$ is a \lipunit{\Lip_\A}{\mu_\A},  and therefore, $(h_p)_{p\in\N}$  is an approximate unit of $\A$.
\end{corollary}

\begin{proof}
	Since $h_{p,m} \in e^\A_p \targetsettunnel{-m}{e^\B_p}{5} e^\A_p$ for all $m\in\N$, Corollary (\ref{core-cor}) applies, and thus, there exists $h_p  \coloneqq \alpha_p(e^\B_p) \in \domsa{\Lip_\A}$ such that 
	\begin{equation*}
		\lim_{m\rightarrow\infty} \Haus{\A}(e^\A_p \targetsettunnel{-f(m)}{e^\B_p}{5} e^\A_p, \{ h_p \} ) = 0 \text.
	\end{equation*}
	In particular, $(h_{p,m})_{m\in\N}$ converges to $h_p$ in the norm of $\A$. So $\norm{h_p}{\A} \leq \norm{e_p}{\D_p}^2 \leq 1+\frac{1}{2+p}$.

	Since $\Lip_\A$ is lower semi-continuous, and since $\Lip_\A(h_{p,m}) \leq \frac{6}{p+2} + \frac{4}{m+2}$, we conclude that $\Lip_\A(h_p) \leq \frac{6}{p+1}$. 
	
	On the other hand, 
	\begin{align*}
	|\mu_\A\circ\pi_m(k_{p,m})-1| 
		&\leq |\mu_\A\circ\pi_m(k_{p,m}) - \mu_\B\circ\rho_m(k_{p,m})| +  |\mu_\B\circ\rho_m(k_{p,m})-1| \\
		&\leq \frac{6}{p+1} \Kantorovich{\Lip_n}(\mu_\A\circ\pi_\A,\mu_\B\circ\pi_\B) + |\mu_\B(e^\B_p) - 1| \\
		&\leq \frac{6}{p+2}\tunnelextent{\tau_m} + \tunnelextent{\tau_p} \\
		&\leq \frac{6}{p+2}\frac{1}{m+2} + \frac{1}{p+2}\text.
	\end{align*}
	Therefore, $\lim_{m\rightarrow\infty} \mu_\A(\pi_m(k_{p,m})) = 1$. Since $\mu_\A$ restricts to a character on $\M_\A$, then
	\begin{equation*}
		\mu_\A(h_p) = \lim_{m\rightarrow\infty} \mu_\A(e^\A_p)^2 \mu_\A(\pi_m(k_{p,m})) \text.
	\end{equation*}
	
	Now, $|1-\mu_\A(e^\A_p)| \leq \tunnelextent{\tau_p} \leq \frac{1}{p+2}$, so $\lim_{p\rightarrow\infty} \mu_\A(e^\A_p) = 1$.

	In particular, for all $p\in\N$, we have shown that
	\begin{equation*}
		|\mu_\A(h_p)|\leq \norm{h_p}{\A} \leq \sqrt{1 + \frac{1}{p+2}}
	\end{equation*}
	so by the squeeze theorem, $\lim_{n\rightarrow\infty}\norm{h_p}{\A} = 1$. 
	
	As $\lim_{p\rightarrow\infty} \Lip_\A(h_p) = 0$, while $\lim_{n\rightarrow\infty}\norm{h_p}{\A} = \lim_{m\rightarrow\infty} \mu_\A(h_{p,m}) = 1$, so $\mu_\A(h_p) =1$ for all $p\in\N$, we conclude by Theorem (\ref{approx-unit-thm}) that $(h_p)_{p\in\N}$ is an approximate unit of $\A$.
\end{proof}

In summary: if $m\in\N$, then $e^\A_m \in \targetsettunnel{-m}{e^\B_m}{2}$, and for all $p \in \N$, we have $s_{p,m} \in \targetsettunnel{-m}{e^\B_p}{2}$; moreover $(h_{p,m})_{m\in\N} = (e^\A_p s_{p,f(m)} e^\A_p)_{m\in\N}$ converges to some $h_p$ such that $(h_p)_{p\in\N}$ is a \lipunit{\Lip_\A}{\mu_\A}. By symmetry, we also note that $e^\B_m \in \targetsettunnel{m}{e^\A}{2}$ and $e^\B_p \in \targetsettunnel{m}{s_{p,m}}{2}$, noting the roles of $m$ and $p$: we can lift $e^\A_m$ to $e^\B_m$ via the tunnel $\tau_m$, and we can lift $s_{p,m}$ to $e^\B_p$ via the same tunnel $\tau_m$ when $p$ is now arbitrary. As mentioned, we could impose $s_{m,m} = e^\A_m$ but this does not prove to be of any additional help.

\medskip

We are now ready to prove the next important step in the construction of our isomorphism.

 \begin{lemma}\label{Cauchy-lemma}
For all $a\in\domsa{\Lip_\A}$, the sequence $(\beta_p(a))_{p\in\N}$ is Cauchy in $\B$. We denote its limit by $\pi(a)$. By construction, $\norm{\pi(a)}{\A}\leq\norm{a}{\A}$, $\norm{\pi(a)}{\Lip_\B,M} \leq \norm{a}{\Lip_\A,M}$ and, if $a \in \M_\A\cap\domsa{\Lip_\A}$, then $\pi(a) \in \M_\B\cap\domsa{\Lip_\B}$, and $\Lip_\B(\pi(a)) \leq \Lip_\A(a)$.
\end{lemma}

\begin{proof}
	We write $b_p \coloneqq \beta_p(a)$ for all $p \in \N$. Fix $l \coloneqq \norm{a}{\Lip_\A,M} + 1$.
	
	Let $\varepsilon > 0$.	Since $(h_p)_{p\in\N}$ is an approximate unit of $\A$, there exists $N_1\in\N$ such that, for all $p,q \geq N_1$,
	\begin{equation*}
		\norm{ h_p a h_p - h_q a h_q }{\A} < \frac{\varepsilon}{7} \text.
	\end{equation*}
	
	Since $(e^\A_n)_{n\in\N}$ is also an approximate unit for $\A$, there exists $N_2\in\N$ such that, for all $n\geq N_2$, we have
	\begin{equation*}
		\norm{a - (e^\A_n)^2 a (e^\A_n)^2}{\A} < \frac{\varepsilon}{56} \text.		
	\end{equation*}

	Let $N\coloneqq\max\{ N_1, N_2 \}$. Fix $p\geq N$ and $q\geq N$.

	\medskip
	
	Let $K_0\in\N$ such that, if $m\geq K_0$, then $\tunnelextent{\tau_m} < \frac{\varepsilon}{14 l (4 M + 2)}$.
	
	By definition of $h_p$ and $h_q$, there exists $K_1\in\N$ such that, for all $m\geq K_1$:
	\begin{equation*}
		\norm{h_p - h_{p,m}}{\A} < \frac{\varepsilon}{14\norm{a}{\A}+1} \text{ and }\norm{h_q - h_{q,m}}{\A} < \frac{\varepsilon}{14\norm{a}{\A}+1} \text.
	\end{equation*}
	
	Moreover, since $(e^\B_m)_{m\in\N}$ is an approximate unit for $\B$, there exists $K_2 \in \N$ such that, for all $m\geq K_2$:
	\begin{equation}\label{Cauchy-Lemma-eq-bp-approx}
		\norm{b_p - e^\B_m b_p e^\B_m}{\B} < \frac{\varepsilon}{14} \text{ and }\norm{b_q - e^\B_m b_q e^\B_m }{\B} < \frac{\varepsilon}{14} \text.
	\end{equation}
	
	By definition of $b_p$ and $b_q$, there exists $K_3 \in \N$ such that, for all $m\geq K_3$,
	\begin{equation}\label{Cauchy-Haus-eq}
		\Haus{\B}\left(e^\B_p \targetsettunnel{m}{a}{l} e^\B_p, \{ b_p \}\right) < \frac{\varepsilon}{56} \text{ and } \Haus{\B}\left(e^\B_q \targetsettunnel{m}{a}{l} e^\B_q, \{ b_q \}\right) < \frac{\varepsilon}{56} \text.
	\end{equation}
	
	Let $m = \max\{K_0, K_1,K_2,K_3\}$.
	
	Let $t_m \in \targetsettunnel{m}{a}{l}$. 
	
	Recall that $s_{p,m} \coloneqq \pi_m(k_{p,m})$. It is important that $s_{p,m} \in \M_\A$ since $k_{p,m} \in \M_n$ and $\pi_n(\M_n) \subseteq \M_\A$.

	 Note that by construction, since $t_m \in \targetsettunnel{m}{a}{l}$, there exists $d_m \in \domsa{\Lip_m}$ such that $\pi_m(d_m) = a$, $\rho_m(d_m) = t_m$ and $\norm{d_m}{\Lip_m,M} \leq l$. Thus 
	$\pi_m(k_{p,m} d_m k_{p,m}) = s_{p,m} a s_{p,m}$, while 
	\begin{align*}
		\Lip_m(k_{p,m} d_m k_{p,m}) 
		&\leq 2\Lip_m(k_{p,m})\norm{d_m}{\D}\norm{k_{p,m}}{\D_n} + \norm{k_{p,m}}{\D}^2 \Lip_m(d_m) \\
		&\leq 4\left(\frac{1}{p+2} + \frac{1}{m+2}\right) M l + (1+\frac{1}{p+2}) l \\
		&\leq ( 4M + 2 ) l  \text.
	\end{align*}
	Therefore, by Definition (\ref{targetset-def}):
	\begin{equation*}
		e_p^\B t_m e_p^\B \in \targetsettunnel{m}{s_{p,m} a s_{p,m}}{(4M + 2) l} \text{ and }e_q^\B t_m e_q^\B \in \targetsettunnel{m}{s_{q,m} a s_{q,m}}{(4M +2) l} \text.
	\end{equation*}
	
	Moreover
	\begin{align}\label{Cauchy-eq-a}
		\norm{ s_{p,m} a s_{p,m} - a }{\A}
		&\leq\norm{ s_{p,m} a s_{p,m} - s_{p,m} (e^\A_p)^2 a (e^\A_p)^2 s_{p,m}}{\A} + \norm{s_{p,m} (e^\A_p)^2 a (e^\A_p)^2 s_{p,m} - a}{\A} \nonumber \\ 
		&\leq \underbracket[1pt]{\norm{ s_{p,m} }{\A}^2}_{\leq 4} \underbracket[1pt]{\norm{ a - (e^\A_p)^2 a (e^\A_p)^2 }{\A}}_{<\frac{\varepsilon}{56}} + \norm{s_{p,m} (e^\A_p)^2 a (e^\A_p)^2 s_{p,m} - a}{\A}\\
		&\leq \frac{\varepsilon}{14} + \norm{s_{p,m} (e^\A_p)^2 a (e^\A_p)^2 s_{p,m} - a}{\A} \text. \nonumber
	\end{align}
	
	Therefore:
	\begin{align*}
		\norm{b_p - b_q}{\B}
		&\leq \norm{e_m^\B (b_p - b_q) e_m^\B}{\B} + \frac{\varepsilon}{7}  \text{ by Expression \eqref{Cauchy-Lemma-eq-bp-approx}}\\
		&\leq \norm{e_m^\B ( e_p^\B t_m e_p^\B - e_q^\B t_m e_q^\B) e_m^\B}{\B} + \frac{2\varepsilon}{7}\\
		&\leq \norm{ s_{p,m} a s_{p,m} - s_{q,m} a s_{q,m} }{\A} + 2l(4M+2)\tunnelextent{\tau_m} + \frac{2\varepsilon}{7} \text{ by Lemma (\ref{main-lemma})} \\
		&\leq \norm{ s_{p,m} a s_{p,m} - s_{q,m} a s_{q,m} }{\A} + \frac{3\varepsilon}{7} \\
		&\leq \norm{ s_{p,m} (e_p^\A)^2 a (e_p^\A)^2 s_{p,m} - s_{q,m} (e_q^\A)^2 a (e_q^\A)^2 s_{q,m} }{\A} + \frac{\varepsilon}{7} + \frac{3\varepsilon}{7} \text{ by Exp. \eqref{Cauchy-eq-a},}\\
		&= \norm{ \underbracket[1pt]{(e^\A_p s_{p,m} e_p^\A)}_{\text{since $[s_{p,m},e^\A] = 0$}} a (e_p^\A s_{p,m} e^\A_p) - (e_q^\A s_{q,m} e_q^\A) a (e_q^\A s_{q,m} e_q^\A) }{\A} + \frac{4\varepsilon}{7} \text{ by Exp. \eqref{Cauchy-Haus-eq},}\\
		&=\norm{ h_{p,m} a h_{p,m} - h_{q,m} a h_{q,m} }{\A} + \frac{4\varepsilon}{7} \\
		&\leq 2 \left(\underbracket[1pt]{\norm{h_{p,m}-h_p}{\B} + \norm{h_{q,m}-h_q}{\A}}_{< \frac{\varepsilon}{7\norm{a}{\A} + 1}} \right)\norm{a}{\A} + \underbracket[1pt]{\norm{h_p a h_p - h_q a h_q}{\A}}_{<\frac{\varepsilon}{7}} + \frac{4\varepsilon}{7} \\
		&< \varepsilon \text.
	\end{align*}
	 This concludes our proof that $(\beta_p(a))_{p\in\N}$ is Cauchy in $\B$. Since $\B$ is complete, we conclude that there exists $\pi(a) \in \B$ such that $\lim_{p\rightarrow\infty} \beta_p(a) = \pi(a)$.
	 
	 Now, $\norm{\pi(a)}{\B} = \lim_{n\rightarrow\infty}\norm{\beta_p(a)}{\B} \leq \norm{a}{\A}$ by Corollary (\ref{core-cor}). Moreover, by lower semi-continuity,
	 \begin{equation*}
	 	\norm{\pi(a)}{\Lip_\B,M} \leq \liminf_{p\rightarrow\infty} \norm{a}{\Lip_\A,M}\left(1+\frac{4}{p+2}\right)^2 = \norm{a}{\Lip_\A,M}\text.
	 \end{equation*}
	 
	 Last, if $a\in\M_\A\cap\domsa{\Lip}$, then, again by lower semicontinuity of $\Lip_\B$, we conclude that $\Lip_\B(\pi(a)) \leq \Lip_A(a)$. This concludes our proof.
\end{proof}

In summary, for all $a\in \domsa{\Lip_\A}$, we have constructed an element $\pi(a) \in \domsa{\Lip_\B}$ with $\norm{\pi(a)}{\B}\leq\norm{a}{\A}$ and $\Lip_\B(\pi(a))\leq \Lip_\A(a)$. 

By employing the same method for the reversed tunnels, we also have that $(\alpha_p(b))_{p\in \N}$ converges to some $\theta(b)$ for all $b\in\domsa{\Lip_\B}$. We will return to this candidate for an inverse; first we extend $\pi$ to a proper *-morphism from $\A$ to $\B$.

\medskip

We now prove that $\pi$ thus constructed is linear.

\begin{lemma}\label{bp-linear-lemma}
	For each $p \in \N$, the map $\beta_p : \domsa{\Lip_\A} \rightarrow \domsa{\Lip_\B}$ is $\R$-linear and of norm at most $1$.
\end{lemma}

\begin{proof}
	Let $a,a' \in \domsa{\Lip_\A}$ and $t\in\R$. Let $l > \max\{ \norm{a}{\Lip_\A,M}, \norm{a'}{\Lip_\A,M}\}$. 
	
	Lemma (\ref{linearity-lemma}) gives us, for all $m\in\N$, the inclusion
	\begin{equation*}
		\targetsettunnel{m}{a}{l} + \targetsettunnel{m}{a'}{l} \subseteq \targetsettunnel{m}{t a + a'}{(1+|t|)l} \text.
	\end{equation*}
	Therefore,
	\begin{multline*}
		0\leq \Haus{\B}\left( e^\B_p \left(\targetsettunnel{m}{a}{l} + \targetsettunnel{m}{a'}{l}\right) e^\B_p,  e^p_\B \targetsettunnel{m}{t a + a'}{(1+|t|)l} e^\B_p \right) \\ \leq \diam{e^\B_p \targetsettunnel{m}{t a + a'}{(1+|t|)l} e^\B_p}{\B} \xrightarrow{m\rightarrow\infty} 0 \text.
	\end{multline*}
	By Lemma (\ref{core-cor}), 
	\begin{equation*}
		\lim_{m\rightarrow\infty} \Haus{\B}\left(e^\B_p \targetsettunnel{m}{t a +a'}{(1+|t|)l} e^\B_p, \{ \beta_p(t a + a') \} \right) = 0\text.
	\end{equation*}
	Therefore, by the triangle inequality for the Hausdorff distance $\Haus{\B}$, we conclude:
	\begin{equation}\label{bp-linear-lemma-cv-eq}
		\lim_{m\rightarrow\infty} \Haus{\B}\left( e^\B_p (\targetsettunnel{m}{a}{l} + \targetsettunnel{m}{a'}{l}) e^\B_p, \{ \beta_p(t a + a') \} \right) = 0 \text.
	\end{equation}
	
	Now, let $c_m \in e^\B_p \targetsettunnel{m}{a}{l} e^\B_p$ and $c'_m \in e^\B_p \targetsettunnel{m}{a'}{l} e^\B_p$. By Lemma (\ref{core-cor}), we conclude that $\lim_{m\rightarrow\infty} c_m = \beta_p(a)$ and $\lim_{m\rightarrow\infty} c'_m = \beta_p(a')$, so $\lim_{m\rightarrow\infty} (t c_m + c'_m) = t \beta_p(a) + \beta_p(a')$. On the other hand, by Equation \eqref{bp-linear-lemma-cv-eq}, we also conclude that $\lim_{m\rightarrow\infty} t c_m + c'_m = \beta_p(t a + a')$. By uniqueness of limits, we conclude $\beta_p(a + t a') = t \beta_p(a) + \beta_p(a')$, as wished.
\end{proof}	

\begin{corollary}
	The map $\pi : \domsa{\Lip_\A} \rightarrow \domsa{\Lip_\B}$ is a continuous linear map with $\opnorm{\pi}{}{}\leq 1$, which therefore has a unique continuous  extension to $\sa{\A}$, necessarily linear, and still denoted by $\pi$. 
\end{corollary}

\begin{proof}
	$\pi$ is, by definition, the pointwise limit of linear maps of $\domsa{\Lip_\A}$, valued in $\domsa{\Lip_\B}$, all with norm uniformly at most $1$, so $\pi$ is indeed linear with norm at most $1$. It thus extends uniquely by uniform continuity to the closure of $\domsa{\Lip_\A}$, which is $\sa{\A}$, and the norm of the resulting extension is at most $1$ again. 
\end{proof}

We now prove that $\pi$ thus constructed is a Jordan-Lie morphism on $\sa{\A}$.  This requires a series of involved inequalities.
\begin{lemma}
	For all $a,b\in \sa{\A}$,
	\begin{equation*}
		\Re(\pi(a)\pi(b)) = \pi(\Re(ab)) \text{ and }\Im(\pi(a)\pi(b)) = \pi(\Im(ab))\text.
	\end{equation*}
\end{lemma}

\begin{proof}
First, let $a,a'\in\domsa{\Lip_\A}$. The argument is the same for $\Re(aa')$ and $\Im (aa')$, so we will just present it for $\Re(aa')$. Fix $l > \max\left\{\norm{a}{\Lip_\A,M},\norm{a'}{\Lip_\A,M}\right\}$. 
	
	Let $\varepsilon > 0$. Since $(h_p)_{p\in\N}$ is an approximate unit in $\A$,  there exists $P_1\in\N$ such that, if $p\geq P_1$, then 
	\begin{equation*}
		\norm{h_p a h_p^2 a' h_p - h_p a a' h_p}{\A} < \frac{\varepsilon}{8} \text.
	\end{equation*}
	
	Since $\lim_{p\rightarrow\infty} \beta_p(\Re(aa')) - \Re(\beta_p(a)\beta_p(a')) = \pi(\Re(aa')) - \Re(\pi(a)\pi(a'))$, there exist $P_2 \in \N$ such that, if $p \geq P_2$, then
	\begin{equation*}
		\norm{\pi(\Re(a a')) - \Re(\pi(a)\pi(a')) - (\beta_p(\Re(a a')) - \Re(\beta_p(a)\beta_p(a')))}{\B} < \frac{\varepsilon}{4} \text. 
	\end{equation*}
	
	Fix $p \geq \max\{P_1, P_2\}$.
	
	First, there exists $K_0 \in \N$ such that, if $m \geq K_0$, then 
	\begin{equation*}
		\tunnelextent{\tau_m} \leq \frac{1}{1+m} < \frac{\varepsilon}{240 M^2 l^2} \text.
	\end{equation*}
	
	By Lemma (\ref{no-em-lemma}), there exists $K_1 \in \N$ such that, if $m\geq K_1$, if $c_m \in \targetsettunnel{m}{a}{l}$ and if $c'_m \in \targetsettunnel{m}{a'}{l}$, then 
	\begin{multline*}
	\Bigg\|\Re(\beta_p(a)\beta_p(a')) - \beta_p(\Re(aa')) \\ - \left[e^\B_p \left( c_m (e^\B_p)^2 c'_m +  c'_m (e^\B_p)^2 c'_m - (c_m c'_m + c'_m c_m) \right) e^\B_p\right]\Bigg\|_{\B}
 < \frac{\varepsilon}{4} \text, 
	\end{multline*}
	since 
	\begin{multline*}
		\lim_{m\rightarrow\infty} \Haus{\B}(e^\B_p \targetsettunnel{m}{a}{l} e^\B_p, \{\beta_p(a)\}) =\lim_{m\rightarrow\infty} \Haus{\B}(e^\B_p \targetsettunnel{m}{a'}{l} e^\B_p, \{\beta_p(a')\}) \\ =\lim_{m\rightarrow\infty} \Haus{\B}(e^\B_p \targetsettunnel{m}{\Re(aa')}{l} e^\B_p, \{\beta_p(\Re(aa'))\})
	\end{multline*}
	and $\Re(c_m c'_m) \in \targetsettunnel{m}{\Re(aa')}{2 M l^2}$ by Lemma (\ref{product-lemma}).

	There exists $K_2\in\N$ such that, if $m\geq K_2$, then $\norm{e^\B_p - e^\B_p e^\B_m}{\B} < \frac{\varepsilon}{160 M^2 l^2}$.
	
%
%
	
	Let $m \geq \max\{ K_0 , K_1, K_2 \}$.
	
	Let $c_m \in \targetsettunnel{m}{a}{l}$ and $c'_m \in \targetsettunnel{m}{a'}{l}$. Note that $\norm{c_m}{\Lip_\B,M} \leq l$ and $\norm{c'_m}{\Lip_\B,M} \leq l$.
	
	We immediately observe that:
	\begin{align*}
		&\norm{\Re(\pi(a)\pi(a')) - \pi(\Re(a a'))}{\B}\\
		&\quad\leq \norm{\Re(\beta_p(a)\beta_p(a')) - \beta_p(\Re(a a'))}{\B} + \frac{\varepsilon}{4} \\
		&\quad\leq \frac{1}{2} \norm{e^\B_p \left( c_m (e^\B_p)^2 c'_m +  c'_m (e^\B_p)^2 c'_m - (c_m c'_m + c'_m c_m) \right) e^\B_p}{\B} + \frac{\varepsilon}{2} \text. 
	\end{align*}
	
	Now, since $\norm{e^\B_p}{\B} \leq \sqrt{1+\frac{1}{p+1}} \leq 2$, and both $\norm{c_m}{\B}\leq M l$ and $\norm{c'_m}{\B} \leq M l$,
	\begin{align*}
		&\norm{e^\B_p \left( c_m (e^\B_p)^2 c'_m +  c'_m (e^\B_p)^2 c'_m - (c_m c'_m + c'_m c_m) \right) e^\B_p}{\B} \\
		&\quad\leq 2 \norm{c_m}{\B}\norm{c'_m}{\B}\norm{e^\B_p}{\B}^2 \left( 1 + \norm{e^\B_p}{\B}^2 \right) \\
		&\quad\leq 40 M^2 l^2 \text.
	\end{align*}
	Moreover, noting that $\Lip_\B(c_m) \leq l$ and $\Lip_\B(c'_m)\leq l$, and using the Leibniz inequality:
	\begin{align*}
		\Lip_\B(e^\B_p(\Re(c_m c'_m))e^\B_p) 
		&\leq 2 \Lip(e^\B_p)\norm{\Re(c_m c'_m)}{\B}\norm{e^\B_p}{\B} + \norm{e^\B_p}{\B}^2 \Lip_\B(\Re(c_m c'_m)) \\
		&\leq 8 M^2 l^2 + 4 M l^2 \leq 12 M^2 l^2 \text. 
	\end{align*}
	
	Again, using the same base inequalities,
	\begin{align*}
		\Lip_\B(c_m (e^\B_p)^2 c'_m)
		&\leq \norm{c_m}{\B}\Lip_\B(c'_m)\norm{e^\B_p}{\B}^2 + \norm{c_m}{\B}\norm{c'_m}{\B}\Lip(e^\B_p)^2 + \Lip_\B(c_m)\norm{e^\B_p}{\B}^2 \norm{c'_m}{\B} \\
		&\leq M l^2\left(1+\frac{1}{p+1}\right) + M^2 l^2\frac{1}{(p+1)^2} + M l^2\left(1+\frac{1}{p+1}\right) \\
		&\leq 5 M^2 l^2\text.
	\end{align*}
	Therefore,
	\begin{align*}
		\Lip_\B(e^\B_p \left( c_m (e^\B_p)^2 c'_m \right) e^\B_p)
		&\leq \norm{e^\B_p}{\B}\left( 2 \Lip(e^\B_p)\norm{c_m (e^\B_p)^2 c'_m}{\B} + \norm{e^\B_p}{\B}\Lip_\B(c_m (e^\B_p)^2 c'_m) \right) \\
		&\leq \sqrt{1+\frac{1}{p+1}}\left(\frac{1}{p+1}\left(1+\frac{1}{1+p}\right) M^2 l^2 + \sqrt{1+\frac{1}{p+1}} 5M^2 l\right)\\
		&\leq 24 M^2 l^2 \text.
	\end{align*}

	Therefore,
	\begin{align*}
		&\Lip_\B(e^\B_p \left( c_m (e^\B_p)^2 c'_m +  c'_m (e^\B_p)^2 c'_m - (c_m c'_m + c'_m c_m) \right) e^\B_p) \\
		&\leq M^2 l^2 (24 + 24 + 12) = 60 M^2 l^2\text.
	\end{align*}
	Therefore,
	\begin{equation}\label{monstruous-norm-eq1}
	\norm{e^\B_p \left( c_m (e^\B_p)^2 c'_m +  c'_m (e^\B_p)^2 c'_m - (c_m c'_m + c'_m c_m) \right)e^\B_p}{\Lip_\B,M} \leq 60 M^2 l^2 \text. 
	\end{equation}
	
	A similar computation shows that 
	\begin{equation}\label{monstruous-norm-eq2}
	\norm{a s_{p,m}^2 a' + a' s_{p,m}^2 a - (aa'+a'a)}{\Lip_\A} < 60 M^2 l^2 \text.
	\end{equation}
	
	Now, a key observation is that, 
	\begin{multline*}
		 e^\B_p \left( c_m (e^\B_p)^2 c'_m +  c'_m (e^\B_p)^2 c'_m - (c_m c'_m + c'_m c_m) \right) e^\B_p \\ 
		 \in \targetsettunnel{-m}{s_{p,m} \left( a s_{p,m}^2 a' + a' s_{p,m}^2 a - (aa'+a'a)\right)s_{p,m}}{60 M^2 l^2} \text.
	\end{multline*} 
	which follows from Exp. \eqref{monstruous-norm-eq1} and \eqref{monstruous-norm-eq2} and a direct computation using the fact that $\pi_n$ and $\rho_n$ are a *-morphisms.

	Therefore, using Lemma (\ref{easy-lemma}), then Lemma (\ref{main-lemma}), we obtain:
	\begin{align*}
		&\norm{e^\B_p \left( c_m (e^\B_p)^2 c'_m +  c'_m (e^\B_p)^2 c'_m - (c_m c'_m + c'_m c_m) \right) e^\B_p}{\B}\\
		&\leq \norm{e_m^\B\left( e^\B_p \left( c_m (e^\B_p)^2 c'_m +  c'_m (e^\B_p)^2 c'_m - (c_m c'_m + c'_m c_m) \right) e^\B_p \right) e_m^\B}{\B} + \frac{\varepsilon}{4} \\
		&\leq \norm{ s_{p,m} (a s_{p,m}^2 a' + a' s_{p,m}^2 a - a a' - a' a )s_{p,m}  }{\A} + \frac{\varepsilon}{4} + \frac{\varepsilon}{4}\text.
	\end{align*}

	We note that, for all $a,b,c\in \A$, the following easy inequality holds, which we hope will clarify the next computation somewhat:
	\begin{align}\label{some-easy-eq-product}
		\norm{a b a - c b c}{\A} 
		&= \norm{(a-c) b c + a b (a-c)}{\A} \nonumber \\
		&\leq \norm{a-c}{\A} \norm{b}{\A} (\norm{a}{\A} + \norm{c}{\A}) 
	\end{align}

	In turn,
	\begin{align*}
		&\norm{ s_{p,m} (a s_{p,m}^2 a'  - a a' )s_{p,m}  }{\A}\\
		&\quad\leq \norm{ s_{p,m} \left((e^\A_m)^2 a (e^\A_m)^2  s_{p,m}^2 (e^\A_m)^2 a' (e^\A_m)^2 s_{p,m} - (e^\A_m)^2 a a' (e^\A_m)^2 \right)s_{p,m}}{\A} \\
		&\quad +\norm{s_{p,m}}{\A}^4 \left(\norm{a-(e^\A_m)^2 a(e^\A_m)^2}{\A}\left(\norm{a}{\A} + \norm{(e^\A_m)^2 a (e^\B_m)^2 a}{\A}  \right) \right) \text{ by Exp. \eqref{some-easy-eq-product}, }\\
		&\quad + \norm{s_{p,m}}{\A}^2 \norm{aa'-(e^\B_m)^2 aa' (e^\B_m)^2}{\A} \\
		&\leq \norm{ h_{p,m} (a h_{p,m}^2 a' - a a') h_{p,m} }{\A} \text{ since $[s_{p,m},e^\A_{m}] = 0$.}\\
		&\quad+\norm{s_{p,m}}{\A}^4 \left(\norm{a-(e^\A_m)^2 a(e^\A_m)^2}{\A}\left(\norm{a}{\A} + \norm{(e^\A_m)^2 a (e^\B_m)^2 a}{\A}  \right) \right) \text{ by Exp. \eqref{some-easy-eq-product}, }\\
		&\quad + \norm{s_{p,m}}{\A}^2 \norm{aa'-(e^\B_m)^2 aa' (e^\B_m)^2}{\A}   \text.
	\end{align*}
	
	Since 
	\begin{equation*}
		\lim_{m\rightarrow\infty} \max\left\{ \norm{a-(e^\A_m)^2 a (e^\A_m)^2}{\A}, \norm{a'-(e^\A_m)^2 a' (e^\A_m)^2}{\A}, \norm{a a'-(e^\A_m)^2 a a' (e^\A_m)^2}{\A}\right\} = 0\text,
	\end{equation*}
	since $(s_{m,p})_{m\in\N}$ is bounded, and since $\lim_{m\rightarrow\infty} h_{p,m} = h_p$, we conclude that
	\begin{equation*}
	\lim_{m\rightarrow\infty}\norm{ h_{p,m} (a h_{p,m}^2 a' - a a') h_{p,m} }{\A} = \norm{h_p a h_p \cdot h_p a' h_p - h_p a a' h_p}{\A} < \frac{\varepsilon}{8}
	\end{equation*}
	and
	\begin{equation*}
	\lim_{m\rightarrow\infty} \norm{s_{p,m}}{\A}^4 \left(\norm{a-(e^\A_m)^2 a(e^\A_m)^2}{\A}\left(\norm{a}{\A} + \norm{(e^\A_p)^2 a (e^\B_p)^2 a}{\A}  \right) \right) = 0
	\end{equation*}
	while
	\begin{equation*}
	\lim_{m\rightarrow\infty} \norm{s_{p,m}}{\A}^2 \norm{aa'-(e^\B_p)^2 aa' (e^\B_p)^2}{\A}   =0\text.
	\end{equation*}
	
	Thus, there exists $K_3 \in \N$ such that, if $m\geq \max\{K_0,K_1,K_2,K_3\}$, then
	\begin{equation*}
		\norm{ s_{p,m} (a s_{p,m}^2 a'  - a a' )s_{p,m}  }{\A} \leq \frac{\varepsilon}{4} \text.
	\end{equation*}

	Similarly, there exists $K_4\in\N$ such that, if $m\geq\max\{K_0,K_1,K_2,K_3,K_4\}$, then
	\begin{equation*}
		\norm{ s_{p,m}(a' s_{p,m}^2 a - a' a)s_{p,m}}{\A} < \frac{\varepsilon}{4} \text.
	\end{equation*}
	
	Consequently, if $m\geq\max\{K_0,K_1,K_2,K_3,K_4\}$, then:
	\begin{align*}
	&\norm{e^\B_p \left( c_m (e^\B_p)^2 c'_m +  c'_m (e^\B_p)^2 c'_m - (c_m c'_m + c'_m c_m) \right) e^\B_p}{\B}\\
	&\quad \norm{ s_{p,m} (a s_{p,m}^2 a'  - a a' )s_{p,m}  }{\A} +	\norm{ s_{p,m} (a s_{p,m}^2 a'  - a a' )s_{p,m}  }{\A} + \frac{\varepsilon}{2} \\
	&\leq \varepsilon \text.
	\end{align*}
	
	Therefore, for all $\varepsilon > 0$,
	\begin{equation*}
		\norm{\pi(\Re(aa')) - \Re(\pi(a)\pi(a'))}{\A} < \varepsilon \text.
	\end{equation*}
	We thus have shown that $\pi(\Re(aa')) = \Re(\pi(a)\pi(a'))$.

Since $\pi$ is continuous over $\sa{\A}$ and $\domsa{\Lip_\A}$ is dense in $\sa{\A}$, we conclude that $\pi(\Re(aa')) = \Re(\pi(a)\pi(a'))$ for all $a,a' \in \sa{\A}$.

The proof for the Lie product is analogous.
\end{proof}

We can now extend $\pi$ to a *-morphism from $\A$ to $\B$. First, for any $a\in\A$, we set
\begin{equation*}
	\pi(a) \coloneqq \pi(\Re a) + i \pi(\Im a) \text.
\end{equation*}
It is then a matter of simple algebra that $\pi(a^\ast)=\pi(a)^\ast$ for all $a\in\A$, and
\begin{align*}
	\pi(a b) &= \pi(\Re (ab) + i\Im (ab) ) = \pi(\Re (ab) ) + i \pi(\Im (ab) ) \\
	&= \Re(\pi(a)\pi(b)) + i\Im(\pi(a)\pi(b)) = \pi(a)\pi(b) \text.
\end{align*}

This constructed map $\pi$ is a *-morphism over $\A$, and since $\A$ is a C*-algebra, $\pi$ is in fact, continuous with norm $1$.

\medskip

Our entire process up to now is symmetric in $\A$ and $\B$. This way, we obtain a *-morphism $\rho : \B \rightarrow \A$ such that, for all $b\in\domsa{\Lip_\B}$ and $l>\norm{b}{\Lip_\B}$, we have
\begin{equation}\label{zero-symm-eq1}
	\rho(b) = \lim_{p\rightarrow\infty} \alpha_p(b)
\end{equation}
where $a_p(b)$ is defined uniquely by
\begin{equation}\label{zero-symm-eq2}
	\lim_{m\rightarrow\infty} \Haus{\A}(e^\A_p \targetsettunnel{-m)}{b}{l} e^\A_p, \{\alpha_p(b)\}) = 0 \text.
\end{equation}
Note that, for all $b\in\domsa{\Lip_\B}$, we have $\Lip_\B(b) \leq \norm{\rho(b)}{\Lip_\A,M} \leq \norm{b}{\Lip_\B,M}$, and if $b \in \M_\B\cap\domsa{\Lip_\B}$, then $\Lip_\A\circ\rho(b) \leq \Lip_\B(a)$.

\medskip

Our last step is to prove that $\rho=\pi^{-1}$. Let $\varepsilon > 0$. 

\medskip 

Since $(h_p)_{p\in\N}$ is an approximate unit of $\A$, there exists $P_0 \in \N$ such that, if $p\geq P_0$, then $\norm{a-h_p a h_p}{\A} < \frac{\varepsilon}{24}$.

By Equation \eqref{zero-symm-eq1}, there exists $P_1\in\N$ such that, if $p\geq P_1$, then 
\begin{equation}\label{inv-eq-1}
	\norm{\rho(\pi(a)) - \alpha_p(\pi(a))}{\A} < \frac{\varepsilon}{20}\text.
\end{equation}

By Lemma (\ref{Cauchy-lemma}), there exists $P_2\in\N$ such that, for all $n\geq P_2$, $\norm{\pi(a) - \beta_p(a)}{\B} < \frac{\varepsilon}{3}$. Since $\alpha_p$ has norm $1$ for any $p\in\N$, we conclude that 
\begin{equation}\label{inv-eq-2}
	\norm{\alpha_p(\pi(a)-\beta_p(a))}{\A} < \frac{\varepsilon}{5}
\end{equation}
whenever $p\geq P_2$.

Last, since $(e^\A_p)_{p\in\N}$ is an approximate unit for $\A$, there exists $P_3 \in \N$ such that, if $p \geq P_3$, then 
\begin{equation}\label{inv-eq-6}
	\norm{a - e^\A_p a e^\A_p}{\A} < \frac{\varepsilon}{320}\text.
\end{equation}

Let $p \coloneqq\max\{P_0,P_1,P_2,P_3\}$.

Let $z_m \in \targetsettunnel{-m}{\beta_p(a)}{l}$. There exists $M_1\in\N$ such that, if $m\geq M_1$, then 
\begin{equation}\label{inv-eq-3}
	\norm{\alpha_p(\beta_p(a)) - e^\A_p e^\A_m z_m e^\A_m e^\A_p}{\A} < \frac{\varepsilon}{5}\text.
\end{equation}

Since $\lim_{m\rightarrow \infty} h_{p,m} = h_p$, there exists $M_2 \in \N$ such that, if $m\geq M_2$, then $\norm{h_{p,m} - h_p}{\A} < \frac{\varepsilon}{40(\norm{a}{\A} + 1)}$. Consequently,
\begin{equation}\label{inv-eq-5}
	\norm{a - h_{p,m} a h_{p,m}}{\A} \leq \norm{a-h_p a h_p}{\A} + 2\norm{h_{p,m}-h_p}{\A}\norm{a}{\A} < \frac{\varepsilon}{20} + \frac{\varepsilon}{20} = \frac{\varepsilon}{10} \text.
\end{equation}

There exists $M_3 \in\N$ such that, if $m\geq M_3$, then
\begin{equation*}
	\Haus{\B}(e^\B_p \targetsettunnel{m}{a}{3 l} e^\B_p, \{ \beta_p(a) \} ) < \frac{\varepsilon}{21} \text.
\end{equation*}

Let $M_4 \in \N$ such that, if $m\geq M_4$, then $\tunnelextent{\tau_m} <\frac{\varepsilon}{ 240 l }$. Note that $\tunnelextent{\tau_{m}^{-1}} = \tunnelextent{\tau_m}$.

Let $y_m \in \targetsettunnel{m}{h_{p,m} a h_{p,m}}{l}$.  By definition of the reverse tunnel and of the target sets,
\begin{equation*}
	h_{p,m} a h_{p,m} \in \targetsettunnel{-m}{y_m}{l} \text.
\end{equation*}
Let $c_m \in \targetsettunnel{m}{a}{l}$ as well. Note that
\begin{equation}\label{inv-eq-6}
	\norm{e_m^\B(y_m - c_m)e_m^\B}{\B} \leq \norm{h_{p,m} a h_{p,m} - a}{\A} + 3 l \tunnelextent{\tau_m} < \frac{\varepsilon}{80}  \text.
\end{equation}

Let $M \in \N$ such that, if $m\geq M$, then $\norm{a - (e^\A_m)^2 a (e^\A_m)^2}{\A} < \frac{\varepsilon}{320}$.

Now, 
\begin{equation*}
	z_m - s_{p,m}e^\A_m a e^\A_m s_{p,m} \in \targetsettunnel{-m}{\beta_p(a) - e^\B_p e^\B_m y_m e^\B_m e^\A_p}{3l}
\end{equation*} 
by construction, using both the linearity and then product properties of the target sets --- this observation is where we ``inverse'' our map, in essence, using the relationship between tunnels and their reverse tunnels. Therefore, by Lemma (\ref{main-lemma}), using the tunnel $\tau_{m}^{-1}$, and noting that $s_{p,m}$, $e^\A_p$ and $e^\A_m$ all commute since they lie in the Abelian C*-algebra $\M_\A$:
\begin{align*}
	\norm{e^\A_m z_m e^\A_m - s_{p,m} e^\A_p a e^\A_p s_{p,m}}{\A}
	&\leq \norm{e^\A_m z_m e^\A_m -  s_{p,m}  a  s_{p,m}}{\A} + \underbracket[1pt]{\norm{s_{p,m}}{\A}^2}_{\leq 2} \underbracket[1pt]{\norm{a-e^\A_p a e^\A_p}{\A}}_{\leq\frac{\varepsilon}{320}} \\
	&\leq \norm{e^\A_m (z_m  -  s_{p,m} e^\A_m a e^\A_m s_{p,m}) e^\A_m}{\A} \\
	&\quad + \underbracket[1pt]{\norm{s_{p,m}}{\A}^2}_{\leq 2} \underbracket[1pt]{\norm{a-(e^\A_m)^2 a (e^\A_m)^2}{\A}}_{\leq\frac{\varepsilon}{320}} + \frac{\varepsilon}{160}  \\
	&\leq \norm{\beta_p(a)-e^\B_p e^\B_m y_m e^\B_m e^\B_p}{\B} + 3 l \tunnelextent{\tau_{m}^{-1}} + \frac{\varepsilon}{80} \nonumber \\
	&\leq \norm{\beta_p(a) - e^\B_p e^\B_m c_m e^\B_m e^\B_p}{\B} \\
	&\quad + \norm{e^\B_p e^\B_m (c_m-y_m)e^\B_m e^\B_p}{\B} + \frac{\varepsilon}{40} \\
	&\leq \frac{\varepsilon}{20} + \underbracket[1pt]{\norm{e^\B_p}{\B}^2}_{<2} \underbracket[1pt]{\norm{e^\B_m (c_m-y_m)e^\B_m }{\B}}_{<\frac{\varepsilon}{80} \text{by Exp. \eqref{inv-eq-6}}} + \frac{\varepsilon}{40} \\
	&\leq \frac{\varepsilon}{40} + \frac{\varepsilon}{20} + \frac{\varepsilon}{40} \leq \frac{\varepsilon}{10}\text. \nonumber
\end{align*}

So
\begin{align}\label{inv-eq-4}
	\norm{e_p^\A e_m^\A z_m e^\A_m e^\A_p - h_{p,m} a h_{p,m}}{\A}
	&= \norm{e^\A_p \left( e_m^\A z_m e^\A_m e^\A_p - s_{p,m} e^\A_p a e^\A_p s_{p,m} \right) e^\A_p}{}\nonumber\\
	&\leq \underbracket[1pt]{\norm{e^\A_p}{\A}^2}_{<2} \underbracket[1pt]{\norm{e^\A_m z_m e^\A_m - s_{p,m} e^\A_p a e^\A_p s_{p,m}}{\A}}_{<\frac{\varepsilon}{10}}\\
	&<\frac{\varepsilon}{5} \text.\nonumber
\end{align}

Therefore, for all $\varepsilon  > 0$:
\begin{align*}
	\norm{\rho(\pi(a)) - a}{\A}
	&\leq \underbracket[1pt]{\norm{\rho(\pi(a)) - \alpha_p(\pi(a))}{\A}}_{<\frac{\varepsilon}{5} \text{ by Exp. \eqref{inv-eq-1}}} + \underbracket[1pt]{\norm{\alpha_p(\pi(a)) - \alpha_p(\beta_p(a))}{\A}}_{<\frac{\varepsilon}{5}\text{ by Exp. \eqref{inv-eq-2}}} \\
	&\quad + \underbracket[1pt]{\norm{\alpha_p(\beta_p(a)) - e^\A_p e^\B_m z_m e^\B_m e^\A_p}{\A}}_{<\frac{\varepsilon}{5} \text{ by Exp. \eqref{inv-eq-3} }} + \underbracket[1pt]{\norm{e^\A_p e^\A_m z_m e^\A_m e^\A_p - h_{p,m} a h_{p,m}}{\A}}_{<\frac{\varepsilon}{5} \text{ by Exp. \eqref{inv-eq-4}}} \\
	&\quad +  \underbracket[1pt]{\norm{h_{p,m} a h_{p,m} - a}{\A}}_{<\frac{\varepsilon}{12} \text{ by Exp. \eqref{inv-eq-5}}}  \\
	&<\varepsilon \text.
\end{align*}

So $\rho\circ\pi(a) = a$. By continuity and linearity of $\rho\circ\pi$ and since $\domsa{\Lip_\A}$ is total in $\A$, we conclude that $\rho\circ\pi(a) = a$ for all $a\in\A$. The same argument applies to prove that $\pi\circ\rho = \mathrm{id}_\B$. So as claimed, $\rho =\pi^{-1}$.

As a consequence, if $a \in \domsa{\Lip_\A}$, we note that
\begin{equation*}
	\norm{a}{\Lip_\A,M} = \norm{\rho(\pi(a))}{\Lip_\A,M} \leq \norm{\pi(a)}{\Lip_\B,M} \leq \norm{a}{\Lip_\A,M}\text{ so }\norm{\pi(a)}{\Lip_\A,M} = \norm{a}{\Lip_\A,M}\text.
\end{equation*}

\medskip

Last, let $a\in\domsa{\Lip_\A}$ and set $l\coloneqq\norm{a}{\Lip_\A,M}+1$. Let $\varepsilon > 0$. There exists $P\in\N$ such that, if $p \geq P$, then 
\begin{equation*}
	\Kantorovich{\Lip_p}(\mu_\A\circ\pi_n,\mu_\B\circ\rho_n) \leq \tunnelextent{\tau_p} < \frac{\varepsilon}{12 l}\text.
\end{equation*}
There exists $P_2 \in\N$ such that, if $p\geq P_2$, then $\norm{\pi(a) - \beta_p(a)}{\B}<\frac{\varepsilon}{4}$.

Proceeding as above once more, there exists $P_3 \in \N$ such that, if $p\geq P_3$ and $m\geq P_3$, then $\norm{a - h_{p,m} a h_{p,m}}<\frac{\varepsilon}{4}$.

For all $m\in\N$, let $d_m \in \domsa{\Lip_n}$ such  that $\pi_n(d_n) = a$ and $\norm{d_n}{\Lip_n,M} \leq l$. Let $c_m\coloneqq \rho_n(d_m)$, so $c_m \in \targetsettunnel{m}{a}{\norm{a}{\Lip_\A,M}+1}$ by definition. There exists $K_1 \in \N$ such that, if $m\geq K_1$, then $\norm{e^\B_p c_m e^\B - \beta_p(a)}{\B} < \frac{\varepsilon}{4}$.

Let $p\geq \max\{P_1,P_2,P_3\}$ and $m\geq \max\{P_3,K_1\}$. Then:
\begin{align*}
	|\mu_\B\circ\pi(a) - \mu_\A(a)|
	&\leq |\mu_\B(\pi(a)-\beta_p(a))| + |\mu_\B(\beta_p(a)) - \mu_\B(e^\B_p c_m e^\B)| \\
	&\quad + |\mu_\B(\rho_m(k_{p,m} d_m k_{p,m})) - \mu_\A(\pi(k_{p,m} d_m p_{k,m}))| \\
	&\quad + |\mu_\A(h_{p,m} a h_{p,m}) - a)|\\
	&\leq \frac{\varepsilon}{4} + 3l \frac{\varepsilon}{12 l} + \frac{\varepsilon}{4} + \frac{\varepsilon}{4} = \varepsilon \text. 
\end{align*}
Therefore, since $\varepsilon > 0$ was arbitrary, $\mu_\B\circ\pi = \mu_\A$.

Our proof of Theorem (\ref{main-thm}) is now complete.
\end{proof}


\section{The Gromov-Hausdorff quantum Metametric}

\subsection{The Metametric}

We have constructed a family of inframetrics $(\metametric{r})_{r\geq 1}$ with desirable properties. However, the dependency on the cut-off parameter $r$ when defining $\metametric{r}$ is unclear. To clarify the importance of this parameter, note that in our work, we use the Leibniz inequality extensively, and thus, we often find ourselves needing to control both Lipschitz seminorms and norm of elements, and our cutoff parameter helps control both. On the other, the price for this convenience is that $\metametric{1}$ is zero between the non-isometric spaces
\begin{equation*}
	[0,1]\cup[3,4] \text{ and }[0,1]\cup[5,6] \text,
\end{equation*}
since it is only sensitive to the metric truncated at $1$. 

We offer a global solution which removes this apparent extrinsic parameter and provide a unique inframetric which enjoys all the basic properties we have proven so far, but moreover, with distance zero implying isometry for the {\MongeKant}. This is the good news. On the less clear news is that this requires to prove convergence for infinitely many inframetrics of the $\metametric{r}$, \emph{but} this is not as unmanageable as it may seem at all, as we will see in examples. Our current, and early, observation, is that it is not more difficult to prove convergence for $\metametric{1}$ than for $\metametric{R}$ for any $R \geq 1$ --- what matters is that a cutoff is chosen, not what it is. So, from this perspective, our present definition is reasonable. 

We make one more observation, prior to our definition. In this entire paper, outside of this section, there is exactly one result which depends on this distance which does not work if we only consider an arbitrary cutoff --- the proof that convergence in our new topology implies convergence in the metametric for {\qcms s}. Notably, the converse does hold without this new tool. In a way which may be interesting to explore further, $\metametric{r}$ is a more local notion of convergence, whereas the following is decidedly global.

\begin{definition}\label{Metametric-def}
The \emph{quantum metametric} $\Eth(\mathds{A},\mathds{B})$ between any two {\pqpms s} $\mathds{A}$ and $\mathds{B}$ is:
	\begin{multline*}
		\Eth(\mathds{A},\mathds{B}) \coloneqq \max\Bigg\{ \inf\left\{ \varepsilon > 0 : \sup_{r \in \left[1,\frac{1}{\varepsilon}\right]} \metametric{r}(\mathds{A},\mathds{B}) < \varepsilon \right\}, \\ \left|\exp\left( -\qdiam{\mathds{A}} \right) - \exp\left(-\qdiam{\mathds{B}}\right) \right| \Bigg\} \text.
	\end{multline*}
\end{definition}

\begin{theorem}
The following assertions hold:
\begin{enumerate}
	\item $\Eth(\mathds{A},\mathds{B}) = \Eth(\mathds{B},\mathds{A})$ for any {\pqpms s} $\mathds{A}$, $\mathds{B}$,
	\item for any {\pqpms s} $\mathds{A}$, $\mathds{B}$ and $\mathds{D}$,
	\begin{equation*}
	\Eth(\mathds{A},\mathds{B}) \leq(1+\Eth(\mathds{A},\mathds{D}))^2 \Eth(\mathds{D},\mathds{B}) + (1+\Eth(\mathds{D},\mathds{B}))^2 \Eth(\mathds{A},\mathds{D})\text,
	\end{equation*}
	\item $\Eth(\mathds{A},\mathds{B}) = 0$ if, and only if, there exists a pointed quantum full isometry from $\mathds{A}$ to $\mathds{B}$.
	\end{enumerate}
\end{theorem}

\begin{proof}
We begin by proving the second assertion. Let $\varepsilon_1 > \Eth(\mathds{A},\mathds{D})$ and $\varepsilon_2 > \Eth(\mathds{D},\mathds{B})$. Let $\varepsilon_m = (1+\varepsilon_1)^2 \varepsilon_2 + (1+\varepsilon_2)^2 \varepsilon_1$. Note that $\varepsilon_m > \max\{\varepsilon_1,\varepsilon_2\}$, so
\begin{equation*}
	\sup_{r \in \left[1,\frac{1}{\varepsilon_m}\right]} \metametric{r}(\mathds{A},\mathds{D}) \leq \varepsilon_1 \text{ and }\sup_{r \in \left[1,\frac{1}{\varepsilon_m}\right]} \metametric{r}(\mathds{D},\mathds{B}) \leq \varepsilon_2 \text.
\end{equation*}

Fix $r \in \left[1,\frac{1}{\varepsilon_m}\right]$. Then
\begin{equation*}
	\metametric{r}(\mathds{A},\mathds{B}) \leq (1+\metametric{}(\mathds{B},\mathds{D})^2 \metametric(\mathds{A},\mathds{D}) + (1+\metametric{}(\mathds{A},\mathds{D})^2 \metametric(\mathds{B},\mathds{D}) \leq (1+\varepsilon_2)^2\varepsilon_1 + (1+\varepsilon_1)^2\varepsilon_2 = \varepsilon_m \text.
\end{equation*}

Of course,
	\begin{align*}
		\left|\exp\left( \qdiam{\mathds{A}} \right) - \exp\left(\qdiam{\mathds{D}}\right) \right|
		&\leq \left|\exp\left( \qdiam{\mathds{A}} \right) - \exp\left(\qdiam{\mathds{B}}\right) \right|\\
		&\quad+ \left|\exp\left( \qdiam{\mathds{B}} \right) - \exp\left(\qdiam{\mathds{D}}\right) \right| \text.
	\end{align*}

Hence, 
\begin{align*}
	\Eth(\mathds{A},\mathds{B}) 
	&\leq \varepsilon_m \\
	&=(1+\varepsilon_1)^2 \varepsilon_2 + (1+\varepsilon_2)^2 \varepsilon_1 \\
	&\leq(1+\Eth(\mathds{A},\mathds{D}))^2 \Eth(\mathds{D},\mathds{B}) + (1+\Eth(\mathds{D},\mathds{B}))^2 \Eth(\mathds{A},\mathds{D})\text,
\end{align*}
and (2) is proven.

\medskip

Assertion (1) is obvious since $\metametric{r}$ is symmetric for all $r \geq 1$.

\medskip

If we assume that $\mathds{A}$ and $\mathds{B}$ are fully quantum isometric, then $\metametric{r}(\mathds{A},\mathds{B}) = 0$ for all $r \geq 1$, so $\Eth(\mathds{A},\mathds{B}) = 0$.

Assume now that $\Eth(\mathds{A},\mathds{B}) = 0$. For each $n \in \N$, since $\metametric{n}(\mathds{A},\mathds{B}) = 0$, there exists a *-isomorphism $\pi_n : \A \rightarrow \B$ such that $\norm{\pi(\cdot)}{\Lip_\B, n} = \norm{\cdot}{\Lip_\A,n}$. 

Therefore, if we fix $h \in \sa{\A}$ strictly positive, the family $\{ \pi_n : n \in\N \}$ of functions is equicontinuous (since $1$-Lipschitz) over the compact $\alg{K}\coloneqq\{ h a h : a \in\dom{\Lip_\A}, \norm{a}{\Lip}\leq 1 \}$. They are valued in $\{ h_\B b h_\B : b\in\dom{\Lip_\B}, \norm{b}{\Lip}\leq 1 \}$ for $h_\B \coloneqq \pi(h)$, which is compact as well. By Arz{`e}la-Ascoli, there exists a map $\theta$ between these compact sets which is the limit of some subsequence $(\pi_{f(n)})_{n\in\N}$ restricted to $\alg{K}$. It is easy to check that since $\alg{K}$  is total in $\A$, and again $\opnorm{\pi_n}{\A}{\B} = 1$, the sequence $(\pi_{f(n)})$ converges pointwise on $\A$ to some map $\theta$ which is then necessarily a *-morphism. Up to extracting a further subsequence if necessary, this map has an inverse, so it is a *-isomorphism. Now, let $a\in\dom{\Lip_\A}$, with $\Lip_\A(a) > 0$. Thus there exists $N\in\N$ such that, for all $n\geq N$, we have $\norm{a}{\Lip_\A,n} = \Lip_\A(a)$. 
Therefore:
\begin{align*}
	\Lip_\B\circ\theta(a) \leq \liminf_{n\rightarrow\infty} \Lip_\B\circ\pi_n(a) 
	&= \liminf_{n\rightarrow\infty} \max\left\{\Lip_\B\circ\pi_n(a),\frac{1}{n}\norm{\pi_n(a)}{\B} \right\} \\
	&= \liminf_{n\rightarrow\infty} \max\left\{\Lip_\A(a),\frac{1}{n}\norm{a}{\A} \right\} \\
	&= \Lip_\A(a)\text.
\end{align*}
So $\Lip_\B\circ\theta \leq \Lip_\A$. Similarly, $\Lip_\A\circ\theta^{-1} \leq \Lip_\B$, so $\Lip_\A \leq \Lip_\B\circ\theta$. Hence, as desired, $\Lip_\B\circ\theta = \Lip_\A$.

\end{proof}

We note that $\Eth$ and $\metametric{r}$ for all $r \geq 1$ satisfy the same relaxed form of triangle inequality. In the next section, we explore how to use these metametrics to define a topology.

\subsection{The quantum Gromov-Hausdorff hypertopology}

We define a \emph{$k$-inframetric} $d$ on a class $X$, for some $k \geq 1$, as a map $d : X\times X \rightarrow [0,\infty)$ such that:
\begin{enumerate}
	\item $d(x,y) = d(y,x)$ for all $x,y \in X$,
	\item $d(x,z) \leq k(d(x,y) + d(y,z))$ for all $x,y,z \in X$,
	\item $d(x,x) = 0$ for all $x\in X$.
\end{enumerate}
Our terminology departs a little from the usual terminology, in that we do not require that $\forall x,y \in X \quad d(x,y)=0 \implies x=y$; this is mostly so that we do not drag the heavy terminology of a pseudo-inframetric. Moreover, and in part explaining the choice of the prefix infra, the relaxed triangle inequality may be expressed as $d(x,z) \leq k\max\{d(x,y),d(y,z)\}$ instead; however this is only a matter of replacing $k$ by $2k$ in our version, which will not matter. From this point of view, an inframetric is a generalization of an ultrametric, though this perspective will not help us.

The reason to introduce this terminology is to connect our work so far, to the body of literature on the subject of inframetrics, via the following observation.
\begin{corollary}
	For any $\varepsilon > 0$ and $M\geq 1$, the class functions $\min\{\metametric{M},\varepsilon\}$ and $\min\{\Eth,\varepsilon\}$ are a $(1+\varepsilon)^2$-inframetrics.
\end{corollary}

\begin{proof}
	We write our proof for $\metametric{M}$, as the proof for $\Eth$ is identitcal. Let $d \coloneqq \min\{\metametric{M},\varepsilon \}$.
	
	If $d(\mathds{A},\mathds{B}) \geq \varepsilon$, or $d(\mathds{B},\mathds{D}) \geq \varepsilon$, then by definition, 
	\begin{align*}
		d(\mathds{A},\mathds{D}) 
		&\leq \varepsilon \text{ by definition of $d$, }\\
		&\leq \max\{d(\mathds{A},\mathds{B}) , d(\mathds{B},\mathds{D}) \}\\
		&\leq d(\mathds{A},\mathds{B}) + d(\mathds{B},\mathds{D}) \\
		&\leq (1+\varepsilon)^2 (d(\mathds{A},\mathds{B}) + d(\mathds{B},\mathds{D})) \text. 
	\end{align*}

	Assume now that both $d(\mathds{A},\mathds{B})$ and $d(\mathds{B},\mathds{D})$ are less than $\varepsilon$. Then, by Theorem (\ref{triangle-thm}),  we conclude that
	\begin{align*}
		d(\mathds{A},\mathds{D}) 
		&\leq \metametric{M}(\mathds{A},\mathds{D}) \\
		&\leq (\metametric{M}(\mathds{B},\mathds{D})+1)^2 \metametric{}(\mathds{A},\mathds{B}) + (\metametric{M}(\mathds{A},\mathds{B})+1)^2 \metametric{M}(\mathds{B},\mathds{D})\\
		&\leq (1+\varepsilon)^2 (\metametric{M}(\mathds{A},\mathds{B}) + \metametric{M}(\mathds{B},\mathds{D})) \\
		&= (1+\varepsilon)^2 (d(\mathds{A},\mathds{B}) + d(\mathds{B},\mathds{D}))\text,
	\end{align*}
	as desired. 
\end{proof}

Inframetrics do induce topologies, {\`a} l'instar of metrics, though one must actually be careful on how to do so. Indeed, there are two natural choices --- using the idea of convergence of sequences, leading to a topology built from a Kuratowsky closure operator, or using the smallest topology containing ``open balls''. The former is always weaker than the latter, and are equal for metrics, but not for inframetrics in general. The proper choice is the closure approach.

We thus introduce the titular topology of this project:
\begin{definition}
	We define the \emph{hypertopology of Gromov-Hausdorff convergence} for {\pqpms s} as the topology on $p\mathscr{P}$ whose closure operator is given for all nonempty $\mathcal{A}\subseteq p\mathscr{P}$ by
	\begin{equation*}
		\closure{\mathcal{A}} \coloneqq \left\{ \mathds{B} \in p\mathscr{P} : \Eth(\mathds{B}, \mathcal{A}) = 0 \right\}
	\end{equation*}
	where 
	\begin{equation*}
		\Eth(\mathds{B},\mathcal{A}) = \inf\left\{ \Eth(\mathds{B},\mathds{A}) : \mathds{A} \in \mathcal{A} \right\} \text,
	\end{equation*}
	and of course, $\closure{\emptyset} = \emptyset$.
\end{definition}

Since $\Eth$ is not a metric and since this is the central concept we introduce here, we present a brief summary of the basic properties of this new topology. Since $\Eth$ and $\metametric{r}$, for any $r\geq 1$, share the same relaxed triangle inequality, the following discussion could be applied equally well to any of these metametrics. We present it at the level of generality we believe capture the essential aspects of our situation.


In general, any inframetric gives rise to a natural topology using a natural Kuratowsky closure operator --- rather than using balls, which lead to a possibly stronger and less intuitive topology. As a result, the topology induced via the closure operator may not make the actual inframetric continuous. However, our situation will enable us to avoid this problem. As the construction of a topology is our central concern, we believe it is worth clarifying this matter in some details.

We let $\Delta$ be an inframetric on the class $p\mathscr{P}$ of all {\pqpms s}, up to some equivalence relation, and a function continuous $f : [0,\infty) \rightarrow [1,\infty)$ such that $f(0) = 1$, such that $\Delta$ satisfies the following relaxed triangle inequality:
\begin{equation*}
	\Delta(\mathds{A},\mathds{B}) \leq f(\Delta(\mathds{A},\mathds{D})) \Delta(\mathds{D},\mathds{B}) + f(\Delta(\mathds{D},\mathds{B})) \Delta(\mathds{A},\mathds{D})
\end{equation*}	
for all {\pqpms s} $\mathds{A}$, $\mathds{B}$ and $\mathds{D}$. For all our examples above, $f : x \in [0,\infty) \mapsto (1 + x)^2$. All the discussion below applies to $\Eth$ as well as $\metametric{r}$ for any $r\geq 1$.

We now define the closure $\closure{\mathcal{A}}$ of any nonempty class $\mathcal{A}$ of {\pqpms s}, and any {\pqpms} $\mathds{A}$, by setting:
\begin{equation*}
	\closure{\mathcal{A}} \coloneqq \left\{ \mathds{B} \text{ \pqpms s} : \Delta(\mathds{A}, \mathcal{A}) = 0 \right\} \text,
\end{equation*}
where
\begin{equation*}
	\Delta(\mathds{A} , \mathcal{A} ) \coloneqq \inf\left\{ \Delta(\mathds{A}, \mathds{B}) : \mathds{B} \in \mathcal{A} \right\}\text.
\end{equation*}
We also set $\closure{\emptyset} = \emptyset$. The inframetrics properties of $\Delta{}$ are sufficient to prove that indeed, $\closure{\cdot}$ is a Kuratowsky operator, thus defining a topology: a set is closed when it is a fixed point for $\closure{\cdot}$, a set is open when its complement is closed. Notably, it is classic to check that, if we define a sequence $(\mathds{A}_n)_{n\in\N}$ of {\pqpms s} to converge to a {\pqpms} $\mathds{A}$ when $\lim_{n\rightarrow\infty} \Delta(\mathds{A}_n,\mathds{A}) = 0$, then
\begin{equation*}
	\closure{\mathcal{A}} = \left\{ \mathds{A} : \exists (\mathds{A}_n)_{n\in\N} \in \mathcal{A}^\N \quad \lim_{n\rightarrow\infty} \Delta(\mathds{A}_n,\mathds{A}) = 0 \right\} \text.
\end{equation*}
In other words, sequences converge in the topology induced by the inframetric exactly when they converge for the inframetric itself. 

Since this topology is the same if we replace $\Delta$ by $\min\{\Delta,1\}$, it is indeed induced by an inframetric; in turn, we may replace $\min\{\Delta,1\}$ by $\min\{\Delta,1\}^\varepsilon$, again without changing our closure operator. By \cite[Proposition 14.5]{Heinonen}, in turn, there exists some $\varepsilon > 0$ such that $\min\{\Delta{},1\}^\varepsilon$ is equivalent to an actual pseudo-metric. Hence, our topology is metrizable.

However, in general, this construction does not make the inframetric itself continuous  --- this topology, which is always weaker than the weakest topology making $\Delta$ continuous, may not be equal to it for a generic inframetric. Our efforts here, however, lead to a better situation.

Notably, and unlike inframetrics in general, $\Delta$ is continuous with respect to its topology. Namely, for any {\pqpms} $\mathds{A}$ and any sequence $(\mathds{B}_n)_{n\in\N}$ converging to $\mathds{B}$ for $\metametric{M}$, then
\begin{align*}	
	\left| \Delta(\mathds{A},\mathds{B}_n) - \Delta(\mathds{A},\mathds{B})\right|
	&\leq (f(\mathds{B}_n,\mathds{B})-1) \Delta(\mathds{A},\mathds{B}) \\
	&\quad+ f(\Delta(\mathds{A},\mathds{B})) \Delta(\mathds{B}_n,\mathds{B}) \\
	&\xrightarrow{n\rightarrow\infty} 0 + 0 = 0 \text{ since $\lim_{x\rightarrow 0} f(x)= 1$.}
\end{align*}
So $\Delta$ is continuous in each variable, using symmetry. Now, if $(\mathds{A}_n)_{n\in\N}$ is a sequence of {\pqpms s} converging to $\mathds{A}$ for $\metametric{M}$, we conclude:
\begin{align*}
\left| \Delta(\mathds{A}_n,\mathds{B}_n) - \Delta(\mathds{A},\mathds{B})\right| 
&\leq \left| \Delta(\mathds{A}_n,\mathds{B}_n) - \Delta(\mathds{A},\mathds{B}_n)\right| + \left| \Delta(\mathds{A},\mathds{B}_n) - \Delta(\mathds{A},\mathds{B})\right| \\
	&\leq f(\Delta(\mathds{A},\mathds{B}_n)) \Delta(\mathds{A}_n,\mathds{A}) \\
	&\quad + \left(f(\Delta(\mathds{A}_n,\mathds{A})) - 1\right) \Delta(\mathds{A},\mathds{B}_n) \\
	&\quad + \left| \Delta(\mathds{A},\mathds{B}_n) - \Delta(\mathds{A},\mathds{B})\right|  \\ 
	&\xrightarrow{n\rightarrow\infty} 0 + 0 + 0 = 0 \text.
\end{align*}
Thus, $\Delta$ is continuous for the topology induced by the closure operator described above, as claimed. In turn, this opens many new interesting results about the metric properties of $\Delta$, as seen for instance in \cite{Xia09}.

As a consequence, the topology described above is in fact, the smallest topology making $\metametric{}$ continuous, and in particular, open balls for the metametric are indeed open, and form a subbase for our topology. In fact, remarkably, open balls for $\metametric{}$ still form a topological base --- again, something not true for general inframetrics. Indeed, define:
\begin{equation*}
	q\mathscr{P}(\mathds{A}, r) \coloneqq \left\{ \mathds{B} \in q\mathscr{P} : \metametric{M}(\mathds{A},\mathds{B}) < r \right\}
\end{equation*}
 for all {\pqpms} $\mathds{A}$ and for all $r > 0$. Assume that, for some $r,r' > 0$ and $\mathds{A},\mathds{B} \in q\mathscr{P}$, there exists $\mathds{D} \in q\mathscr{P}(\mathds{A},r) \cap q\mathscr{P}(\mathscr{B},r')$.
 
Let $r_0 \coloneqq \min\{ \metametric{}(\mathds{D},\mathds{A}), \metametric{}(\mathds{D},\mathds{B}) \}$. Note that $r_0 < \min\{r,r'\}$. Since $\lim_{\varepsilon\rightarrow 0^+} \varepsilon f(r_0) +  + r_0 f(\varepsilon) = r_0 < \min\{r,r'\}$, we conclude that there exists $\varepsilon > 0$ such that $\varepsilon f(r_0) + r_0 f(\varepsilon) < \min\{r,r'\}$.

Let now $\mathds{E} \in q\mathscr{P}(\mathds{D}, \varepsilon)$. We compute:
\begin{align*}
	\Delta(\mathds{E},\mathds{A})
	&\leq \Delta(\mathds{E},\mathds{D}) f(\Delta(\mathds{A},\mathds{D}))  \\
	&\quad + \Delta(\mathds{A},\mathds{D})f(\Delta(\mathds{E},\mathds{D}))\\
	&\leq \varepsilon f(r_0) + r_0 f(\varepsilon) < \min\{r,r'\} < r\text.
\end{align*}
 Therefore, $\mathds{E} \in q\mathscr{P}(\mathds{A},r)$. Similarly, $\mathds{E}\in q\mathscr{P}(\mathds{B},r')$. Hence,
 \begin{equation*}
 	q\mathscr{P}(\mathds{D},\varepsilon) \subseteq q\mathscr{P}(\mathds{A},r) \cap q\mathscr{P}(\mathscr{B},r') \text,
 \end{equation*}
 and we have shown that open balls for $\Delta$ form a topological basis. Consequently, all open sets for $\Delta$ are unions of such balls.
 
 In particular, the Gromov-Hausdorff topology on $p\mathscr{P}$ is metrizable, generated by the topological basis of the open balls for $\Eth$, and convergence of a sequence $(\mathds{A}_n)_{n\in\N}$ to $\mathds{A}_\infty$ in this topology simply means $(\Eth(\mathds{A}_n,\mathds{A}))_{n\in\N}$ converges to $0$. The function $\Eth$ is continuous, and in fact our topology is the smallest topology to make it so.

\section{The Classical case and The Quantum Compact Case}

We now reconcile our construction with previous, known classes metrized by some form of Gromov-Hausdorff distance. In doing so, we immediately obtain many known examples of convergence for our new topology. In the next section, we will study a completely new example of convergence for our new topology involving a noncommutative, noncompact {\pqpms}.

\subsection{The Quantum Compact Case}

\emph{In this article, all {\qcms s} are assumed to be Leibniz unless otherwise specified.} As convergence for {\pqms s} involve the choice of a base point, we begin by defining the pointed version of the propinquity \cite{Latremoliere13,Latremoliere13b,Latremoliere14,Latremoliere15}. 

\begin{definition}
	A \emph{pointed {\qcms}} $(\A,\Lip,\mu)$ is a {\qcms} $(\A,\Lip)$ and a state $\mu \in \StateSpace(\A)$.
\end{definition}

Of triple $(\A,\Lip,\mu)$ is a pointed {\qcms} if, and only if, $(\A,\Lip,\C\unit_\A,\mu)$ is a {\pqpms}. Indeed, it is immediate that if $(\A,\Lip,\C\unit_\A,\mu)$ is a {\lcqms}, then $(\A,\Lip_{|\domsa{\Lip}},\mu)$ is a {\qcms}. On the other hand, if $(\A,\Lip,\mu)$ is a pointed {\qcms}, then we have two cases: if $\Lip$ is, in fact, given by a seminorm defined on a dense *-subalgebra of $\A$ --- which often occurs, for instance when $\Lip$ is computed from a spectral triple, e.g.  \cite{Connes89,Rieffel01,Rieffel02,Antonescu04,Ozawa05,Latremoliere13c,Latremoliere15c, Latremoliere20a,Latremoliere21a,Latremoliere23a,LAtremoliere23b,Latremoliere24a}, or a group action, such as in \cite{Rieffel98a, Latremoliere13c,Latremoliere16}, then $(\A,\Lip,\C\unit_\A,\mu)$ is a {\pqpms}. Otherwise, since L-seminorms on {\qcms s} are formally only defined on some dense Jordan-Lie subalgebra $\domsa{\Lip}$ of $\sa{\A}$, we can set $\Lip'(a) \coloneqq \max\{\Lip\circ\Re(a),\Lip\circ\Im(a)\}$ for all $a \in \A$ such that $\Re(a),\Im(a) \in \domsa{\Lip}$, and easily check that $(\A,\Lip,\C\unit_\A,\mu)$ is a {\pqpms}.

\begin{convention}
	In this section, for any {\qcms} $(\A,\Lip)$, the seminorm $\Lip$ will henceforth be assumed to be defined on a dense *-subalgebra $\dom{\Lip}$ of $\A$, as well as being a norm modulo constant and hermitian, which, as seen above, does not introduce any restriction. 
\end{convention}

\medskip

We now recall how the propinquity is defined. We call the tunnels introduced in \cite{Latremoliere13b, Latremoliere14} \emph{compact tunnels} to distinguish them from our current version. We also allow for a relaxed form of the Leibniz inequality, and say that $(\D,\Lip)$ is a $(C,D)$-Leibniz {\qcms} when for all $a,b \in \domsa{\Lip}$, we have 
\begin{equation*}
	\max\left\{ \Lip(\Re(ab)), \Lip(\Im(ab)) \right\} \leq C(\Lip(a)\norm{b}{\A} + \norm{a}{\A}\Lip(b)) + D\Lip(a)\Lip(b) \text. 
\end{equation*}
This will prove necessary in the proof of one of our results in this section.

We then define:
\begin{definition}[{\cite[Definition 3.1]{Latremoliere13b}}]
	Let $(\A,\Lip_\A)$ and $(\B,\Lip_\B)$ be two {\qcms s}. A \emph{compact $(C,D)$-tunnel} $\tau \coloneqq (\D,\Lip_\D,\pi,\rho)$ is given by a {$(C,D)$ Leibniz \qcms} $(\D,\Lip_\D)$   and two quantum isometries $\pi : (\D,\Lip_\D) \rightarrow (\A,\Lip_\A)$ and $\rho: (\D,\Lip_\D) \rightarrow (\B,\Lip_\B)$.
\end{definition}	
\begin{definition}
	Let $(\A,\Lip_\A)$ and $(\B,\Lip_\B)$ be two {\qcms s}. The \emph{extent} of a tunnel $\tau\coloneqq (\D,\Lip_\D,\pi,\rho)$ from $(\A,\Lip_\A)$ to $(\B,\Lip_\B)$ is 
	\begin{equation*}
		\max\left\{ \Haus{\Kantorovich{\Lip_\D}}\left(\StateSpace(\D),\pi^\ast(\StateSpace(\A))\right), \Haus{\Kantorovich{\Lip_\D}}\left(\StateSpace(\D),\rho^\ast(\StateSpace(\B))\right) \right\} \text.
	\end{equation*}
\end{definition}
We observe that
\begin{equation*}
	\tunnelextent{\tau} = \inf\left\{\varepsilon>0 : \StateSpace(\D) \subseteq_\varepsilon^{\Kantorovich{\Lip_\D}} \pi^\ast(\StateSpace(\A)) \text{ and }\StateSpace(\D) \subseteq_\varepsilon^{\Kantorovich{\Lip_\D}} \rho^\ast(\StateSpace(\B)) \right\} \text.
\end{equation*}
\begin{definition}[{\cite[Definition 2.11]{Latremoliere14}}]
	The \emph{$(C,D)$- (dual) propinquity} between any two {\qcms s} $\mathds{A}$ and $\mathds{B}$ is defined as:
	\begin{equation*}
		\dpropinquity{C,D}(\mathds{A},\mathds{B}) \coloneqq \inf\left\{ \tunnelextent{\tau} : \tau \text{ is a $(C,D)$-tunnel from }\mathds{A}{ to }\mathds{B} \right\}\text.
	\end{equation*}
\end{definition}
We refer to \cite{Latremoliere13,Latremoliere13b,Latremoliere14,Latremoliere15,Latremoliere13c,Latremoliere15c,Latremoliere15d,Latremoliere16,Latremoliere16b} for our study of this metric.

We now introduce a pointed version of the propinquity.
\begin{definition}
	Let $(\A,\Lip_\A,\mu_\A)$ and $(\B,\Lip_\B,\mu_\B)$ be two pointed {\qcms s}. The \emph{pointed} extent $\tunnelextent{\tau,\mu_\A,\mu_\B}$ of a tunnel $\tau\coloneqq(\D,\Lip_\D,\pi,\rho)$ from $(\A,\Lip_\A)$ to $(\B,\Lip_\B)$ is
	\begin{equation*}
		\tunnelextent{\tau,\mu_\A,\mu_\B} \coloneqq \max\{ \tunnelextent{\tau}, \Kantorovich{\Lip_\D}(\mu_\A\circ\pi, \mu_\B\circ\rho) \} \text.
	\end{equation*}
\end{definition}

\begin{definition}
	The \emph{pointed propinquity} $\dppropinquity{C,D}((\A,\Lip_\A,\mu_\A),(\B,\Lip_\B,\mu_\B))$ between two pointed {\qcms s} $(\A,\Lip_\A,\mu_\A)$ and $(\B,\Lip_\B,\mu_\B)$ is defined as:
	\begin{equation*}
		 \inf\left\{ \tunnelextent{\tau,\mu_\A,\mu_\B} : \tau \text{ is a tunnel from }(\A,\Lip_\A) { to } (\B,\Lip_\B) \right\}\text.
	\end{equation*}
\end{definition}

It is immediate that $\dppropinquity{C,D}$ is pseudo-metric, following the model of $\dpropinquity{C,D}$ --- in fact, $\dppropinquity{C,D}$ dominates $\dpropinquity{C,D}$, and is constructed using the same tunnels, so the properties of $\dppropinquity{C,D}$ follow almost immediately from our work on the propinquity.

 The distance zero question is the subject of the following result, which depends heavily on \cite{Latremoliere13b}.
\begin{definition}
	Let $(\A,\Lip_\A,\mu_\A)$ and $(\B,\Lip_\B,\mu_\B)$ be two pointed {\qcms s}, and $M\geq 1$. A \emph{pointed quantum isometry} $\pi:(\A,\Lip_\A,\mu)\rightarrow(\B,\Lip_\B,\mu_\B)$ is a quantum isometry $\pi : (\A,\Lip_\A)\rightarrow(\B,\Lip_\B)$ such that $\mu_\B\circ\pi = \mu_\A$.
	A \emph{full pointed quantum isometry} $\pi$ is a pointed quantum isometry, with an inverse $\pi^{-1}$ which is also a pointed quantum isometry.
\end{definition}
\begin{theorem}
	$\dppropinquity{C,D}$ is a metric on the class of pointed {\qcms s}, up to full pointed quantum isometry.
\end{theorem}
\begin{proof}
	Let $(\A,\Lip_\A,\mu_\A)$ and $(\B,\Lip_\B,\mu_\B)$ be two pointed {\qcms s} such that $\dppropinquity{}((\A,\Lip_\A,\mu_\A),(\B,\Lip_\B,\mu_\B)) = 0$. Therefore, there exists a sequence $(\tau_n)_{n\in\N}$ of compact tunnels from $(\A,\Lip_\A,\mu_\A)$ to $(\B,\Lip_\B,\mu_\B)$ such that
	\begin{equation*}
		\lim_{n\rightarrow\infty} \tunnelextent{\tau_n,\mu_\A,\mu_\B} = 0 \text.
	\end{equation*}

	By \cite{Latremoliere13b,Latremoliere14}, up to extracting a subsequence out of $(\tau_n)_{n\in\N}$ (we will simply assumed we replaced our sequence of tunnels by this subsequence), there exist a full quantum isometry $\theta : (\A,\Lip_\A)\rightarrow(\B,\Lip_\B)$, such that, if $a \in \dom{\Lip_\A}$, if $l\geq \Lip_\A(a)$, and if $T_n(a,l) \coloneqq \{ \rho_n(d) : \Lip_\D(d)\leq l, \pi_n(d) = a\}$, then 
		\begin{equation*}
			\lim_{n\rightarrow\infty} \Haus{\B}(T_n(a,l) , \{ \theta(a) \}) = 0\text.		\end{equation*}

	Let $a \in \dom{\Lip_\A}$ and $l = \Lip_\A(a)$. Let $d \in \dom{\Lip_n}$ with $\Lip_n(d) \leq l$ and $\pi_n(d) = a$. 
	\begin{align*}
		|\mu_\B\circ\theta(a) - \mu_\A(a)|
		&\leq |\mu_\B(\theta(a) - \rho_n(d))| + |\mu_\B\circ\rho_n(d) - \mu_\A\circ\pi_n(d)| \\
		&\leq \norm{\theta(a) - \rho_n(d)}{\B} + \Kantorovich{\Lip_n}(\mu_\A\circ\pi_n,\mu_\B\circ\rho_n) \\
		&\leq \Haus{\B}(T_n(a,l), \{\theta(a)\}) + \tunnelextent{\tau_n} \\
		&\xrightarrow{n\rightarrow\infty} 0 \text.  
	\end{align*}
	
	Thus $\mu_\B\circ\theta = \mu_\A$ on $\dom{\Lip_\A}$; by continuity and linearity of $\mu_\B\circ\theta$ and $\mu_\A$ on $\A$, and the fact that $\dom{\Lip_\A}$ is total in $\A$, our proof is complete.
\end{proof}

Last, we connect the pointed propinquity and the propinquity for pointed {\qcms s}.
\begin{lemma}
	Let $(\A_n,\Lip_n)_{n\in\N\cup\{\infty\}}$ be a family of {\qcms s}. The following assertions are equivalent.
	\begin{enumerate}
		\item $\lim_{n\rightarrow\infty} \dpropinquity{}((\A_n,\Lip_n),(\A_\infty,\Lip_\infty)) = 0 \text,$
		\item for all $\mu_\infty \in \StateSpace(\A_\infty)$, there exists states $\mu_n \in \StateSpace(\A_n)$ for all $n\in\N$, such that
		\begin{equation*}
		\lim_{n\rightarrow\infty} \dppropinquity{1,0}((\A_n,\Lip_n,\mu_n),(\A_\infty,\Lip_\infty,\mu_\infty)) = 0 \text,
		\end{equation*}
		\item $\lim_{n\rightarrow\infty} \dppropinquity{1,0}((\A_n,\Lip_n,\mu_n),(\A_\infty,\Lip_\infty,\mu_\infty)) = 0 \text,$
			for some states $\mu_n \in \StateSpace(\A_n)$ for $n\in\N\cup\{\infty\}$. 
	\end{enumerate}
\end{lemma}

\begin{proof}
	Assume (1). Let $\mu_\infty \in \StateSpace(\A_\infty)$. Let $(\tau_n)_{n\in\N} = (\D_n,\Lip_n,\pi_n,\rho_n)_{n\in\N}$ be a sequence of compact tunnels such that $\lim_{n\rightarrow\infty} \tunnelextent{\tau_n} = 0$, where $\tau_n$ is from $(\A_n,\Lip_n)$ to $(\A_\infty)$.
	
	For each $n\in\N$, by definition of the extent of a compact tunnel, there exists $\mu_n \in \StateSpace(\B)$ such that $\Kantorovich{\Lip_n}(\mu_\infty\circ\rho_n,\mu_n\circ\pi_n) \leq \tunnelextent{\tau_n}$. Thus Assertion (1) implies Assertion (2). Assertion (2) implies (3) and (3) implies (1) trivially.
\end{proof}

We begin by proving that convergence for the propinquity implies convergence for the quantum metametric. The key step is to construct a tunnel from a compact tunnel, relating their extents.

\begin{lemma}\label{tunnel-to-tunnel-lemma}
	Let $(\A,\Lip_\A,\mu_\A)$ and $(\B,\Lip_\B,\mu_\B)$ be two pointed {\qcms s}. If $\tau$ is a compact $(1,0)$-tunnel from $(\A,\Lip_\A)$ to $(\B,\Lip_\B)$, then there exist a {\qcms} $(\alg{E},\Lip[T])$, and two *-epimorphisms $\pi' : \alg{E}\rightarrow\A$ and $\rho':\alg{E}\rightarrow\B$ such that $\tau'=(\alg{E},\Lip[T],\C\unit_{\alg{E}},\pi',\rho')$ is an $M$-tunnel from $(\A,\Lip_\A,\C\unit_\A,\mu_\A)$ to $(\B,\Lip_\B,\C\unit_\B,\mu_\B)$ with
	\begin{equation*}
		\tunnelextent{\tau'} \leq \frac{3 \varepsilon + \varepsilon^2}{1+\varepsilon} \text,
	\end{equation*}
	where $\varepsilon\coloneqq\tunnelextent{\tau,\mu_\A,\mu_\B}$.
\end{lemma}

\begin{proof}
We write $\tau \coloneqq (\D,\Lip_\D,\pi,\rho)$ --- note $(\D,\Lip_\D)$ is a {\qcms}.

\medskip

Let $\alg{E} \coloneqq \A \oplus \D \oplus \B$. Let $\pi' : (a,d,b) \in \alg{E} \mapsto a \in\A$ and $\rho' : (a,d,b) \in \alg{E} \mapsto b\in\B$. We also introduce $\sigma : (a,d,e) \in \alg{E} \mapsto d\in\D$.

Let $\varepsilon = \tunnelextent{\tau,\mu_\A,\mu_\B}$, to simplify our notations.  

\medskip

For all $a \in \dom{\Lip_\A}$, $d\in \dom{\Lip_\D}$ and $b\in\dom{\Lip_\B}$, we set:
\begin{equation*}
	\Lip[T] (a,d,b) \coloneqq \max\left\{ \Lip_\A(a), \Lip_\D(d), \Lip_\B(b), \frac{1+\varepsilon}{\varepsilon}\norm{a - \pi(d)}{\A}, \frac{1+\varepsilon}{\varepsilon}\norm{b - \rho(d)}{\B} \right\} \text.
\end{equation*}
Thus $\dom{\Lip[T]} = \dom{\Lip_\A}\oplus\dom{\Lip_\D}\oplus\dom{\Lip_\B}$ by construction.

It is immediate that $\Lip[T]$ is hermitian, Leibniz, closed, and $\Lip[T]_n (a,d,b) = 0$ implies there exists $t\in \C$ such that $a = t\unit_\A$, $d=t\unit_\D$ and $b = t\unit_\B$. Moreover, let $\mu_\D \in \StateSpace(\D)$. Note that if $\mu_\D(d) = 0$, then $\norm{d}{\D} \leq \qdiam{\D,\Lip_\D}$ \cite[Proposition 1.6]{Rieffel98a}. Therefore:
\begin{align*}
	&\{ (a,d,b) \in \dom{\Lip[T]} : \Lip[T](a,d,b) \leq 1, \mu_\D\circ\sigma(a,d,b) = 0 \} \\
	&\quad \subseteq
	\{ a \in \dom{\Lip_\A} : \Lip_\A(a)\leq 1, \norm{a}{\A} \leq \frac{\varepsilon}{1+\varepsilon} + \norm{d}{\D} \} \\
	&\quad \times
	\{ d \in \dom{\Lip_\D} : \Lip_\D(d) \leq 1, \mu_\D(d) = 0 \} \\
	&\quad \times
	 \{ b \in \dom{\Lip_\B} : \Lip_\B(b)\leq 1, \norm{b}{\B} \leq \frac{\varepsilon}{1+\varepsilon} + \norm{d}{\D} \}\\
	 &\subseteq\{ a \in \dom{\Lip_\A} : \Lip_\A(a)\leq 1, \norm{a}{\A} \leq \frac{\varepsilon}{1+\varepsilon} + \qdiam{\D,\Lip_\D} \} \\
	&\quad \times
	\{ d \in \dom{\Lip_\D} : \Lip_\D(d) \leq 1, \mu_\D(d) = 0 \} \\
	&\quad \times
	 \{ b \in \dom{\Lip_\B} : \Lip_\B(b)\leq 1, \norm{b}{\B} \leq \frac{\varepsilon}{1+\varepsilon} + \qdiam{\D,\Lip_\D} \}\text.
\end{align*}
Since $(\A,\Lip_\A)$, $(\B,\Lip_\B)$ and $(\D,\Lip_\D)$ are all {\qcms s}, each factor on the right hand side is compact, and thus, as a closed subset of a compact set, so is $\{ (a,d,b) \in \dom{\Lip[T]} : \Lip[T](a,d,b) \leq 1, \mu_\D\circ\sigma(a,d,b) = 0 \} $. Thus $(\alg{E},\Lip[T])$ is a {\qcms}.

\medskip

Let $M \geq 1$. Let $a\in\dom{\Lip_\A,M}$ with $\norm{a}{\Lip_\A,M} \leq 1$. Since $\Lip_\A(a)\leq 1$, and since $\pi$ is a quantum isometry, there exists $d' \in \dom{\Lip_\D}$ such that $\pi(d') = a$ with $\Lip_\D(d') = \Lip_\A(a)$. Moreover, by \cite[Proposition 4.4]{Latremoliere13b}, \cite[Proposition 2.12]{Latremoliere14}, we conclude that 
\begin{equation*}
	\norm{d'}{\D} \leq \norm{a}{\A} + \varepsilon \leq 1 + \varepsilon \text.
\end{equation*}
Let $d \coloneqq \frac{1}{1+\varepsilon} d'$, so that $\Lip(d) \leq 1$ and $\norm{d}{\D} \leq 1$. Let $b \coloneqq \rho(d)$ so that $\norm{b}{\B} \leq 1$ and $\Lip_\B(b) \leq 1$. 
Now
\begin{equation*}
	\frac{1+\varepsilon}{\varepsilon} \norm{a - \pi(d')}{\A} = \frac{1+\varepsilon}{\varepsilon} \norm{1 - \frac{1}{1+\varepsilon} a}{\A} = \frac{\varepsilon}{1+\varepsilon} \frac{\varepsilon}{1+\varepsilon} \norm{a}{\A} = 1 \text.
\end{equation*}
Similarly, $\frac{1+\varepsilon}{\varepsilon}\norm{b - \rho(d')}{\B} = 1$. 

Altogether,
\begin{align*}
	\Lip[T](a,d,b)
	&=\max\left\{ \Lip_\A(a), \Lip(d), \Lip(b), \frac{1 + \varepsilon}{\varepsilon} \norm{a - \pi(d)}{\A}, \frac{1 + \varepsilon}{\varepsilon} \norm{b - \rho(d)}{\B} \right\}
	&=1 \text.
\end{align*}
Thus, for all $a\in\dom{\Lip_\A}$ with $\norm{a}{\Lip_\A}$ = 1, there exists $(a,d,b) \in \dom{\Lip[T]}$ such that $\norm{(a,d,b)}{\Lip[T],M} = 1$. It is immediate that $\Lip[T] \geq \Lip_\A\circ\pi'$, and therefore, since $\pi'$ is a unital *-morphism, so of norm $1$, that $\norm{\pi'(\cdot)}{\Lip_\A} \leq \norm{\cdot}{\Lip[T]}$. So $\pi'$ is a quantum isometry. Since $\pi'$ is unital, it is immediate as well that $\pi'$ is a topographic quantum isometry.

 The same reasonong applies to show that $\rho'$ is also a topographic $M$-isometry. 
 
 Thus, it is immediate that
 \begin{equation*}
 	\tau' \coloneqq \left( \alg{E}, \Lip[T], \C\unit_{\alg{E}}, \pi', \rho', \unit_{\alg{E}} \right) 
 \end{equation*}
is a tunnel.

\medskip

It remains to compute its extent. Note that $\Lip[T](\unit_{\alg{E}}) = 0$, and $\mu_\A\circ \pi'(\unit_{\alg{E}}) = 1 = \mu_\B\circ \rho(\unit_{\alg{E}})$. 

Let $\varphi \in \StateSpace(\alg{E})$. There exists $t_1,t_2,t_3 \in [0,1]$ such that $t_1 + t_2 + t_3 = 1$, and $\varphi_1 \in\StateSpace(\A)$, $\varphi_2 \in \StateSpace(\D)$ and $\varphi_3 \in \StateSpace(\B)$ such that 
\begin{equation*}
	\varphi = t_1\varphi_1\circ\pi + t_2\varphi_2 + t_3 \varphi_3\circ\rho \text.
\end{equation*}

Since $\varphi_1\circ\pi \subseteq\StateSpace(\D)$, there exists $\psi_1 \in \QuasiStateSpace(\B)$ such that
\begin{equation*}
	\boundedLipschitz{\Lip_\D,M}(\varphi_1\circ\pi,\psi_1\circ\rho) \leq \Kantorovich{\Lip_\D}(\varphi_1\circ\pi,\psi_1\circ\rho) \leq \varepsilon\text.
\end{equation*}

We then compute:
\begin{align*}
\boundedLipschitz{\Lip[T],M}(\varphi_1\circ\pi',\psi_1\circ\rho')
&=\sup\{|\varphi_1\pi'(a,d,b) - \psi_1\circ\rho'(a,d,b)| : \norm{(a,d,b)}{\Lip[T]} \leq 1 \} \\
&\leq\sup\{|\varphi_1(a) - \psi_1(b)| : \norm{(a,d,b)}{\Lip[T]} \leq 1 \} \\
&\leq\sup\{|\varphi_1(a) - \varphi_1\circ\pi(d)|: \norm{(a,d,b)}{\Lip[T]} \leq 1 \} \\
&\quad + \sup\{|\varphi_1\circ\pi(d) - \psi_1\circ\pi(d)|:\norm{(a,d,b)}{\Lip[T]} \leq 1\} \\
&\quad + \sup\{|\psi_1\circ\pi(d) - \psi_1(b)| :\norm{(a,d,b)}{\Lip[T]} \leq 1 \}\\
&\leq \norm{a-\pi(d)}{\A} + \boundedLipschitz{\Lip_\D}(\varphi_1,\psi_1) + \norm{b-\rho(d)}{\B} \\
&\leq \frac{\varepsilon}{1+\varepsilon} + \varepsilon + \frac{\varepsilon}{1+\varepsilon} \\
&= \frac{3\varepsilon + \varepsilon^2}{1+\varepsilon} \text.
\end{align*}

Since $\varphi_2 \in \StateSpace(\D)$, there exists $\psi_2 \in \StateSpace(\B)$ such that
\begin{equation*}
	\boundedLipschitz{\Lip_\D,M}(\varphi_2,\psi_2\circ\rho) \leq \Kantorovich{\Lip_\D}(\varphi_2,\psi_2\circ\rho) \leq \varepsilon\text.
\end{equation*}
By a similar computation, we then note:
\begin{equation*}
	\boundedLipschitz{\Lip[T],M}(\varphi_2\circ\sigma, \psi_2\circ\rho') \leq \varepsilon + \frac{\varepsilon}{1+\varepsilon}\text. 
\end{equation*}

Set $\psi_3 \coloneqq \varphi_3$ and $\psi \coloneqq \sum_{j=1}^3 t_j \psi_j \in \StateSpace(\A_\infty)$. We compute:
\begin{align*}
	\boundedLipschitz{\Lip[T],M}(\varphi\circ\pi',\psi\circ\rho')
	&\leq \sum_{j=1}^3 \boundedLipschitz{\Lip[T]}(\varphi_1\circ\pi',\psi_1\circ\rho') \\
	&\leq t_1 \frac{3\varepsilon + \varepsilon^2}{1+\varepsilon} + t_2\frac{2\varepsilon + \varepsilon^2}{1+\varepsilon} + t_3 \cdot 0 \leq \frac{3\varepsilon + \varepsilon^2}{1+\varepsilon}  \text.
\end{align*}

Therefore, 
\begin{equation*}
	\StateSpace(\alg{E}) \subseteq_{\frac{3\varepsilon + \varepsilon^2}{1+\varepsilon}}^{\boundedLipschitz{\Lip[T]},M} (\rho')^\ast(\StateSpace(\B)) \text.
\end{equation*}

Since $\boundedLipschitz{\Lip[T],M}$ is positive $1$-homogeneous, we immediately conclude that
\begin{equation*}
	\QuasiStateSpace(\alg{E}) \subseteq_{\frac{3\varepsilon + \varepsilon^2}{1+\varepsilon}}^{\boundedLipschitz{\Lip[T]},M} (\rho')^\ast(\QuasiStateSpace(\B)) \text.
\end{equation*}

A similar computation shows that
\begin{equation*}
	\QuasiStateSpace(\alg{E}) \subseteq_{\frac{3\varepsilon + \varepsilon^2}{1+\varepsilon}}^{\boundedLipschitz{\Lip[T]},M} (\pi')^\ast(\QuasiStateSpace(\A)) \text.
\end{equation*}

Last,
\begin{align*}
	\Kantorovich{\Lip[T]}(\mu_\A\circ\pi', \mu_\B\circ\rho') \leq \frac{2\varepsilon}{1+\varepsilon} + \varepsilon = \frac{3\varepsilon + \varepsilon^2}{1+\varepsilon}
\end{align*}
by a similar computation as above.  Last, $\Kantorovich{\Lip_\D}(\mu_\A\circ\pi,\mu_\B\circ\rho) \leq \varepsilon$ by definition of the extent for pointed {\qcms s}.

 Therefore, the extent of $\tau'$ is at most $\frac{3\varepsilon + \varepsilon^2}{1+\varepsilon}$, as claimed.
\end{proof}

\begin{corollary}
	If $(\A_n,\Lip_n,\mu_n)_{n\in\N}$ is a sequence of pointed {\qcms s} converging for the pointed propinquity to $(\A,\Lip,\mu)$, then 
	\begin{equation*}
		\lim_{n\rightarrow\infty} \sup_{M\geq 1} \metametric{M}((\A_n,\Lip_n,\C\unit_n,\mu_n), (\A,\Lip,\C\unit_\A,\mu)) = 0 \text,
	\end{equation*}
	and in particular,
	\begin{equation*}
	\lim_{n\rightarrow\infty} \Eth((\A_n,\Lip_n,\C\unit_n,\mu_n), (\A,\Lip,\C\unit_\A,\mu)) = 0 \text.
	\end{equation*}
\end{corollary}

\begin{proof}
	Let $\varepsilon > 0$, $M\geq 1$. Set $x \coloneqq \min\{1,\frac{\varepsilon}{4}\}$ --- note that $x>0$. There exists $N\in\N$ such that, for all $n\geq N_1$, we have $\dpropinquity{}((\A,\Lip),(\A_n,\Lip_n)) < \frac{x}{2}$. Moreover, since $\exp(-\qdiam{\cdot})$ is continuous with respect to the propinquity, there exists $N_2 \in \N$ such that, for all $n\geq N_2$, we have
	\begin{equation*}
		\left|\exp(-\qdiam{\A,\Lip}) - \exp(-\qdiam{\A_n,\Lip_n}) \right| < \varepsilon \text.
	\end{equation*}
	
	 Let $n \geq N \coloneqq \max\{N_1,N_2\}$. There exists, by definition of the metametric, a compact tunnel $\tau_n$ from $(\A,\Lip)$ to $(\A_n,\Lip_n)$ with extent at most $x$. By Lemma (\ref{tunnel-to-tunnel-lemma}), there exists a tunnel $\tau'$ from $(\A_n,\Lip_n,\C\unit_n,\mu_n)$ to $(\A,\Lip,\C\unit_\A,\mu)$ with extent at most
	 \begin{equation*}
	 	\frac{3 x+x^2}{1+x} \leq \varepsilon \text.
	 \end{equation*}
	 
	 Therefore, by Definition (\ref{prop-def}), for all $n\geq N$,
	 \begin{equation*}
	 	\metametric{M}((\A_n,\Lip_n,\C\unit_n,\mu_n), (\A,\Lip,\C\unit_\A,\mu)) \leq \max\{ \tunnelextent{\tau'} , \varepsilon \} = \varepsilon \text.
	 \end{equation*}
	 Therefore, if $n\geq\max\{N_1,N_2\}$, then $\sup_{M\geq 1}\metametric{M}((\A_n,\Lip_n,\C\unit_n,\mu_n), (\A,\Lip,\C\unit_\A,\mu)) \leq \varepsilon$.
	 
	 By Definition (\ref{Metametric-def}), this implies, in turn, that if $n\geq\max\{N_1,N_2\}$, then
	 \begin{equation*}
	 	\Eth((\A_n,\Lip_n,\C\unit_n,\mu_n), (\A,\Lip,\C\unit_\A,\mu)) < \varepsilon \text.
	\end{equation*}
	
	 This proves our corollary.
\end{proof}

\begin{corollary}
	If $(\A_n,\Lip_n)_{n\in\N}$ is a sequence of pointed {\qcms s} converging for the propinquity to $(\A,\Lip)$, and if $\mu \in \StateSpace(\A)$, then there exists $\mu_n \in \StateSpace(\A_n)$ for all $n\in\N$ such that 
	\begin{equation*}
		\lim_{n\rightarrow\infty} \sup_{M\geq 1} \metametric{M}((\A_n,\Lip_n,\C\unit_n,\mu_n), (\A,\Lip,\C\unit_\A,\mu)) = 0 \text,
	\end{equation*}
	and in particular,
	\begin{equation*}
	\lim_{n\rightarrow\infty} \Eth((\A_n,\Lip_n,\C\unit_n,\mu_n), (\A,\Lip,\C\unit_\A,\mu)) = 0 \text.
	\end{equation*}
\end{corollary}

\begin{proof}
	This follows from our previous corollary and the relationship between the metametric and the pointed metametric.
\end{proof}

\medskip

We turn to the converse. We point out that there exists tunnels between {\qcms s} of the form $(\D,\Lip,\M,\pi,\rho,e)$ where $\D$ is not unital, in general. Indeed, one can use $e$ to cut off the ``excess'' space. In particular, even when domains and codomains are {\qcms s}, the element $e$ need not be a unit, or a multiple of a unit, or any such. Moreover, neither $\rho$ nor $\pi$ are quantum isometries, i.e. induce isometries from the {\MongeKant s} on the domain and codomain of our tunnel. 

We offer the following conversion lemma which turns a tunnel between {\qcms s} into a compact tunnel, at a cost of relaxing the Leibniz condition. We focus on tunnels with relatively small extent, as our eventual concern is about convergence.

\begin{lemma}\label{nothing-easy-lemma}
	Let $\mathds{A} \coloneqq (\A,\Lip_\A,\mu_\A)$ and $\mathds{B} \coloneqq (\B,\Lip_\B,\mu_\B)$ be two pointed {\qcms s}. If $\tau$ be an $r$-tunnel  from $(\A,\Lip_\A,\M_\A,\mu_\A)$ to $(\B,\Lip_\B,\M_\B\mu_\B)$, for some topographies $\M_\A$ and $\M_\B$ of, respectively, $\mathds{A}$ and $\mathds{B}$, and if 
	\begin{equation*}
		r > \max\{\qdiam{\A,\Lip_\A},\qdiam{\B,\Lip_\B}\}\text,
	\end{equation*}
	and
	\begin{equation*}
		\tunnelextent{\tau} < \max\left\{3,\frac{1}{4r+6}\right\} \text,
	\end{equation*}
 then there exists a compact $(2,4r)$-tunnel $\tau'$ from $\mathds{A}$ to $\mathds{B}$ of extent at most 
 	\begin{equation*}
 		\tunnelextent{\tau',\mu_\A,\mu_\B} \leq \tunnelextent{\tau}\left(3 + 9 r + 4 r^2 \right)\text.
	\end{equation*}
\end{lemma}

\begin{proof}
	Let 
	\begin{equation*}
		\tunnel{\tau}{(\A,\Lip_\A,\M_\A,\mu_\A)}{\pi}{(\D,\Lip,\M,e)}{\rho}{(\B,\Lip_\B,\M_\B,\mu_\B)}\text.
	\end{equation*}
	
	Let $\varepsilon \coloneqq \tunnelextent{\tau}$. Let $D\coloneqq\max\{\qdiam{\A,\Lip_\A},\qdiam{\B,\Lip_\B}\}$, and assume $D<r$. We assume that $\varepsilon\leq\frac{1}{2(r+1)}$, so in particular, $\varepsilon(D+2) < \frac{1}{2}$ by assumption.

	Our first step is to extent $\pi$ and $\rho$ as unital *-epimorphism form $\unital{\D}$ onto $\A$ and $\B$, respectively, keeping the same notation for these (necessarily unique) extentions. We then extend $\Lip$ to $\unital{\D}$ by setting $\Lip(d + t\unit_\D) = \Lip(d)$ for all $d\in\dom{\Lip}$ and $t\in \R$. Note that if $\D$ is already unital, then this definition is consistent. However, also note that, if $\D$ is not unital, then $\Lip$ is no longer an L-seminorm.

	Let $e^\A \coloneqq \pi(e)$. Since $(\A,\Lip_\A)$ is a {\qcms}, we conclude by \cite[Proposition 1.6]{Rieffel98a}, that
	\begin{equation*}
		\norm{e^\A- \mu_\A(e^\A)\unit_\A}{\A} \leq \qdiam{\A,\Lip_\A} \Lip_\A(e^A) \leq D\varepsilon \text.
	\end{equation*}
Moreover, $|\mu_\A(e^\A)-1| \leq \varepsilon$ by Definition (\ref{extent-def}). Hence
\begin{equation*}
	\norm{e^\A - \unit_\A}{\A} \leq (D+1)\varepsilon \text.
\end{equation*}
Let now $\varphi \in \StateSpace(\A)$. By Definition (\ref{extent-def}) again, there exists $\psi \in \QuasiStateSpace(\B)$ such that
\begin{equation}\label{convoluted-eq}
	\boundedLipschitz{\Lip,M}(\varphi\circ\pi(e\cdot e), \psi\circ\rho) \leq \varepsilon \text.
\end{equation}

Let $d \in \dom{\Lip}$ with $\norm{d}{\Lip,r}\leq 1$. We then compute:
\begin{align}\label{varphi-psi-bl-eq}
	|\varphi\circ\pi(d) - \psi\circ\rho(d)| 
	&\leq \underbracket[1pt]{|\varphi\circ\pi(d)-\varphi\circ\pi(ede)|}_{\leq\norm{\pi(d-ede)}{\D}} + |\varphi\circ\pi(ede) - \psi\rho(d)| \nonumber \\
	&\leq \norm{a-e^\A a e^\A}{\A} + |\varphi\circ\pi(e d e) - \psi\circ\rho(d)| \leq (D+2) \varepsilon \text.
\end{align}
Yet $\psi$ is a quasi-state and not necessarily a state, so we proceed to adjust $\psi$. Record first that $\norm{\psi}{\B^\ast}\leq 1$.

By  Expression \eqref{convoluted-eq}, and since there exists a \lipunit{\Lip}{\mu_\A\circ\pi} $(h_n)_{n\in\N}$ in $\M$, with $\norm{h_n}{\Lip,r}\leq 1 + \frac{1}{D(n+1)} \leq 2$, and thus:
\begin{equation*}
	|\varphi\circ\pi(e h_n e) - \psi\circ\rho(h_n)| \leq 2 \varepsilon\text.
\end{equation*}
Therefore, taking the limit at $n$ goes to $\infty$, since $\varphi\circ\pi(e\cdot e)$ and $\psi\circ\rho$ are both positive linear functionals over $\D$,
\begin{equation*}
\left|\norm{\varphi\circ\pi(e\cdot e)}{\A^\ast} - \norm{\psi}{\B^\ast}\right| \leq 2\varepsilon
\end{equation*}

Consequently, 
\begin{equation*}
	\norm{\psi}{\B^\ast} \geq \varphi\circ\pi(e^2) - 2\varepsilon\text.
\end{equation*}

Now, using the fact that $\mu_\A\circ\pi$ is a character of $\M$ and $e\in \M$, as well as $|1-\mu_\A\circ\pi(e)| \leq \varepsilon$ once more,
\begin{align*}
	|1-\varphi\circ\pi(e^2)| 
	&\leq |1-\mu_\A\circ(e^2)| + |\mu_\A\circ\pi(e^2) - \varphi\circ\pi(e^2)| \\
	&\leq |1-\mu_\A\circ\pi(e)| |1+\mu_\A\circ\pi(e)| + \Lip(e^2)\Kantorovich{\Lip_\A}(\varphi,\mu_\A) \\
	&\leq 2 \varepsilon + 2\Lip(e)\norm{e}{\D}D \\
	&\leq 2\varepsilon + \varepsilon\sqrt{1+\varepsilon} D \leq \varepsilon(2+D\sqrt{1+\varepsilon}) \leq 2(1+D) \varepsilon \text, 
\end{align*}
where we used the fact that $\frac{x}{1-x}<2$ when $x\in\left[0,\frac{1}{2}\right]$.

Hence
\begin{equation*}
	\varphi\circ\pi(e^2) \geq 1 - 2(D+1)\varepsilon \text{ so }\norm{\psi}{\B^\ast}\geq 1 - 2(D+2)\varepsilon \text.
\end{equation*}
Consequently,
\begin{equation*}
	\left| 1 - \frac{1}{\norm{\psi}{\B^\ast}} \right| = \frac{1}{\norm{\psi}{\B^\ast}} - 1 \leq \frac{1}{1-2(D+2)\varepsilon} - 1 = \frac{2(D+2)\varepsilon}{1-2(D+2)\varepsilon} \leq 4(2D+2) \varepsilon \text.
\end{equation*}

Altogether, we see that, if $\theta\coloneqq \frac{1}{\norm{\psi}{\B^\ast}}\psi$, then $\theta \in \StateSpace(\B)$, and if $d\in\D$ with $\norm{d}{\Lip,r} \leq 1$, then
\begin{equation*}
	|\varphi\circ\pi(d) - \theta\circ\rho(d)| \leq \underbracket[1pt]{(D+2)\varepsilon }_{\text{by Exp. \eqref{varphi-psi-bl-eq}}}+ \left|1-\frac{1}{\norm{\psi}{\B^\ast}}\right| r \leq ((D+2) + (4D+8)r)\varepsilon \leq \varepsilon\left(2 + 9r + 4r^2 \right) \text. 
\end{equation*}
Thus, 
\begin{equation}\label{varphi-theta-close-bl-eq}
	\forall \varphi \in \StateSpace(\A)\quad \exists\theta\in \StateSpace(\B) \quad\quad\boundedLipschitz{\Lip}(\varphi\circ\pi,\psi\circ\rho) \leq \left(2 + 7r + 4r^2 \right)\varepsilon\text.
\end{equation}
 We obtain a similar estimate exchanging the roles of $\A$ and $\B$, by symmetry.

\medskip

Define 
\begin{equation*}
	\Lip[S](d_1,d_2) \coloneqq \max\left\{\norm{d_1-\mu_\A\circ\pi(d_1)\unit_\D}{\Lip, r} , \norm{d_2-\mu_\B\circ\pi(d_2)\unit_\D}{\Lip, r}, \frac{1}{\varepsilon}\norm{d_1-d_2}{\D} \right\}
\end{equation*}
for all $(d_1,d_2) \in \dom{\Lip}\oplus\dom{\Lip}$, and where $\norm{d}{\Lip,r} \coloneqq \max\left\{ \frac{1}{r}\norm{d}{\unital{\D}}, \Lip(d) \right\}$ using our extension of $\Lip$ to $\dom{\Lip} + \R\unit_\D$ in $\unital{D}$. 

We note that $\Lip_\A\circ\pi(d) \leq \Lip(d)  = \Lip(d-\mu_\A(\pi(d))\unit_\D)\leq \norm{d-\mu_\A(\pi(d))\unit_\D}{\Lip,r} \leq \Lip[S](d,f)$ for all $d,f \in \dom{\Lip}$.

Let now $\alg{E} \coloneqq \A\oplus\B$. For all $(a,b) \in \dom{\Lip_\A}\oplus\dom{\Lip_\B}$, we define:
\begin{equation*}
	\Lip[Q](a,b) \coloneqq \inf\left\{ \Lip[S](d_1,d_2) : \pi(d_1) = a, \rho(d_2) = b \right\}\text.
\end{equation*}
Let $\eta_\A : (a,b) \in \alg{E}\mapsto a\in\A$ and $\eta_\B:(a,b)\in\alg{E}\mapsto b\in \B$.

We first note that $\Lip[Q]$ is indeed defined and finite on $\dom{\Lip_\A}\oplus\dom{\Lip_\B}$ and thus densely defined in $\A\oplus\B$. Moreover, let $\Lip[Q](a,b)\leq 1$ and $\mu_\A(a) = 0$. Thus, there exists $d_1,d_2\in\dom{\Lip}$ such that $\pi(d_1) = a$, $\rho(d_2) = b$, $\norm{d_1-d_2}{\D}\leq\varepsilon$, $\norm{d_1}{\Lip}\leq 2$ and $\norm{d_2}{\Lip} \leq 2$. Now $\Lip_\A\circ\pi\leq\Lip_1$ so $\Lip_\A(a)\leq 2$, and similarly, $\Lip_\B(b)\leq 2$. Last, 
\begin{align*}
	|\mu_\B(b)| 
	&= \left|\mu_\B\circ\rho(d_2)\right| \\
	&\leq \left|\mu_\B\circ\rho(d_2-d_1) + \mu_\B\circ\rho(d_1)\right| 
	&\leq \varepsilon + \left|\mu_\B\circ\rho(d_1) - \mu_\A\circ\pi(d_1)\right| + \left|\mu_\A\circ\pi(d_1)\right| \\
	&\leq 1 + 2 \varepsilon + |\mu_\A(a)| = 1+2\varepsilon \text.
\end{align*}
Hence:
\begin{multline*}
	\left\{ (a,b) \in \dom{\Lip[Q]} : \Lip[Q](a,b)\leq 1\right\} \subseteq \left\{a\in\dom{\Lip_\A}:\mu_\A(a) =0,\Lip_\A(a)\leq 2 \right\} \\
	\times\left\{ b\in\dom{\Lip_\B} : |\mu_\B(b)|\leq 7, \Lip_\B(b)\leq 2 \right\}
\end{multline*}
and, since both $(\A,\Lip_\A)$ and $(\B,\Lip_\B)$ are {\qcms s}, the set on the right hand side is compact. Hence $\left\{ (a,b) \in \dom{\Lip[Q]} : \Lip[Q](a,b)\leq 1\right\}$ is totally bounded. Moreover, it is easily checked to be closed. 

Assume now that $\Lip[Q](a,b) = 0$. Since $\max\{\Lip_\A(a),\Lip_\B(b)\}\leq \Lip[Q](a,b) = 0$, we conclude that $a = t \unit_\A$ and $b = s \unit_\B$, but we wish to show that $s=t$. Let $\delta>0$. There exists $d_1,d_2 \in \dom{\Lip[S]}$ such that $\pi(d_1)=a$, $\pi(d_2)=b$, and $\Lip[S](d_1,d_2) < \delta$. In particular, 
\begin{equation*}
	|t-s| = |\mu_\A\circ_\A(d_1) - \mu_\B\circ\B(d_2)| \leq 2\delta + \norm{d_1 - d_2}{\D} \leq 3\delta \text.
\end{equation*}
Since $\delta> 0$ is arbitrary, we conclude indeed that $s=t$, i.e. $\Lip[Q]$ is a norm modulo constants.

So $(\alg{E},\Lip[Q])$ is a {\qcms}, though it remains to investigate its Leibniz property.

\medskip

Now, we turn to the Leibniz property of $\Lip[Q]$, First, as noted in \cite[Lemma 3.2]{Latremoliere15c}, we observe that, for all $d_1,d_2 \in \D$:
\begin{align*}
	\norm{d_1 d_2 - \mu_\A(\pi(d_1 d_2))\unit_\D}{\D}
	&\leq \norm{(d_1 - \mu_\A(\pi(d_1)))d_2}{\D} + \norm{\mu_\A(\pi(d_1))d_2 - \mu_\A(\pi(d_1)\pi(d_2))}{\D} \\
	&\leq \norm{d_2}{\D} \norm{d_1 - \mu_\A(\pi(d_1))}{\D} + \norm{\mu_\A(\pi(d_1)) d_2 - \mu_\A(\pi(d_1))\mu_\A(\pi(d_2))}{\D} \\
	&+ \norm{\mu_\A(\pi(d_1))\mu_\A(\pi(d_2)) - \mu_\A(\pi(d_1)\pi(d_2))}{\D}\\
	&\leq \norm{d_1 - \mu(\pi(d_1))}{\D}\norm{d_2}{\D} + \norm{d_1}{\D}\norm{d_2-\mu(\pi(d_2))}{\D} \\ 
	&\quad + |\mu((\mu(\pi(d_1)) - \pi(d_1))\pi(d_2))| \\
	&\leq 2\norm{d_1 - \mu(\pi(d_1))}{\D}\norm{d_2}{\D} + \norm{d_1}{\D}\norm{d_2-\mu(\pi(d_2))}{\D} \text.
\end{align*}
Since $\Lip$ is Leibniz,
\begin{align*}
	\Lip(d_1 d_2 - \mu_\A(\pi(d_1 d_2))\unit_\D) = \Lip(d_1 d_2) \leq \norm{d_1}{\D}\Lip(d_2) + \Lip(d_1)\norm{d_2}{\D}\text.
\end{align*}
Consequently, $\Lip[S]$ is $(2,0)$-Leibniz.

Now, let $(a,b), (a',b') \in \dom{\Lip[Q]}$. By construction of $\Lip[S]$, there exists $(d,f), (d',f') \in \dom{\Lip[S]}$ such that $\pi(d) = a$, $\pi(d') = a$, $\rho(d) = b$, $\rho(d') = b'$, while $\Lip[S](d,f) \leq \Lip[Q](a,b) + \delta$ and $\Lip[S](d',f') \leq \Lip[Q](a',b') + \delta$. 
)
Consequently, $\norm{d-\mu_\A(\pi(d))}{\D} \leq r (\Lip[Q](a,b) + \delta)$ by construction of $\Lip[S]$. So in particular,  
\begin{align*}
	\norm{d}{\D} 
	&\leq \norm{d-\mu_\A(\pi(d))\unit_\D}{\D} + |\mu(\pi(d))| \\
	&\leq r (\Lip[Q](a,b) + \delta) + \norm{a}{\A} \\
	&\leq r (\Lip[Q](a,b) + \delta) + \norm{(a,b)}{\alg{E}} \text.
\end{align*}

The same estimates apply to show
\begin{equation*}
	\norm{f}{\D} \leq r (\Lip[Q](a,b) + \delta) + \norm{(a,b)}{\alg{E}}
\end{equation*}
and
\begin{equation*}
	\max\{\norm{d'}{\D},\norm{f'}{\D} \leq r (\Lip[Q](a',b') + \delta) + \norm{(a',b')}{\alg{E}}\text. 
\end{equation*}

We note that the delicate choice we made for $\Lip[S]$ was precisely designed to ensure the above control over norms.

We can then compute:
\begin{align*}
	\Lip[Q](aa',bb')
	&\leq \Lip[S](d d', f f') \\
	&\leq 2\left( \Lip[S](d,f) \norm{d',f'}{\D\oplus\D} + \norm{d,f}{\D\oplus\D}\Lip[S](d',f') \right) \\
	&\leq 2\left( (\Lip[Q](a,b)+\delta) (\norm{a',b'}{\A} + r\Lip[Q](a',b') + r \delta) + (\norm{(a',b')}{\alg{E}} + r\Lip[Q](a',b') + r \delta)(\Lip[Q](a,b)+\delta) \right)\text.
\end{align*}
Since $\delta>0$ was arbitrary, we conclude:
\begin{equation*}
	\Lip[Q](a a', b b') \leq 2 \left( \Lip[Q](a,b) \norm{(a',b')}{\alg{E}} + \norm{(a,b)}{\alg{E}}\Lip[Q](a',b') \right) + 4 r \Lip[Q](a,b) \Lip[Q](a',b') \text.
\end{equation*}

So $(\alg{E},\Lip[Q])$ is a $(2,4r)$-{\qcms}. 

\medskip

Note that by construction, for all $a\in\dom{\Lip_\A}$ and $b\in\dom{\Lip_\B}$, we have $\Lip_\A(a) \leq \Lip[Q](a,b)$ since $\Lip_\A(a)\leq\Lip[S](d_1,d_2)$ for any $d_1,d_2\in\dom{\Lip}$ with $\pi(d_1) = a$; similarly, $\Lip_\B(b)\leq \Lip[Q](a,b)$.

Let now $a \in \dom{\Lip_\A}$ with $\Lip_\A(a) \leq 1$. Then $\norm{a-\mu_\A(a)\unit_\A}{\A} \leq D \leq r$. Therefore, $\norm{a-\mu(a)\unit_\A}{\Lip,r} \leq 1$. Let $\delta> 0$. By definition of $\pi$, there exists $d' \in \dom{\Lip}$ such that $\pi(d') = a-\mu(a)\unit_\A$, and $\norm{d'}{\Lip,r} \leq 1+\delta$ --- in particular, $\Lip(d')\leq 1+\delta$. Therefore, if we set $d \coloneqq d' + \mu_\A(a) \unit_\D$, then $\pi(d) = a$, and by construction, $\Lip[S](d,d) = \norm{d'}{\Lip,r} \leq 1+\delta$. If $b\coloneqq \rho(d)$, then $\Lip[Q](a,b) \leq \Lip[S](d,d) \leq 1+\delta$.

So, for all $a\in\dom{\Lip}$, there exists $b \in \dom{\Lip_\B}$ such that $\Lip[Q](a,b) \leq\Lip_\A(a)(1 +\delta)$. We thus have shown that:
\begin{equation*}
	\Lip_\A(a) = \inf\left\{ \Lip[Q](a,b) : b \in \dom{\Lip_\B} \right\} \text.
\end{equation*}
Thus, $\eta_\A$ is a quantum isometry from $(\alg{E},\Lip[Q])$ onto $(\A,\Lip_\A)$.

The same reasoning applies to show that $\eta_\B$ is a quantum isometry from $(\alg{E},\Lip[Q])$ onto $(\B,\Lip_\B)$. Therefore, $(\alg{E},\Lip[Q],\eta_\A,\eta_\B)$ is indeed a compact tunnel from $(\A,\Lip_\A)$ to $(\B,\Lip_\B)$. It remains to compute its estimate.

\medskip

Let $\varphi \in \StateSpace(\A)$. As seen in Expression \eqref{varphi-theta-close-bl-eq}, there exists $\theta\in\StateSpace(\B)$ such that $\boundedLipschitz{\Lip,r}(\varphi\circ\pi,\theta\circ\rho) \leq \left(2 + 9 r + 4r^2 \right)\varepsilon$.

Therefore, for all $\delta>0$,
\begin{align*}
	\Kantorovich{\Lip[Q]}&(\varphi\circ\eta_\A,\theta\circ\eta_\B)\\
	&=\sup\left\{ |\varphi(a) - \theta(b)| : \Lip[Q](a,b) \leq 1\right\}\\
	&\leq\frac{1}{1+\delta}\sup\Big\{ |\varphi\circ\pi(d_1) - \theta\circ\rho(d_2)| : \\
	&\quad\quad \norm{d_1-\mu_\A(\pi(d_1))}{\Lip,r} \leq 1+\delta, \norm{d_2-\mu_\B(\rho(d_2))}{\Lip,r} \leq 1+\delta, \norm{d_1-d_2}{\D} \leq \varepsilon \Big\}\\
	&\leq \frac{1}{1+\delta}\sup\left\{ |\varphi\circ\pi(d_1) - \varphi\circ\pi(d_2)| : \norm{d_1-d_2}{\D} \leq \varepsilon\right\}\\
	&\quad + \frac{1}{1+\delta}\sup\left\{ |\varphi\circ\pi(d_2) - \theta\circ\rho(d_2)| : \norm{d_2-\mu_\B\circ\rho(d_2)}{\Lip,r} \leq 1\right\}\\
	&\leq \frac{\varepsilon}{1+\delta} + \frac{1}{1+\delta}\sup\left\{ |\varphi\circ\pi(d_2-\mu_\B\circ\rho(d_2)\unit_\D) - \theta\circ\rho(d_2-\mu_\B\circ\rho(d_2)\unit_\D)| : \norm{d_2}{\Lip,r} \leq 1\right\}\\
	&\leq \frac{\varepsilon}{1+\delta} + \frac{1}{1+\delta}\sup\left\{ |\varphi\circ\pi(d_2) - \theta\circ\rho(d_2)| : \norm{d_2}{\Lip,r} \leq 1\right\}\\
	&\leq \frac{1}{1+\delta}\left(\varepsilon + \left(2 + 9r + 4r^2 \right)\varepsilon\right) \text.
\end{align*}
Hence, since $\delta>0$ was arbitrary, 
\begin{equation*}
	\Haus{\Kantorovich{\Lip[Q]}}(\eta_\A^\ast(\StateSpace(\A)),\eta_\B^\ast(\StateSpace(\B))) \leq (3+9r+4r^2) \varepsilon\text,
\end{equation*}
from which it immediately follows by convexity that 
\begin{equation*}
	\Haus{\Kantorovich{\Lip[Q]}}(\StateSpace(\alg{E}),\eta_\B^\ast(\StateSpace(\B))) \leq (3+9r+4r^2) \varepsilon\text,
\end{equation*} 
since every state of $\alg{E}$ is a convex combination of a state of $\A$ and a state of $\B$.

Our reasoning is again symmetric in $\A$ and $\B$. So $(\alg{E},\Lip[Q],\eta_\A,\eta_\B)$ is a compact tunnel of extent at most $(3+9r+4r^2) \varepsilon$. Since, by Definition (\ref{extent-def}), we also have $\Kantorovich{\Lip}(\mu_\A\circ\pi,\mu_\B\circ\rho) \leq \varepsilon$, this concludes our proof.
\end{proof}

We can now obtain a first result on the continuity of the propinquity with respect to the local metametrics.
\begin{lemma}\label{prop-implies-prop-lemma}
	If, for some $r \geq 1$,
	\begin{equation*}
	\lim_{n\rightarrow\infty} \metametric{r}((\A_n,\Lip_n,\M_n,\mu_n), (\A_\infty,\Lip_\infty,\M_\infty,\mu_\infty)) = 0
	\end{equation*}
	for some family $(\A_n,\Lip_n,\mu_n)_{n\in\N\cup\{\infty}$ of pointed {\qcms s} and, for each $n\in\N\cup\{\infty\}$, some topography $\M_n$ of $(A_n\Lip_n)$,  and if $\qdiam{\A_\infty,\Lip_\infty} < r$, then
	\begin{equation*}
	\lim_{n\rightarrow\infty} \dppropinquity{2,4r}((\A_n,\Lip_n,\mu_n), (\A_\infty,\Lip_\infty,\mu_\infty)) = 0\text.
	\end{equation*}
\end{lemma}

\begin{proof}
Let $\varepsilon > 0$. By assumption, there exists $N_0 \in \N$ such that, if $n\geq N_0$, then $\qdiam{\A_n,\Lip_n} < r$. Moreover, there exists $N_1 \in \N$ such that if $n\geq N_1$, we can find a tunnel $\tau_n$ from $(\A_n,\Lip_n,\M_n,\mu_n)$ to $(\A_\infty,\Lip_\infty,\M_\infty,\mu_\infty)$ of extent at most $\frac{\varepsilon}{4r^2+9r+3}$. By the previous result, there exists a compact $(2,4r)$-tunnel $\tau'$ from $(\A_n,\Lip_n)$ to $(\A_\infty,\Lip_\infty)$ of pointed extent at most $(4r^2+9r+3)\frac{\varepsilon}{4r^2+9r+3} = \varepsilon$.

Therefore, $(\A_n,\Lip_n,\mu_n)$ converges to $(\A_\infty,\Lip_\infty,\mu_\infty)$ for the pointed propinquity $\dppropinquity{2,4r}$.
\end{proof}

We now can conclude that convergence for the quantum metametric implies convergence for the propinquity, up to a minor relaxing of the Leibniz constraint on compact tunnels.

\begin{corollary}
	If $(\A_n,\Lip_n,\mu_n)_{n\in\N\cup\{\infty\}}$ is a family of pointed {\qcms s}, and if
	\begin{equation*}
		\lim_{n\rightarrow\infty} \Eth((\A_n,\Lip_n,\M_n,\mu_n),(\A_\infty,\Lip_\infty,\M_\infty,\mu_\infty)) = 0\text,
	\end{equation*}
	for some choice of topography $\M_n$ of $(\A_n,\Lip_n)$ for each $n\in\N\cup\{\infty\}$, then, for $r > \qdiam{A_\infty,\Lip_\infty}$, 
	\begin{equation*}
		\lim_{n\rightarrow\infty} \dppropinquity{2,4 r}((\A_n,\Lip_n,\mu_n),(\A_\infty,\Lip_\infty,\mu_\infty)) = 0\text.
	\end{equation*}	
\end{corollary}

\begin{proof}
	First, note that by Definition (\ref{Metametric-def}) of $\Eth$, we have
	\begin{equation*}
		\lim_{n\rightarrow\infty} \qdiam{\A_n,\Lip_n} = \qdiam{\A_\infty,\Lip_\infty} \text.
	\end{equation*}

	Let $D\coloneqq\qdiam{\A_\infty,\Lip_\infty}$. Let $r > D$. By definition of $\Eth$, we conclude 
	\begin{equation*}
		\lim_{n\rightarrow\infty} \metametric{r}((\A_n,\Lip_n,\M_n,\mu_n),(\A,\Lip,\M_\infty,\mu_\infty)) = 0\text.
	\end{equation*}
 Our conclusion follows from Lemma (\ref{prop-implies-prop-lemma}).
\end{proof}

\subsection{The Classical Case}

Gromov introduced in \cite{Gromov81} a distance, up to full isometry, between pointed proper spaces, now known as the Gromov-Hausdorff distance, and of course, the basis for our present construction. When restricted to compact metric spaces, this metric is equivalent to the metric originally introduced by Edwards in \cite{Edwards75}. 

A pointed proper metric space $(X,d_X,x_0)$ is given by a proper (or boundedly compact) metric space $(X,d_X)$ and a point $x_0 \in X$. The closed ball of center $x \in X$ and radius $r \geq 0$ is denoted by $X[x,r] \coloneqq \{ t \in X : d_X(x,t)\leq r \}$.

The Gromov-Hausdorff distance between two pointed proper metric spaces $(X,d_X,x_0)$ and $(Y,d_Y,y_0)$ is defined as:
\begin{equation*}
	\inf\left\{ \varepsilon > 0 \middle\vert \exists j_X : X\rightarrow Z, j_Y : Y\rightarrow Z \text{ isometries st }\begin{array}{l}
		X[x_0, \frac{1}{\varepsilon}] \subseteq_{Z}^\varepsilon Y \\
		Y[y_0, \frac{1}{\varepsilon}] \subseteq_{Z}^\varepsilon X
		\end{array}
		\right\}\text.
\end{equation*}

\begin{lemma}\label{GH-meta-lemma}
	If $(X,d_X,x_0)$ and $(Y,d_Y,y_0)$ are two pointed proper metric spaces, then 
	\begin{equation*}
		\metametric{r}\left((C_0(X),\Lip_{d_X}, C_0(X),\delta_{x_0}),(C_0(Y),\Lip_{d_Y}, C_0(Y),\delta_{y_0}) \right) \leq 2 \GH\left( (X,d_X,x_0), (Y,d_Y,y_0) \right)
	\end{equation*}
	for all $r \in \left[ 1 , \GH( (X,d_X,x_0), (Y,d_Y,y_0) )^{-1} \right]$ if $\GH( (X,d_X,x_0), (Y,d_Y,y_0) )>0$, or simply for all $r \geq 1$ otherwise.
\end{lemma}

\begin{proof}
Let $d$ be a metric on the disjoint union $Z\coloneqq X\coprod Y$, such that the restriction of $d$ to $X\times X$ (respectively $Y\times Y$) is the metric of $X$ (respectively of $Y$). Assume there exists $\varepsilon \in (0,1)$ such that $d(x_0,y_0) \leq \varepsilon$, and 
\begin{equation*}
	X\left[x_0,\frac{1}{\varepsilon}\right] \subseteq_\varepsilon^d Y \text{ and } Y\left[y_0,\frac{1}{\varepsilon}\right] \subseteq_\varepsilon^d X \text.
\end{equation*}
Let $B\coloneqq X\left[x_0,\frac{1}{\varepsilon}\right]\coprod Y\left[y_0,\frac{1}{\varepsilon}\right] \subseteq Z$. We define
\begin{equation*}
	h : z \in \Z \mapsto \max\{1, \varepsilon d(x, Z\setminus B) \} \text,
\end{equation*}
so that $\Lip_d(e)\leq\varepsilon$ and $e \in C_0(Z)$. 

Let $z \in Z$. If $z\notin B$, and if $x \in X$ then $d(x_0,z) > \frac{1}{\varepsilon}$. If $z\notin B$ and if $z\in Y$, then $d(y_0,z) > \frac{1}{\varepsilon}$.

Now, if $z\notin B$ and $z\in Y$, then
\begin{equation*}
	\varepsilon \leq d(x_0,y_0) \leq d(x_0,z) + d(z,y) \leq d(x_0,z) + \frac{1}{\varepsilon} \text. 
\end{equation*}
So $d(x_0,z) \geq \frac{1}{\varepsilon} - \varepsilon$. In conclusion, if $z \notin B$, then $d(x_0,z) \geq\frac{1}{\varepsilon} - \varepsilon$ , so $d(x_0,B) \geq \frac{1}{\varepsilon} - \varepsilon$. Similarly, $d(y_0,z)\geq \frac{1}{\varepsilon} - \varepsilon$.

Consequently, $1 \geq e(x_0) \geq 1 - \varepsilon^2 \geq 1 - \varepsilon$. Similarly, $1\geq e(y_0) \geq 1 - \varepsilon$.

\medskip

We continue to use the notation $\delta_x$ for the character of $C(Z)$ given by evaluation at $x \in Z$. We thus have proven, so far, that
\begin{equation*}
	\max\{ |1 - \delta_{x_0}(e)|, |1 - \delta_{y_0}| \} \leq \varepsilon\text{, }\Kantorovich{\Lip_d}(\delta_{x_0},\delta_{y_0}) \leq \varepsilon\text{ and }\Lip_d(e) \leq \varepsilon \text.
\end{equation*}

Let now $z \in Z$ and $r \in \left[0,\varepsilon^{-1}\right]$. Let $f\in C_0(Z)$, with $\norm{f}{\Lip_d,r} \leq 1$. If $z \in B$, and  if $z \in Y$, in which case trivially, setting $y\coloneqq z$, we note that $|\delta_z( h f h) - \delta_y (h f h)| = 0$. If instead $z \in B$ and $z \in X$, then since $X\left[0,\frac{1}{\varepsilon}\right]\subseteq_\varepsilon^d Y$, there exists $y \in Y$ such that $d(x,y) \leq \varepsilon$. Since $\Lip_d(h f h) \leq 2$, we conclude that $|\delta_x(h f h) - \delta_y(h f h)| = |(hfh)(x) - (hfh)(y)| \leq 2 \varepsilon$. 

If, instead, $z\notin B$, then $h(z) = 0$, so $|\delta_x(hfh) - 0| = 0$.

Thus, 
\begin{equation*}
	\forall x \in Z \quad \exists \delta \in \QuasiStateSpace(C_0(Y)) \quad  \boundedLipschitz{\Lip_d,r}(\delta_x,\delta) \leq 2 \varepsilon \text.
\end{equation*}

By the Krein-Milman theorem, and the convexity of the Fortet-Mourier distances, we conclude that
\begin{equation*}
	\forall \varphi \in \QuasiStateSpace(C_0(X)) \quad \exists \psi \in \QuasiStateSpace(C_0(Y)) \quad \boundedLipschitz{\Lip_d,r}(\varphi,\psi) \leq 2\varepsilon \text.
\end{equation*}

The reasoning is symmetric if we switch $X$ and $Y$. Therefore,
\begin{equation*}
	h\QuasiStateSpace(C_0(Z))h \subseteq_{2\varepsilon}^{\boundedLipschitz{\Lip_d,r}} \QuasiStateSpace(C_0(Y)) \text{ and }h\QuasiStateSpace(C_0(Z))h \subseteq_{2\varepsilon}^{\boundedLipschitz{\Lip_d,r}} \QuasiStateSpace(C_0(X)) \text.
\end{equation*}

 Let $\pi : f \in C_0(Z) \mapsto f_{|X} \in C_0(X)$ and $\rho : f \in C_0(Z) \mapsto f_{|Y} \in C_0(Y)$. Note that if $f \in C_0(X)$ is Lipschitz and real valued, then McShane's extension theorem \cite{McShane34} implies the existence of a Lipschitz function $g : Z\rightarrow\R$ such that $g_{|X} = f$, and $\Lip_d(g) = \Lip_d(f)$. Letting $u \in C_0(Y)$ with $\norm{u}{C_0(Y)} \leq \norm{f}{C_0(X)}$, and $\Lip_{d_Y}(u)\leq\Lip_{d_X}(f)$, and setting 
 \begin{equation*}
 	t : x \in Z\mapsto \begin{cases} f(x) \text{ if $x\in X$, } \\ u(x) \text{ otherwise,} \end{cases}\text,
\end{equation*}
we see that $\pi(\min\{g,t\}) = f$, with $\Lip(\min\{g,t\}) = \Lip_{d_x}(f)$ and $\norm{\min\{g,t\}}{C_0(Z)} = \norm{f}{C_0(X)}$. So $\pi$ is a topographic quantum $M$-isometry for all $M \geq 1$, and similarly, so is $\rho$.

In conclusion, then:
\begin{equation*}
	\tunnel{\tau}{(C_0(X),\Lip_{d_X}, C_0(X),\delta_{x_0})}{\pi}{(C_0(Z),\Lip_d,C_0(Z),h)}{\rho}{(C_0(Y),\Lip_{d_Y}, C_0(Y),\delta_{y_0})}
\end{equation*}
is an $r$-tunnel of extent at most $2\varepsilon$, for all $r \in [0,\varepsilon^{-1}]$.

Consequently, 
\begin{equation*}
	\sup_{0 \leq r \leq \varepsilon^{-1}} \metametric{r}((C_0(X),\Lip_{d_X}, C_0(X),\delta_{x_0}),(C_0(Y),\Lip_{d_Y}, C_0(Y),\delta_{y_0})) \leq 2\varepsilon \text,
\end{equation*}
as claimed.
\end{proof}

We therefore can conclude:
\begin{theorem}
	Let $(X_n,d_n,x_n)_{n\in\N\cup\{\infty\}}$ be a family of pointed proper metric spaces. If $\lim_{n\rightarrow\infty} \GH((X_n,d_n,x_n),(Y,d_Y,y)) = 0$, then
	\begin{equation*}
		\lim_{n\rightarrow\infty} \Eth{}((C_0(X_n),\Lip_{d_n},C_0(X_n), \delta_{x_n}), (C_0(X_\infty), \Lip_{d_\infty}, C_0(X_\infty), \delta_{x_\infty})) = 0 \text.
	\end{equation*}
\end{theorem}

\begin{proof}
	This follows immediately from Lemma (\ref{GH-meta-lemma}).
\end{proof}



\section{A noncompact noncommutative example}

We now apply our theory to prove that certain C*-crossed-products obtained from letting subgroups of $\Z$ act on $\Z$ by translation have finite dimensional metric approximations, when the metrics used are induced by the action and the standard  metric on $\Z$ in a natural manner. These examples provide finite dimensional approximations of finite sums of the algebra of compact operators by full matrix algebras in a \emph{metric} sense. Our purpose in this section is not strive for the greatest generality, but rather to present a detailed accessible example which captures various aspects of our theory, while being neither a {\qcms} not a commutative example.

\medskip

We wish to point out some of the relevant aspects this example will illustrate.
\begin{enumerate}
	\item The algebra of compact operators is simple, thus as far from a commutative example as possible.
	\item The algebra of compact operators is, of course, not unital, thus the theory for {\qcms s} can not be applied to it.
	\item Our metric on the algebra of compact operators gives rise to a {\pqms} which is not compact, and moreover, does not extend to some {\qcms} structure on its unitlization. In fact, the radius of the state space for our metric will be infinite. So this example is very much not a compact quantum space, though we will prove that it is the limit of {\qcms s} for the metametric.
	\item The {\MongeKant} induced by our chosen L-seminorm does not metrize the weak* topology of the state space, and in fact, is not a metric as it takes the value $\infty$ between some pairs of states. Thus, we encounter in this example the core difficulties raised by leaving the {\qcms} realm.
	\item Our metric is constructed from a C*-crossed-product presentation of the algebra of compact operator algebras, using a very natural construction found for {\qcms s} but never extended to the locally compact situation. It is not an ad-hoc metric built for illustration purposes but a very natural one arising from a dynamical system.
	\item The same exact finite dimensional algebras used as approximations of the algebra of compact operators are also used to approximate quantum tori, though of course with different metrics. However, the metric we use here and the metric used for quantum tori approximations (called fuzzy tori, then) are very similar in some ways, with only one important change. This is interesting as well, as an illustration of the role of the metric.
	\item When $\Z$ acts properly and freely on a space, the C*-crossed-product will appear as some space of sections of some bundle over the orbit space, where each fiber is basically of the form $c_0(\Z)\rtimes\Z$ for the action by translation of $\Z$ on itself. Thus, this example will re-appear in such situations, among others.
	\item Commutativity can appear at the limit, as in particular, we can approximate direct sums of copies of the algebra of compact operators by full matrix algebras.
\end{enumerate}

\medskip

As a matter of style, since it is concerned with a single family of examples, this section is presented as a single narrative, with the notations introduced along the exposition remaining valid until the end of the section.

\subsection{Preliminary: $c_0(\Z)$}

Let $c_0(\Z)$ be the C*-algrebra of $\C$-valued functions over $\Z$ vanishing at infinity, i.e. functions $f : \Z \rightarrow \C$ such that for all $\varepsilon > 0$, there exists a finite subset $F\subseteq\Z$ such that $|f(n)|<\varepsilon$ for all $n\in\Z\setminus F$.

It will prove helpful, later on, to invoke the larger C*-algebra $c_1(\Z)$ of functions $f:\Z\rightarrow\C$ over $\Z$ such that $f - l \in c_0(\Z)$ for some $l \in \R$, that is, the space of continuous functions over the one-point compactification of $\Z$. Note that $c_0(\Z)\subseteq c_1(\Z)$; in fact, $c_1(\Z)$ is the smallest unitization of $c_0(\Z)$.

The usual metric on $\Z$ induces a natural ``discrete calculus'' on $c_1(\Z)$ --- and most importantly for us, on $c_0(\Z)$. Let $\pi$ be the faithful non-degenerate *-representation of $c_1(\Z)$ on $\ell^2(\Z)$ by multiplication operator, i.e. for all $f \in c_1(\Z)$ and $\xi \coloneqq (\xi_n)_{n\in\Z}$:
\begin{equation*}
	\pi(f)\xi = \left( f(n)\xi_n \right)_{n\in\Z}\text.
\end{equation*}
We henceforth will also write $\pi$ for the restriction of $\pi$ to $c_0(\Z)$.

We also define the unitary $U$ on $\ell^2(\Z)$ to be the bilateral shift, by setting for all $(\xi_n)_{n\in\Z} \in \ell^2(\Z)$:
\begin{equation*}
	U\xi = (\xi_{n-1})_{n\in\Z} \text.
\end{equation*} 
We then define $\partial_Z$ simply as the linear endomorphism defined on the C*-algebra $\B$ of all bounded operator on $\ell^2(\Z)$ by: 
\begin{equation*}
	\partial_\Z \coloneqq a \in \B \mapsto [ U, a ] \text.
\end{equation*}
In particular, for all $f \in c_1(\Z)$, we compute:
\begin{equation*}
	\partial_\Z \pi(f) = \pi(n \in \Z \mapsto f(n+1) - f(n))  U \text.
\end{equation*}

We also set $\Lip_\Z : f \in c_1(\Z) \mapsto \norm{\partial_\Z\pi(f)}{\B}$ (which is a norm on $c_0(\Z)$). We note that $\partial_\Z$ is a derivation of $\B$, and thus, for all $f,g \in c_1(\Z)$,
\begin{equation*}
	\Lip_\Z (fg)  \leq \norm{f}{c_0(\Z)\rtimes_\alpha\Z} \Lip_\Z(g) + \Lip_\Z(f) \norm{g}{c_0(\Z)\rtimes_\alpha\Z} \text.
\end{equation*}

Henceforth, we shall identity $c_1(\Z)$ and $c_0(\Z)$ with, respectively, $\pi(c_1(\Z))$ and $\pi(c_0(\Z))$, as $\pi$ is faithful, and drop the notation $\pi$ altogether.

\medskip

In particular, note that if $\delta_p$ is the indicator function of $\{ p \}$ for all $p\in\Z$, seen as a character of $c_0(\Z)$, and if $f \in c_0(\Z)$ and $\Lip_\Z (f) \leq 1$, then for all $p < q \in \Z$,
\begin{equation*}
	| \delta_p(f) - \delta_q(f) |
	\leq |f(p) - f(q)| \leq \sum_{j=p}^{q-1} \underbracket[1pt]{|f(j+1)-f(j)|}_{\leq 1} = q-p = |p-q| \text.
\end{equation*}
Moreover, let $f_{p,q}(n) \coloneqq \begin{cases}
	n-p \text{ if $p \leq n \leq q$}, \\
	0 \text{ otherwise,}
\end{cases}$ 
for all $n\in\Z$, so $f \in c_0(\Z)$, then $\Lip_\Z(f) = 1$, and of course, for all $p,q\in\Z$:
\begin{equation*}
	|\delta_p( f_{p,q} ) - \delta_q( f_{p,q} ) | = |p - q| \text.
\end{equation*} 
So, in conclusion, if $\Kantorovich{\Lip_\Z}$ is the {\MongeKant} induced by $\Lip_\Z$ on $\StateSpace(c_0(\Z))$, then, for all $p,q \in \Z$, identifying $\Z$ with the space of characters of $c_0(\Z)$ via $p \in \Z \mapsto \delta_p$, we have shown that:
\begin{equation*}
	\Kantorovich{\Lip_\Z}(\delta_p,\delta_q) = |p-q| \text.
\end{equation*}
Note in particular, that $\qdiam{c_0(\Z),\Lip_\Z} = \infty$. We also note that $\Lip_\Z$, while defined on $c_1(\Z)$, is not an L-seminorm, since $\Kantorovich{\Lip}$ gives infinite radius to $\StateSpace(c_1(\Z))$.

\medskip

It is, on the other hand, straightforward that $(c_0(\Z),\Lip_\Z, c_0(\Z),\delta_0)$ is a {\pqpms}. Indeed, let $h : n \in \Z \mapsto \frac{1}{|n|+1} \in c_0(\Z)$ --- of course, $h > 0$. Then it is immediate that the set
\begin{equation*}
	S\coloneqq \left\{ h f h : f \in c_0(\Z), \Lip_\Z(f) \leq 1, \norm{f}{c_0(\Z)} \leq 1 \right\}
\end{equation*}
is compact. Indeed, $S$ is closed (note that $\Lip_\Z$ is continuous with respect to $\norm{\cdot}{c_0(\Z)}$ here). For all $\varepsilon > 0$, choosing $N \in \N$ so that if $n\geq N$, then $\frac{1}{n^2} < \frac{\varepsilon}{2}$, then $|h f h (n)| = \frac{1}{n^2} |f(n)| < \frac{\varepsilon}{2}$. Let 
\begin{equation*}
	B_n \coloneqq \{ g \in c_0(\Z) : \norm{g}{c_0(\Z)}\leq 1\text{ and }\forall n\geq N \quad g(n) = 0 \}\text,
\end{equation*}
so that $\Haus{c_0(\Z)}(S,B_n) < \frac{\varepsilon}{2}$.

Now, $B_n$ is compact, hence totally bounded, since it is homeomorphic to the closed unit ball of the finite dimensional space $\C^n$. So there exists a finite set $F \subseteq B_n$ such that $\Haus{c_0(\Z)}(B_n, F) < \frac{\varepsilon}{2}$. Consequently, $\Haus{c_0(\Z)}(S,F) < \varepsilon$. Thus $S$ is totally bounded, since $\varepsilon > 0$ was arbitrary.

Moreover, if $\delta_0(f) = 0$ for $f \in c_1(\Z)$, and $\Lip_\Z(f) \leq 1$, then $|f(n)| = |f(n) - f(0)| \leq |n-0| = n$, and thus $|h f h (n)| \leq \frac{1}{|n|}$. So $\{ h f h : \Lip_\Z(f) \leq 1, \mu(f) = 0 \}$ is bounded. Hence, $(c_0(\Z),\Lip_\Z,\delta_0)$ is a ponted {\lcqms}.

Moreover, if we consider the sequence:
\begin{equation}\label{example-fk-eq}
	f_k : n \in \Z \mapsto
	\begin{cases}
		\frac{n + 1 - |k|}{|n|+1} \text{ if $|k| \leq n+1$, }\\
		0 \text{ otherwise.}
	\end{cases}
\end{equation}
then we note that $(f_k)_{k\in\N} \in c_0(\Z)^\N$, that $\norm{f_k}{c_0(\Z)} = 1$ while $\Lip_\Z(p_k) = \frac{1}{k}$ for all $k\in \N$, so $\lim_{n\rightarrow\infty} \Lip_\Z(f_k) = 0$. Moreover $\delta_0(f_k)=1$ for all $k\in\N$.  Thus, $(c_0(\Z),\Lip_\Z,c_0(\Z),\delta_0)$ is a {\pqpms}.

\subsection{$c_0(\Z)\rtimes_\alpha \Z$ as a {\pqpms}}\label{Zcross-sub}

Fix $t \in \N\setminus\{0\}$ henceforth. Let now $\alpha$ be the action of $\Z$ on $c_0(\Z)$ induced by the action of $t\Z$ on $\Z$ by translation, so for all $f\in c_0(\Z)$, and for all $m\in\Z$,
\begin{equation*}
	\alpha^m( f ) : n\in\Z\mapsto f(n - t m) \text.
\end{equation*}
Note that $\alpha$ acts by isometry on $c_0(\Z)$, more formally
\begin{equation*}
	\forall f\in c_0(\Z) \quad \forall m \in \Z \quad \partial_\Z \alpha^m(f) = \alpha^m(\partial_\Z f)
\end{equation*}
so $\Kantorovich{\Lip_\Z}(\varphi\circ\alpha^p, \psi\circ\alpha^p) = \Kantorovich{\Lip_\Z}(\varphi, \psi)$ and $\boundedLipschitz{\Lip_\Z}(\varphi\circ\alpha^p, \psi\circ\alpha^p) = \boundedLipschitz{\Lip_\Z}(\varphi,\psi)$ for all $\varphi,\psi \in \StateSpace(c_0(\Z))$ and $p\in\Z$.

\medskip

Now, a direct computation shows that, if we set 
\begin{equation*}
	u\coloneqq U^t \text,
\end{equation*}
then for all $f\in c_1(\Z)$,
\begin{align*}
	u f u^\ast \xi 
	&= u f (\xi_{n+t})_{n\in\Z}	\\
	&= u (f(n)\xi_{n+t})_{n\in\Z} \\
	&= (f(n-t)\xi_n)_{n\in\Z} = \alpha^1(f)\xi \text.
\end{align*}
Therefore, $(\pi,u)$ is a covariant pair for the dynamical system $(c_1(\Z),\alpha,\Z)$, so it gives a *-representation of $c_1(\Z)\rtimes\Z$. In fact, the C*-algebra generated by $c_0(\Z)$ and $u$ is *-isomorphic to the C*-crossed-product $c_0(\Z)\rtimes \Z$, and we identify them henceforth as well. By our construction, the set $\{ f u^j : f \in c_0(\Z), j \in \Z \}$ is total in $c_0(\Z)\rtimes_\alpha\Z$. 

Note that $\partial_Z$ acts on $c_0(\Z)\rtimes_\alpha\Z$, and in particular,
\begin{equation*}
	\partial_\Z( f u^j ) = \partial_\Z(f) u^j \text{ since $U$ commutes with $u=U^t$.}
\end{equation*}

\medskip

The dual action of $\T$ on $c_0(\Z)\rtimes_\alpha\Z$ is spatially implemented as follows. For each $g \in \T$,  $\xi=(\xi_n)_{n\in\Z}$, and any bounded operator $a$ on $\ell^2(\Z)$, we set
\begin{equation*}
	v^g \xi = (g^{n_t} \xi_n)_{n\in\Z} \text{ and }\beta^g (a) = v^g a (v^{g})^\ast \text,
\end{equation*}
where $n_t$ is uniquely defined as the quotient of $n$ by $t$ for the Euclidean division in $\Z$. Since $v^g$ is a unitary for all $g \in \T$, the map $\beta^g$ is a *-automorphism of $\B$ and $g\in\T\mapsto \beta^g$ is a group morphism.

We check that, in particular, $\beta^g(f) = f$ for all $g \in \T$ and $f \in c_1(\Z)$, since $v^g$ commutes with $c_1(\Z)$ for all $g \in \T$, while $\beta^g(u) = g u$. Thus $\beta^g(c_0(\Z)\rtimes_\alpha\Z) = c_0(\Z) \rtimes_\alpha\Z$. It is of course the so-called dual action of $\T$ on $c_0(\Z)\rtimes_\alpha\Z$.

\medskip

 There exists a dense *-subalgebra $\dom{\partial_\T}$ of $\B$ such that the following limit exists for $a\in\B$ if, and only if, $a\in\dom{\partial_\T}$, allowing the definition of the unbounded operator $\partial_\T$:
 \begin{equation*}
 	\partial_\T (a) \coloneqq \lim_{t\rightarrow 0} \frac{1}{t}\left(\beta^{\exp(it)}(a) - a\right) \text.
 \end{equation*}
In particular, $\partial_\T(f) = f$ for all $f \in c_1(\Z)$, and $\partial_\T(u) = u$. Since $\partial_T$ is a derivation, we conclude that
\begin{equation*}
	\forall f \in c_0(\Z) \quad \forall j \in \Z \quad \partial_\T (f U^j) = j f U^j \text.
\end{equation*}
In particular, the linear span of $\{ f U^j : f \in c_0(\Z), j \in \Z \}$, which is dense in $c_0(\Z)\rtimes_\alpha\Z$ by construction, is indeed a subspace of $\dom{\partial_\T}$.

\medskip

We now put $\partial_\Z$ and $\partial_\T$ together. Let $M_2$ be the C*-algebra of $2\times 2$ matrices over $\C$. Let $\gamma_1, \gamma_2 \in M_2$ be two $2\times 2$ self-adjoint anticommuting unitary matrices, and then define the following operator:
\begin{equation*}
	\gradiant \coloneqq \partial_\Z \otimes \gamma_1 + \partial_\T \otimes \gamma_2 
\end{equation*}
on $\dom{\partial_\T}$ with values in $\B \otimes M_2$.

Let 
\begin{equation*}
	\dom{\Lip} \coloneqq c_0(\Z)\rtimes_\alpha\Z \cap \dom{\partial_\T}
\end{equation*}
and, for all $a\in\dom{\Lip}$,
\begin{equation*}
	\Lip(a) \coloneqq \norm{\gradiant(a)}{\B\otimes M_2}\text.
\end{equation*}

Again, $\dom{\Lip}$ contains the dense *-subalgebra of all linear combinations of $\left\{ f U^j : f \in c_0(\Z), j \in \Z \right\}$, and thus is dense. In fact, it is a dense *-subalgebra of $c_0(\Z)\rtimes_\alpha\Z$.

\medskip

We record, for later use, that for all $a\in \dom{\Lip}$,
\begin{align*}
	\Lip_\Z(a) 
	&= \norm{\partial_\Z a}{\B} \\
	&= \norm{\frac{1}{2}\left(\gamma_1 \gradiant(a) + \gradiant(a) \gamma_1\right)}{\B\otimes M_2} \\
	&\leq \norm{\gradiant(a)}{\B} = \Lip(a) \text. 
\end{align*}
Similarly,
\begin{equation*}
	\norm{\partial_\T a}{\B} \leq \Lip(a) \text.
\end{equation*}

Moreover, let $a,b\in \dom{\Lip}$. It is immediate to check that $\gradiant{}$ is a derivation (since both $\partial_\Z$ and $\partial_\T$ are), so:
\begin{align*}
	\Lip(a b)
	&=\norm{\gradiant(ab)}{\B\otimes M_2} \\
	&=\norm{(a\otimes 1)\gradiant(b) + \gradiant(a)(b\otimes 1)}{\B} \\
	&\leq \norm{a}{\B} \norm{\gradiant(b)}{\B\otimes M_2} + \norm{b}{\B}\norm{\gradiant(a)}{\B\otimes M_2} \\
	&\leq \norm{a}{\B} \Lip(b) + \norm{b}{\B} \Lip(a) \text.
\end{align*}

\begin{remark}
We pause to note that, if we set $D \coloneqq \{ (\xi_n)_{n\in\Z} \in \ell^2(\Z) : (n \xi_n)_{n\in\Z} \in \ell^2(\Z) \}$, and we define on $D\otimes \C^2 \subseteq \ell^2(\Z)\otimes \C^2$ the operator:
\begin{equation*}
	\Dirac \coloneqq U\otimes\gamma_1 + \nabla \otimes\gamma_2
\end{equation*}
where $\nabla(\xi_n)_{n\in\Z} = (n_t \xi_n)_{n\in\Z}$ for all $(\xi_n)_{n\in\Z} \in D$, then a simple computation shows that:
\begin{equation*}
	\gradiant(a) = [\Dirac,a\otimes 1] \text{ so }\Lip(a) = \norm{[\Dirac,a\otimes 1]}{\B\otimes M_2}\text.
\end{equation*}

The domain of $\Dirac$ is densely-defined in $\ell^2(\Z)\otimes\C^2$, and thus our construction can be seen as akin to some spectral triple. However, we note that $\Dirac$ is not self-adjoint (though this could be easily fixed), and this triple would be defined for a non-unital C*-algebras. In any case, this remark provides additional reasons to believe our construction of a quantized calculus on $c_0(\Z)\rtimes_\alpha\Z$ above is natural, by connecting it with the usual construction for a spectral-triple on a C*-crossed-product by $\Z$ from a spectral triple on its base C*-algebra by a isometric action. In our case, if we look at $(c_0,\ell^2(\Z),U)$ as a spectral-triple like structure over $c_0(\Z)$ (again, with $U$ not self-adjoint, which again is a not a concern for our present purpose), whose Connes' metric, as we checked above, gives the usual metric over $\Z$, then this section introduces the induced spectral triple analogue on $c_0(\Z)\rtimes_\alpha \Z$. We will not further need this observation in this paper, however, as our focus is solely on the metric aspects of the underlying geometry of $c_0(\Z)\rtimes_\alpha\Z$.
\end{remark}

\medskip

We now wish to prove that $(c_0(\Z)\rtimes_\alpha \Z, \Lip)$ is a {\pqms}, with topography $c_0(\Z)$, and we also wish to fix some base state.

Again, we let $h : n \in \Z \mapsto \frac{1}{|n|+1}$, then $h \in c_0(\Z)\subseteq c_0(\Z)\rtimes \Z$, with $h > 0$. Moreover, $C^\ast(h) = c_0(\Z)$. Our first goal is to prove that $\{ h a h : \norm{a}{\Lip}\leq 1 \}$ is compact in $c_0(\Z)\rtimes_\alpha\Z$.

To this end, we shall use the following standard construction, which will also play a role in the proof of our convergence result. Fix the unique Haar probabily measure $\lambda$ on $\T$. We of course identify the Pontryagin dual $\widehat{T}$ of $\T$ with $\Z$ via $j \in \Z \mapsto (g\in\T\mapsto g^j) \in \widehat{T}$. For all $j \in \Z$, a trivial computation shows that
\begin{equation*}
	\int_\T g^j \beta^g(f u^k) \, d\lambda(g) = \delta_j^k f u^j \text.
\end{equation*}

In particular, $\mathds{E} : a\in c_0(\Z)\rtimes_\alpha \Z \mapsto \int_\T \alpha^g(a) \, d\lambda(g)$ is a conditional expectation onto $c_0(\Z)$. We thus set
\begin{equation*}
	\mu \coloneqq \delta_0 \circ \mathds{E}
\end{equation*}
and note that $\mu$ is a state of $c_0(\Z)\rtimes_\alpha\Z$.

\medskip

Let now $\varepsilon > 0$. As is well-known (see, e.g., \cite{Latremoliere05} for a construction), there exists a function $\varphi : \T \rightarrow [0,\infty)$ such that:
\begin{enumerate}
	\item $\varphi_\varepsilon$ is a linear combination of characters of $\T$, i.e. there exists a finite set $S_\varepsilon \subseteq \Z$ such that $\varphi_\varepsilon : g \in \T\mapsto \sum_{j \in S_\varepsilon} c_j g^j$ for some constants $(c_j)_{j \in S_\varepsilon}$,
	\item $\int_\T\varphi \,d\lambda = 1$,
	\item $\int_\T \varphi(g)|1-g|\,d\lambda(g) < \varepsilon$.
\end{enumerate}
By integration, we then define the following operator on $c_0(\Z)\rtimes_\alpha\Z$:
\begin{equation*}
	\beta^{\varphi_\varepsilon} : a\in c_0(\Z)\rtimes_\alpha\Z \mapsto \int_\T \varphi^\varepsilon(g) \beta^g(a) \, d\lambda(g) \text.
\end{equation*}
Of course, $\beta^{\varphi_\varepsilon}$ is a linear endomorphism of $c_0(\Z)\rtimes_\alpha\Z$ with 
\begin{equation*}
	\opnorm{\beta^{\varphi_\varepsilon}}{}{c_0(\Z)\rtimes_\alpha\Z} \leq \norm{\varphi_\varepsilon}{L^1(\T,\lambda)} = 1\text.
\end{equation*}

Let now $a \in \dom{\Lip}$ with $\Lip(f) \leq 1$. Then note that
\begin{equation*}
	\norm{a - \beta^{\exp(it)}(a)}{c_0(\Z)\rtimes_\alpha\Z} \leq \int_0^t \norm{\frac{d}{dt}\beta^{\exp(is)}(a)} \, ds \leq |1-\exp(it)| \text.
\end{equation*}
Therefore, by construction,
\begin{align*}
	\norm{a - \beta^{\varphi_\varepsilon}(a)}{} \leq \int_\T \varphi_\varepsilon(g) \norm{\beta^g(a)-a}{} \, d\lambda(g) \leq \int_\T \varphi^\varepsilon(g) |g-1| \, d\lambda(g) < \varepsilon \text. 
\end{align*}
Thus for all $a\in \dom{\Lip}$,
\begin{equation*}
	\norm{a-\beta^{\varphi_\varepsilon}(a)}{} \leq \varepsilon\Lip(a) \text.
\end{equation*}

Therefore, for all $a \in \dom{\Lip}$ with $\norm{a}{\Lip} \leq 1$, and for all $\varepsilon > 0$, we have shown that
\begin{equation*}
	\Haus{\B}(\{ h a h : a\in\dom{\Lip},\norm{a}{\Lip}\leq 1\}, \{ h \beta^{\varphi_\varepsilon}(a) h : a\in\dom{\Lip},\norm{a}{\Lip}\leq 1 \}) < \varepsilon \text.
\end{equation*}

On the other hand, $\beta^{\varphi_\varepsilon}(a)$ lies in the linear span of $\{ f u^k : f \in c_0(\Z), k \in S_\varepsilon \} \text.$

Moreover, $\Lip(\beta^{\varphi_\varepsilon}(a)) \leq \Lip(a)$. 

Let $b = f u^j$ for $j \in S_\varepsilon$ and $f \in c_0(\Z)$. A simple computation shows that
\begin{equation*}
	h ( f u^j ) h = (h f \alpha^{-j}(h))  u^j \text.
\end{equation*}

So 
\begin{equation*}
\{ h \beta^{\varphi_\varepsilon}(a) h : a\in\dom{\Lip},\norm{a}{\Lip}\leq 1 \} \subseteq \left\{ \sum_{j \in S_\varepsilon} (h f \alpha^{-j} (h)) u^j : \norm{f}{\Lip} \leq 1 \right\} \text.
\end{equation*}

With a similar reasoning as in the previous subsection, the set $\left\{ h f \alpha^{-j} (h) : \norm{f}{c_0(\Z)} \leq 1 \right\}$ is compact, therefore 
\begin{equation*}
	\left\{ \sum_{j \in S_\varepsilon} (h f \alpha^{-j} (h)) u^j : \norm{f}{\Lip} \leq 1 \right\}
\end{equation*} 
is as well, since $S_\varepsilon$ is finite. Hence $\{ h \beta^{\varphi_\varepsilon}(a) h : a\in\dom{\Lip},\norm{a}{\Lip}\leq 1 \}$ is totally bounded in $\B$. So for our $\varepsilon > 0$, there exists a $\varepsilon$-dense subset $F$ of $\{ h \beta^{\varphi_\varepsilon}(a) h : a\in\dom{\Lip},\norm{a}{\Lip}\leq 1 \}$, which in turn is a $2\varepsilon$-dense subset of $\{ h a h : a\in\dom{\Lip}, \norm{a}{\Lip} \leq 1\}$.

Therefore, $\{ h a h : a\in\dom{\Lip}, \norm{a}{\Lip} \leq 1\}$ is totally bounded. It is easy to check that this set is also closed, since $\norm{\cdot}{\Lip}$ extends to a lower semicontinuous function over $\B$ by setting it to $\infty$ whenever $a\notin  \dom{\Lip}$.

 As $c_0(\Z)\rtimes\Z$ is complete, $\{ h a h : a\in\dom{\Lip}, \norm{a}{\Lip} \leq 1\}$ is indeed compact, as required.

\medskip

Let $a \in \dom{\Lip}\oplus\C\unit$ with $\Lip(a) \leq 1$ and $\mu(a) = 0$. As above, 
\begin{equation*}
	\norm{a - \beta^{\varphi_1}(a)}{\B} \leq 1 \text,
\end{equation*}
and $\beta^{\varphi_1}(a) = \sum_{j \in S_1} c_j u^j$ for some $(c_j)_{j\in S_1} \in C^{S_1}$, with $\norm{\beta^{\varphi_1(a)}}{\Lip} \leq 1$.

Now, as above, $h (\sum_{j \in S_1} c_j u^j ) h = \sum_{j \in S_1} (h \alpha^{-j}(h) f_j) u^j$ so 
\begin{equation*}
	\norm{\beta^{\varphi_1}(a)}{} = \norm{h (\sum_{j \in S_1} f_j u^j ) h}{\B} \leq \sum_{j \in S_1} \norm{h \alpha^{-j}(h) f}{\B}\text. 
\end{equation*}

First, let $j \in S_1\setminus\{0\}$. Then 
\begin{equation*}
	\norm{j f u^j}{\B} = \norm{\mathds{E}(\partial_\T(\alpha^{\varphi_1}))}{\B} \leq \norm{\partial_\T(\alpha^{\varphi_1})}{\B} \leq \norm{\partial_\T(a)}{\B} \leq \Lip(a) \leq 1 \text.
\end{equation*}
Therefore, $\norm{f_j}{c_0(\Z)} \leq \frac{1}{j} \leq 1$.

On the other hand, $\norm{\partial_\Z(f_0)}{\B} \leq\Lip(f) \leq 1$, and by assumption, $f_0(0) = 0$. So for all $n\in\Z$, we conclude $|f(n)| \leq |n|$. Thus $|f_0 h(n)| \leq \frac{n}{2^n} \leq 1$.

Consequently, $\norm{\beta^{\varphi_1}(a)}{} \leq |S_1|$, and thus $\norm{a}{} \leq |S_1| + 1$. So
\begin{equation*}
	\left\{ h a h : a\in\dom{\Lip}, \Lip(a) \leq 1, \mu(a) = 0 \right\}
\end{equation*}
is bounded.

\medskip

The sequence $(f_k)_{k\in\N}$ in Expression \eqref{example-fk-eq} allow us, once more, to conclude that we have shown that $(c_0(\Z)\rtimes_\alpha\Z, c_0(\Z), \Lip,\mu)$ is a {\pqpms}.

\subsection{Approximations}

We now turn to finding natural, finite dimensional approximations for $c_0(\Z)\rtimes_\alpha\Z$. Let $\bigslant{\Z}{n}$ the cyclic group of order $n$, and let $c\left(\bigslant{\Z}{n}\right)$ be the (finite dimensional) C*-algebra of all $\C$-valued functions over $\bigslant{\Z}{n}$.

We propose to work with the C*-crossed-product 
\begin{equation*}
	\A_p \coloneqq  c\left(\bigslant{\Z}{p}\right)\rtimes_{\alpha_p} \bigslant{\Z}{p}
\end{equation*}
with $\alpha_p^z (f) : m \in \bigslant{\Z}{p} \mapsto f(m - [t]z)$, where $[t]$ denote the equivalence class of $t$ modulo $p$ in $\bigslant{\Z}{p}$.

As above, we identify $\A_p$ with the C*-algebra generated by $\pi_p(c(\bigslant{\Z}{p}))$ and the unitary $u_p$ on $\ell^2\left(\bigslant{\Z}{p}\right)$, where for all $f \in c(\bigslant{\Z}{p})$ and $(\xi_p)_{p\in\bigslant{\Z}{p}}$: 
\begin{equation*}
	\pi_p(f)(\xi_p)_{n\in\bigslant{\Z}{p}} = (f(p)\xi_p)_{p\in\bigslant{\Z}{p}}
\end{equation*}
and
\begin{equation*}
	u_p \coloneqq U_p^{[t]} \text{ where }U_p (\xi_m)_{m\in\bigslant{\Z}{p}} = (\xi_{m-1})_{m\in\bigslant{\Z}{p}} \text.
\end{equation*}

For each $n\in\N$, we denote the dual action of $\U_p \coloneqq \{ z\in \C : z^p = 1\}$ on $\A_p$ by $\beta_p$. As before, for any $\varphi \in L^1(\U_p)$, and denoting the unique Haar probability measure of $\U_p$ by $\lambda_p$, we also define $\beta_p^\varphi : a\in \A_p \mapsto \int_{\U_p} \varphi(g) \beta_p^g(a) \, d\lambda_p(g)$.

For each $p\in\N$, and for each $a\in \A_p$, we define
\begin{equation*}
	\partial_{\bigslant{\Z}{p}}(a) \coloneqq [u_p, a]
\end{equation*}
and
\begin{equation*}
	\partial_{\U_p}(a) \coloneqq [w_p, a]
\end{equation*}
where
\begin{equation*}
	w_p (\xi_{n\in\bigslant{\Z}{p}}) = \left( \frac{p}{2i\pi}\left( \exp\left(\frac{2 i \pi n}{p}\right) - 1\right)\xi(n) \right)_{n\in\bigslant{\Z}{p}} \text,
\end{equation*}
which is well defined since, for any two $n,m \in \Z$ with $n\equiv m \mod p$, then $\exp\left(\frac{2 i \pi n}{p}\right) = \exp\left(\frac{2 i \pi m}{p}\right)$. 
Note that $w_p$ commutes with $c\left(\bigslant{\Z}{p}\right)$, and thus $\partial_{\U_p}( f ) = f$ for all $f \in c\left(\bigslant{\Z}{p}\right)$.

We define, for all $a\in \A_p$,
\begin{equation*}
	\Lip_p (a) = \norm{ \partial_{\bigslant{\Z}{p}}(a) \otimes \gamma_1 + \partial_{\U_p}(a) \otimes \gamma_2 }{\A_p\otimes M_2} \text.
\end{equation*}

It is immediate that $(\A_p,\Lip_p)$ is a finite dimensional {\qcms}. Nonetheless, as we intend to employ the metametric for {\pqpms s}, we still need some additional structure.  For our choice of a topography of $\A_p$, we naturally choose $\alg{M}_p \coloneqq c\left(\bigslant{\Z}{p}\right)$. We also define
\begin{equation*}
	\mu_p : a\in \A_p \mapsto \delta_0\left(\int_{\U_p} \beta^g(a) \, d\lambda_p(g)\right)
\end{equation*}
so that $\mu_p \in \StateSpace(\A_p)$. It is now immediate that $(\A_p,\Lip,\M_p,\mu_p)$ is a {\pqpms} for all $p\in\N$ --- note, in particular, that $\A_p$ is unital and its unit has L-seminorm $0$.

\medskip

We are now ready to discuss the convergence of $(\A_p,\Lip_p,\M_p,\mu_p)_{p\in\N}$ to $(\A,\Lip,\M,\mu)$ for the metametric, where $\A\coloneqq (c_0(\Z)\rtimes_\alpha \Z$ and $\M = c_0(\Z)$. To this end, we will construct appropriate tunnels between these spaces.

\medskip

Let $\varepsilon > 0$. Our goal is to find $N \in \N$ such that, if $p\geq N$, there exists an $r$-tunnel $\tau_n$ from $(\A_p,\Lip_p,\M_p,\mu_p)$ to $(c_0(\Z)\rtimes_\alpha\Z, \Lip,c_0(\Z),\mu)$ of extent at most $\varepsilon$ for each $r\in \left[1,\frac{1}{\varepsilon}\right]$. To begin with, we shall henceforth use the following \emph{degenerate} representation of $\A_p$. Let
\begin{equation*}
	I_p \coloneqq \left\{ n \in \Z : \lceil{\frac{p}{2}}\rceil \leq n \leq \lceil \frac{p}{2} \rceil \right\} \text. 
\end{equation*}

If $\xi \coloneqq (\xi_n)_{n\in\bigslant{\Z}{p}} \in \ell^2(\bigslant{\Z}{p})$, we set
\begin{equation*}
	J\xi : n \in \Z\mapsto
	\begin{cases}
		\xi_{[n]} \text{ if $n \in I_p$, }\\
		0 \text{ otherwise. }
	\end{cases}
\end{equation*}
By construction, $J$ is an isometry from $\ell^2(\bigslant{\Z}{p})$ into $\ell^2(\Z)$. The restriction of $a \in \B(\ell^2(\bigslant{\Z}{p})) \mapsto J a J^\ast$ to $\A_p$ is a degenerate *-representation $\rho_p$ of $\A_p$ on $\ell^2(\Z)$. Note that $\rho_n(\A_p) \subseteq \A$, and $\rho_p$ is faithful. We now identify $\A_p$ with $\rho_p(\A_p)$ henceforth, and we also identify $\partial_{\bigslant{\Z}{p}}$ and $\partial_{\U_p}$ with $J\partial_{\bigslant{\Z}{p}}J^\ast$ and $J\partial_{\U_p} J^\ast$, respectively.

\medskip

We proceed as follows. 

Let $N_1 \in \N$ such that if $n \geq P_0$, then $\frac{1}{n} < \frac{\varepsilon}{12}$.

Let $\psi \coloneqq \varphi^{\frac{\varepsilon}{24}}$, using the notation in Subsection (\ref{Zcross-sub}). Let 
\begin{equation*}
	N_2 \coloneqq \max\left\{ N_1, |n| : n \in S_{\frac{\varepsilon}{24}} \right\} + 1 \text.
\end{equation*}

By \cite[Lemma 3.6]{Latremoliere05}, there exists $N_3 \in \N$ such that, if $p\geq N_3$, then
\begin{equation*}
	\int_{\U_p} \psi(g) |g-1| \, d\lambda_p(g) < \frac{\varepsilon}{12} \text.
\end{equation*}

Let now, for all $n\in\N\setminus\{0\}$, define:
\begin{equation*}
	h_n : z \in \Z \mapsto \begin{cases}
		1 \text{ if $|z|\leq n$, } \\
		1 - \frac{\varepsilon(z-n)}{n t} \text{ if $n\leq |z| \leq n(t\varepsilon^{-1}+1)$, } \\
		0 \text{ otherwise. }
	\end{cases}
\end{equation*}
By construction, $h_n \in \sa{c_0(\Z)} \subseteq c_0(\Z)\rtimes_\alpha\Z$, and $\partial_\Z(h_n) \leq \frac{\varepsilon}{n}$ while $\partial_{\T}(h_n) = 0$. So $\Lip(h_n) \leq \frac{1}{n}$ for all $n\in\N\setminus\{0\}$. In particular, $\Lip(h_{N_2}) = \frac{1}{N_2 r} < \frac{\varepsilon}{6}$.

We then let $N_4 = N_2(2 t+1)  \in\N$, so that $\alpha^j(h_{N_4}) h_{N_2} = h_{N_2}$ for all $|j| \leq N_2$. As a consequence, we note that
\begin{equation*}
	h_{N_4} u^j h_{N_2} - u^j h_{N_2}
	= u^j \alpha^j(h_{N_4}) h_{N_2}  
	= u^j h_{N_2} \text.
\end{equation*}
Hence, 
\begin{equation}\label{hn1hn3-commute-eq}
	\forall a \in \beta^\psi(\A) \quad h_{N_4} a h_{N_2} = a h_{N_2} \text.
\end{equation}
Also, we record that
\begin{equation*}
	\Lip(h_{N_4}) = \frac{1}{r N_4} \leq \frac{1}{2 N_2 t + N_2} \leq \frac{1}{2N_2} < \frac{\varepsilon}{24} \text.
\end{equation*}

\medskip

We set
\begin{equation*}
	N_5 \coloneqq 2 ( N_4(t+1) + N_2 t ) \text,
\end{equation*}
and note that by construction, $N_5 \geq \max\{N_1,N_2,N_3,N_4\}$.

Let us henceforth assume that $p > N_5$. Let $\xi \coloneqq (\xi_n)_{n\in\Z}$. By construction, 
\begin{equation*}
	h_{N_4} \xi : n \in \Z \mapsto 
	\begin{cases} 
		\xi_n \text{ if $|n| < N_4$, }\\
		\left(1-\frac{\varepsilon(n - N_4)}{N_4 t}\right)\xi_n \text{ if $N_4\leq |n|\leq N_4(t\varepsilon^{-1}+1)$,} \\
		0 \text{ otherwise. }
	\end{cases}
\end{equation*}
In particular, if $P$ is the projection in $\ell^2(\Z)$ on the space $\{ (\xi_n)_{n\in\Z} \in \ell^2(\Z) : \forall n \in \Z \quad |n|\geq N_5 \implies \xi_n = 0\}$, then by construction, since $p > N_5$, we note that $U_p^j P = U^j P$ for all $j\in\Z, |j|\leq N_2 t$, since 

Moreover $P h_{N_4} = h_{N_4}$, so
\begin{equation}\label{Uhn4-eq1}
	U^j h_{N_4} = U^j (P h_{N_4}) = (U_p^j P) h_{N_4} = U_p^j h_{N_4} 
\end{equation}
for all $j \in \Z$, $|j| \leq N_2 t$.

Furthermore, since $h_{N_4}$ is self-adjoint, it follows that
\begin{equation}\label{Uhn4-eq2}
	h_{N_4} U^j = (U^{-j} h_{N_4})^\ast = (U_p^{-j} h_{N_4})^\ast = h_{N_4} U_p^j 
\end{equation}
for all $j \in \Z$, $|j| \leq N_2$. (Note that $h_{N_4}$ does not commute with $U$ or $U_p$ though the commutator with either is the same).

In particular,
\begin{equation}\label{Uhn4-eq3}
	[h_{N_4}, U-U_p] = 0 \text. 
\end{equation}

Last, since $\frac{d\exp(2i\pi x j)}{dx}_{|x=0} = 2i \pi j$ and $\lim_{x\rightarrow 0}\exp(2i\pi x) = 1$, let $N_6 \in \N$ such that, for all $p \geq N_6$, and for all $j \in \Z$, $|j| \leq N_5$ and $|n|\leq N_2$, 
\begin{equation*}
	\left| \frac{p\exp\left(\frac{2i\pi n}{p}\right)}{2i\pi} \left(\exp\left(\frac{2i j \pi}{p}\right)-1\right) - j \right| < \frac{\varepsilon}{12(2N_2 + 1)} \text. 
\end{equation*}

We shall henceforth assume that $p\geq N_7 \coloneqq \max\{ N_5, N_6 \}$. Since $u = U^t$ and $u_p = U_p^t$, It follows from Exp. \eqref{Uhn4-eq1}, \eqref{Uhn4-eq2}, \eqref{Uhn4-eq3} that:
\begin{align*}
	\partial_\Z( h_{N_4} f u^j h_{N_4} ) &- \partial_{\bigslant{\Z}{p}}( h_{N_4} f \underbracket[1pt]{u_p^j h_{N_4}}_{=u^p h_{N_4}} ) \\ 
	&= [h_{N_4} f u^j h_{N_4}, U - U_p] \\
	&= h_{N_4} f u^j \underbracket[1pt]{[h_{N_4}, U-U_p]}_{=0 \text{ by \eqref{Uhn4-eq3}}} + [h_{N_4} f u^j, U-U_p] h_{N_4} \\
	&=\underbracket[1pt]{[h_{N_4} f, U-U_p]u^j}_{\text{since }[U,u_j] = 0} \\
	&=\underbracket[1pt]{[h_{N_4},U-U_p]}_{=0 \text{ by \eqref{Uhn4-eq3}}} f + h_{N_4}[f,U-U_p] \\
	&= f \underbracket[1pt]{h_{N_4}(U-U_p)}_{=0 \text{ by \eqref{Uhn4-eq1}}} + \underbracket[1pt]{h_{N_4}(U-U_p)}_{=0 \text{ by \eqref{Uhn4-eq2}}} f \\
	&=0\text.
\end{align*}

Moreover,
\begin{align*}
	(\partial_\T(u^j h_{N_4}) &- \partial_{\Z_n}(u_p^j h_{N_4}) (\xi_{n\in\Z}) \\
	&= \left(j u^j \xi(n) - \left[ \left(\frac{p}{2i\pi}\exp\left(\frac{2i\pi n}{p}\right)-1\right) -  \left(\frac{p}{2i\pi}\exp\left(\frac{2i\pi n - j t}{p}\right)-1\right) \right] (h_{N_4}\xi)_{n-jt}\right)_{n\in\Z} \\
	&= \left( \left( j  - \frac{p}{2i\pi}\left(\exp\left(\frac{2i\pi n}{p}\right) -  \exp\left(\frac{2i\pi n - j t}{p}\right) \right) \right) (h_{N_4}\xi)_{n - t j} \right)_{n\in\Z}\\
	&= \left( \left( j - \frac{p\exp\left(\frac{2i\pi n}{p}\right)}{2i\pi} \left(\exp\left(-\frac{2 i \pi j t}{p}\right) - 1 \right) \right) (h_{N_4}\xi)_{n-j t}\right)_{n\in\Z}\text. 
\end{align*}
So, for all $\xi = (\xi_n)_{n\in\Z} \in \ell^2(\Z)$ with $\norm{\xi}{\ell^2(\Z)} \leq 1$, we compute:
\begin{align*}
	&\norm{(\partial_\T(u^j h_{N_4}) - \partial_{\Z_n}(u_p^j h_{N_4}) \xi}{\ell^2(\Z)}^2  \\
	&\quad = \sum_{j \in \Z} \left| j - \frac{p\exp\left(\frac{2i\pi n}{p}\right)}{2i\pi} \left(\exp\left(-\frac{2 i \pi j t}{p}\right) - 1 \right) \right|^2 |(h_{N_4}(n-j t) \xi_{n-j t}|^2 \\
	&\quad = \sum_{j = -N_4 + jt}^{N_4 + jt} \left| j - \frac{p\exp\left(\frac{2i\pi n}{p}\right)}{2i\pi} \left(\exp\left(-\frac{2 i \pi j t}{p}\right) - 1 \right) \right|^2 |h_{N_4}(n-jt)\xi_{n-j t}|^2 \\
	&\quad \leq \frac{\varepsilon^2}{ 144(2N_2 + 1)^2 } \sum_{j=-N_5}^{N_5}  |h_{N_4}(n-jt)\xi_{n-jt}|^2 \\
	&\quad \leq \frac{\varepsilon^2}{144(2N_2 + 1)^2} \text.
\end{align*}
Therefore, 
\begin{equation*}
	\norm{\partial_\Z(u^j h_{N_4}) - \partial_{\U_n}(u_p^j h_{N_4})}{\B} \leq \frac{\varepsilon}{12 (2N_2 + 1)} \text.
\end{equation*}

Therefore, if $a \in \beta^\psi(\A) = \beta^\psi(\A_p)$ for any $p \geq N_6$, we note that 
\begin{equation}\label{lip-cont-eq}
	|\Lip_p(a) - \Lip(a)| < \frac{\varepsilon}{12}\Lip(a) \text.
\end{equation}

We now fix $p \geq N_6$. Let $\D_p \coloneqq \A \oplus \A_p$, $\alg{T}_p \coloneqq \M \oplus \M_p$, and $k_p \coloneqq (h_{N_1}, h_{N_1})$. For any $(d_1,d_2) \in \dom{\Lip}\oplus\dom{\Lip_p}$, we define:
\begin{equation*}
	\Lip[T]_p (d_1,d_2) \coloneqq \max\left\{ \Lip(d_1), \Lip_p(d_2), \frac{3}{\varepsilon} \norm{  h_{N_1}^2 d_1  - d_2 h_{N_1}^2 }{\B} \right\} \text.
\end{equation*}
We also define $\varpi_p : (d_1,d_2) \in \D \mapsto d_1 \in \A$ and $\varrho_p : (d_1,d_2) \in \D \mapsto d_2 \in \A_p$. 

We shall now prove that $\tau_p \coloneqq (\D_p, \Lip[T]_p, \alg{T}_p, k_p, \varpi_p, \varrho_p)$ is a tunnel of extent at most $\varepsilon$.

\medskip

It is immediate that $(\D_p,\Lip[T]_p)$ is a {\pqms}. First, the unit ball of $\Lip[T]$ is trivially closed. Moreover, if $y \coloneqq (h,1)$, then
\begin{multline*}
	\left\{ (y (a,b) y) : \Lip[T]_p(a,b) \leq 1, \norm{(a,b)}{\D_p} \leq 1\right\} 
	\subseteq \{ h a h : \Lip(a) \leq 1, \norm{a}{\A} \leq 1 \} \times \{ b : \Lip_p(b) \leq 1, \norm{b}{\A_p} \leq 1 \}
\end{multline*}
and the right hand side is compact as the product of two compact sets, so the left hand side is compact as well (as it is closed).

By construction, $\Lip[T]_p$ is Leibniz.

Let $(a,b) \in \D_p$ with $\Lip[T]_p(a,b) \leq 1$ and $\mu(a) = 0$. We already know that $\{ h a h : a\in\dom{\Lip},\Lip(a)\leq 1, \mu(a) = 0\}$ is bounded. 

Moreover, since $\mu(h_{N_2}^2) = 1$, we note that 
\begin{equation*}
	|\mu(b)| = |\mu(h_{N_2} b)| \leq |\mu(h_{N_2} b - a h_{N_2})| + |\mu(a h_{N_2})| \leq \frac{\varepsilon}{3} + 0 \text.
\end{equation*}
Since $(\A_n,\Lip_n)$ is a {\qcms}, the set $\{ b \in \dom{\Lip_n} : |\mu(b)| \leq \frac{\varepsilon}{3}, \Lip_n(b) \leq 1 \}$ is compact, hence bounded. So 
\begin{equation*}
	\left\{ y (a,b) y : (a,b)\in\dom{\Lip[T]_p}, \Lip[T]_p(a,b) \leq 1, \mu(a) = 0\right\}
\end{equation*}
is bounded. 

\medskip

Fix $r \in \left[1, \frac{1}{\varepsilon}\right]$.

\medskip

Now, let $a \in \A$ with $\norm{a}{\Lip,r} \leq 1$. As seen above, $\norm{a - \beta^{\psi}(a)}{\B} < \frac{\varepsilon}{12}\Lip(a)$ while $\norm{\beta^\psi(a)}{\B} \leq r$ and $\Lip_\Z(\beta^\psi(a)) \leq \Lip(a)$. Then
\begin{equation*}
	\Lip( h_{N_4} \beta^\psi(a) h_{N_4} ) \leq \norm{h_{N_4}}{\B}\left( 2\Lip(h_{N_4}) \norm{\beta^\psi(a)}{\B} + \norm{h_{N_4}}{\B} \Lip(\beta^\psi(a))\right) \leq \frac{\varepsilon}{12} + 1 \text.
\end{equation*}

Let now $a \in \beta^\psi(\A)$, so that $a = \sum_{j \in S} f_j u^j$ for some $(f_j)_{j \in S} \in c_0(\Z)^S$. We then estimate: 
\begin{align}\label{hn1-commute-eq}
	\norm{[h_{N_2}^2, a]}{\B} 
	&\leq \sum_{j \in S\setminus\{0\}} \norm{f_j [h_{N_2}^2,u^j]}{\B} \text{ since $[f_j,h_{N_1}] = 0$, } \nonumber \\
	&\leq \sum_{j \in S} \norm{f_j}{\B} \norm{[h_{N_2}^2,u^j]}{\B} \\
	&\leq \sum_{j \in S} \frac{1}{N_2^2} 1 = \frac{1}{N_2} \leq \frac{\varepsilon}{12}  \nonumber \text.
\end{align}

Moroever, by Expression \eqref{lip-cont-eq}, we also have
\begin{align*}
	\norm{h_{N_2}^2 a - h_{N_4} \beta^\psi(a) h_{N_4} h_{N_2}^2}{\B}
	&\leq \norm{h_{N_2}^2 (a - \beta^\psi(a)) }{\B} + \norm{ h_{N_2}^2 \beta^\psi(a) - h_{N_4} \beta^\psi(a)) h_{N_2}^2 }{\B}   \\
	&\leq \frac{\varepsilon}{12} + \norm{ h_{N_2}^2 \beta^\psi(a) - \beta^\psi(a) h_{N_2}^2}{\B} \text{ by Exp. \eqref{hn1hn3-commute-eq}, } \\
	&\leq \frac{\varepsilon}{12} + \norm{ [h_{N_2}^2,\beta^\psi(a)]}{\B} \\
	&\leq \frac{\varepsilon}{12} + \frac{\varepsilon}{12} = \frac{\varepsilon}{6} \text{ by Exp. \eqref{hn1-commute-eq} .}
\end{align*}

It follows that:
\begin{align*}
	\norm{a h_{N_2} - h_{N_2} \frac{6}{6+\varepsilon} h_{N_4}\beta^\psi(a) h_{N_4} }{\B}
	&\leq \norm{a h_{N_2} - \frac{6}{6+\varepsilon} a h_{N_2}}{\B} \\
	&\quad + \frac{6}{6+\varepsilon} \norm{a h_{N_2} - h_{N_2}  h_{N_4}\beta^\psi(a) h_{N_4} }{\B} \\
	&\leq \frac{\varepsilon}{6} + \frac{\varepsilon}{6} = \frac{\varepsilon}{3} \text.
\end{align*}

Therefore, we have shown that:
\begin{align*}
	\Lip[T]_p\left( a, \frac{6}{6+\varepsilon} h_{N_4} \beta^\psi(a) h_{N_4} \right) \leq 1 \text,
\end{align*}
while of course $\norm{(a,\frac{6}{6+\varepsilon} h_{N_4} \beta^\psi(a) h_{N_4})}{\D_p} \leq 1$.

This proves that $\varpi_p$ is a quantum $M$-isometry. Now, let $a \in c_0(\Z)$. We note that the restriction $b$ of $a$ to $I_p$ is an element of $c\left(\bigslant{\Z}{p}\right)$ with $\Lip_\Z(b) = \Lip_\Z(a)$ and $\Lip_\T(b) = 0 = \Lip_\T(a)$. So $\varpi_p$ is a topographic quantum $M$-isometry.

\medskip

If $a \in \dom{\Lip_p}$ with $\norm{a}{\Lip_p} \leq 1$ and $p \geq N_6$, then $\norm{h_{N_3} \beta_p^\psi(a) h_{N_3}}{\Lip_p} \leq 1 + \varepsilon$. Furthermore, $\norm{h_{N_3} \beta_p^\psi(a) h_{N_3}}{\Lip} \leq 1 + 2\varepsilon$.

We also note that
\begin{align*}
	\norm{ h_{N_1}^2 h_{N_3} \beta_p^\psi(a) h_{N_3} - a h_{N_1}^2 }{\B}
	&= \norm{h_{N_1}^2 \beta^\psi(a) h_{N_3} - a h_{N_1}^2}{\B} \\
	&= \norm{\beta^\psi(a) h_{N_1}^2 - a h_{N_1}^2}{\B} \\
	&=  \norm{\beta^\psi(a) - a}{\B} \norm{h_{N_1}^2}{\B} \\
	&\leq \varepsilon \text.
\end{align*}

Again, we obtain that $\Lip[T]_p(h_{N_3} \beta^\psi(a) h_{N_3}, a) \leq 1$. If $a\in c\left(\bigslant{\Z}{p}\right)$, then we identify it with a function over $I_p$. There exists an element $b \in c_0(\Z)$ with $\Lip_\Z(b) = \Lip_\Z(a)$, $\Lip_\T(b) = 0$, and which restricts to $a$ on $I_p$.

Thus, $\varrho_p$ is also a topographic quantum $M$-isometry.  

Thus $\tau_p$ is indeed a tunnel.

\medskip

By construction, $\Lip[T]_p(k_p) = \frac{1}{N_1} < \varepsilon$. Moreover $\delta_0(h_{N_2}) = 1$.

\medskip

Let $\varphi \in \QuasiStateSpace(\A)$. Let $(d_1,d_2) \in \D_p$ with $\norm{(d_1,d_2)}{\Lip[T]_p} \leq 1$. Let $\psi \coloneqq \varphi(h_{N_2}\cdot h_{N_2}) \in \QuasiStateSpace(\A_p)$. 
\begin{align*}
	|\varphi(h_{N_2} d_1 h_{N_2}) - \psi(h_{N_2} d_2 h_{N_2})|
	&=|\varphi(h_{N_2} d_1 h_{N_2}) - h_{N_2} d_2 h_{N_2})| \\
	&\leq \norm{h_{N_2} (d_1 - d_2) h_{N_2}}{\B} \\
	&\leq \norm{h_{N_2}(d_1 - \beta^\psi(d_1)) h_{N_2}}{\B} + \norm{h_{N_2}(d_2 - \beta^\psi(d_2)) h_{N_2}}{\B}\\
	&\quad + \norm{h_{N_2}(\beta^\psi(d_1) - \beta^\psi(d_2)) h_{N_2}}{\B} \\
	&\leq 2\frac{\varepsilon}{12} + 2\frac{\varepsilon}{12} + \norm{h_{N_2}^2 d_1 - d_2 h_{N_2}^2}{\B} \\
	&\leq \frac{\varepsilon}{6} + \frac{\varepsilon}{6} + \frac{\varepsilon}{3} < \varepsilon \text.
\end{align*}
Hence
\begin{equation*}
	\left\{\varphi(h_{N_2} \cdot h_{N_2}) : \varphi \in \StateSpace(\A) \right\} \subseteq^{\varepsilon} \QuasiStateSpace(\B) \text.
\end{equation*}

A similar reasoning works if $\psi \in \QuasiStateSpace(\A_n)$,  by setting $\varphi:a \in \A\mapsto \psi(h_{N_4} a h_{N_4})$, to show that
\begin{equation*}
	\left\{\psi(h_{N_2} \cdot h_{N_2}) : \psi \in \StateSpace(\B) \right\} \subseteq^{\varepsilon} \QuasiStateSpace(\A) \text.
\end{equation*}

Therefore, the extent of $\tau_p$, as an $r$-tunnel, is at less than $\varepsilon$. This completes our proof, so that we can conclude that if $p\geq N_6$, for all $r \in \left[1,\frac{1}{\varepsilon}\right]$, we have:
\begin{equation*}
	\metametric{r}\left((c_0(\Z)\rtimes_\alpha\Z, \Lip, c_0(\Z),\mu),  \left( c\left(\bigslant{Z}{p}\right) \rtimes_{\alpha_p} \bigslant{\Z}{p}, \Lip_p, c\left(\bigslant{\Z}{p}\right), \mu_p\right) \right) < \varepsilon \text.
\end{equation*}

Therefore:
\begin{equation*}
	\lim_{p\rightarrow\infty} \Eth\left((c_0(\Z)\rtimes_\alpha\Z, \Lip, c_0(\Z),\mu),  \left( c\left(\bigslant{Z}{p}\right) \rtimes_{\alpha_p} \bigslant{\Z}{p}, \Lip_p, c\left(\bigslant{\Z}{p}\right), \mu_p\right) \right) = 0 \text.
\end{equation*}

\bibliographystyle{amsplain} \bibliography{./thesis}

\vfill
\end{document}